\documentclass[reqno]{amsart}
\usepackage{mathrsfs}
\usepackage{color}
\usepackage{amsmath}
\usepackage{amsfonts}
\usepackage{amssymb}
\usepackage{graphicx}

\usepackage{lineno}
\usepackage{hyperref}

 \newtheorem{Theorem}{Theorem}[section]
 
 \newtheorem{Lemma}[Theorem]{Lemma}
 \newtheorem{Proposition}[Theorem]{Proposition}

\newtheorem{Problem}[Theorem]{Problem}

 \newtheorem{Remark}[Theorem]{Remark}

 \numberwithin{equation}{section}

\begin{document}


\title[concavity property of minimal $L^{2}$ integrals]
 {concavity property of minimal $L^{2}$ integrals with Lebesgue measurable gain \uppercase\expandafter{\romannumeral7}---Negligible weights}

\author{Shijie Bao}
	\address{Shijie Bao: School of
		Mathematical Sciences, Peking University, Beijing 100871, China.}
	\email{bsjie@pku.edu.cn}
	
	\author{Qi'an Guan}
	\address{Qi'an Guan: School of
		Mathematical Sciences, Peking University, Beijing 100871, China.}
	\email{guanqian@math.pku.edu.cn}

\author{Zhitong Mi}
\address{Zhitong Mi: Institute of Mathematics, Academy of Mathematics
and Systems Science, Chinese Academy of Sciences, Beijing, China
}
\email{zhitongmi@amss.ac.cn}
	
	\author{Zheng Yuan}
	\address{Zheng Yuan: School of
		Mathematical Sciences, Peking University, Beijing 100871, China.}
	\email{zyuan@pku.edu.cn}

\thanks{}



\date{\today}

\dedicatory{In Memory of Jean-Pierre Demailly (1957-2022)}

\commby{}


\begin{abstract}
In this article, we present characterizations of the concavity property of minimal $L^2$ integrals with negligible weights degenerating to linearity on the fibrations over open Riemann surfaces and the fibrations over products of open Riemann surfaces.
As applications, we obtain characterizations of the holding of equality in optimal jets $L^2$ extension problem
with negligible weights on the fibrations over open Riemann surfaces and the fibrations over products of open Riemann surfaces.
\end{abstract}

\maketitle

\section{introduction}

Recall that the strong openness property of multiplier ideal sheaves \cite{GZSOC}, i.e. $\mathcal{I}(\varphi)=\mathcal{I}_+(\varphi):=\mathop{\cup} \limits_{\epsilon>0}\mathcal{I}((1+\epsilon)\varphi)$
	(conjectured by Demailly \cite{DemaillySoc}) has been widely used and discussed
	in several complex variables, complex algebraic geometry and complex differential geometry
	(see e.g. \cite{GZSOC,K16,cao17,cdM17,FoW18,DEL18,ZZ2018,GZ20,berndtsson20,ZZ2019,ZhouZhu20siu's,FoW20,KS20,DEL21}),
	where multiplier ideal sheaf $\mathcal{I}(\varphi)$ is the sheaf of germs of holomorphic functions $f$ such that $|f|^2e^{-\varphi}$ is locally integrable (see e.g. \cite{Tian,Nadel,Siu96,DEL,DK01,DemaillySoc,DP03,Lazarsfeld,Siu05,Siu09,DemaillyAG,Guenancia}),
and $\varphi$ is a plurisubharmonic function on a complex manifold $M$ (see \cite{Demaillybook}).

	When $\mathcal{I}(\varphi)=\mathcal{O}$, the strong openness property is the openness property (conjectured by Demailly-Koll\'ar \cite{DK01}).
	Berndtsson \cite{Berndtsson2} established an effectiveness result of the openness property, and obtained the openness property.
	Stimulated by Berndtsson's effectiveness result, and continuing the solution of the strong openness property \cite{GZSOC},
	Guan-Zhou \cite{GZeff} established an effectiveness result of the strong openness property by considering the minimal $L^{2}$ integral on the pseudoconvex domain $D$.

	Considering the minimal $L^{2}$ integrals on the sub-level sets of the weight $\varphi$,
	Guan \cite{G16} obtained a sharp version of Guan-Zhou's effectiveness result,
	and established a concavity property of the minimal $L^2$ integrals on the sublevel sets of the weight $\varphi$ (with constant gain),
	which was applied to give a proof of Saitoh's conjecture for conjugate Hardy $H^2$ kernels \cite{Guan2019},
	and the sufficient and necessary condition of the existence of decreasing equisingular approximations with analytic singularities for the multiplier ideal sheaves with weights $\log(|z_{1}|^{a_{1}}+\cdots+|z_{n}|^{a_{n}})$ \cite{guan-20}.
	
	For smooth gain, Guan \cite{G2018} (see also \cite{GM}) obtained the concavity property on Stein manifolds (weakly pseudoconvex K\"{a}hler case was obtained by Guan-Mi \cite{GM_Sci}),
which was applied by Guan-Yuan to give an optimal support function related to the strong openness property \cite{GY-support} and an effectiveness result of the strong openness property in $L^p$ \cite{GY-lp-effe}.
	For Lebesgue measurable gain, Guan-Yuan \cite{GY-concavity} obtained the concavity property on Stein manifolds (weakly pseudoconvex K\"{a}hler case was obtained by Guan-Mi-Yuan \cite{GMY}),
	which deduced a twisted $L^p$ strong openness property \cite{GY-twisted}.
	
	Note that the linearity is a degenerate concavity. A natural problem was posed in \cite{GY-concavity3}:
	
	\begin{Problem}[\cite{GY-concavity3}]\label{Q:chara}
		How to characterize the concavity property degenerating to linearity?
	\end{Problem}
For open Riemann surfaces, Guan-Yuan gave an answer to Problem \ref{Q:chara} for single points \cite{GY-concavity} (for the case of subharmonic weights, see Guan-Mi \cite{GM}),
and gave an answer to Problem \ref{Q:chara} for finite points \cite{GY-concavity3}.
For products of open Riemann surfaces, Guan-Yuan \cite{GY-concavity4} gave an answer to Problem \ref{Q:chara} for products of finite points.

For fibrations over open Riemann surfaces, Bao-Guan-Yuan \cite{BGY-concavity5} gave an answer to Problem \ref{Q:chara} with negligible weights pulled back from the open Riemann surfaces.
For fibrations over products of open Riemann surfaces, Bao-Guan-Yuan \cite{BGY-concavity6} gave an answer to Problem \ref{Q:chara} with negligible weights vanishing identically.

In this article, for the fibrations over open Riemann surfaces and the fibrations over products of open Riemann surfaces,
we give answers to Problem \ref{Q:chara} with negligible weights on fibrations.

We would like to recall the definition of minimal $L^2$ integral as follows.

Let $\Omega_j$  be an open Riemann surface, which admits a nontrivial Green function $G_{\Omega_j}$ for any  $1\le j\le n_1$. Let $Y$ be an $n_2-$dimensional weakly pseudoconvex K\"ahler manifold, and let $K_Y$ be the canonical (holomorphic) line bundle on $Y$. Let $M=\left(\prod_{1\le j\le n_1}\Omega_j\right)\times Y$ be an $n-$dimensional complex manifold, where $n=n_1+n_2$. Let $\pi_{1}$, $\pi_{1,j}$ and $\pi_2$ be the natural projections from $M$ to $\prod_{1\le j\le n_1}\Omega_j$, $\Omega_j$ and $Y$ respectively. Let $K_M$ be the canonical (holomorphic) line bundle on $M$.

 Let $Z_j$ be a (closed) analytic subset of $\Omega_j$ for any $j\in\{1,\ldots,n_1\}$, and denote that $Z_0:=\left(\prod_{1\le j\le n_1}Z_j\right)\times Y\subset M$.

Let $\psi< 0$ be a plurisubharmonic function on $M$ such that $\{\psi<-t\}\backslash Z_0$ is a weakly pseudoconvex K\"ahler manifold for any $t\in\mathbb{R}$ and $Z_{0}\subset \{\psi=-\infty\}$.
Let $\varphi_1$ be a Lebesgue measurable function on $\left(\prod_{1\le j\le n_1}\Omega_j\right)$ such that $\pi_1^*(\varphi_1)+\psi$ is a plurisubharmonic function on $M$.
Let $\varphi_2$ be a plurisubharmonic funciton on $Y$. Denote $\varphi:=\pi_1^*(\varphi_1)+\pi_1^*(\varphi_2)$.  Let $\mathcal{F}_{(z,y)}\supset \mathcal{I}(\varphi+\psi)_{(z,y)}$ be an ideal of $\mathcal{O}_{(z,y)}$ for any $(z,y)\in Z_0$. Let $f$ be a holomorphic $(n,0)$ form on a neighborhood $U$ of $Z_0$. Let $c(t)$ be a positive Lebesgue measurable function on $(0,+\infty)$. Denote
\begin{flalign*}
		\begin{split}
			\inf\bigg\{\int_{\{\psi<-t\}}|\tilde{f}|^2e^{-\varphi}c(-\psi) :& \big(\tilde{f}-f,(z,y)\big)\in \big(\mathcal{O}(K_M)\otimes\mathcal{F}\big)_{(z,y)} \\
			&\text{\ for \ any } (z,y)\in Z_0,  \\
			& \& \ \tilde{f}\in H^0\big(\{\psi<-t\},\mathcal{O}(K_M)\big).\bigg\}
		\end{split}
	\end{flalign*}
	by $G(t;c,f,\varphi,\psi,\mathcal{F})$ for any $t\in [0,+\infty)$. Here $|f|^2:=(\sqrt{-1})^{n^2}f\wedge\bar{f}$ for any $(n,0)$ form $f$.  We simply denote $G(t;c,f,\varphi,\psi,\mathcal{F})$ by $G(t)$ when there is no misunderstands and denote $G(t;c,\varphi,\psi,\mathcal{F})$ by $G(t;c)$, $G(t;\varphi)$, $G(t;\psi)$ and $G(t;\mathcal{F})$ when we focus on various choices of $c(t),\varphi,\psi$ and $\mathcal{F}$ respectively.

We generally assume that $c(t)$ is a positive function on $(0,+\infty)$ such that $\int_{0}^{+\infty}c(t)e^{-t}dt<+\infty$, $c(t)e^{-t}$ is decreasing with respect to $t$ on $(0,+\infty)$  and $e^{-\varphi}c(-\psi)$ has a positive lower bound on any compact subset of $M\backslash Z_0$ in this paper (when other assumption
for $c(t)$ is used, we introduce it explicitly).
Then $G\big(h^{-1}(r)\big)$ is concave with respect to $r\in[0,\int_0^{+\infty}c(t_1)e^{-t_1}dt_1]$ (see \cite{GMY}, see also Theorem \ref{Concave}), where $h(t)$ $=\int_t^{+\infty}c(s)e^{-s}ds$ for any $t\in [0,+\infty)$.

\subsection{Main results}

In this section, we present characterizations of the concavity property of minimal $L^2$ integrals with negligible weights degenerating to linearity on the fibrations over open Riemann surfaces and products of open Riemann surfaces.

\subsubsection{Linearity of the minimal $L^2$ integrals on fibrations over open Riemann surfaces}
In this section, we present characterizations of the concavity property of minimal $L^2$ integrals degenerating to linearity on the fibrations over open Riemann surface.

To state our results, we firstly recall the following notations (see \cite{OF81}, see also \cite{guan-zhou13ap,GY-concavity,GMY}).

Let $\Omega$  be an open Riemann surface, which admits a nontrivial Green function $G_{\Omega}$.
A character $\chi$ on $\pi_1(\Omega)$ is a homomorphism from $\pi_1(\Omega)$ to $\mathbb{C}^*=\mathbb{C}\backslash{\{0\}}$ which takes values in the unit circle $\{z\in\mathbb{C};|z|=1\}$.
\par
Let $P:\Delta\to \Omega$ be the universal covering from unit disc $\Delta\subset \mathbb{C}$ to $\Omega$.
We call the holomorphic function $f$ (resp. holomorphic $(1,0)$ form $F$) on $\Delta$ is a multiplicative function \big(resp. multiplicative differential (Prym differential)\big) if there is a character $\chi$ on $\pi_1(\Omega)$, such that $g^*f=\chi(g)f$ (resp. $g^*F=\chi(g)F$) for every $g\in \pi_1(\Omega)$ which naturally acts on the universal covering of $\Omega$. Denote the set of such kinds of $f$ (resp. F) by $\mathcal{O}^{\chi}(\Omega)$ \big(resp. $\Gamma^{\chi}(\Omega)$\big).
\par
As $P$ is a universal covering, then for any harmonic function $h$ on $\Omega$, there exists a character $\chi_h$ associated to $h$ and a multiplicative function $f_h \in \mathcal{O}^{\chi_h}(\Omega)$, such that $|f_h|=P^*e^{h}$. And if $g \in \mathcal{O} (\Omega)$ and $g$ has no zero points on $\Omega$, then we have $\chi_h=\chi_{h+\log|g|}$.

For Green function $G_{\Omega}(\cdot,z_0)$, one can find a $\chi_{z_0}$ and a multiplicative function $f_{z_0} \in \mathcal{O}^{\chi_{z_0}}(\Omega)$, such that $|f_{z_0}|=P^*e^{G_{\Omega}(\cdot,z_0)}$ (see \cite{suita72}).

Now we assume that $n_1=1$ and then $M=\left(\prod_{1\le j\le n_1}\Omega_j\right)\times Y=\Omega\times Y$, where $Y$ is an $(n-1)-$dimensional weakly pseudoconvex K\"ahler manifold.
Denote $Z_{\Omega}=\{z_j: 1\le j< \gamma\}$ be a subset of $\Omega$ of discrete points, where $\gamma>1$ is a positive integer or $\gamma=+\infty$. Denote $Z_0:=Z_{\Omega}\times Y$. Denote $Z_j:=\{z_j\}\times Y$ for any $j$.

Let $\psi$ be a plurisubharmonic function on $M$ such that $\{\psi<-t\}\backslash Z_0$ is a weakly pseudoconvex K\"ahler manifold for any $t\in\mathbb{R}$ and $Z_0\subset \{\psi=-\infty\}$.
It follows from Siu's decomposition theorem that $$dd^{c}\psi=\sum\limits_{j\ge1}2p_j[Z_j]+\sum\limits_{i\ge 1}\lambda_i[A_i]+R,$$
where $[Z_j]$ and $[A_i]$ are the currents of integration over an irreducible $(n-1)-$dimensional analytic set, and where $R$ is a closed positive current with the property that $dimE_c(R)<n-1$ for every $c>0$, where $E_c(R)=\{x\in M: v(R,x)\ge c\}$ is the upperlevel sets of Lelong number. We assume that $p_j>0$ for any $1\le j< \gamma$.

Then $N:=\psi-\pi^{*}_1\big(\sum\limits_{j\ge1}2p_jG_{\Omega}(z,z_j)\big)$ is a plurisubharmonic function on $M$, where $\pi_1:M\to \Omega$ be the natural projection. We assume that $N\le0$ and $N|_{Z_j}$ is not identically $-\infty$ for any $j$.

Let $\varphi_1$ be a Lebesgue measurable function on $\Omega$ such that $\psi+\pi^{*}_1(\varphi)$ is a plurisubharmonic function on $M$. By Siu's decomposition theorem, we have
$$dd^{c}(\psi+\pi^{*}_1(\varphi))=\sum\limits_{j\ge1}2\tilde{q}_j[Z_j]+\sum\limits_{i\ge 1}\tilde{\lambda}_i[\tilde{A}_i]+\tilde{R},$$
where $\tilde{q}_j\ge 0$ for any $1\le j< \gamma$.

By Weierstrass theorem on open Riemann surfaces, there exists a holomorphic function $g$ on $\Omega$ such that $ord_{z_j}(g)=q_j:=[\tilde{q}_j]$ for any $z_j\in Z_{\Omega}$ and $g(z)\neq 0$ for any $z\notin Z_{\Omega}$, where $[q]$ equals to the integral part of the nonnegative real number $q$.
Then we know that there exists a plurisubharmonic function $\tilde{\psi}_2\in Psh(M)$ such that $$\psi+\pi^{*}_1(\varphi_1)=\pi^{*}_1(2\log|g|)+\tilde{\psi}_2.$$

Let $\varphi_2\in Psh(Y)$. Denote $\pi_2:M\to Y$ be the natural projection and $\varphi:=\pi^{*}_1(\varphi_1)+\pi^{*}_2(\varphi_2)$.

For any $1\le j<\gamma$, let $\tilde{z}_j$ be a local coordinate on a neighborhood $V_{j}$ of $z_j$ satisfying $\tilde{z}_j(z_0)=0$ and $V_{j}\cap V_{k}=\emptyset$, for any $j\neq k$. Denote $V_0:=\cup_{1\le j < \gamma}V_j$. Let $f$ be a holomorphic $(n,0)$ form on $V_{0}\times Y$ which is a neighborhood of $Z_0$. Denote $\mathcal{F}_{(z_j,y)}=\mathcal{I}(\pi_1^{*}(2\log|g|)+\pi_2^*(\varphi_2))_{(z_j,y)}$ for any $(z_j,y)\in Z_0$. Let $G(t)$ be the minimal $L^2$ integral on $\{\psi<-t\}$ with respect to $\varphi$, $f$, $\mathcal{F}$ and $c$ for any $t\geq 0$.

Let $\gamma=m+1$ is an positive integer, i.e., $Z_{\Omega}=\{z_j: 1\le j< \gamma\}$ contains $m$ points. We obtain the following characterization of the concavity of $G\big(h^{-1}(r)\big)$ degenerating to linearity on the fibrations over open Riemann surfaces.

\begin{Theorem}\label{finite points}
		Assume that $G(0)\in (0,+\infty)$. Then $G\big(h^{-1}(r)\big)$ is linear with respect to $r\in (0,\int_0^{+\infty}c(t)e^{-t}dt]$ if and only if the following statements hold:
		
	(1). $N\equiv 0$ and  $\psi=\pi_1^*\left(2\sum\limits_{j=1}^mp_j G_{\Omega}(\cdot,z_j)\right)$;
		
		(2). for any $j\in\{1,2,\ldots,m\}$, $f=\pi_1^*(a_j\tilde{z}_j^{k_j}d\tilde{z}_j)\wedge \pi_2^*(f_{Y})+f_j$ on $V_{j}\times Y$, where $a_j\in\mathbb{C}\setminus \{0\}$ is a constant, $k_j$ is a nonnegative integer, $f_{Y}$ is a holomorphic $(n-1,0)$ form on $Y$ such that $\int_{Y}|f_{Y}|^2e^{-\varphi_2}\in (0,+\infty)$, and $\big(f_j,(z_j,y)\big)\in \big(\mathcal{O}(K_M)\big)_{(z_j,y)}\otimes\mathcal{I}(\varphi+\psi)_{(z_j,y)}$ for any $j\in\{1,2,\ldots,m\}$ and $y\in Y$;
		
		(3). $\varphi_1+2\sum\limits_{j=1}^mp_j G_{\Omega}(\cdot,z_j)=2\log |g|+2\sum\limits_{j=1}^mG_{\Omega}(\cdot,z_j)+2u$, where $g$ is a holomorphic function on $\Omega$ such that $ord_{z_j}(g)=k_j$ and $u$ is a harmonic function on $\Omega$;
		
		(4). $\prod\limits_{j=1}^m\chi_{z_j}=\chi_{-u}$, where $\chi_{-u}$ and $\chi_{z_j}$ are the characters associated to the functions $-u$ and $G_{\Omega}(\cdot,z_j)$ respectively;
		
		(5). for any $j\in\{1,2,\ldots,m\}$,
		\begin{equation}
			\lim_{z\rightarrow z_j}\frac{a_j\tilde{z}_j^{k_j}d\tilde{z}_j}{gP_*\left(f_u\left(\prod\limits_{l=1}^mf_{z_l}\right)\left(\sum\limits_{l=1}^mp_l\dfrac{d{f_{z_{l}}}}{f_{z_{l}}}\right)\right)}=c_0,
		\end{equation}
		where $c_0\in\mathbb{C}\setminus\{0\}$ is a constant independent of $j$, $f_{u}$ is a holomorphic function $\Delta$ such that $|f_{u}|=P^*(e^{u})$ and $f_{z_{l}}$ is a holomorphic function on $\Delta$ such that $|f_{z_{l}}|=P^*\left(e^{G_{\Omega}(\cdot,z_{l})}\right)$ for any $l\in\{1,\ldots,m\}$.
	\end{Theorem}

\begin{Remark}
When $N\equiv \pi_1^{*}(N_1)$, it follows from $\mathcal{I}(\varphi+\psi)_{(z_j,y)}=\mathcal{I}\big(\pi_1^{*}(2\log|g|)+\pi_2^*(\varphi_2)\big)_{(z_j,y)}$ (see Lemma \ref{equiv of multiplier ideal sheaf})
that Theorem \ref{finite points} reduces to Theorem 1.4 in \cite{BGY-concavity5} (see also Theorem \ref{BGY-RESULT-FINITE POINTS}),
where $N_1\le0$ is a subharmonic function on $\Omega$ and $N_1(z_j)>-\infty$ for any $j$.
\end{Remark}

Let $\gamma=+\infty$, i.e., $Z_{\Omega}=\{z_j: 1\le j< \gamma\}$ is an infinite subset of $\Omega$ of discrete points. Assume that $2\sum\limits_{j\ge 1}p_j G_{\Omega}(\cdot,z_j)\not\equiv -\infty$. We present a necessary condition such that $G(h^{-1}(r))$ is linear as follows.

\begin{Proposition}\label{infinite points}
		Assume that $G(0)\in (0,+\infty)$. If $G(h^{-1}(r))$ is linear with respect to $r\in (0,\int_0^{+\infty}c(t)e^{-t}dt]$, then the following statements hold:
		
		(1).  $N\equiv 0$ and  $\psi=\pi_1^*\left(2\sum\limits_{j=1}^{\gamma}p_j G_{\Omega}(\cdot,z_j)\right)$;
		
		(2). for any $j\in\mathbb{N}_+$, $f=\pi_1^*(a_j\tilde{z}_j^{k_j}d\tilde{z}_j)\wedge \pi_2^*(f_{Y})+f_j$ on $V_{j}\times Y$, where $a_j\in\mathbb{C}\setminus \{0\}$ is a constant, $k_j$ is a nonnegative integer, $f_{Y}$ is a holomorphic $(n-1,0)$ form on $Y$ such that $\int_{Y}|f_{Y}|^2e^{-\varphi_2}\in (0,+\infty)$, and $\big(f_j,(z_j,y)\big)\in \big(\mathcal{O}(K_M)\big)_{(z_j,y)}\otimes\mathcal{I}(\varphi+\psi)_{(z_j,y)}$ for any $j\in\mathbb{N}_+$ and $y\in Y$;
		
		(3). $\varphi_1+2\sum\limits_{j=1}^mp_j G_{\Omega}(\cdot,z_j)=2\log |g|$, where $g$ is a holomorphic function on $\Omega$ such that $ord_{z_j}(g)=k_j+1$ for any $j\in\mathbb{N}_+$;
		
		(4). for any $j\in\mathbb{N}_+$,
		\begin{equation}
			\frac{p_j}{ord_{z_j}g}\lim_{z\rightarrow z_j}\frac{dg}{a_j\tilde{z}_j^{k_j}d\tilde{z}_j}=c_0,
		\end{equation}
		where $c_0\in\mathbb{C}\setminus\{0\}$ is a constant independent of $j$;
		
		(5). $\sum\limits_{j\in\mathbb{N}_+}p_j<+\infty$.
	\end{Proposition}

\begin{Remark}When $N\equiv \pi_1^{*}(N_1)$, Proposition \ref{infinite points} is Proposition 1.6 in \cite{BGY-concavity5} (see also Theorem \ref{BGY-RESULT-INFINITE POINTS}), where $N_1\le0$ is a subharmonic function on $\Omega$ and $N_1(z_j)>-\infty$ for any $j$.
\end{Remark}

Let $\tilde{M}\subset M$ be an $n-$dimensional weakly pseudoconvex submanifold satisfying that $Z_0\subset\tilde{M}$. Let $f$ be a holomorphic $(n,0)$ form on a neighborhood $\tilde{U}_0$ of $Z_0$ in $\tilde{M}$.

Let $\psi$ be a plurisubharmonic function on $\tilde{M}$ such that $\{\psi<-t\}\backslash Z_0$ is a weakly pseudoconvex K\"ahler manifold for any $t\in\mathbb{R}$.
It follows from Siu's decomposition theorem that $$dd^{c}\psi=\sum\limits_{j\ge1}2p_j[Z_j]+\sum\limits_{i\ge 1}\lambda_i[A_i]+R,$$
where $[Z_j]$ and $[A_i]$ are the currents of integration over an irreducible $(n-1)-$dimensional analytic set, and where $R$ is a closed positive current with the property that $dimE_c(R)<n-1$ for every $c>0$, where $E_c(R)=\{x\in \tilde M: v(R,x)\ge c\}$ is the upperlevel sets of Lelong number. We assume that $p_j>0$ for any $1\le j< \gamma$.

Then $N:=\psi-\pi^{*}_1(\sum\limits_{j\ge1}2p_jG_{\Omega}(z,z_j))$ is a plurisubharmonic function on $\tilde M$. We assume that $N\le0$ and $N|_{Z_j}$ is not identically $-\infty$ for any $j$.

Let $\varphi_1$ be a Lebesgue measurable function on $\Omega$ such that $\psi+\pi^{*}_1(\varphi)$ is a plurisubharmonic function on $\tilde M$. By Siu's decomposition theorem, we have
$$dd^{c}(\psi+\pi^{*}_1(\varphi))=\sum\limits_{j\ge1}2\tilde{q}_j[Z_j]+\sum\limits_{i\ge 1}\tilde{\lambda}_i[\tilde{A}_i]+\tilde{R},$$
where $\tilde{q}_j\ge 0$ for any $1\le j< \gamma$.

By Weierstrass theorem on open Riemann surfaces, there exists a holomorphic function $g$ on $\Omega$ such that $ord_{z_j}(g)=q_j:=[\tilde{q}_j]$ for any $z_j\in Z_{\Omega}$ and $g(z)\neq 0$ for any $z\notin Z_{\Omega}$, where $[q]$ equals to the integral part of the nonnegative real number $q$.
Then we know that there exists a plurisubharmonic function $\tilde{\psi}_2\in Psh(\tilde M)$ such that $$\psi+\pi^{*}_1(\varphi_1)=\pi^{*}_1(2\log|g|)+\tilde{\psi}_2.$$

Let $\varphi_2\in Psh(Y)$. Denote $\varphi:=\pi^{*}_1(\varphi_1)+\pi^{*}_2(\varphi_2)$.

Let $c(t)$ be a positive function on $(0,+\infty)$ such that $\int_{0}^{+\infty}c(t)e^{-t}dt<+\infty$, $c(t)e^{-t}$ is decreasing with respect to $t$ on $(0,+\infty)$  and $e^{-\varphi}c(-\psi)$ has a positive lower bound on any compact subset of $\tilde{M}\backslash Z_0$.

Denote
\begin{flalign*}
		\begin{split}
			\inf\bigg\{\int_{\{\psi<-t\}}|\tilde{f}|^2e^{-\varphi}c(-\psi) :& \big(\tilde{f}-f,(z,y)\big)\in \bigg(\mathcal{O}(K_{\tilde{M}})\otimes\mathcal{I}\big(\pi_1^{*}(2\log|g|)+\pi_2^*(\varphi_2)\big)\bigg)_{(z,y)} \\
			&\text{\ for \ any } (z,y)\in Z_0,  \\
			& \& \ \tilde{f}\in H^0(\{\psi<-t\},\mathcal{O}(K_{\tilde{M}})).\bigg\}
		\end{split}
	\end{flalign*}
	by $\tilde{G}(t)$ for any $t\in [0,+\infty)$. Note that $\tilde{G}(t)$ is a minimal $L^2$ integrals on $\tilde{M}$, where $\tilde{M}$ is a submanifold of $M$.

We present a necessary condition such that $\tilde{G}\big(h^{-1}(r)\big)$ is linear with respect to $r\in (0,\int_0^{+\infty}c(t)e^{-t}dt]$ as follows.

	\begin{Proposition}\label{tildeM}
		Assume that $\tilde{G}(0)\in (0,+\infty)$. If $\tilde{G}\big(h^{-1}(r)\big)$ is linear with respect to $r\in (0,\int_0^{+\infty}c(t)e^{-t}dt]$, then $\tilde{M}=M$.
	\end{Proposition}

\begin{Remark} When $N= \pi_1^{*}(N_1)|_{\tilde{M}}$, Proposition \ref{tildeM} is Proposition 1.7 in \cite{BGY-concavity5}, where $N_1\le0$ is a subharmonic function on $\Omega$ and $N_1(z_j)>-\infty$ for any $j$.
\end{Remark}

\subsubsection{Linearity of the minimal $L^2$ integrals on fibrations over products of open Riemann surfaces}

In this section, we present characterizations of the concavity property of minimal $L^2$ integrals degenerating to linearity on the fibrations over products of open Riemann surfaces.

 When $n_1\ge 1$, $M=\left(\prod_{1\le j\le n_1}\Omega_j\right)\times Y$ is an $n-$dimensional complex manifold, where $n=n_1+n_2$. Let $\pi_{1}$, $\pi_{1,j}$ and $\pi_2$ be the natural projections from $M$ to $\prod_{1\le j\le n_1}\Omega_j$, $\Omega_j$ and $Y$ respectively. Let $K_M$ be the canonical (holomorphic) line bundle on $M$. Denote $P_j:\Delta\to \Omega_j$ be the universal covering from unit disc $\Delta$ to $\Omega_j$ for $1\le j \le n_1$.

 Let $Z_j$ be a (closed) analytic subset of $\Omega_j$ for any $j\in\{1,\ldots,n_1\}$, and denote that $Z_0:=\left(\prod_{1\le j\le n_1}Z_j\right)\times Y\subset M$.  Let $N\le 0$ be a plurisubharmonic function on $M$ satisfying $N|_{Z_0}\not\equiv-\infty$. For any $j\in\{1,\ldots,n_1\}$, let $\varphi_j$ be an upper semi-continuous function on $\Omega_{j}$ such that $\varphi_j(z)>-\infty$ for any $z\in Z_j$. Assume that $\sum_{1\le j\le n_1}\pi_{1,j}^*(\varphi_j)+N$ is a plurisubharmonic function on $M$.  Let $\varphi_Y$ be a plurisubharmonic function on $Y$, and denote that $\varphi:=\sum_{1\le j\le n_1}\pi_{1,j}^*(\varphi_j)+\pi_2^*(\varphi_Y)$.

Let $c$ be a positive function on $(0,+\infty)$ such that $\int_{0}^{+\infty}c(t)e^{-t}dt<+\infty$, $c(t)e^{-t}$ is decreasing on $(0,+\infty)$ and $c(-\psi)$ has a positive lower bound on any compact subset of $M\backslash Z_0$. Let $f$ be a holomorphic $(n,0)$ form on a neighborhood of $Z_0$.

Assume that $Z_0=\{z_0\}\times Y=\{(z_1,\ldots,z_{n_1})\}\times Y\subset M$.  Denote $$\hat{G}:=\max_{1\le j\le n_1}\left\{2p_j\pi_{1,j}^{*}\big(G_{\Omega_j}(\cdot,z_j)\big)\right\},$$ where $p_j$ is positive real number for $1\le j\le n_1$.
Denote $$\psi:=\max_{1\le j\le n_1}\left\{2p_j\pi_{1,j}^{*}\big(G_{\Omega_j}(\cdot,z_j)\big)\right\}+N.$$
We assume that $\{\psi<-t\}\backslash Z_0$ is a weakly pseudoconvex K\"ahler manifold for any $t\in\mathbb{R}$. Denote $\mathcal{F}_{(z_0,y)}=\mathcal{I}\big(\hat G+\pi_2^*(\varphi_Y)\big)_{(z_0,y)}$ for any $(z_0,y)\in Z_0$. Let $G(t)$ be the minimal $L^2$ integral on $\{\psi<-t\}$ with respect to $\varphi$, $f$, $\mathcal{F}$ and $c$ for any $t\geq 0$.

Let $w_j$ be a local coordinate on a neighborhood $V_{z_j}$ of $z_j\in\Omega_j$ satisfying $w_j(z_j)=0$. Denote that $V_0:=\prod_{1\le j\le n_1}V_{z_j}$, and $w:=(w_1,\ldots,w_{n_1})$ is a local coordinate on $V_0$ of $z_0\in \prod_{1\le j\le n_1}\Omega_j$. Denote that $E:=\left\{(\alpha_1,\ldots,\alpha_{n_1}):\sum_{1\le j\le n_1}\frac{\alpha_j+1}{p_j}=1\,\&\,\alpha_j\in\mathbb{Z}_{\ge0}\right\}$.
Let $f$ be a holomorphic $(n,0)$ form on $V_0\times Y\subset M$.

We present a characterization of the concavity of $G\big(h^{-1}(r)\big)$ degenerating to linearity for the case $Z_0=\{z_0\}\times Y$.

\begin{Theorem}
	\label{thm:linear-fibra-single}
	Assume that $G(0)\in(0,+\infty)$.  $G\big(h^{-1}(r)\big)$ is linear with respect to $r\in(0,\int_{0}^{+\infty}c(t)e^{-t}dt]$  if and only if the  following statements hold:
	
$(1)$ $N\equiv 0$ and $\psi=\max_{1\le j\le n_1}\left\{2p_j\pi_{1,j}^{*}(G_{\Omega_j}(\cdot,z_j))\right\}$;

	$(2)$ $f=\sum_{\alpha\in E}\pi_{1}^*\left(w^{\alpha}dw_1\wedge\ldots\wedge dw_{n_1}\right)\wedge \pi_2^*(f_{\alpha})+g_0$ on $V_0\times Y$, where  $g_0$ is a holomorphic $(n,0)$ form on $V_0\times Y$ satisfying $(g_0,z)\in\big(\mathcal{O}(K_M)\otimes\mathcal{I}(\varphi+\psi)\big)_{z}$ for any point $z\in Z_0$ and $f_{\alpha}$ is a holomorphic $(n_2,0)$ form on $Y$ such that $\sum_{\alpha\in E}\int_{Y}|f_{\alpha}|^2e^{-\varphi_Y}\in(0,+\infty)$;
	
	$(3)$ $\varphi_j=2\log|g_j|+2u_j$, where $g_j$ is a holomorphic function on $\Omega_j$ such that $g_j(z_j)\not=0$ and $u_j$ is a harmonic function on $\Omega_j$ for any $1\le j\le n_1$;

    $(4)$ $\chi_{j,z_j}^{\alpha_j+1}=\chi_{j,-u_j}$ for any $j\in\{1,2,...,n\}$ and $\alpha\in E$ satisfying $f_{\alpha}\not\equiv 0$, where  $\chi_{j,-u_j}$ be the character associated to $-u_j$ on $\Omega_j$ and $\chi_{j,z_j}$ be the character associated to $G_{\Omega_j}(\cdot,z_j)$ on $\Omega_j$.
\end{Theorem}

\begin{Remark} When $N\equiv 0$ ($\varphi_j$ is a subharmonic function on $\Omega_j$ satisfying $\varphi_j(z_j)>-\infty$ for any $1\le j\le n_1$),
it follows from $\mathcal{I}(\varphi+\psi)_{(z_j,y)}=\mathcal{I}(\hat G+\pi_2^*(\varphi_Y))_{(z_j,y)}$ for any $(z_j,y)\in Z_0$ (see Lemma \ref{l:phi1+phi2}) that Theorem \ref{thm:linear-fibra-single} can be referred to Theorem 1.2 in \cite{BGY-concavity6} (see also Theorem \ref{GBY6-single pt}).
\end{Remark}

 Let $Z_j=\{z_{j,1},\ldots,z_{j,m_j}\}\subset\Omega_j$ for any  $j\in\{1,\ldots,n_1\}$, where $m_j$ is a positive integer. Denote $Z_0:=(\prod_{1\le j\le n_1}Z_j)\times Y$.
Let $N\le0$ be a plurisubharmonic function on $M$ satisfying $N|_{Z_0}\not\equiv-\infty$.
Denote
$$\hat{G}:=\max_{1\le j\le n_1}\left\{\pi_{1,j}^*\left(2\sum_{1\le k\le m_j}p_{j,k}G_{\Omega_j}(\cdot,z_{j,k})\right)\right\},$$
where $p_{j,k}$ is a positive real number.
Denote
$$\psi:=\max_{1\le j\le n_1}\left\{\pi_{1,j}^*\left(2\sum_{1\le k\le m_j}p_{j,k}G_{\Omega_j}(\cdot,z_{j,k})\right)\right\}+N.$$
We assume that $\{\psi<-t\}\backslash Z_0$ is a weakly pseudoconvex K\"ahler manifold for any $t\in\mathbb{R}$.

Let $w_{j,k}$ be a local coordinate on a neighborhood $V_{z_{j,k}}\Subset\Omega_{j}$ of $z_{j,k}\in\Omega_j$ satisfying $w_{j,k}(z_{j,k})=0$ for any $j\in\{1,\ldots,n_1\}$ and $k\in\{1,\ldots,m_j\}$, where $V_{z_{j,k}}\cap V_{z_{j,k'}}=\emptyset$ for any $j$ and $k\not=k'$. Denote that $\tilde{I}_1:=\big\{(\beta_1,\ldots,\beta_{n_1}):1\le \beta_j\le m_j  \text{ for any } j\in\{1,\ldots,n_1\} \big\}$, $V_{\beta}:=\prod_{1\le j\le n_1}V_{z_{j,\beta_j}}$ for any $\beta=(\beta_1,\ldots,\beta_{n_1})\in \tilde{I}_1$ and $w_{\beta}:=(w_{1,\beta_1},\ldots,w_{n_1,\beta_{n_1}})$ is a local coordinate on $V_{\beta}$ of $z_{\beta}:=(z_{1,\beta_1},\ldots,z_{n_1,\beta_{n_1}})\in \prod_{1\le j\le n_1}\Omega_j$ satisfying $w_\beta(z_\beta)=0$.

Let $\beta^*=(1,\ldots,1)\in \tilde{I}_1$, and let $\alpha_{\beta^*}=(\alpha_{\beta^*,1},\ldots,\alpha_{\beta^*,n_1})\in\mathbb{Z}_{\ge0}^{n_1}$.
Denote that $E':=\left\{\alpha\in\mathbb{Z}_{\ge0}^{n_1}:\sum_{1\le j\le n_1}\frac{\alpha_j+1}{p_{j,1}}>\sum_{1\le j\le n_1}\frac{\alpha_{\beta^*,j}+1}{p_{j,1}}\right\}$. Let $f$ be a holomorphic $(n,0)$ form on $\cup_{\beta\in \tilde{I}_1}V_{\beta}\times Y$ satisfying $f=\pi_1^*\left(w_{\beta^*}^{\alpha_{\beta^*}}dw_{1,1}\wedge\ldots\wedge dw_{n_1,1}\right)\wedge\pi_2^*\left(f_{\alpha_{\beta^*}}\right)+\sum_{\alpha\in E'}\pi_1^*(w^{\alpha}dw_{1,1}\wedge\ldots\wedge dw_{n_1,1})\wedge\pi_2^*(f_{\alpha})$ on $V_{\beta^*}\times Y$, where $f_{\alpha_{\beta^*}}$ and $f_{\alpha}$ are  holomorphic $(n_2,0)$ forms on $Y$. Denote $\mathcal{F}_{(z,y)}=\mathcal{I}\big(\hat G+\pi_2^*(\varphi_Y)\big)_{(z,y)}$ for any point $(z,y)\in Z_0$. Let $G(t;c)$ be the minimal $L^2$ integral on $\{\psi<-t\}$ with respect to $\varphi$, $f$ and $\mathcal{F}$ for any $t\geq 0$.

We present a characterization of the concavity of $G\big(h^{-1}(r)\big)$ degenerating to linearity for the case that $Z_j$ is a set of finite points.

\begin{Theorem}
	\label{thm:linear-fibra-finite}Assume that $G(0)\in(0,+\infty)$.  $G(h^{-1}(r))$ is linear with respect to $r\in(0,\int_0^{+\infty} c(s)e^{-s}ds]$ if and only if the following statements hold:
	
$(1)$ $N\equiv 0$ and $\psi=\max_{1\le j\le n_1}\left\{\pi_{1,j}^*\left(2\sum_{1\le k\le m_j}p_{j,k}G_{\Omega_j}(\cdot,z_{j,k})\right)\right\}$;
	
	$(2)$ $\varphi_j=2\log|g_j|+2u_j$ for any $j\in\{1,\ldots,n_1\}$, where $u_j$ is a harmonic function on $\Omega_j$ and $g_j$ is a holomorphic function on $\Omega_j$ satisfying $g_j(z_{j,k})\not=0$ for any $k\in\{1,\ldots,m_j\}$;
	
	$(3)$ There exists a nonnegative integer $\gamma_{j,k}$ for any $j\in\{1,\ldots,n_1\}$ and $k\in\{1,\ldots,m_j\}$, which satisfies that $\prod_{1\le k\leq m_j}\chi_{j,z_{j,k}}^{\gamma_{j,k}+1}=\chi_{j,-u_j}$ and $\sum_{1\le j\le n_1}\frac{\gamma_{j,\beta_j}+1}{p_{j,\beta_j}}=1$ for any $\beta\in \tilde{I}_1$;
	
	$(4)$ $f=\pi_1^*\left(c_{\beta}\left(\prod_{1\le j\le n_1}w_{j,\beta_j}^{\gamma_{j,\beta_j}}\right)dw_{1,\beta_1}\wedge\ldots\wedge dw_{n,\beta_n}\right)\wedge\pi_2^*(f_0)+g_\beta$ on $V_{\beta}\times Y$ for any $\beta\in \tilde{I}_1$, where $c_{\beta}$ is a constant, $f_0\not\equiv0$ is a holomorphic $(n_2,0)$ form on $Y$ satisfying $\int_Y|f_0|^2e^{-\varphi_2}<+\infty$, and $g_{\beta}$ is a holomorphic $(n,0)$ form on $V_{\beta}\times Y$ such that $(g_{\beta},z)\in\big(\mathcal{O}(K_M)\otimes\mathcal{I}(\varphi+\psi)\big)_{z}$ for any $z\in\{z_\beta\}\times Y$;
	
	$(4)$ $c_{\beta}\prod_{1\le j\le n_1}\left(\lim_{z\rightarrow z_{j,\beta_j}}\frac{w_{j,\beta_j}^{\gamma_{j,\beta_j}}dw_{j,\beta_j}}{g_j(P_{j})_*\left(f_{u_j}\left(\prod_{1\le k\le m_j}f_{z_{j,k}}^{\gamma_{j,k}+1}\right)\left(\sum_{1\le k\le m_j}p_{j,k}\frac{df_{z_{j,k}}}{f_{z_{j,k}}}\right)\right)}\right)=c_0$ for any $\beta\in \tilde{I}_1$, where $c_0\in\mathbb{C}\backslash\{0\}$ is a constant independent of $\beta$, $f_{u_j}$ is a holomorphic function $\Delta$ such that $|f_{u_j}|=P_j^*(e^{u_j})$ and $f_{z_{j,k}}$ is a holomorphic function on $\Delta$ such that $|f_{z_{j,k}}|=P_j^*\left(e^{G_{\Omega_j}(\cdot,z_{j,k})}\right)$ for any $j\in\{1,\ldots,n_1\}$ and $k\in\{1,\ldots,m_j\}$.
\end{Theorem}

\begin{Remark} When $N\equiv 0$ ($\varphi_j$ is a subharmonic function on $\Omega_j$ satisfying $\varphi_j(z_{j,\beta_j})>-\infty$ for any $1\le j\le n_1$ and $1\le \beta_j \le m_j$),
it follows from $\mathcal{I}(\varphi+\psi)_{(z,y)}=\mathcal{I}(\hat{G}+\pi_2^*(\varphi_Y))_{(z,y)}$ for any $(z,y)\in Z_0$ (see Lemma \ref{l:phi1+phi2})
that Theorem \ref{thm:linear-fibra-finite} can be referred to Theorem 1.5 in \cite{BGY-concavity6} (see also Theorem \ref{GBY6-finitie pts}).
\end{Remark}

 Let ${Z}_j=\{z_{j,k}:1\le k<\tilde m_j\}$ be a discrete subset of $\Omega_j$ for any  $j\in\{1,\ldots,n_1\}$, where $\tilde{m}_j\in\mathbb{Z}_{\ge2}\cup\{+\infty\}$. Denote $Z_0:=(\prod_{1\le j\le n_1}Z_j)\times Y$.
Let $p_{j,k}$ be a positive number for any $1\le j\le n_1$ and $1\le k<\tilde m_j$ such that  $\sum_{1\le k<\tilde{m}_j}p_{j,k}G_{\Omega_j}(\cdot,z_{j,k})\not\equiv-\infty$ for any $j$.
Let
$$\hat{G}=\max_{1\le j\le n_1}\left\{\pi_{1,j}^*\left(2\sum_{1\le k<\tilde{m}_j}p_{j,k}G_{\Omega_j}(\cdot,z_{j,k})\right)\right\},$$
and
$$\psi=\max_{1\le j\le n_1}\left\{\pi_{1,j}^*\left(2\sum_{1\le k<\tilde{m}_j}p_{j,k}G_{\Omega_j}(\cdot,z_{j,k})\right)\right\}+N.$$
We assume that $\{\psi<-t\}\backslash Z_0$ is a weakly pseudoconvex K\"ahler manifold for any $t\in\mathbb{R}$ and $\limsup_{t\rightarrow+\infty}c(t)<+\infty$.

Let $w_{j,k}$ be a local coordinate on a neighborhood $V_{z_{j,k}}\Subset\Omega_{j}$ of $z_{j,k}\in\Omega_j$ satisfying $w_{j,k}(z_{j,k})=0$ for any $j\in\{1,\ldots,n_1\}$ and $1\le k<\tilde{m}_j$, where $V_{z_{j,k}}\cap V_{z_{j,k'}}=\emptyset$ for any $j$ and $k\not=k'$. Denote that $\tilde I_1:=\big\{(\beta_1,\ldots,\beta_{n_1}):1\le \beta_j< \tilde m_j$ for any $j\in\{1,\ldots,n_1\}\big\}$, $V_{\beta}:=\prod_{1\le j\le n_1}V_{z_{j,\beta_j}}$ for any $\beta=(\beta_1,\ldots,\beta_{n_1})\in\tilde I_1$ and $w_{\beta}:=(w_{1,\beta_1},\ldots,w_{n_1,\beta_{n_1}})$ is a local coordinate on $V_{\beta}$ of $z_{\beta}:=(z_{1,\beta_1},\ldots,z_{n_1,\beta_{n_1}})\in \prod_{1\le j\le n_1}\Omega_j$.

Let $\beta^*=(1,\ldots,1)\in \tilde{I}_1$, and let $\alpha_{\beta^*}=(\alpha_{\beta^*,1},\ldots,\alpha_{\beta^*,n_1})\in\mathbb{Z}_{\ge0}^{n_1}$. Denote that $E':=\left\{\alpha\in\mathbb{Z}_{\ge0}^{n_1}:\sum_{1\le j\le n_1}\frac{\alpha_j+1}{p_{j,1}}>\sum_{1\le j\le n_1}\frac{\alpha_{\beta^*,j}+1}{p_{j,1}}\right\}$.
Let $f$ be a holomorphic $(n,0)$ form on $\cup_{\beta\in \tilde{I}_1}V_{\beta}\times Y$ satisfying $f=\pi_1^*\left(w_{\beta^*}^{\alpha_{\beta^*}}dw_{1,1}\wedge\ldots\wedge dw_{n_1,1}\right)\wedge\pi_2^*\left(f_{\alpha_{\beta^*}}\right)+\sum_{\alpha\in E'}\pi_1^*(w^{\alpha}dw_{1,1}\wedge\ldots\wedge dw_{n_1,1})\wedge\pi_2^*(f_{\alpha})$ on $V_{\beta^*}\times Y$, where $f_{\alpha_{\beta^*}}$ and $f_{\alpha}$ are  holomorphic $(n_2,0)$ forms on $Y$. Denote $\mathcal{F}_{(z,y)}=\mathcal{I}\big(\hat G+\pi_2^*(\varphi_Y)\big)_{(z,y)}$ for any $(z,y)\in Z_0$. Let $G(t)$ be the minimal $L^2$ integral on $\{\psi<-t\}$ with respect to $\varphi$, $f$, $\mathcal{F}$ and $c$ for any $t\geq 0$.

When there exists $j_0\in\{1,\ldots,n_1\}$ such that $\tilde m_{j_0}=+\infty$, we present that $G\big(h^{-1}(r)\big)$ is not linear.

\begin{Theorem}
	\label{thm:linear-fibra-infinite}If $G(0)\in(0,+\infty)$ and there exists $j_0\in\{1,\ldots,n_1\}$ such that $\tilde m_{j_0}=+\infty$, then $G\big(h^{-1}(r)\big)$ is not linear with respect to $r\in(0,\int_0^{+\infty} c(s)e^{-s}ds]$.
\end{Theorem}

\begin{Remark} When $N\equiv 0$ ($\varphi_j$ is a subharmonic function on $\Omega_j$ satisfying $\varphi_j(z_{j,\beta_j})>-\infty$ for any $1\le j\le n_1$ and $1\le \beta_j \le m_j$),
it follows from $\mathcal{I}(\varphi+\psi)_{(z,y)}=\mathcal{I}(\hat G+\pi_2^*(\varphi_Y))_{(z,y)}$ for any $(z,y)\in Z_0$ (see Lemma \ref{l:phi1+phi2})
that Theorem \ref{thm:linear-fibra-infinite} can be referred to Theorem 1.7 in \cite{BGY-concavity6} (see also Theorem \ref{GBY6-infinite points}).
\end{Remark}

 Let ${Z}_j=\{z_{j,k}:1\le k<\tilde m_j\}$ be a discrete subset of $\Omega_j$ for any  $j\in\{1,\ldots,n_1\}$, where $\tilde{m}_j\in\mathbb{Z}_{\ge2}\cup\{+\infty\}$. Denote $Z_0:=(\prod_{1\le j\le n_1}Z_j)\times Y$.

Let $\tilde{M}\subset M$ be an $n-$dimensional weakly pseudoconvex K\"ahler manifold satisfying that $Z_0\subset \tilde{M}$.  Let $f$ be a holomorphic $(n,0)$ form on a neighborhood $U_0\subset \tilde{M}$ of $Z_0$.

Let $N\le 0$ be a plurisubharmonic function on $\tilde{M}$ satisfying $N|_{Z_0}\not\equiv-\infty$. For any $j\in\{1,\ldots,n_1\}$, let $\varphi_j$ be an upper semi-continuous function on $\Omega_{j}$ such that $\varphi_j(z)>-\infty$ for any $z\in Z_j$. Assume that $\sum_{1\le j\le n_1}\pi_{1,j}^*(\varphi_j)+N$ is a plurisubharmonic function on $\tilde{M}$.  Let $\varphi_Y$ be a plurisubharmonic function on $Y$, and denote that $\varphi:=\sum_{1\le j\le n_1}\pi_{1,j}^*(\varphi_j)+\pi_2^*(\varphi_Y)$.

Let $p_{j,k}$ be a positive number for any $1\le j\le n_1$ and $1\le k<\tilde m_j$ such that  $\sum_{1\le k<\tilde{m}_j}p_{j,k}G_{\Omega_j}(\cdot,z_{j,k})\not\equiv-\infty$ for any $j$.
Denote
$$\hat G=\max_{1\le j\le n_1}\left\{\pi_{1,j}^*\left(2\sum_{1\le k<\tilde{m}_j}p_{j,k}G_{\Omega_j}(\cdot,z_{j,k})\right)\right\},$$
and
$$\psi=\max_{1\le j\le n_1}\left\{\pi_{1,j}^*\left(2\sum_{1\le k<\tilde{m}_j}p_{j,k}G_{\Omega_j}(\cdot,z_{j,k})\right)\right\}+N.$$
We assume that $\{\psi<-t\}\backslash Z_0$ is a weakly pseudoconvex K\"ahler manifold for any $t\in\mathbb{R}$.

Let $c$ be a positive function on $(0,+\infty)$ such that $\int_{0}^{+\infty}c(t)e^{-t}dt<+\infty$, $c(t)e^{-t}$ is decreasing on $(0,+\infty)$ and $c(-\psi)$ has a positive lower bound on any compact subset of $\tilde M\backslash Z_0$. Let $f$ be a holomorphic $(n,0)$ form on a neighborhood of $Z_0$.

Denote
\begin{flalign*}
		\begin{split}
			\inf\bigg\{\int_{\{\psi<-t\}}|\tilde{f}|^2e^{-\varphi}c(-\psi) :& \big(\tilde{f}-f,(z,y)\big)\in \bigg(\mathcal{O}(K_{\tilde{M}})\otimes\mathcal{I}\big(\hat{G}+\pi_2^*(\varphi_Y)\big)\bigg)_{(z,y)} \\
			&\text{\ for \ any } (z,y)\in Z_0,  \\
			& \& \ \tilde{f}\in H^0\big(\{\psi<-t\},\mathcal{O}(K_{\tilde{M}})\big).\bigg\}
		\end{split}
	\end{flalign*}
	by $\tilde{G}(t)$ for any $t\in [0,+\infty)$. Note that $\tilde{G}(t)$ is a minimal $L^2$ integrals on $\tilde{M}$, where $\tilde{M}$ is a submanifold of $M$.

We present a necessary condition such that $\tilde{G}\big(h^{-1}(r)\big)$ is linear with respect to $r\in (0,\int_0^{+\infty}c(t)e^{-t}dt]$ as follows.

\begin{Proposition}
	\label{p:M=M_1}If $\tilde{G}(0)\in(0,+\infty)$ and $\tilde{G}\big(h^{-1}(r)\big)$ is  linear with respect to $r\in(0,\int_0^{+\infty} c(s)e^{-s}ds]$, we have $\tilde{M}=M$.
\end{Proposition}
\begin{Remark} When $N\equiv 0$ ($\varphi_j$ is a subharmonic function on $\Omega_j$ satisfying $\varphi_j(z_{j,\beta_j})>-\infty$ for any $1\le j\le n_1$ and $1\le \beta_j \le m_j$),
it follows from $\mathcal{I}(\varphi+\psi)_{(z,y)}=\mathcal{I}(\hat G+\pi_2^*(\varphi_Y))_{(z,y)}$ for any $(z,y)\in Z_0$ (see Lemma \ref{l:phi1+phi2})
that Proposition \ref{p:M=M_1}can be referred to Proposition 1.8 in \cite{BGY-concavity6}.
\end{Remark}

\subsection{Applications}

In this section,
we present characterizations of the holding of equality in optimal jets $L^2$ extension problem with the negligible weights on fibraions.

\subsubsection{Background: equality in optimal jets $L^2$ extension problem}

Let $\Omega$ be an open Riemann surface with a nontrivial Green function $G_{\Omega}$.  Let $w$ be a local coordinate on a neighborhood $V_{z_0}$ of $z_0\in\Omega$ satisfying $w(z_0)=0$.
Let $c_{\beta}(z)$ be the logarithmic capacity (see \cite{S-O69}) on $\Omega$, i.e.
$$c_{\beta}(z_0):=\exp \lim\limits_{\xi\to z_0}G_{\Omega}(z,z_0)-\log|w(z)|.$$

Let $B_{\Omega}(z_0)$ be the Bergman kernel function on $\Omega$.
An open question was posed by Sario-Oikawa \cite{S-O69}: find a relation between the magnitudes of the quantities $\sqrt{\pi B_{\Omega}(z)}$, $c_{\beta}(z)$.

In \cite{suita72}, Suita conjectured: $ \pi B_{\Omega}(z_0)\ge\big(c_{\beta}(z_0)\big)^2$ holds, and the equality holds if and only if $\Omega$ is
conformally equivalent to the unit disc less a (possible) closed set of inner capacity zero.

The inequality part of Suita conjecture for bounded planar domains was proved by B\l ocki \cite{Blo13}, and original form of the inequality part was proved by Guan-Zhou \cite{gz12}.
The equality part of Suita conjecture was proved by Guan-Zhou \cite{guan-zhou13ap}, then Suita conjecture was completed proved.

It follows from the extremal property of the Bergman kernel function that the holding of the following two equalities are equivalent

(1) $ \pi B_{\Omega}(z_0)=(c_{\beta}(z_0))^2$;

(2) $\inf\{\int_{\Omega}|F|^2: F$ is a holomorphic  $(1,0)$ form on $\Omega$ such that $F(z_0)=dw\}=\frac{2\pi}{(c_{\beta}(z_0))^2}$.

Note that (2) is equivalent to the holding of the equality in optimal 0-jet $L^2$ extension problem for open Riemann surface
$\Omega$ and single point $z_{0}$ with trivial weights $\varphi\equiv0$ and trivial gain $c\equiv1$.
Then it is natural to ask

\begin{Problem}\label{pb:equality}
How to characterize the holding of the equality in optimal $k$-jets $L^2$ extension problem,
where $k$ is a nonnegative integer?
\end{Problem}

For open Riemann surfaces and single points, when the weights are harmonic and gain is constant,
Guan-Zhou \cite{guan-zhou13ap} gave an answer to 0-jet version of Problem \ref{pb:equality},
i.e. a proof of the extended Suita conjecture posed by Yamada \cite{Yamada}.

For open Riemann surfaces and single points, Guan-Yuan \cite{GY-concavity} gave an answer to 0-jet version of Problem \ref{pb:equality}
(when the weights are subharmonic and gain is smooth,
Guan-Mi \cite{GM} gave an answer to 0-jet version of Problem \ref{pb:equality}), and Guan-Mi-Yuan \cite{GMY} gave an answer to Problem \ref{pb:equality}.

For open Riemann surfaces and analytic sets (with finite or infinite points),
Guan-Yuan \cite{GY-concavity3} gave an answer to Problem \ref{pb:equality}.
For products of open Riemann surfaces and products of analytic sets, Guan-Yuan \cite{GY-concavity4} gave an answer to Problem \ref{pb:equality}.

For fibrations over open Riemann surfaces, Bao-Guan-Yuan \cite{BGY-concavity5} gave an answer to Problem \ref{pb:equality} with negligible weights pulled back from the open Riemann surfaces.
For fibrations over products of open Riemann surfaces, Bao-Guan-Yuan \cite{BGY-concavity6} gave an answer to Problem \ref{pb:equality} with negligible weights pulled back from the products of open Riemann surfaces.

In the following sections, for fibrations over open Riemann surfaces and fibrations over products of open Riemann surfaces,
we give answers to Problem \ref{pb:equality} with negligible weights on fibrations.

\subsubsection{Fibrations over open Riemann surfaces}

In this section, we give characterizations of the holding of equality in optimal jets $L^2$ extension problem with negligible weights from fibers over analytic subsets to fibrations over open Riemann surfaces.

Let $\Omega$ be an open Riemann surface with nontrivial Green functions. Let $Y$ be an $(n-1)-$dimensional weakly pseudoconvex K\"ahler manifold. Denote $M=\Omega \times Y$. Let $K_M$ be the canonical line bundle on $M$. Let $\pi_1$ and $\pi_2$ be the natural projections from $M$ to $\Omega$ and $Y$ respectively.

Let $Z_{\Omega}=\{z_j:j \in \mathbb{N}_+ \& 1\le j<\gamma\}$ be a subset of $\Omega$ of discrete points. Denote $Z_0:=Z_{\Omega}\times Y$. Denote $Z_j:=\{z_j\}\times Y$ for any $j\ge 1$.

 Assume that $\tilde{M}\subset M$ is an $n-$dimensional weakly pseudoconvex submanifold satisfying that $Z_0\subset \tilde{M}$.

Let $\psi$ be a plurisubharmonic function on $\tilde M$ such that $\{\psi<-t\}\backslash Z_0$ is a weakly pseudoconvex K\"ahler manifold for any $t\in\mathbb{R}$ and $Z_0\subset \{\psi=-\infty\}$.
It follows from Siu's decomposition theorem that $$dd^{c}\psi=\sum\limits_{j\ge1}2p_j[Z_j]+\sum\limits_{i\ge 1}\lambda_i[A_i]+R,$$
where $[Z_j]$ and $[A_i]$ are the currents of integration over an irreducible $n-$dimensional analytic set, and where $R$ is a closed positive current with the property that $dimE_c(R)<n$ for every $c>0$, where $E_c(R)=\{x\in M: v(R,x)\ge c\}$ is the upperlevel sets of Lelong number. We assume that $p_j\ge0$ is a positive number for any $1\le j< \gamma$ and  $2\sum_{j\ge 1}p_j G_{\Omega}(\cdot,z_j)\not\equiv -\infty$.

Then $N:=\psi-\pi^{*}_1(\sum\limits_{j\ge1}2p_jG_{\Omega}(z,z_j))$ is a plurisubharmonic function on $\tilde M$. We assume that $N\le0$.

Let $k_j$ be a nonnegative integer for any $1\le j< \gamma$.
Let $\varphi_1$ be a Lebesgue measurable function on $\Omega$ such that $\pi_1^*(\varphi_1)+\psi$ is a plurisubharmonic function on $\tilde{M}$.
We also assume that there exists a holomorphic function $g\in\mathcal{O}(\Omega)$ and a plurisubharmonic function $\tilde{\psi}_2\in Psh(\tilde{M})$ such that $$\psi+\pi^{*}_1(\varphi_1)=\pi^{*}_1(2\log|g|)+\tilde{\psi}_2,$$
where $ord_{z_j}(g)=k_j+1$, for any $1\le j<\gamma$.

Let $\varphi_2$ be a plurisubharmonic funciton on $Y$. Denote $\varphi:=\pi_1^*(\varphi_1)+\pi_1^*(\varphi_2)$.

For any $1\le j<\gamma$, let $\tilde{z}_j$ be a local coordinate on a neighborhood $V_{j}$ of $z_j$ satisfying $\tilde{z}_j(z_j)=0$ and $V_{j}\cap V_{k}=\emptyset$, for any $j\neq k$. We assume that $g=d_j\tilde{z}_j^{k_j+1}h_j(\tilde{z}_j)$ on each $V_j$, where $h_j(z_j)=1$. Denote $V_0:=\cup_{1\le j < \gamma}V_j$.

	Let $\gamma=m+1$ be a positive integer. We give an application of Theorem \ref{finite points} as below.

\begin{Theorem}\label{app-L2 equality}
		 Let $\psi$, $\varphi_1$ and $\varphi_2$ be as above. Let $c(t)$ be a positive measurable function on $(0,+\infty)$ satisfying that $c(t)e^{-t}$ is decreasing on $(0,+\infty)$ and $\int_{0}^{+\infty}c(s)e^{-s}ds<+\infty$. Let $a_j$ be a constant for any $j$. Let $F_j$ be a holomorphic $(n,0)$ form on $Y$. Assume that $$\sum_{j=1}^{m}\frac{2\pi|a_j|^2}{p_j|d_j|^2}\int_Y|F_j|^2e^{-\varphi_2-\tilde{\psi}_2(z_j,w)}\in(0,+\infty).$$
		
		Let $f$ be a holomorphic $(n,0)$ form on $V_0\times Y$ satisfying that $f=\pi_1^*(a_j\tilde{z}_j^{k_j}d\tilde{z}_j)\wedge\pi_2^*(F_j)$ on $V_{j}\times Y$. Then there exists a holomorphic $(n,0)$ form $F$ on $\tilde{M}$ such that $(F-f,(z_j,y))\in(\mathcal{O}(K_{\tilde{M}})\otimes\mathcal{I}(\pi_1^*(2\log |g|)+\pi_2^*(\varphi_2))_{(z_j,y)}$ for any $(z_j,y)\in Z_0$ and
		\begin{equation}\label{L2result}
			\int_{\tilde{M}}|F|^2e^{-\varphi}c(-\psi)\leq
\left(\int_0^{+\infty}c(s)e^{-s}ds\right)
 \sum_{j=1}^{m}\frac{2\pi|a_j|^2}{p_j|d_j|^2}\int_Y|F_j|^2e^{-\varphi_2-\tilde{\psi}_2(z_j,w)}.
		\end{equation}
		
		Moreover, equality $\left(\int_0^{+\infty}c(s)e^{-s}ds\right)
 \sum_{j=1}^{m}\frac{2\pi|a_j|^2}{p_j|d_j|^2}\int_Y|F_j|^2e^{-\varphi_2-\tilde{\psi}_2(z_j,w)}=\inf\big\{$ $ \int_{\tilde{M}}|\tilde{F}|^2e^{-\varphi}c(-\psi):\tilde{F}$ is a holomorphic $(n,0)$ form on $\tilde{M}$ such that $(\tilde{F}-f,(z_j,y))\in(\mathcal{O}(K_{\tilde{M}})\otimes\mathcal{I}(\pi_1^*(2\log |g|)+\pi_2^*(\varphi_2))_{(z_j,y)}$ for any $(z_j,y)\in Z_0\big\}$ holds if and only if the following statements hold:
		
		(1). $N\equiv 0$ and $\psi=\pi_1^*(2\sum\limits_{j=1}^mp_j G_{\Omega}(\cdot,z_j))$;
		
		(2). $\varphi_1+2\sum\limits_{j=1}^mk_j G_{\Omega}(\cdot,z_j)=2\log |g|+2\sum\limits_{j=1}^m(k_j+1)G_{\Omega}(\cdot,z_j)+2u$, where $g$ is a holomorphic function on $\Omega$ such that $g(z_j)\neq 0$ for any $j\in\{1,2,\ldots,m\}$ and $u$ is a harmonic function on $\Omega$;
		
		(3). $\prod\limits_{j=1}^m\chi_{z_j}^{k_j+1}=\chi_{-u}$, where $\chi_{-u}$ and $\chi_{z_j}$ are the characters associated to the functions $-u$ and $G_{\Omega}(\cdot,z_j)$ respectively;
		
		(4). for any $j\in\{1,2,\ldots,m\}$,
		\begin{equation}
			\lim_{z\rightarrow z_j}\frac{a_j\tilde{z}_j^{k_j}d\tilde{z}_j}{gP_*\left(f_u\left(\prod\limits_{l=1}^mf_{z_l}^{k_l+1}\right)\left(\sum\limits_{l=1}^mp_l\dfrac{d{f_{z_{l}}}}{f_{z_{l}}}\right)\right)}=c_j\in\mathbb{C}\setminus\{0\},
		\end{equation}
		and there exist $c_0\in\mathbb{C}\setminus\{0\}$ and a holomorphic $(n-1,0)$ form $F_Y$ on $Y$ which are independent of $j$ such that $c_0F_Y=c_jF_j$ for any $j\in\{1,2,\ldots,m\}$;
		
		(5). $\tilde{M}=M$.
		
	\end{Theorem}

\begin{Remark} When  $N\equiv \pi_1^{*}(N_1)$, it follows from $\mathcal{I}(\varphi+\psi)_{(z_j,y)}=\mathcal{I}(\pi_1^{*}(2\log|g|)+\pi_2^*(\varphi_2))_{(z_j,y)}$ (see Lemma \ref{equiv of multiplier ideal sheaf})
that Theorem \ref{app-L2 equality} can be referred to Theorem 1.8 in \cite{BGY-concavity5} (see also Theorem \ref{GBY-APP1}), where $N_1\le0$ is a subharmonic function on $\Omega$.
\end{Remark}
	When $\gamma=+\infty$, i.e., $Z_{\Omega}=\{z_j: 1\le j< \gamma\}$ is an infinite subset of $\Omega$ of discrete points. Let $k_j$ be a nonnegative integer for any $j\in\mathbb{N}^+$. Assume that $p_j=k_j+1$, i.e., we have $\psi=\pi^{*}_1\big(\sum\limits_{j=1}^{+\infty}2(k_j+1)G_{\Omega}(z,z_j)\big)+N$, where $N\le 0$ is a plurisubharmonic function on $\tilde M$.

We give an $L^2$ extension result from fibers over analytic subsets to fibrations over open Riemann surfaces, where the analytic subsets are infinite subsets of discrete points on open Riemann surfaces.

\begin{Theorem}\label{app-L2 inequality}
		 Let $\psi$, $\varphi_1$ and $\varphi_2$ be as above. Let $c(t)$ be a positive measurable function on $(0,+\infty)$ satisfying that $c(t)e^{-t}$ is decreasing on $(0,+\infty)$ and $\int_{0}^{+\infty}c(s)e^{-s}ds<+\infty$. Let $a_j$ be a constant for any $j$. Let $F_j$ be a holomorphic $(n,0)$ form on $Y$. Assume that  $$\sum_{j=1}^{+\infty}\frac{2\pi|a_j|^2}{(k_j+1)|d_j|^2}\int_Y|F_j|^2e^{-\varphi_2-\tilde{\psi}_2(z_j,w)}\in(0,+\infty).$$
		
		Let $f$ be a holomorphic $(n,0)$ form on $V_0\times Y$ satisfying that $f=\pi_1^*(a_j\tilde{z}_j^{k_j}d\tilde{z}_j)\wedge\pi_2^*(F_j)$ on $V_{z_j}\times Y$. Then there exists a holomorphic $(n,0)$ form $F$ on $\tilde{M}$ such that $\big(F-f,(z_j,y)\big)\in\big(\mathcal{O}(K_{M})\otimes\mathcal{I}(\pi_1^*(2\log |g|)+\pi_2^*(\varphi_2)\big)_{(z_j,y)}$ for any $(z_j,y)\in Z_0$ and
		\begin{equation}\label{L2result}
			\int_{\tilde{M}}|F|^2e^{-\varphi}c(-\psi)<
\left(\int_0^{+\infty}c(s)e^{-s}ds\right)
 \sum_{j=1}^{+\infty}\frac{2\pi|a_j|^2}{(k_j+1)|d_j|^2}\int_Y|F_j|^2e^{-\varphi_2-\tilde{\psi}_2(z_j,w)}.
		\end{equation}	
	\end{Theorem}

\begin{Remark} When  $N\equiv \pi_1^{*}(N_1)$, it follows from $\mathcal{I}(\varphi+\psi)_{(z_j,y)}=\mathcal{I}(\pi_1^{*}(2\log|g|)+\pi_2^*(\varphi_2))_{(z_j,y)}$ (see Lemma \ref{equiv of multiplier ideal sheaf})
that Theorem \ref{app-L2 inequality} can be referred to Theorem 1.10 in \cite{BGY-concavity5} (see also Theorem \ref{GBY-APP2}), where $N_1\le0$ is a subharmonic function on $\Omega$.
\end{Remark}

\subsubsection{Fibrations over products of open Riemann surfaces}

In this section, we present characterizations of the holding of equality in optimal jets $L^2$ extension problem with negligible weights from fibers over analytic subsets to fibrations over products of open Riemann surfaces.

Let $\Omega_j$  be an open Riemann surface, which admits a nontrivial Green function $G_{\Omega_j}$ for any  $1\le j\le n_1$. Let $Y$ be an $n_2-$dimensional weakly pseudoconvex K\"ahler manifold, and let $K_Y$ be the canonical (holomorphic) line bundle on $Y$. Let
$M=\left(\prod_{1\le j\le n_1}\Omega_j\right)\times Y$
 be an $n-$dimensional complex manifold, where $n=n_1+n_2$. Let $\pi_{1}$, $\pi_{1,j}$ and $\pi_2$ be the natural projections from $M$ to $\prod_{1\le j\le n_1}\Omega_j$, $\Omega_j$ and $Y$ respectively. Let $K_{M}$ be the canonical (holomorphic) line bundle on $M$.
Let $Z_j$ be a (closed) analytic subset of $\Omega_j$ for any $j\in\{1,\ldots,n_1\}$, and denote that $Z_0:=\left(\prod_{1\le j\le n_1}Z_j\right)\times Y$. Let $\tilde{M}\subset M$ be an $n-$dimensional complex manifold satisfying that $Z_0\subset \tilde{M}$, and let $K_{\tilde{M}}$ be the canonical (holomorphic) line bundle on $\tilde{M}$.

Let $Z_j=\{z_{j,k}:1\le k<\tilde m_j\}$ be a discrete subset of $\Omega_j$ for any  $j\in\{1,\ldots,n_1\}$, where $\tilde m_j\in\mathbb{Z}_{\ge2}\cup\{+\infty\}$.
Let $w_{j,k}$ be a local coordinate on a neighborhood $V_{z_{j,k}}\Subset\Omega_{j}$ of $z_{j,k}\in\Omega_j$ satisfying $w_{j,k}(z_{j,k})=0$ for any $1\le j\le n_1$ and $1\le k<\tilde m_j$, where $V_{z_{j,k}}\cap V_{z_{j,k'}}=\emptyset$ for any $j$ and $k\not=k'$. Denote that $\tilde I_1:=\big\{(\beta_1,\ldots,\beta_{n_1}):1\le \beta_j<\tilde m_j$ for any $j\in\{1,\ldots,n_1\}\big\}$, $V_{\beta}:=\prod_{1\le j\le n_1}V_{z_{j,\beta_j}}$  and $w_{\beta}:=(w_{1,\beta_1},\ldots,w_{n_1,\beta_{n_1}})$ is a local coordinate on $V_{\beta}$ of $z_{\beta}:=(z_{1,\beta_1},\ldots,z_{n_1,\beta_{n_1}})\in\prod_{1\le j\le n_1}\Omega_j$ for any $\beta=(\beta_1,\ldots,\beta_{n_1})\in\tilde I_1$. Then $Z_0=\{(z_{\beta},y):\beta\in\tilde I_1\,\&\,y\in Y\}\subset \tilde{M}$.

Let $N\le0$ be a plrusubharmonic function on $\tilde M$ and let $\varphi_j$ be a Lebesgue measurable function on $\Omega_j$ such that $N+\sum_{1\le j\le n_1}\pi_{1,j}^*(\varphi_j)$ is a plurisubharmonic function on $\tilde M$ satisfying $\left(N+\sum_{1\le j\le n_1}\pi_{1,j}^*(\varphi_j)\right)|_{Z_0}\not\equiv-\infty$. Let $\varphi_Y$ be a plurisubharmonic function on $Y$.

Let $p_{j,k}$ be a positive number for any $1\le j\le n_1$ and $1\le k<\tilde m_j$, which satisfies that $\sum_{1\le k<\tilde m_j}p_{j,k}G_{\Omega_j}(\cdot,z_{j,k})\not\equiv-\infty$ for any $1\le j\le n_1$. Denote
$$\hat{G}:=\max_{1\le j\le n_1}\left\{2\sum_{1\le k<\tilde m_j}p_{j,k}\pi_{1,j}^{*}(G_{\Omega_j}(\cdot,z_{j,k}))\right\},$$
and denote that
$$\psi:=\max_{1\le j\le n_1}\left\{2\sum_{1\le k<\tilde m_j}p_{j,k}\pi_{1,j}^{*}(G_{\Omega_j}(\cdot,z_{j,k}))\right\}+N.$$
We assume that $\{\psi<-t\}\backslash Z_0$ is a weakly pseudoconvex K\"ahler manifold for any $t\in\mathbb{R}$.
Let $\varphi:=\sum_{1\le j\le n_1}\pi_{1,j}^*(\varphi_j)+\pi_2^*(\varphi_Y)$ on $\tilde M$.

 Denote that $E_{\beta}:=\left\{(\alpha_1,\ldots,\alpha_{n_1}):\sum_{1\le j\le n_1}\frac{\alpha_j+1}{p_{j,\beta_j}}=1\,\&\,\alpha_j\in\mathbb{Z}_{\ge0}\right\}$ and $\tilde E_{\beta}:=\left\{(\alpha_1,\ldots,\alpha_{n_1}):\sum_{1\le j\le n_1}\frac{\alpha_j+1}{p_{j,\beta_j}}\ge1\,\&\,\alpha_j\in\mathbb{Z}_{\ge0}\right\}$ for any $\beta\in\tilde I_1$.

Let $f$ be a holomorphic $(n,0)$ form on a neighborhood $U_0\subset \tilde{M}$ of $Z_0$ such that
\begin{equation}\label{decom in introduction}
f=\sum_{\alpha\in\tilde E_{\beta}}\pi_1^*(w_{\beta}^{\alpha}dw_{1,\beta_1}\wedge\ldots\wedge dw_{n_1,\beta_{n_1}})\wedge\pi_2^*(f_{\alpha,\beta})
\end{equation} on $U_0\cap(V_{\beta}\times Y)$, where $f_{\alpha,\beta}$ is a holomorphic $(n_2,0)$ form on $Y$ for any $\alpha\in E_{\beta}$ and $\beta\in\tilde I_1$.

Denote that
\begin{equation*}
c_{j,k}:=\exp\lim_{z\rightarrow z_{j,k}}\left(\frac{\sum_{1\le k_1<\tilde m_j}p_{j,k_1}G_{\Omega_j}(z,z_{j,k_1})}{p_{j,k}}-\log|w_{j,k}(z)|\right)
\end{equation*}
 for any $j\in\{1,\ldots,n\}$ and $1\le k<\tilde m_j$ (following from Lemma \ref{l:green-sup} and Lemma \ref{l:green-sup2}, we get that the above limit exists).

Let $c_j(z)$ be the logarithmic capacity (see \cite{S-O69}) on $\Omega_j$, which is locally defined by
$$c_j(z_j):=\exp\lim_{z\rightarrow z_j}(G_{\Omega_j}(z,z_j)-\log|w_j(z)|).$$

For the case $Z_0=\{z_0\}\times Y\subset \tilde{M}$, where $z_0=(z_1,\ldots,z_{n_1})\in\prod_{1\le j\le n_1}\Omega_j$, we denote that $E:=\left\{(\alpha_1,\ldots,\alpha_{n_1}):\sum_{1\le j\le n_1}\frac{\alpha_j+1}{p_j}=1\,\&\,\alpha_j\in\mathbb{Z}_{\ge0}\right\}$. Let $f_{\alpha}$ be the holomorphic $(n_2,0)$ form on $Y$ which comes from formula \eqref{decom in introduction}. We obtain a characterization of the holding of equality in optimal jets $L^2$ extension problem.

\begin{Theorem}
\label{thm:exten-fibra-single}
Let $c$ be a positive function on $(0,+\infty)$ such that $\int_{0}^{+\infty}c(t)e^{-t}dt<+\infty$ and $c(t)e^{-t}$ is decreasing on $(0,+\infty)$. Assume that
$$\sum_{\alpha\in E}\frac{(2\pi)^{n_1}\int_Y|f_{\alpha}|^2e^{-\varphi_Y-\left(N+\pi_{1,j}^*(\sum_{1\le j\le n_1}\varphi_j)\right)(z_0,w)}}{\prod_{1\le j\le n_1}(\alpha_j+1)c_{j}(z_j)^{2\alpha_{j}+2}}\in(0,+\infty).$$
Then there exists a holomorphic $(n,0)$ form $F$ on $\tilde{M}$ satisfying that $(F-f,z)\in\bigg(\mathcal{O}(K_{\tilde{M}})\otimes\mathcal{I}\big(\max_{1\le j\le n_1}\left\{2p_j\pi_{1,j}^{*}(G_{\Omega_j}(\cdot,z_j))\right\}\big)\bigg)_{z}$ for any $z\in Z_0$ and
\begin{displaymath}
	\begin{split}
	&\int_{\tilde{M}}|F|^2e^{-\varphi}c(-\psi)\\
	\le&\left(\int_0^{+\infty}c(s)e^{-s}ds\right)\sum_{\alpha\in E}\frac{(2\pi)^{n_1}\int_Y|f_{\alpha}|^2e^{-\varphi_Y-\left(N+\pi_{1,j}^*(\sum_{1\le j\le n_1}\varphi_j)\right)(z_o,w)}}{\prod_{1\le j\le n_1}(\alpha_j+1)c_{j}(z_j)^{2\alpha_{j}+2}}.	
	\end{split}
\end{displaymath}
	
	Moreover, equality $\inf\big\{\int_{\tilde{M}}|\tilde{F}|^2e^{-\varphi}c(-\psi):\tilde{F}\in H^0(\tilde{M},\mathcal{O}(K_{\tilde{M}}))\,\&\, (\tilde{F}-f,z)\in\bigg(\mathcal{O}(K_{\tilde{M}})\otimes\mathcal{I}\big(\max_{1\le j\le n_1}\left\{2p_j\pi_{1,j}^{*}(G_{\Omega_j}(\cdot,z_j))\right\}\big)\bigg)_{z}$ for any $z\in Z_0\big\}=\left(\int_0^{+\infty}c(s)e^{-s}ds\right)\times\sum_{\alpha\in E}\frac{(2\pi)^{n_1}\int_Y|f_{\alpha}|^2e^{-\varphi_Y-\left(N+\pi_{1,j}^*(\sum_{1\le j\le n_1}\varphi_j)\right)(z_0,w)}}{\prod_{1\le j\le n_1}(\alpha_j+1)c_{j}(z_j)^{2\alpha_{j}+2}}$ holds if and only if the following statements hold:

	$(1)$ $\tilde{M}=\left(\prod_{1\le j\le n_1}\Omega_j\right)\times Y$ and $N\equiv0$;

	$(2)$ $\varphi_j=2\log|g_j|+2u_j$, where $g_j$ is a holomorphic function on $\Omega_j$ such that $g_j(z_j)\not=0$ and $u_j$ is a harmonic function on $\Omega_j$ for any $1\le j\le n_1$;

    $(3)$ $\chi_{j,z_j}^{\alpha_j+1}=\chi_{j,-u_j}$ for any $j\in\{1,2,...,n\}$ and $\alpha\in E$ satisfying $f_{\alpha}\not\equiv 0$.
\end{Theorem}
\begin{Remark}
	\label{rem:exten-fibra-single}
If $(f_{\alpha},y)\in\big(\mathcal{O}(K_Y)\otimes\mathcal{I}(\varphi_Y)\big)_y$ for any $y\in Y$ and $\alpha\in \tilde E\backslash E$, the above result also holds when we replace  the ideal sheaf\
$\mathcal{I}\left(\max_{1\le j\le n_1}\left\{2p_j\pi_{1,j}^{*}\big(G_{\Omega_j}(\cdot,z_j)\big)\right\}\right)$  by $\mathcal{I}\big(\hat G+\pi_2^*(\varphi_2)\big)$. We prove the remark in Section \ref{section of multi extension}.
\end{Remark}

Let $Z_j=\{z_{j,1},\ldots,z_{j,m_j}\}\subset\Omega_j$ for any  $j\in\{1,\ldots,n_1\}$, where $m_j$ is a positive integer.
Let $f$ be a holomorphic $(n,0)$ form on a neighborhood $U_0\subset \tilde{M}$ of $Z_0$ such that
$$f=\sum_{\alpha\in\tilde E_{\beta}}\pi_1^*(w_{\beta}^{\alpha}dw_{1,\beta_1}\wedge\ldots\wedge dw_{n_1,\beta_{n_1}})\wedge\pi_2^*(f_{\alpha,\beta})$$ on $U_0\cap(V_{\beta}\times Y)$, where $f_{\alpha,\beta}$ is a holomorphic $(n_2,0)$ form on $Y$ for any $\alpha\in E_{\beta}$ and $\beta\in\tilde I_1$. Let $\beta^*=(1,\ldots,1)\in\tilde I_1$, and let $\alpha_{\beta^*}=(\alpha_{\beta^*,1},\ldots,\alpha_{\beta^*,n_1})\in E_{\beta^*}$. Denote that $E':=\left\{\alpha\in\mathbb{Z}_{\ge0}^{n_1}:\sum_{1\le j\le n_1}\frac{\alpha_j+1}{p_{j,1}}>1\right\}$. Assume that $f=\pi_1^*\left(w_{\beta^*}^{\alpha_{\beta^*}}dw_{1,1}\wedge\ldots\wedge dw_{n_1,1}\right)\wedge\pi_2^*\left(f_{\alpha_{\beta^*},\beta^*}\right)+\sum_{\alpha\in E'}\pi_1^*(w^{\alpha}dw_{1,1}\wedge\ldots\wedge dw_{n_1,1})\wedge\pi_2^*(f_{\alpha,\beta})$ on $U_0\cap(V_{\beta^*}\times Y)$.

We obtain a characterization of the holding of equality in optimal jets $L^2$ extension problem for the case that $Z_j$ is finite.

\begin{Theorem}
\label{thm:exten-fibra-finite}
Let $c$ be a positive function on $(0,+\infty)$ such that $\int_{0}^{+\infty}c(t)e^{-t}dt<+\infty$ and $c(t)e^{-t}$ is decreasing on $(0,+\infty)$. Assume that
$$\sum_{\beta\in \tilde{I}_1}\sum_{\alpha\in E_{\beta}}\frac{(2\pi)^{n_1}\int_Y|f_{\alpha,\beta}|^2e^{-\varphi_Y-\left(N+\pi_{1,j}^*(\sum_{1\le j\le n_1}\varphi_j)\right)(z_{\beta},w)}}{\prod_{1\le j\le n_1}(\alpha_j+1)c_{j,\beta_j}^{2\alpha_{j}+2}}\in(0,+\infty).$$
Then there exists a holomorphic $(n,0)$ form $F$ on $\tilde{M}$ satisfying that $(F-f,z)\in\left(\mathcal{O}(K_{\tilde{M}})\otimes\mathcal{I}\big(\max_{1\le j\le n_1}\left\{2\sum_{1\le k\le m_j}p_{j,k}\pi_{1,j}^{*}(G_{\Omega_j}(\cdot,z_{j,k}))\right\}\big)\right)_{z}$ for any $z\in Z_0$ and
\begin{displaymath}
	\begin{split}
	&\int_{\tilde{M}}|F|^2e^{-\varphi}c(-\psi)\\
	\le&\left(\int_0^{+\infty}c(s)e^{-s}ds\right)\sum_{\beta\in \tilde{I}_1}\sum_{\alpha\in E_{\beta}}\frac{(2\pi)^{n_1}\int_Y|f_{\alpha,\beta}|^2e^{-\varphi_Y-\left(N+\pi_{1,j}^*(\sum_{1\le j\le n_1}\varphi_j)\right)(z_{\beta},w)}}{\prod_{1\le j\le n_1}(\alpha_j+1)c_{j,\beta_j}^{2\alpha_{j}+2}}.	
	\end{split}
\end{displaymath}
	
	Moreover, equality $\inf\bigg\{\int_{\tilde{M}}|\tilde{F}|^2e^{-\varphi}c(-\psi):\tilde{F}\in H^0(\tilde{M},\mathcal{O}(K_{\tilde{M}}))\,\&\, (\tilde{F}-f,z)\in\left(\mathcal{O}(K_{\tilde{M}})\otimes\mathcal{I}\big(\max_{1\le j\le n_1}\left\{2\sum_{1\le k\le m_j}p_{j,k}\pi_{1,j}^{*}(G_{\Omega_j}(\cdot,z_{j,k}))\right\}\big)\right)_{z}$ for any $z\in Z_0\bigg\}=\left(\int_0^{+\infty}c(s)e^{-s}ds\right)\sum_{\beta\in \tilde{I}_1}\sum_{\alpha\in E_{\beta}}\frac{(2\pi)^{n_1}\int_Y|f_{\alpha,\beta}|^2e^{-\varphi_Y-\left(N+\pi_{1,j}^*(\sum_{1\le j\le n_1}\varphi_j)\right)(z_{\beta},w)}}{\prod_{1\le j\le n_1}(\alpha_j+1)c_{j,\beta_j}^{2\alpha_{j}+2}}$ holds if and only if the following statements hold:

	$(1)$ $\tilde{M}=\left(\prod_{1\le j\le n_1}\Omega_j\right)\times Y$ and $N\equiv0$;
	
	$(2)$ $\varphi_j=2\log|g_j|+2u_j$ for any $j\in\{1,\ldots,n_1\}$, where $u_j$ is a harmonic function on $\Omega_j$ and $g_j$ is a holomorphic function on $\Omega_j$ satisfying $g_j(z_{j,k})\not=0$ for any $k\in\{1,\ldots,m_j\}$;
	
	$(3)$ There exists a nonnegative integer $\gamma_{j,k}$ for any $j\in\{1,\ldots,n_1\}$ and $k\in\{1,\ldots,m_j\}$, which satisfies that $\prod_{1\le k\leq m_j}\chi_{j,z_{j,k}}^{\gamma_{j,k}+1}=\chi_{j,-u_j}$ and $\sum_{1\le j\le n_1}\frac{\gamma_{j,\beta_j}+1}{p_{j,\beta_j}}=1$ for any $\beta\in \tilde{I}_1$;
	
	$(4)$ $f_{\alpha,\beta}=c_{\beta}f_0$ holds for $\alpha=(\gamma_{1,\beta_1},\ldots,\gamma_{n_1,\beta_{n_1}})$ and $f_{\alpha,\beta}\equiv0$ holds for any $\alpha\in E_{\beta}\backslash\{(\gamma_{1,\beta_1},\ldots,\gamma_{n_1,\beta_{n_1}})\}$, where $\beta\in \tilde{I}_1$, $c_{\beta}$ is a constant and $f_0\not\equiv0$ is a holomorphic $(n_2,0)$ form on $Y$ satisfying $\int_Y|f_0|^2e^{-\varphi_2}<+\infty$;
		
	$(5)$ $c_{\beta}\prod_{1\le j\le n_1}\left(\lim_{z\rightarrow z_{j,\beta_j}}\frac{w_{j,\beta_j}^{\gamma_{j,\beta_j}}dw_{j,\beta_j}}{g_j(P_{j})_*\left(f_{u_j}\left(\prod_{1\le k\le m_j}f_{z_{j,k}}^{\gamma_{j,k}+1}\right)\left(\sum_{1\le k\le m_j}p_{j,k}\frac{df_{z_{j,k}}}{f_{z_{j,k}}}\right)\right)}\right)=c_0$ for any $\beta\in \tilde{I}_1$, where $c_0\in\mathbb{C}\backslash\{0\}$ is a constant independent of $\beta$, $f_{u_j}$ is a holomorphic function $\Delta$ such that $|f_{u_j}|=P_j^*(e^{u_j})$ and $f_{z_{j,k}}$ is a holomorphic function on $\Delta$ such that $|f_{z_{j,k}}|=P_j^*\left(e^{G_{\Omega_j}(\cdot,z_{j,k})}\right)$ for any $j\in\{1,\ldots,n_1\}$ and $k\in\{1,\ldots,m_j\}$.
\end{Theorem}
\begin{Remark}
	\label{rem:exten-fibra-finite} If $(f_{\alpha,\beta},y)\in(\mathcal{O}(K_Y)\otimes\mathcal{I}(\varphi_Y))_y$ holds for any $y\in Y$,  $\alpha\in \tilde E_{\beta}\backslash E_{\beta}$ and $\beta\in \tilde{I}_1$, the above result also holds when we replace  the ideal sheaf $\mathcal{I}\left(\max_{1\le j\le n_1}\left\{2\sum_{1\le k\le m_j}p_{j,k}\pi_{1,j}^{*}(G_{\Omega_j}(\cdot,z_{j,k}))\right\}\right)$  by $\mathcal{I}\big(\hat G+\pi_2^*(\varphi_2)\big)$. We prove the remark in Section \ref{section of multi extension}.
\end{Remark}

Let $Z_j=\{z_{j,k}:1\le k<\tilde m_j\}$ be a discrete subset of $\Omega_j$ for any  $j\in\{1,\ldots,n_1\}$, where $\tilde m_j\in\mathbb{Z}_{\ge2}\cup\{+\infty\}$. Let $f$ be a holomorphic $(n,0)$ form on a neighborhood $U_0\subset \tilde{M}$ of $Z_0$ such that
$$f=\sum_{\alpha\in\tilde E_{\beta}}\pi_1^*(w_{\beta}^{\alpha}dw_{1,\beta_1}\wedge\ldots\wedge dw_{n_1,\beta_{n_1}})\wedge\pi_2^*(f_{\alpha,\beta})$$ on $U_0\cap(V_{\beta}\times Y)$, where $f_{\alpha,\beta}$ is a holomorphic $(n_2,0)$ form on $Y$ for any $\alpha\in E_{\beta}$ and $\beta\in\tilde{I}_1$. Let $\beta^*=(1,\ldots,1)\in\tilde{I}_1$, and let $\alpha_{\beta^*}=(\alpha_{\beta^*,1},\ldots,\alpha_{\beta^*,n_1})\in E_{\beta^*}$. Denote that $E':=\left\{\alpha\in\mathbb{Z}_{\ge0}^{n_1}:\sum_{1\le j\le n_1}\frac{\alpha_j+1}{p_{j,1}}>1\right\}$. Assume that $f=\pi_1^*\left(w_{\beta^*}^{\alpha_{\beta^*}}dw_{1,1}\wedge\ldots\wedge dw_{n_1,1}\right)\wedge\pi_2^*\left(f_{\alpha_{\beta^*},\beta^*}\right)+\sum_{\alpha\in E'}\pi_1^*(w^{\alpha}dw_{1,1}\wedge\ldots\wedge dw_{n_1,1})\wedge\pi_2^*(f_{\alpha,\beta})$ on $U_0\cap(V_{\beta^*}\times Y)$.

When there exists $j_0\in\{1,\ldots,n_1\}$ such that $\tilde m_{j_0}=+\infty$,
we obtain that the equality in optimal jets $L^2$ extension problem \emph{ could not hold}.

\begin{Theorem}
\label{thm:exten-fibra-infinite}Let $c$ be a positive function on $(0,+\infty)$ such that $\int_{0}^{+\infty}c(t)e^{-t}dt<+\infty$ and $c(t)e^{-t}$ is decreasing on $(0,+\infty)$. Assume that
$$\sum_{\beta\in \tilde{I}_1}\sum_{\alpha\in E_{\beta}}\frac{(2\pi)^{n_1}\int_Y|f_{\alpha,\beta}|^2e^{-\varphi_Y-\left(N+\pi_{1,j}^*(\sum_{1\le j\le n_1}\varphi_j)\right)(z_{\beta},w)}}{\prod_{1\le j\le n_1}(\alpha_j+1)c_{j,\beta_j}^{2\alpha_{j}+2}}\in(0,+\infty)$$
and there exists $j_0\in\{1,\ldots,n_1\}$ such that $\tilde m_{j_0}=+\infty$.

Then there exists a holomorphic $(n,0)$ form $F$ on $\tilde{M}$ satisfying that $(F-f,z)\in\left(\mathcal{O}(K_{\tilde{M}})\otimes\mathcal{I}\left(\max_{1\le j\le n_1}\left\{2\sum_{1\le k<\tilde m_j}p_{j,k}\pi_{1,j}^{*}(G_{\Omega_j}(\cdot,z_{j,k}))\right\}\right)\right)_{z}$ for any $z\in Z_0$ and
\begin{displaymath}
	\begin{split}
	&\int_{\tilde{M}}|F|^2e^{-\varphi}c(-\psi)\\
<&\left(\int_0^{+\infty}c(s)e^{-s}ds\right)\sum_{\beta\in\tilde{I}_1}\sum_{\alpha\in E_{\beta}}\frac{(2\pi)^{n_1}\int_Y|f_{\alpha,\beta}|^2e^{-\varphi_Y-\left(N+\pi_{1,j}^*(\sum_{1\le j\le n_1}\varphi_j)\right)(z_{\beta},w)}}{\prod_{1\le j\le n_1}(\alpha_j+1)c_{j,\beta_j}^{2\alpha_{j}+2}}.	
	\end{split}
\end{displaymath}
\end{Theorem}
\begin{Remark}
	\label{rem:exten-fibra-infinite}If $(f_{\alpha,\beta},y)\in(\mathcal{O}(K_Y)\otimes\mathcal{I}(\varphi_Y))_y$ holds for any $y\in Y$, $\alpha\in \tilde E_{\beta}\backslash E_{\beta}$ and $\beta\in\tilde{I}_1$, the above result also holds when we replace  the ideal sheaf $\mathcal{I}\left(\max_{1\le j\le n_1}\left\{2\sum_{1\le k<\tilde m_j}p_{j,k}\pi_{1,j}^{*}(G_{\Omega_j}(\cdot,z_{j,k}))\right\}\right)$  by $\mathcal{I}\big(\hat G+\pi_2^*(\varphi_2)\big)$. We prove the remark in Section \ref{section of multi extension}.
\end{Remark}

\begin{Remark} We  note that, by using our recent progress (see \cite{GMY-boundary2}) on minimal $L^2$ integrals related to boundary points, the main results in the present article  hold without assuming the condtion ``$\{\psi<-t\}\backslash Z_0$ is a weakly pseudoconvex K\"ahler manifold for any $t\in\mathbb{R}$''.

\end{Remark}

\section{Preparations I: Minimal $L^2$ integrals}

In this section, we recall and present some lemmas related to minimal $L^{2}$ integrals.

\subsection{Minimal $L^2$ integrals on weakly pseudoconvex K\"{a}hler manifolds}
In this section, we recall some results about the concavity property of minimal $L^2$ integrals on weakly pseudoconvex K\"{a}hler manifolds in \cite{GMY}.

Let $M$ be an $n-$dimensional complex manifold. Let $X$ and $Z$ be closed subsets of $M$. A triple $(M,X,Z)$ satisfies condition $(A)$, if the following statements hold:
	
	(1). $X$ is a closed subset of $M$ and $X$ is locally neligible with respect to $L^2$ holomorphic functions, i.e., for any coordinated neighborhood $U\subset M$ and for any $L^2$ holomorphic function $f$ on $U\setminus X$, there exists an $L^2$ holomorphic function $\tilde{f}$ on $U$ such that $\tilde{f}|_{U\setminus X}=f$ with the same $L^2$ norm;
	
	(2). $Z$ is an analytic subset of $M$ and $M\setminus(X\cup Z)$ is a weakly pseudoconvex K\"{a}hler manifold.
	
	Let $(M,X,Z)$ be a triple satisfying condition $(A)$. Let $K_M$ be the canonical line bundle on $M$. Let $\psi$ be a plurisubharmonic function on $M$ such that $\{\psi<-t\}\backslash (X\cup Z)$ is a weakly pseudoconvex K\"ahler manifold for any $t\in\mathbb{R}$. Let $\varphi$ be a Lebesgue measurable function on $M$ such that $\varphi+\psi$ is a plurisubharmonic function on $M$. Denote $T=-\sup_M\psi$.
	
	We recall the concept of ``gain'' in \cite{GMY}. A positive measurable function $c$ on $(T,+\infty)$ is in the class $\mathcal{P}_{T,M}$ if the following two statements hold:
	
	(1). $c(t)e^{-t}$ is decreasing with respect to $t$;
	
	(2). there is a closed subset $E$ of $M$ such that $E\subset Z\cap\{\psi=-\infty\}$ and for any compact subset $K\subset M\setminus E$, $e^{-\varphi}c(-\psi)$ has a positive lower bound on $K$.
	
	Let $Z_0$ be a subset of $M$ such that $Z_0\cap\text{Supp} (\mathcal{O}/\mathcal{I}(\varphi+\psi))\neq\emptyset$. Let $U\supset Z_0$ be an open subset of $M$, and let $f$ be a holomorphic $(n,0)$ form on $U$. Let $\mathcal{F}_{z_0}\supset\mathcal{I}(\varphi+\psi)_{z_0}$ be an ideal of $\mathcal{O}_{z_0}$ for any $z_0\in Z_0$.
	
	Denote that
	\begin{flalign}
		\begin{split}
			G(t;c):=\inf\bigg\{\int_{\{\psi<-t\}}|\tilde{f}|^2e^{-\varphi}c(-\psi) : \tilde{f}\in H^0(\{\psi<-t\}, \mathcal{O}(K_M))&\\
			\& (\tilde{f}-f)\in H^0(Z_0, (\mathcal{O}(K_M)\otimes\mathcal{F})|_{Z_0})\bigg\}&,
		\end{split}
	\end{flalign}
	and
	\begin{flalign}
		\begin{split}
			\mathcal{H}^2(c,t):=\bigg\{\tilde{f} : \int_{\{\psi<-t\}}|\tilde{f}|^2e^{-\varphi}c(-\psi)<+\infty, \tilde{f}\in H^0(\{\psi<-t\}, \mathcal{O}(K_M))&\\
			\& (\tilde{f}-f)\in H^0(Z_0, (\mathcal{O}(K_M)\otimes\mathcal{F})|_{Z_0})\bigg\}&,
		\end{split}
	\end{flalign}
	where $t\in [T,+\infty)$, and $c$ is a nonnegative measurable function on $(T,+\infty)$. Here $|\tilde{f}|^2:=\sqrt{-1}^{n^2}\tilde{f}\wedge\bar{\tilde{f}}$ for any $(n,0)$ form $\tilde{f}$, and $(\tilde{f}-f)\in H^0(Z_0, (\mathcal{O}(K_M)\otimes\mathcal{F})|_{Z_0})$ means that $(\tilde{f}-f,z_0)\in (\mathcal{O}(K_M)\otimes\mathcal{F})_{z_0}$ for any $z_0\in Z_0$. If there is no holomorphic $(n,0)$ form $\tilde{f}$ on $\{\psi<-t\}$ satisfying $(\tilde{f}-f)\in H^0(Z_0, (\mathcal{O}(K_M)\otimes\mathcal{F})|_{Z_0})$, we set $G(t;c)=+\infty$.
	
	In \cite{GMY}, Guan-Mi-Yuan obtained the following concavity of $G(t;c)$.
	
	\begin{Theorem}[\cite{GMY}]\label{Concave}
		Let $c\in\mathcal{P}_{T,M}$ such that $\int_T^{+\infty}c(s)e^{-s}ds<+\infty$. If there exists $t\in [T,+\infty)$ satisfying that $G(t)<+\infty$, then $G(h^{-1}(r))$ is concave with respect to $r\in (0,\int_T^{+\infty}c(t)e^{-t}dt)$, $\lim\limits_{t\rightarrow T+0}G(t)=G(T)$ and $\lim\limits_{t\rightarrow +\infty}G(t)=0$, where $h(t)=\int_t^{+\infty}c(t_1)e^{-t_1}dt_1$.
	\end{Theorem}
\begin{Lemma}[\cite{GMY}]
\label{lem:A}Let $c\in\mathcal{P}_{T,M}$ satisfying $\int_T^{+\infty}c(s)e^{-s}ds<+\infty$.
Assume that $G(t)<+\infty$ for some $t\in[T,+\infty)$.
Then there exists a unique holomorphic $(n,0)$ form $F_{t}$ on
$\{\psi<-t\}$ satisfying $(F_{t}-f)\in H^{0}(Z_0,(\mathcal{O}(K_{M})\otimes\mathcal{F})|_{Z_0})$ and $\int_{\{\psi<-t\}}|F_{t}|^{2}e^{-\varphi}c(-\psi)=G(t)$.
Furthermore,
for any holomorphic $(n,0)$ form $\hat{F}$ on $\{\psi<-t\}$ satisfying $(\hat{F}-f)\in H^{0}(Z_0,(\mathcal{O}(K_{M})\otimes\mathcal{F})|_{Z_0})$ and
$\int_{\{\psi<-t\}}|\hat{F}|^{2}e^{-\varphi}c(-\psi)<+\infty$,
we have the following equality
\begin{equation}
\label{equ:20170913e}
\begin{split}
&\int_{\{\psi<-t\}}|F_{t}|^{2}e^{-\varphi}c(-\psi)+\int_{\{\psi<-t\}}|\hat{F}-F_{t}|^{2}e^{-\varphi}c(-\psi)
\\=&
\int_{\{\psi<-t\}}|\hat{F}|^{2}e^{-\varphi}c(-\psi).
\end{split}
\end{equation}
\end{Lemma}
	
	Guan-Mi-Yuan also obtained the following corollary of Theorem \ref{Concave}, which is a necessary condition for the concavity degenerating to linearity.
	
	\begin{Lemma}[\cite{GMY}]\label{linear}
		Let $c(t)\in\mathcal{P}_{T,M}$ such that $\int_T^{+\infty}c(s)e^{-s}ds<+\infty$. If $G(t)\in (0,+\infty)$ for some $t\geq T$ and $G(h^{-1}(r))$ is linear with respect to $r\in (0,\int_T^{+\infty}c(s)e^{-s}ds)$, where $h(t)=\int_t^{+\infty}c(l)e^{-l}dl$, then there exists a unique holomorphic $(n,0)$ form $F$ on $M$ satisfying $(F-f)\in H^0(Z_0,(\mathcal{O}(K_M)\otimes\mathcal{F})|_{Z_0})$, and $G(t;c)=\int_{\{\psi<-t\}}|F|^2e^{-\varphi}c(-\psi)$ for any $t\geq T$.
		
		Furthermore, we have
		\begin{equation}
			\int_{\{-t_2\leq\psi<-t_1\}}|F|^2e^{-\varphi}a(-\psi)=\frac{G(T_1;c)}{\int_{T_1}^{+\infty}c(t)e^{-t}dt}\int_{t_1}^{t_2}a(t)e^{-t}dt,
		\end{equation}
		for any nonnegative measurable function $a$ on $(T,+\infty)$, where $T\leq t_1<t_2\leq +\infty$.
		
		Especially, if $\mathcal H^2(\tilde{c},t_0)\subset\mathcal H^2(c,t_0)$ for some $t_0\geq T$, where $\tilde{c}$ is a nonnegative measurable function on $(T,+\infty)$, we have
		\begin{equation}
			G(t_0;\tilde{c})=\int_{\{\psi<-t_0\}}|F|^2e^{-\varphi}\tilde{c}(-\psi)=\frac{G(T_1;c)}{\int_{T_1}^{+\infty}c(s)e^{-s}ds}\int_{t_0}^{+\infty} \tilde{c}(s)e^{-s}ds.
		\end{equation}		
	\end{Lemma}
	
	\begin{Remark}[\cite{GMY}]\label{ctildec}
		Let $c(t)\in\mathcal{P}_{T,M}$. If $\mathcal{H}^2(\tilde{c},t_1)\subset\mathcal{H}^2(c,t_1)$, then $\mathcal{H}^2(\tilde{c},t_2)\subset\mathcal{H}^2(c,t_2)$, where $t_1>t_2>T$. In the following, we give some sufficient conditions of $\mathcal{H}^2(\tilde{c},t_0)\subset\mathcal{H}^2(c,t_0)$ for $t_0>T$:
		
		(1). $\tilde{c}\in\mathcal{P}_{T,M}$ and $\lim\limits_{t\rightarrow+\infty}\frac{\tilde{c}(t)}{c(t)}>0$. Especially, $\tilde{c}\in\mathcal{P}_{T,M}$, $c$ and $\tilde{c}$ are smooth on $(T,+\infty)$ and $\frac{d}{dt}(\log\tilde{c}(t))\geq\frac{d}{dt}(\log c(t))$;
		
		(2). $\tilde{c}\in\mathcal{P}_{T,M}$, $\mathcal{H}^2(c,t_0)\neq\emptyset$ and there exists $t>t_0$ such that $\{\psi<-t\}\Subset\{\psi<-t_0\}$, $\{z\in\overline{\{\psi<-t\}} : \mathcal{I}(\varphi+\psi)_z\neq\mathcal{O}_z\}\subset Z_0$ and $\mathcal{F}|_{\overline{\{\psi<-t\}}}=\mathcal{I}(\varphi+\psi)|_{\overline{\{\psi<-t\}}}$.
	\end{Remark}

\subsection{The sufficient and necessary conditions of the concavity of $G\big(h^{-1}(r)\big)$ degenerating to linearity}
In this section, we recall some result on the characterizations of the concavity of $G\big(h^{-1}(r)\big)$ degenerating to linearity on the fibrations over open Riemann surfaces and products of open Riemann surfaces.

The following result can be referred to \cite{BGY-concavity5}.

Let $Z_0^1:=\{z_j:j\in \mathbb{N} \& 1\leq j\leq m\}$ be a finite subset of the open Riemann surface $\Omega$. Let $Y$ be an $n-1$ dimensional weakly pseudoconvex K\"{a}hler manifold. Let $M=\Omega\times Y$ be a complex manifold, and $K_M$ be the canonical line bundle on $M$. Let $\pi_1$, $\pi_2$ be the natural projections from $M$ to $\Omega$ and $Y$ and $Z_0:=\pi_1^{-1}(Z_0^1)$. Let $\psi_1$ be a subharmonic function on $\Omega$ such that $p_j=\frac{1}{2}v(dd^c\psi_1,z_j)>0$, and let $\varphi_1$ be a Lebesgue measurable function on $\Omega$ such that $\varphi_1+\psi_1$ is subharmonic on $\Omega$. Let $\varphi_2$ be a plurisubharmonic function on $Y$. Denote that $\psi:=\pi_1^*(\psi_1)$, $\varphi:=\pi_1^*(\varphi_1)+\pi_2^*(\varphi_2)$.
	
	Let $w_j$ be a local coordinate on a neighborhood $V_{z_j}\Subset\Omega$ of $z_j$ satisfying $w_j(z_j)=0$ for $z_j\in Z_0^1$, where $V_{z_j}\cap V_{z_k}=\emptyset$ for any $j,k$, $j\neq k$. Denote that $V_0:=\bigcup_{1\leq j\leq m}V_{z_j}$. Let $f$ be a holomorphic $(n,0)$ form on $V_0\times Y$. Denote $\mathcal{F}_{(z_j,y)}=\mathcal{I}(\psi+\varphi)_{(z_j,y)}$ for any $(z_j,y)\in Z_0$. Let $G(t;c)$ be the minimal $L^2$ integral on $\{\psi<-t\}$ with respect to $\varphi$, $f$ and $\mathcal{F}$ for any $t\geq 0$.

	We recall a characterization of the concavity of $G\big(h^{-1}(r)\big)$ degenerating to linearity for the fibers over sets of finite points as follows.

	\begin{Theorem}[\cite{BGY-concavity5}]
\label{BGY-RESULT-FINITE POINTS}
		Assume that $G(0)\in (0,+\infty)$ and $\big(\psi_1-2p_jG_{\Omega}(\cdot,z_j)\big)(z_j)>-\infty$, where $p_j=\frac{1}{2}v(dd^c(\psi_1),z_j)>0$ for any $j\in\{1,2,\ldots,m\}$. Then $G\big(h^{-1}(r)\big)$ is linear with respect to $r\in (0,\int_0^{+\infty}c(t)e^{-t}dt]$ if and only if the following statements hold:
		
		(1). $\psi_1=2\sum\limits_{j=1}^mp_j G_{\Omega}(\cdot,z_j)$;
		
		(2). for any $j\in\{1,2,\ldots,m\}$, $f=\pi_1^*(a_jw_j^{k_j}dw_j)\wedge \pi_2^*(f_{Y})+f_j$ on $V_{z_j}\times Y$, where $a_j\in\mathbb{C}\setminus \{0\}$ is a constant, $k_j$ is a nonnegative integer, $f_{Y}$ is a holomorphic $(n-1,0)$ form on $Y$ such that $\int_{Y}|f_{Y}|^2e^{-\varphi_2}\in (0,+\infty)$, and $\big(f_j,(z_j,y)\big)\in \big(\mathcal{O}(K_M)\otimes\mathcal{I}(\varphi+\psi)\big)_{(z_j,y)}$ for any $j\in\{1,2,\ldots,m\}$ and $y\in Y$;
		
		(3). $\varphi_1+\psi_1=2\log |g|+2\sum\limits_{j=1}^mG_{\Omega}(\cdot,z_j)+2u$, where $g$ is a holomorphic function on $\Omega$ such that $ord_{z_j}(g)=k_j$ and $u$ is a harmonic function on $\Omega$;
		
		(4). $\prod\limits_{j=1}^m\chi_{z_j}=\chi_{-u}$, where $\chi_{-u}$ and $\chi_{z_j}$ are the characters associated to the functions $-u$ and $G_{\Omega}(\cdot,z_j)$ respectively;
		
		(5). for any $j\in\{1,2,\ldots,m\}$,
		\begin{equation}
			\lim_{z\rightarrow z_j}\frac{a_jw_j^{k_j}dw_j}{gP_*\left(f_u\left(\prod\limits_{l=1}^mf_{z_l}\right)\left(\sum\limits_{l=1}^mp_l\dfrac{d{f_{z_{l}}}}{f_{z_{l}}}\right)\right)}=c_0,
		\end{equation}
		where $c_0\in\mathbb{C}\setminus\{0\}$ is a constant independent of $j$, $f_{u}$ is a holomorphic function $\Delta$ such that $|f_{u}|=P^*(e^{u})$ and $f_{z_{j,k}}$ is a holomorphic function on $\Delta$ such that $|f_{z_{l}}|=P^*\left(e^{G_{\Omega}(\cdot,z_{l})}\right)$ for any $l\in\{1,\ldots,m\}$ .
	\end{Theorem}

When $Z_0^1:=\{z_j:j\in \mathbb{N} \& 1\leq j<+\infty\}$ is an infinite subset of the open Riemann surface $\Omega$ of discrete set. Assume that $2\sum\limits_{j=1 }^{+\infty}p_j G_{\Omega}(\cdot,z_j)\not\equiv -\infty$. We recall the following necessary condition such that $G\big(h^{-1}(r)\big)$ is linear.
\begin{Proposition}[\cite{BGY-concavity5}]
\label{BGY-RESULT-INFINITE POINTS}
		Assume that $G(0)\in (0,+\infty)$ and $\big(\psi_1-2p_jG_{\Omega}(\cdot,z_j)\big)(z_j)>-\infty$, where $p_j=\frac{1}{2}v(dd^c(\psi_1),z_j)>0$ for any $j\in\mathbb{N}_+$. Assume that $G\big(h^{-1}(r)\big)$ is linear with respect to $r\in (0,\int_0^{+\infty}c(t)e^{-t}dt]$, then the following statements hold:
		
		(1). $\psi_1=2\sum\limits_{j=1}^{\infty}p_j G_{\Omega}(\cdot,z_j)$;
		
		(2). for any $j\in\mathbb{N}_+$, $f=\pi_1^*(a_jw_j^{k_j}dw_j)\wedge \pi_2^*(f_{Y})+f_j$ on $V_{z_j}\times Y$, where $a_j\in\mathbb{C}\setminus \{0\}$ is a constant, $k_j$ is a nonnegative integer, $f_{Y}$ is a holomorphic $(n-1,0)$ form on $Y$ such that $\int_{Y}|f_{Y}|^2e^{-\varphi_2}\in (0,+\infty)$, and $\big(f_j,(z_j,y)\big)\in \big(\mathcal{O}(K_M)\otimes\mathcal{I}(\varphi+\psi)\big)_{(z_j,y)}$ for any $j\in\mathbb{N}_+$ and $y\in Y$;
		
		(3). $\varphi_1+\psi_1=2\log |g|$, where $g$ is a holomorphic function on $\Omega$ such that $ord_{z_j}(g)=k_j+1$ for any $j\in\mathbb{N}_+$;
		
		(4). for any $j\in\mathbb{N}_+$,
		\begin{equation}
			\frac{p_j}{ord_{z_j}g}\lim_{z\rightarrow z_j}\frac{dg}{a_jw_j^{k_j}dw_j}=c_0,
		\end{equation}
		where $c_0\in\mathbb{C}\setminus\{0\}$ is a constant independent of $j$;
		
		(5). $\sum\limits_{j\in\mathbb{N}_+}p_j<+\infty$.
	\end{Proposition}

	Let $Z_0^1:=\{z_j:j\in \mathbb{N} \& 1\leq j\leq m\}$ be a finite subset of the open Riemann surface $\Omega$. Let $Y$ be an $n-1$ dimensional weakly pseudoconvex K\"{a}hler manifold. Let $M=\Omega\times Y$ be a complex manifold, and $K_M$ be the canonical line bundle on $M$. Let $\pi_1$, $\pi_2$ be the natural projections from $M$ to $\Omega$ and $Y$, and $Z_0:=\pi_1^{-1}(Z_0^1)$.
	
	Let $w_j$ be a local coordinate on a neighborhood $V_{z_j}\Subset Omega$ of $z_j$ satisfying $w_j(z_j)=0$ for $z_j\in Z_0^1$, where $V_{z_j}\cap V_{z_k}=\emptyset$ for any $j,k$, $j\neq k$. Let $c_{\beta}(z)$ be the logarithmic capacity (see \cite{S-O69}) on $\Omega$ which is defined by
	\begin{equation*}
		c_{\beta}(z_j):=\exp \lim_{z\rightarrow z_0}(G_{\Omega}(z,z_j)-\log|w_j(z)|).
	\end{equation*}
	Denote that $V_0:=\bigcup_{1\leq j\leq m}V_{z_j}$. Assume that $\tilde{M}\subset M$ is an $n-$dimensional weakly pseudoconvex submanifold satisfying that $Z_0\subset \tilde{M}$.

 We recall the following characterization of the holding of the equality in optimal $L^2$ extension from fibers over analytic subsets to fibrations over open Riemann surfaces (see \cite{BGY-concavity5}).
	
	\begin{Theorem}[\cite{BGY-concavity5}]
\label{GBY-APP1}
		Let $k_j$ be a nonnegative integer for any $j\in\{1,2,...,m\}$. Let $\psi_1$ be a negative subharmonic function on $\Omega$ satisfying that   $\frac{1}{2}v(dd^{c}\psi_1,z_j)=p_j>0$ for any $j\in\{1,2,...,m\}$. Denote $\psi:=\pi_1^*(\psi_1)$. Let $\varphi_1$ be a Lebesgue measurable function on $\Omega$ such that $\varphi_1+\psi_1$ is subharmonic on $\Omega$, $\frac{1}{2}v(dd^c(\varphi_1+\psi_1),z_j)=k_j+1$ and $\alpha_j:=\big(\varphi_1+\psi_1-2(k_j+1)G_{\Omega}(\cdot,z_j)\big)(z_j)>-\infty$ for any $j$. Let $\varphi_2$ be a plurisubharmonic function on $Y$. Let $c(t)$ be a positive measurable function on $(0,+\infty)$ satisfying that $c(t)e^{-t}$ is decreasing on $(0,+\infty)$ and $\int_{0}^{+\infty}c(s)e^{-s}ds<+\infty$. Let $a_j$ be a constant for any $j$. Let $F_j$ be a holomorphic $(n-1,0)$ form on $Y$ such that $\int_Y|F_j|^2e^{-\varphi_2}<+\infty$ for any $j$.
		
		Let $f$ be a holomorphic $(n,0)$ form on $V_0\times Y$ satisfying that $f=\pi_1^*(a_jw_j^{k_j}dw_j)\wedge\pi_2^*(F_j)$ on $V_{z_j}\times Y$. Then there exists a holomorphic $(n,0)$ form $F$ on $\tilde{M}$ such that $(F-f,(z_j,y))\in(\mathcal{O}(K_{\tilde{M}})\otimes\mathcal{I}(\varphi+\psi))_{(z_j,y)}$ for any $(z_j,y)\in Z_0$ and
		\begin{equation}\label{L2result}
			\int_{\tilde{M}}|F|^2e^{-\varphi}c(-\psi)\leq\left(\int_0^{+\infty}c(s)e^{-s}ds\right)\sum_{j=1}^m\frac{2\pi|a_j|^2e^{-\alpha_j}}{p_jc_{\beta}(z_j)^{2(k_j+1)}}\int_Y|F_j|^2e^{-\varphi_2}.
		\end{equation}
		
		Moreover, equality $\left(\int_0^{+\infty}c(s)e^{-s}ds\right)\sum_{j=1}^m\frac{2\pi|a_j|^2e^{-\alpha_j}}{p_jc_{\beta}(z_j)^{2(k_j+1)}}\int_Y|F_j|^2e^{-\varphi_2}=\inf\big\{$ $ \int_M|\tilde{F}|^2e^{-\varphi}c(-\psi):\tilde{F}$ is a holomorphic $(n,0)$ form on $\tilde{M}$ such that $\big(\tilde{F}-f,(z_j,y)\big)\in\big(\mathcal{O}(K_{\tilde{M}})\otimes\mathcal{I}(\varphi+\psi)\big)_{(z_j,y)}$ for any $(z_j,y)\in Z_0\big\}$ holds if and only if the following statements hold:
		
		(1). $\psi_1=2\sum\limits_{j=1}^mp_j G_{\Omega}(\cdot,z_j)$;
		
		(2). $\varphi_1+\psi_1=2\log |g|+2\sum\limits_{j=1}^m(k_j+1)G_{\Omega}(\cdot,z_j)+2u$, where $g$ is a holomorphic function on $\Omega$ such that $g(z_j)\neq 0$ for any $j\in\{1,2,\ldots,m\}$ and $u$ is a harmonic function on $\Omega$;
		
		(3). $\prod\limits_{j=1}^m\chi_{z_j}^{k_j+1}=\chi_{-u}$, where $\chi_{-u}$ and $\chi_{z_j}$ are the characters associated to the functions $-u$ and $G_{\Omega}(\cdot,z_j)$ respectively;
		
		(4). for any $j\in\{1,2,\ldots,m\}$,
		\begin{equation}
			\lim_{z\rightarrow z_j}\frac{a_jw_j^{k_j}dw_j}{gP_*\left(f_u\left(\prod\limits_{l=1}^mf_{z_l}^{k_l+1}\right)\left(\sum\limits_{l=1}^mp_l\dfrac{d{f_{z_{l}}}}{f_{z_{l}}}\right)\right)}=c_j\in\mathbb{C}\setminus\{0\},
		\end{equation}
		and there exist $c_0\in\mathbb{C}\setminus\{0\}$ and a holomorphic $(n-1,0)$ form $F_Y$ on $Y$ which are independent of $j$ such that $c_0F_Y=c_jF_j$ for any $j\in\{1,2,\ldots,m\}$;
		
		(5). $\tilde{M}=M$.
		
	\end{Theorem}

 Let $Z_0^1:=\{z_j:1\le j <+\infty\}$ be an infinite discrete subset of the open Riemann surface $\Omega$. Let $Y$ be an $n-1$ dimensional weakly pseudoconvex K\"{a}hler manifold. Let $M=\Omega\times Y$ be a complex manifold, and $K_M$ be the canonical line bundle on $M$. Let $\pi_1$, $\pi_2$ be the natural projections from $M$ to $\Omega$ and $Y$, and $Z_0:=\pi_1^{-1}(Z_0^1)$.
	
	Let $w_j$ be a local coordinate on a neighborhood $V_{z_j}\Subset\Omega$ of $z_j$ satisfying $w_j(z_j)=0$ for $z_j\in Z_0^1$, where $V_{z_j}\cap V_{z_k}=\emptyset$ for any $j,k$, $j\neq k$. Denote that $V_0:=\bigcup_{j=1}^{\infty}V_{z_j}$.

 We recall the following $L^2$ extension result from fibers over analytic subsets to fibrations over open Riemann surfaces, where the analytic subsets are infinite points on open Riemann surfaces (see \cite{BGY-concavity5}).

	\begin{Theorem}[\cite{BGY-concavity5}]
\label{GBY-APP2}

Let $k_j$ be a nonnegative integer for any $j\in\mathbb{N}_+$. Let $\psi_1$ be a negative subharmonic function on $\Omega$ satisfying that   $\frac{1}{2}v(dd^{c}\psi_1,z_j)=k_j+1>0$ for any $j\in\mathbb{N}_+$. Denote $\psi:=\pi_1^*(\psi_1)$. Let $\varphi_1$ be a Lebesgue measurable function on $\Omega$ such that $\varphi_1+\psi_1$ is subharmonic on $\Omega$, $\frac{1}{2}v(dd^c(\varphi_1+\psi_1),z_j)=k_j+1$ and $\alpha_j:=\big(\varphi_1+\psi_1-2(k_j+1)G_{\Omega}(\cdot,z_j)\big)(z_j)>-\infty$ for any $j$. Let $\varphi_2$ be a plurisubharmonic function on $Y$. Let $c(t)$ be a positive measurable function on $(0,+\infty)$ satisfying that $c(t)e^{-t}$ is decreasing on $(0,+\infty)$ and $\int_{0}^{+\infty}c(s)e^{-s}ds<+\infty$. Let $a_j$ be a constant for any $j$. Let $F_j$ be a holomorphic $(n-1,0)$ form on $Y$ such that $\int_Y|F_j|^2e^{-\varphi_2}<+\infty$ for any $j$.
		
		Let $f$ be a holomorphic $(n,0)$ form on $V_0\times Y$ satisfying that $f=\pi_1^*(a_jw_j^{k_j}dw_j)\wedge\pi_2^*(F_j)$ on $V_{z_j}\times Y$.
		If
		\begin{equation*}
			\sum_{j=1}^{\infty}\frac{2\pi|a_j|^2e^{-\alpha_j}}{(k_j+1)c_{\beta}(z_j)^{2(k_j+1)}}\int_Y|F_j|^2e^{-\varphi_2}<+\infty,
		\end{equation*}	
		then there exists a holomorphic $(n,0)$ form $F$ on $M$ such that $(F-f,(z_j,y))\in(\mathcal{O}(K_M)\otimes\mathcal{I}(\varphi+\psi))_{(z_j,y)}$ for any $(z_j,y)\in Z_0$ and
		\begin{equation}\label{L2result-infinite}
			\int_M|F|^2e^{-\varphi}c(-\psi)<\left(\int_0^{+\infty}c(s)e^{-s}ds\right)\sum_{j=1}^{\infty}\frac{2\pi|a_j|^2e^{-\alpha_j}}{(k_j+1)c_{\beta}(z_j)^{2(k_j+1)}}\int_Y|F_j|^2e^{-\varphi_2}.
		\end{equation}
	\end{Theorem}

The following results can be referred to \cite{BGY-concavity6}.

Let $\Omega_j$  be an open Riemann surface, which admits a nontrivial Green function $G_{\Omega_j}$ for any  $1\le j\le n_1$. Let $Y$ be an $n_2-$dimensional weakly pseudoconvex K\"ahler manifold, and let $K_Y$ be the canonical (holomorphic) line bundle on $Y$. Let
$M=\left(\prod_{1\le j\le n_1}\Omega_j\right)\times Y$
 be an $n-$dimensional complex manifold, where $n=n_1+n_2$. Let $\pi_{1}$, $\pi_{1,j}$ and $\pi_2$ be the natural projections from $M$ to $\prod_{1\le j\le n_1}\Omega_j$, $\Omega_j$ and $Y$ respectively. Let $K_{M}$ be the canonical (holomorphic) line bundle on $M$.
Let $Z_j$ be a (closed) analytic subset of $\Omega_j$ for any $j\in\{1,\ldots,n_1\}$, and denote that $Z_0:=\left(\prod_{1\le j\le n_1}Z_j\right)\times Y$.

Let $Z_j=\{z_{j,k}:1\le k<\tilde m_j\}$ be a discrete subset of $\Omega_j$ for any  $j\in\{1,\ldots,n_1\}$, where $\tilde m_j\in\mathbb{Z}_{\ge2}\cup\{+\infty\}$.
Let $w_{j,k}$ be a local coordinate on a neighborhood $V_{z_{j,k}}\Subset\Omega_{j}$ of $z_{j,k}\in\Omega_j$ satisfying $w_{j,k}(z_{j,k})=0$ for any $1\le j\le n_1$ and $1\le k<\tilde m_j$, where $V_{z_{j,k}}\cap V_{z_{j,k'}}=\emptyset$ for any $j$ and $k\not=k'$. Denote that $\tilde I_1:=\big\{(\beta_1,\ldots,\beta_{n_1}):1\le \beta_j<\tilde m_j$ for any $j\in\{1,\ldots,n_1\}\big\}$, $V_{\beta}:=\prod_{1\le j\le n_1}V_{z_{j,\beta_j}}$  and $w_{\beta}:=(w_{1,\beta_1},\ldots,w_{n_1,\beta_{n_1}})$ is a local coordinate on $V_{\beta}$ of $z_{\beta}:=(z_{1,\beta_1},\ldots,z_{n_1,\beta_{n_1}})\in\prod_{1\le j\le n_1}\Omega_j$ for any $\beta=(\beta_1,\ldots,\beta_{n_1})\in\tilde I_1$. Then $Z_0=\{(z_{\beta},y):\beta\in\tilde I_1\,\&\,y\in Y\}\subset \tilde{M}$.

Let  $\varphi_j$ be a subharmonic function on $\Omega_j$ such that $\varphi_j(z_{j,k})>-\infty$ for any $1\le k\le \tilde{m}_j$. Let $\varphi_Y$ be a plurisubharmonic function on $Y$.

Let $p_{j,k}$ be a positive number for any $1\le j\le n_1$ and $1\le k<\tilde m_j$, which satisfies that $\sum_{1\le k<\tilde m_j}p_{j,k}G_{\Omega_j}(\cdot,z_{j,k})\not\equiv-\infty$ for any $1\le j\le n_1$.
Denote that
$$\psi:=\max_{1\le j\le n_1}\left\{2\sum_{1\le k<\tilde m_j}p_{j,k}\pi_{1,j}^{*}(G_{\Omega_j}(\cdot,z_{j,k}))\right\}$$
 and $\varphi:=\sum_{1\le j\le n_1}\pi_{1,j}^*(\varphi_j)+\pi_2^*(\varphi_Y)$ on $M$.

 Denote that $E_{\beta}:=\left\{(\alpha_1,\ldots,\alpha_{n_1}):\sum_{1\le j\le n_1}\frac{\alpha_j+1}{p_{j,\beta_j}}=1\,\&\,\alpha_j\in\mathbb{Z}_{\ge0}\right\}$ and $\tilde E_{\beta}:=\left\{(\alpha_1,\ldots,\alpha_{n_1}):\sum_{1\le j\le n_1}\frac{\alpha_j+1}{p_{j,\beta_j}}\ge1\,\&\,\alpha_j\in\mathbb{Z}_{\ge0}\right\}$ for any $\beta\in\tilde I_1$.

Let $f$ be a holomorphic $(n,0)$ form on a neighborhood $U_0\subset \tilde{M}$ of $Z_0$ such that
$$f=\sum_{\alpha\in\tilde E_{\beta}}\pi_1^*(w_{\beta}^{\alpha}dw_{1,\beta_1}\wedge\ldots\wedge dw_{n_1,\beta_{n_1}})\wedge\pi_2^*(f_{\alpha,\beta})$$ on $U_0\cap(V_{\beta}\times Y)$, where $f_{\alpha,\beta}$ is a holomorphic $(n_2,0)$ form on $Y$ for any $\alpha\in E_{\beta}$ and $\beta\in\tilde I_1$.

Let $Z_0=\{z_0\}\times Y=\{(z_1,\ldots,z_{n_1})\}\times Y\subset M$.
Let $$\psi=\max_{1\le j\le n_1}\left\{2p_j\pi_{1,j}^{*}\big(G_{\Omega_j}(\cdot,z_j)\big)\right\},$$ where $p_j$ is positive real number for $1\le j\le n_1$.
Let $w_j$ be a local coordinate on a neighborhood $V_{z_j}$ of $z_j\in\Omega_j$ satisfying $w_j(z_j)=0$. Denote that $V_0:=\prod_{1\le j\le n_1}V_{z_j}$, and $w:=(w_1,\ldots,w_{n_1})$ is a local coordinate on $V_0$ of $z_0\in \prod_{1\le j\le n_1}\Omega_j$. Denote that $E:=\left\{(\alpha_1,\ldots,\alpha_{n_1}):\sum_{1\le j\le n_1}\frac{\alpha_j+1}{p_j}=1\,\&\,\alpha_j\in\mathbb{Z}_{\ge0}\right\}$.
Let $f$ be a holomorphic $(n,0)$ form on $V_0\times Y\subset M$.

We recall a characterization of the concavity of $G\big(h^{-1}(r)\big)$ degenerating to linearity for the case $Z_0=\{z_0\}\times Y$.

\begin{Theorem}[\cite{BGY-concavity6}]
	\label{GBY6-single pt}
	Assume that $G(0)\in(0,+\infty)$.  $G\big(h^{-1}(r)\big)$ is linear with respect to $r\in(0,\int_{0}^{+\infty}c(t)e^{-t}dt]$  if and only if the  following statements hold:
	
	$(1)$ $f=\sum_{\alpha\in E}\pi_{1}^*\left(w^{\alpha}dw_1\wedge\ldots\wedge dw_{n_1}\right)\wedge \pi_2^*(f_{\alpha})+g_0$ on $V_0\times Y$, where  $g_0$ is a holomorphic $(n,0)$ form on $V_0\times Y$ satisfying $(g_0,z)\in\big(\mathcal{O}(K_M)\otimes\mathcal{I}(\varphi+\psi)\big)_{z}$ for any $z\in Z_0$ and $f_{\alpha}$ is a holomorphic $(n_2,0)$ form on $Y$ such that $\sum_{\alpha\in E}\int_{Y}|f_{\alpha}|^2e^{-\varphi_Y}\in(0,+\infty)$;
	
	$(2)$ $\varphi_j=2\log|g_j|+2u_j$, where $g_j$ is a holomorphic function on $\Omega_j$ such that $g_j(z_j)\not=0$ and $u_j$ is a harmonic function on $\Omega_j$ for any $1\le j\le n_1$;

    $(3)$ $\chi_{j,z_j}^{\alpha_j+1}=\chi_{j,-u_j}$ for any $j\in\{1,2,...,n\}$ and $\alpha\in E$ satisfying $f_{\alpha}\not\equiv 0$.
\end{Theorem}

	Let $Z_j=\{z_{j,1},\ldots,z_{j,m_j}\}\subset\Omega_j$ for any  $j\in\{1,\ldots,n_1\}$, where $m_j$ is a positive integer.
Let
$$\psi=\max_{1\le j\le n_1}\left\{\pi_{1,j}^*\left(2\sum_{1\le k\le m_j}p_{j,k}G_{\Omega_j}(\cdot,z_{j,k})\right)\right\},$$
where $p_{j,k}$ is a positive real number.
Let $w_{j,k}$ be a local coordinate on a neighborhood $V_{z_{j,k}}\Subset\Omega_{j}$ of $z_{j,k}\in\Omega_j$ satisfying $w_{j,k}(z_{j,k})=0$ for any $j\in\{1,\ldots,n_1\}$ and $k\in\{1,\ldots,m_j\}$, where $V_{z_{j,k}}\cap V_{z_{j,k'}}=\emptyset$ for any $j$ and $k\not=k'$. Denote that $\tilde{I}_1:=\big\{(\beta_1,\ldots,\beta_{n_1}):1\le \beta_j\le m_j$ for any $j\in\{1,\ldots,n_1\}\big\}$, $V_{\beta}:=\prod_{1\le j\le n_1}V_{z_{j,\beta_j}}$ for any $\beta=(\beta_1,\ldots,\beta_{n_1})\in \tilde{I}_1$ and $w_{\beta}:=(w_{1,\beta_1},\ldots,w_{n_1,\beta_{n_1}})$ is a local coordinate on $V_{\beta}$ of $z_{\beta}:=(z_{1,\beta_1},\ldots,z_{n_1,\beta_{n_1}})\in \prod_{1\le j\le n_1}\Omega_j$ satisfying $w_\beta(z_\beta)=0$.

Let $\beta^*=(1,\ldots,1)\in \tilde{I}_1$, and let $\alpha_{\beta^*}=(\alpha_{\beta^*,1},\ldots,\alpha_{\beta^*,n_1})\in\mathbb{Z}_{\ge0}^{n_1}$.
Denote that $E':=\left\{\alpha\in\mathbb{Z}_{\ge0}^{n_1}:\sum_{1\le j\le n_1}\frac{\alpha_j+1}{p_{j,1}}>\sum_{1\le j\le n_1}\frac{\alpha_{\beta^*,j}+1}{p_{j,1}}\right\}$. Let $f$ be a holomorphic $(n,0)$ form on $\cup_{\beta\in \tilde{I}_1}V_{\beta}\times Y$ satisfying $f=\pi_1^*\left(w_{\beta^*}^{\alpha_{\beta^*}}dw_{1,1}\wedge\ldots\wedge dw_{n_1,1}\right)\wedge\pi_2^*\left(f_{\alpha_{\beta^*}}\right)+\sum_{\alpha\in E'}\pi_1^*(w^{\alpha}dw_{1,1}\wedge\ldots\wedge dw_{n_1,1})\wedge\pi_2^*(f_{\alpha})$ on $V_{\beta^*}\times Y$, where $f_{\alpha_{\beta^*}}$ and $f_{\alpha}$ are  holomorphic $(n_2,0)$ forms on $Y$.

We recall a characterization of the concavity of $G\big(h^{-1}(r)\big)$ degenerating to linearity for the case  $Z_j$ is a set of finite points.

\begin{Theorem}[\cite{BGY-concavity6}]
	\label{GBY6-finitie pts}Assume that $G(0)\in(0,+\infty)$.  $G\big(h^{-1}(r)\big)$ is linear with respect to $r\in(0,\int_0^{+\infty} c(s)e^{-s}ds]$ if and only if the following statements hold:

	$(1)$ $\varphi_j=2\log|g_j|+2u_j$ for any $j\in\{1,\ldots,n_1\}$, where $u_j$ is a harmonic function on $\Omega_j$ and $g_j$ is a holomorphic function on $\Omega_j$ satisfying $g_j(z_{j,k})\not=0$ for any $k\in\{1,\ldots,m_j\}$;
	
	$(2)$ There exists a nonnegative integer $\gamma_{j,k}$ for any $j\in\{1,\ldots,n_1\}$ and $k\in\{1,\ldots,m_j\}$, which satisfies that $\prod_{1\le k\leq m_j}\chi_{j,z_{j,k}}^{\gamma_{j,k}+1}=\chi_{j,-u_j}$ and $\sum_{1\le j\le n_1}\frac{\gamma_{j,\beta_j}+1}{p_{j,\beta_j}}=1$ for any $\beta\in \tilde{I}_1$;
	
	$(3)$ $f=\pi_1^*\left(c_{\beta}\left(\prod_{1\le j\le n_1}w_{j,\beta_j}^{\gamma_{j,\beta_j}}\right)dw_{1,\beta_1}\wedge\ldots\wedge dw_{n,\beta_n}\right)\wedge\pi_2^*(f_0)+g_\beta$ on $V_{\beta}\times Y$ for any $\beta\in \tilde{I}_1$, where $c_{\beta}$ is a constant, $f_0\not\equiv0$ is a holomorphic $(n_2,0)$ form on $Y$ satisfying $\int_Y|f_0|^2e^{-\varphi_2}<+\infty$, and $g_{\beta}$ is a holomorphic $(n,0)$ form on $V_{\beta}\times Y$ such that $(g_{\beta},z)\in\big(\mathcal{O}(K_M)\otimes\mathcal{I}(\varphi+\psi)\big)_{z}$ for any $z\in\{z_\beta\}\times Y$;
	
	$(4)$ $c_{\beta}\prod_{1\le j\le n_1}\left(\lim_{z\rightarrow z_{j,\beta_j}}\frac{w_{j,\beta_j}^{\gamma_{j,\beta_j}}dw_{j,\beta_j}}{g_j(P_{j})_*\left(f_{u_j}\left(\prod_{1\le k\le m_j}f_{z_{j,k}}^{\gamma_{j,k}+1}\right)\left(\sum_{1\le k\le m_j}p_{j,k}\frac{df_{z_{j,k}}}{f_{z_{j,k}}}\right)\right)}\right)=c_0$ for any $\beta\in \tilde{I}_1$, where $c_0\in\mathbb{C}\backslash\{0\}$ is a constant independent of $\beta$, $f_{u_j}$ is a holomorphic function $\Delta$ such that $|f_{u_j}|=P_j^*(e^{u_j})$ and $f_{z_{j,k}}$ is a holomorphic function on $\Delta$ such that $|f_{z_{j,k}}|=P_j^*\left(e^{G_{\Omega_j}(\cdot,z_{j,k})}\right)$ for any $j\in\{1,\ldots,n_1\}$ and $k\in\{1,\ldots,m_j\}$.
\end{Theorem}

  Let ${Z}_j=\{z_{j,k}:1\le k<\tilde m_j\}$ be a discrete subset of $\Omega_j$ for any  $j\in\{1,\ldots,n_1\}$, where $\tilde{m}_j\in\mathbb{Z}_{\ge2}\cup\{+\infty\}$.
Let $p_{j,k}$ be a positive number for any $1\le j\le n_1$ and $1\le k<\tilde m_j$ such that  $\sum_{1\le k<\tilde{m}_j}p_{j,k}G_{\Omega_j}(\cdot,z_{j,k})\not\equiv-\infty$ for any $j$.
Let
$$\psi=\max_{1\le j\le n_1}\left\{\pi_{1,j}^*\left(2\sum_{1\le k<\tilde{m}_j}p_{j,k}G_{\Omega_j}(\cdot,z_{j,k})\right)\right\}.$$ Assume that $\limsup_{t\rightarrow+\infty}c(t)<+\infty$.

Let $w_{j,k}$ be a local coordinate on a neighborhood $V_{z_{j,k}}\Subset\Omega_{j}$ of $z_{j,k}\in\Omega_j$ satisfying $w_{j,k}(z_{j,k})=0$ for any $j\in\{1,\ldots,n_1\}$ and $1\le k<\tilde{m}_j$, where $V_{z_{j,k}}\cap V_{z_{j,k'}}=\emptyset$ for any $j$ and $k\not=k'$. Denote that $\tilde I_1:=\{(\beta_1,\ldots,\beta_{n_1}):1\le \beta_j< \tilde m_j$ for any $j\in\{1,\ldots,n_1\}\}$, $V_{\beta}:=\prod_{1\le j\le n_1}V_{z_{j,\beta_j}}$ for any $\beta=(\beta_1,\ldots,\beta_{n_1})\in\tilde I_1$ and $w_{\beta}:=(w_{1,\beta_1},\ldots,w_{n_1,\beta_{n_1}})$ is a local coordinate on $V_{\beta}$ of $z_{\beta}:=(z_{1,\beta_1},\ldots,z_{n_1,\beta_{n_1}})\in \prod_{1\le j\le n_1}\Omega_j$.

Let $\beta^*=(1,\ldots,1)\in I_1$, and let $\alpha_{\beta^*}=(\alpha_{\beta^*,1},\ldots,\alpha_{\beta^*,n_1})\in\mathbb{Z}_{\ge0}^{n_1}$. Denote that $E':=\left\{\alpha\in\mathbb{Z}_{\ge0}^{n_1}:\sum_{1\le j\le n_1}\frac{\alpha_j+1}{p_{j,1}}>\sum_{1\le j\le n_1}\frac{\alpha_{\beta^*,j}+1}{p_{j,1}}\right\}$.
Let $f$ be a holomorphic $(n,0)$ form on $\cup_{\beta\in I_1}V_{\beta}\times Y$ satisfying $f=\pi_1^*\left(w_{\beta^*}^{\alpha_{\beta^*}}dw_{1,1}\wedge\ldots\wedge dw_{n_1,1}\right)\wedge\pi_2^*\left(f_{\alpha_{\beta^*}}\right)+\sum_{\alpha\in E'}\pi_1^*(w^{\alpha}dw_{1,1}\wedge\ldots\wedge dw_{n_1,1})\wedge\pi_2^*(f_{\alpha})$ on $V_{\beta^*}\times Y$, where $f_{\alpha_{\beta^*}}$ and $f_{\alpha}$ are  holomorphic $(n_2,0)$ forms on $Y$.

We recall that $G\big(h^{-1}(r)\big)$ is not linear when there exists $j_0\in\{1,\ldots,n_1\}$ such that $\tilde m_{j_0}=+\infty$ as follows.

\begin{Theorem}[\cite{BGY-concavity6}]
	\label{GBY6-infinite points}If $G(0)\in(0,+\infty)$ and there exists $j_0\in\{1,\ldots,n_1\}$ such that $\tilde m_{j_0}=+\infty$, then $G\big(h^{-1}(r)\big)$ is not linear with respect to $r\in(0,\int_0^{+\infty} c(s)e^{-s}ds]$.
\end{Theorem}

Let $\tilde{M}\subset M$ be an $n-$dimensional complex manifold satisfying that $Z_0\subset \tilde{M}$, and let $K_{\tilde{M}}$ be the canonical (holomorphic) line bundle on $\tilde{M}$.

Let $\Psi\le0$ be a  plurisubharmonic function on $\prod_{1\le j\le n_1}\Omega_j$, and let $\varphi_j$ be a Lebesgue measurable function on $\Omega_j$ such that $\Psi+\sum_{1\le j\le n_1}\tilde\pi_{j}^*(\varphi_j)$ is plurisubharmonic on $\prod_{1\le j\le n_1}\Omega_j$, where $\tilde\pi_j$ is the natural projection from $\prod_{1\le j\le n_1}\Omega_j$ to $\Omega_j$. Let $\varphi_Y$ be a plurisubharmonic function on $Y$.
Let $p_{j,k}$ be a positive number for any $1\le j\le n_1$ and $1\le k<\tilde m_j$, which satisfies that $\sum_{1\le k<\tilde m_j}p_{j,k}G_{\Omega_j}(\cdot,z_{j,k})\not\equiv-\infty$ for any $1\le j\le n_1$.
Denote that
$$\psi:=\max_{1\le j\le n_1}\left\{2\sum_{1\le k<\tilde m_j}p_{j,k}\pi_{1,j}^{*}\big(G_{\Omega_j}(\cdot,z_{j,k})\big)\right\}+\pi_1^*(\Psi)$$
 and $\varphi:=\sum_{1\le j\le n_1}\pi_{1,j}^*(\varphi_j)+\pi_2^*(\varphi_Y)$ on $M$.

 Denote that $E_{\beta}:=\left\{(\alpha_1,\ldots,\alpha_{n_1}):\sum_{1\le j\le n_1}\frac{\alpha_j+1}{p_{j,\beta_j}}=1\,\&\,\alpha_j\in\mathbb{Z}_{\ge0}\right\}$ and $\tilde E_{\beta}:=\left\{(\alpha_1,\ldots,\alpha_{n_1}):\sum_{1\le j\le n_1}\frac{\alpha_j+1}{p_{j,\beta_j}}\ge1\,\&\,\alpha_j\in\mathbb{Z}_{\ge0}\right\}$ for any $\beta\in\tilde I_1$.
Let $f$ be a holomorphic $(n,0)$ form on a neighborhood $U_0\subset \tilde{M}$ of $Z_0$ such that
$$f=\sum_{\alpha\in\tilde E_{\beta}}\pi_1^*(w_{\beta}^{\alpha}dw_{1,\beta_1}\wedge\ldots\wedge dw_{n_1,\beta_{n_1}})\wedge\pi_2^*(f_{\alpha,\beta})$$ on $U_0\cap(V_{\beta}\times Y)$, where $f_{\alpha,\beta}$ is a holomorphic $(n_2,0)$ form on $Y$ for any $\alpha\in E_{\beta}$ and $\beta\in\tilde I_1$.
Denote that
\begin{equation*}
c_{j,k}:=\exp\lim_{z\rightarrow z_{j,k}}\left(\frac{\sum_{1\le k_1<\tilde m_j}p_{j,k_1}G_{\Omega_j}(z,z_{j,k_1})}{p_{j,k}}-\log|w_{j,k}(z)|\right)
\end{equation*}
 for any $j\in\{1,\ldots,n\}$ and $1\le k<\tilde m_j$ (following from Lemma \ref{l:green-sup} and Lemma \ref{l:green-sup2}, we get that the above limit exists).

When $Z_0=\{z_0\}\times Y\subset \tilde{M}$, where $z_0=(z_1,\ldots,z_{n_1})\in\prod_{1\le j\le n_1}\Omega_j$.

We recall a characterization of the holding of equality in optimal jets $L^2$ extension problem for the case $Z_0=\{z_0\}\times Y$.

\begin{Theorem}[\cite{BGY-concavity6}]
\label{GBY6-exten-fibra-single}
Let $c$ be a positive function on $(0,+\infty)$ such that $\int_{0}^{+\infty}c(t)e^{-t}dt<+\infty$ and $c(t)e^{-t}$ is decreasing on $(0,+\infty)$. Assume that
$$\sum_{\alpha\in E}\frac{(2\pi)^{n_1}e^{-\left(\Psi+\sum_{1\le j\le n_1}\tilde\pi_{j}^*(\varphi_j)\right)(z_0)}\int_Y|f_{\alpha}|^2e^{-\varphi_Y}}{\prod_{1\le j\le n_1}(\alpha_j+1)c_{j}(z_j)^{2\alpha_{j}+2}}\in(0,+\infty).$$
Then there exists a holomorphic $(n,0)$ form $F$ on $\tilde{M}$ satisfying that $(F-f,z)\in\left(\mathcal{O}(K_{\tilde{M}})\otimes\mathcal{I}\left(\max_{1\le j\le n_1}\left\{2p_j\pi_{1,j}^{*}(G_{\Omega_j}(\cdot,z_j))\right\}\right)\right)_{z}$ for any $z\in Z_0$ and
\begin{displaymath}
	\begin{split}
	&\int_{\tilde{M}}|F|^2e^{-\varphi}c(-\psi)\\
	\le&\left(\int_0^{+\infty}c(s)e^{-s}ds\right)\sum_{\alpha\in E}\frac{(2\pi)^{n_1}e^{-\left(\Psi+\sum_{1\le j\le n_1}\tilde\pi_{j}^*(\varphi_j)\right)(z_0)}\int_Y|f_{\alpha}|^2e^{-\varphi_Y}}{\prod_{1\le j\le n_1}(\alpha_j+1)c_{j}(z_j)^{2\alpha_{j}+2}}.	
	\end{split}
\end{displaymath}
	
	Moreover, equality $\inf\big\{\int_{\tilde{M}}|\tilde{F}|^2e^{-\varphi}c(-\psi):\tilde{F}\in H^0(\tilde{M},\mathcal{O}(K_{\tilde{M}}))\,\&\, (\tilde{F}-f,z)\in(\mathcal{O}\left(K_{\tilde{M}})\otimes\mathcal{I}\left(\max_{1\le j\le n_1}\left\{2p_j\pi_{1,j}^{*}(G_{\Omega_j}(\cdot,z_j))\right\}\right)\right)_{z}$ for any $z\in Z_0\big\}=\left(\int_0^{+\infty}c(s)e^{-s}ds\right)\times\sum_{\alpha\in E}\frac{(2\pi)^{n_1}e^{-\left(\Psi+\sum_{1\le j\le n_1}\tilde\pi_{j}^*(\varphi_j)\right)(z_0)}\int_Y|f_{\alpha}|^2e^{-\varphi_Y}}{\prod_{1\le j\le n_1}(\alpha_j+1)c_{j}(z_j)^{2\alpha_{j}+2}}$ holds if and only if the following statements hold:

	$(1)$ $\tilde{M}=\left(\prod_{1\le j\le n_1}\Omega_j\right)\times Y$ and $\Psi\equiv0$;

	$(2)$ $\varphi_j=2\log|g_j|+2u_j$, where $g_j$ is a holomorphic function on $\Omega_j$ such that $g_j(z_j)\not=0$ and $u_j$ is a harmonic function on $\Omega_j$ for any $1\le j\le n_1$;

    $(3)$ $\chi_{j,z_j}^{\alpha_j+1}=\chi_{j,-u_j}$ for any $j\in\{1,2,...,n\}$ and $\alpha\in E$ satisfying $f_{\alpha}\not\equiv 0$.
\end{Theorem}
\begin{Remark}[\cite{BGY-concavity6}]
	\label{GBY6-exten-fibra-single-rem}If $(f_{\alpha},y)\in(\mathcal{O}(K_Y)\otimes\mathcal{I}(\varphi_Y))_y$ for any $y\in Y$ and $\alpha\in \tilde E\backslash E$, the above result also holds when
 we replace  the ideal sheaf\ \  $\mathcal{I}\left(\max_{1\le j\le n_1}\left\{2p_j\pi_{1,j}^{*}(G_{\Omega_j}(\cdot,z_j))\right\}\right)$  by $\mathcal{I}(\varphi+\psi)$.
\end{Remark}

Let $Z_j=\{z_{j,1},\ldots,z_{j,m_j}\}\subset\Omega_j$ for any  $j\in\{1,\ldots,n_1\}$, where $m_j$ is a positive integer.
Let $f$ be a holomorphic $(n,0)$ form on a neighborhood $U_0\subset \tilde{M}$ of $Z_0$ such that
$$f=\sum_{\alpha\in\tilde E_{\beta}}\pi_1^*(w_{\beta}^{\alpha}dw_{1,\beta_1}\wedge\ldots\wedge dw_{n_1,\beta_{n_1}})\wedge\pi_2^*(f_{\alpha,\beta})$$ on $U_0\cap(V_{\beta}\times Y)$, where $f_{\alpha,\beta}$ is a holomorphic $(n_2,0)$ form on $Y$ for any $\alpha\in E_{\beta}$ and $\beta\in\tilde I_1$. Let $\beta^*=(1,\ldots,1)\in\tilde I_1$, and let $\alpha_{\beta^*}=(\alpha_{\beta^*,1},\ldots,\alpha_{\beta^*,n_1})\in E_{\beta^*}$.  Assume that $f=\pi_1^*\left(w_{\beta^*}^{\alpha_{\beta^*}}dw_{1,1}\wedge\ldots\wedge dw_{n_1,1}\right)\wedge\pi_2^*\left(f_{\alpha_{\beta^*},\beta^*}\right)+\sum_{\alpha\in E'}\pi_1^*(w^{\alpha}dw_{1,1}\wedge\ldots\wedge dw_{n_1,1})\wedge\pi_2^*(f_{\alpha,\beta})$ on $U_0\cap(V_{\beta^*}\times Y)$.

We recall a characterization of the holding of equality in optimal jets $L^2$ extension problem for the case that $Z_j$ is finite.

\begin{Theorem}[\cite{BGY-concavity6}]
\label{GBY6:exten-fibra-finite}
Let $c$ be a positive function on $(0,+\infty)$ such that $\int_{0}^{+\infty}c(t)e^{-t}dt<+\infty$ and $c(t)e^{-t}$ is decreasing on $(0,+\infty)$. Assume that
$$\sum_{\beta\in \tilde{I}_1}\sum_{\alpha\in E_{\beta}}\frac{(2\pi)^{n_1}e^{-\left(\Psi+\sum_{1\le j\le n_1}\tilde\pi_j^*(\varphi_j)\right)(z_{\beta})}\int_Y|f_{\alpha,\beta}|^2e^{-\varphi_Y}}{\prod_{1\le j\le n_1}(\alpha_j+1)c_{j,\beta_j}^{2\alpha_{j}+2}}\in(0,+\infty).$$
Then there exists a holomorphic $(n,0)$ form $F$ on $\tilde{M}$ satisfying that $(F-f,z)\in\left(\mathcal{O}(K_{\tilde{M}})\otimes\mathcal{I}\left(\max_{1\le j\le n_1}\left\{2\sum_{1\le k\le m_j}p_{j,k}\pi_{1,j}^{*}(G_{\Omega_j}(\cdot,z_{j,k}))\right\}\right)\right)_{z}$ for any $z\in Z_0$ and
\begin{displaymath}
	\begin{split}
	&\int_{\tilde{M}}|F|^2e^{-\varphi}c(-\psi)\\
	\le&\left(\int_0^{+\infty}c(s)e^{-s}ds\right)\sum_{\beta\in \tilde{I}_1}\sum_{\alpha\in E_{\beta}}\frac{(2\pi)^{n_1}e^{-\left(\Psi+\sum_{1\le j\le n_1}\tilde\pi_j^*(\varphi_j)\right)(z_{\beta})}\int_Y|f_{\alpha,\beta}|^2e^{-\varphi_Y}}{\prod_{1\le j\le n_1}(\alpha_j+1)c_{j,\beta_j}^{2\alpha_{j}+2}}.	
	\end{split}
\end{displaymath}
	
	Moreover, equality $\inf\bigg\{\int_{\tilde{M}}|\tilde{F}|^2e^{-\varphi}c(-\psi):\tilde{F}\in H^0(\tilde{M},\mathcal{O}(K_{\tilde{M}}))\,\&\, (\tilde{F}-f,z)\in\left(\mathcal{O}(K_{\tilde{M}})\otimes\mathcal{I}\left(\max_{1\le j\le n_1}\left\{2\sum_{1\le k\le m_j}p_{j,k}\pi_{1,j}^{*}(G_{\Omega_j}(\cdot,z_{j,k}))\right\}\right)\right)_{z}$ for any $z\in Z_0\bigg\}=\left(\int_0^{+\infty}c(s)e^{-s}ds\right)\sum_{\beta\in \tilde{I}_1}\sum_{\alpha\in E_{\beta}}\frac{(2\pi)^{n_1}e^{-\left(\Psi+\sum_{1\le j\le n_1}\tilde\pi_j^*(\varphi_j)\right)(z_{\beta})}\int_Y|f_{\alpha,\beta}|^2e^{-\varphi_Y}}{\prod_{1\le j\le n_1}(\alpha_j+1)c_{j,\beta_j}^{2\alpha_{j}+2}}$ holds if and only if the following statements hold:

	$(1)$ $\tilde{M}=\left(\prod_{1\le j\le n_1}\Omega_j\right)\times Y$ and $\Psi\equiv0$;
	
	$(2)$ $\varphi_j=2\log|g_j|+2u_j$ for any $j\in\{1,\ldots,n_1\}$, where $u_j$ is a harmonic function on $\Omega_j$ and $g_j$ is a holomorphic function on $\Omega_j$ satisfying $g_j(z_{j,k})\not=0$ for any $k\in\{1,\ldots,m_j\}$;
	
	$(3)$ There exists a nonnegative integer $\gamma_{j,k}$ for any $j\in\{1,\ldots,n_1\}$ and $k\in\{1,\ldots,m_j\}$, which satisfies that $\prod_{1\le k\leq m_j}\chi_{j,z_{j,k}}^{\gamma_{j,k}+1}=\chi_{j,-u_j}$ and $\sum_{1\le j\le n_1}\frac{\gamma_{j,\beta_j}+1}{p_{j,\beta_j}}=1$ for any $\beta\in \tilde{I}_1$;
	
	$(4)$ $f_{\alpha,\beta}=c_{\beta}f_0$ holds for $\alpha=(\gamma_{1,\beta_1},\ldots,\gamma_{n_1,\beta_{n_1}})$ and $f_{\alpha,\beta}\equiv0$ holds for any $\alpha\in E_{\beta}\backslash\{(\gamma_{1,\beta_1},\ldots,\gamma_{n_1,\beta_{n_1}})\}$, where $\beta\in \tilde{I}_1$, $c_{\beta}$ is a constant and $f_0\not\equiv0$ is a holomorphic $(n_2,0)$ form on $Y$ satisfying $\int_Y|f_0|^2e^{-\varphi_2}<+\infty$;
		
	$(5)$ $c_{\beta}\prod_{1\le j\le n_1}\left(\lim_{z\rightarrow z_{j,\beta_j}}\frac{w_{j,\beta_j}^{\gamma_{j,\beta_j}}dw_{j,\beta_j}}{g_j(P_{j})_*\left(f_{u_j}\left(\prod_{1\le k\le m_j}f_{z_{j,k}}^{\gamma_{j,k}+1}\right)\left(\sum_{1\le k\le m_j}p_{j,k}\frac{df_{z_{j,k}}}{f_{z_{j,k}}}\right)\right)}\right)=c_0$ for any $\beta\in \tilde{I}_1$, where $c_0\in\mathbb{C}\backslash\{0\}$ is a constant independent of $\beta$, $f_{u_j}$ is a holomorphic function $\Delta$ such that $|f_{u_j}|=P_j^*(e^{u_j})$ and $f_{z_{j,k}}$ is a holomorphic function on $\Delta$ such that $|f_{z_{j,k}}|=P_j^*\left(e^{G_{\Omega_j}(\cdot,z_{j,k})}\right)$ for any $j\in\{1,\ldots,n_1\}$ and $k\in\{1,\ldots,m_j\}$.
\end{Theorem}

\begin{Remark}[\cite{BGY-concavity6}]
	\label{GBY6:exten-fibra-finite-rem} If $(f_{\alpha,\beta},y)\in(\mathcal{O}(K_Y)\otimes\mathcal{I}(\varphi_Y))_y$ holds for any $y\in Y$,  $\alpha\in \tilde E_{\beta}\backslash E_{\beta}$ and $\beta\in \tilde{I}_1$, the above result also holds when we replace  the ideal sheaf $\mathcal{I}\left(\max_{1\le j\le n_1}\left\{2\sum_{1\le k\le m_j}p_{j,k}\pi_{1,j}^{*}(G_{\Omega_j}(\cdot,z_{j,k}))\right\}\right)$  by $\mathcal{I}(\varphi+\psi)$.
\end{Remark}

Let $Z_j=\{z_{j,k}:1\le k<\tilde m_j\}$ be a discrete subset of $\Omega_j$ for any  $j\in\{1,\ldots,n_1\}$, where $\tilde m_j\in\mathbb{Z}_{\ge2}\cup\{+\infty\}$. Let $f$ be a holomorphic $(n,0)$ form on a neighborhood $U_0\subset\tilde{M}$ of $Z_0$ such that
$$f=\sum_{\alpha\in\tilde E_{\beta}}\pi_1^*(w_{\beta}^{\alpha}dw_{1,\beta_1}\wedge\ldots\wedge dw_{n_1,\beta_{n_1}})\wedge\pi_2^*(f_{\alpha,\beta})$$ on $U_0\cap(V_{\beta}\times Y)$, where $f_{\alpha,\beta}$ is a holomorphic $(n_2,0)$ form on $Y$ for any $\alpha\in E_{\beta}$ and $\beta\in\tilde I_1$. Let $\beta^*=(1,\ldots,1)\in\tilde I_1$, and let $\alpha_{\beta^*}=(\alpha_{\beta^*,1},\ldots,\alpha_{\beta^*,n_1})\in E_{\beta^*}$.  Assume that $f=\pi_1^*\left(w_{\beta^*}^{\alpha_{\beta^*}}dw_{1,1}\wedge\ldots\wedge dw_{n_1,1}\right)\wedge\pi_2^*\left(f_{\alpha_{\beta^*},\beta^*}\right)+\sum_{\alpha\in E'}\pi_1^*(w^{\alpha}dw_{1,1}\wedge\ldots\wedge dw_{n_1,1})\wedge\pi_2^*(f_{\alpha,\beta})$ on $U_0\cap(V_{\beta^*}\times Y)$.

We recall that the equality in optimal jets $L^2$ extension problem could not hold when there exists $j_0\in\{1,\ldots,n_1\}$ such that $\tilde m_{j_0}=+\infty$.

\begin{Theorem}[\cite{BGY-concavity6}]
\label{GBY6:exten-fibra-infinite}Let $c$ be a positive function on $(0,+\infty)$ such that $\int_{0}^{+\infty}c(t)e^{-t}dt<+\infty$ and $c(t)e^{-t}$ is decreasing on $(0,+\infty)$. Assume that
$$\sum_{\beta\in\tilde I_1}\sum_{\alpha\in E_{\beta}}\frac{(2\pi)^{n_1}e^{-\left(\Psi+\sum_{1\le j\le n_1}\tilde\pi_j^*(\varphi_j)\right)(z_{\beta})}\int_Y|f_{\alpha,\beta}|^2e^{-\varphi_Y}}{\prod_{1\le j\le n_1}(\alpha_j+1)c_{j,\beta_j}^{2\alpha_{j}+2}}\in(0,+\infty)$$
and there exists $j_0\in\{1,\ldots,n_1\}$ such that $\tilde m_{j_0}=+\infty$.

Then there exists a holomorphic $(n,0)$ form $F$ on $\tilde{M}$ satisfying that $(F-f,z)\in\left(\mathcal{O}(K_{\tilde{M}})\otimes\mathcal{I}\left(\max_{1\le j\le n_1}\left\{2\sum_{1\le k<\tilde m_j}p_{j,k}\pi_{1,j}^{*}(G_{\Omega_j}(\cdot,z_{j,k}))\right\}\right)\right)_{z}$ for any $z\in Z_0$ and
\begin{displaymath}
	\begin{split}
	&\int_{\tilde{M}}|F|^2e^{-\varphi}c(-\psi)\\
<&\left(\int_0^{+\infty}c(s)e^{-s}ds\right)\sum_{\beta\in\tilde I_1}\sum_{\alpha\in E_{\beta}}\frac{(2\pi)^{n_1}e^{-\left(\Psi+\sum_{1\le j\le n_1}\tilde\pi_j^*(\varphi_j)\right)(z_{\beta})}\int_Y|f_{\alpha,\beta}|^2e^{-\varphi_Y}}{\prod_{1\le j\le n_1}(\alpha_j+1)c_{j,\beta_j}^{2\alpha_{j}+2}}.	
	\end{split}
\end{displaymath}
\end{Theorem}

\begin{Remark}[\cite{BGY-concavity6}]
	\label{GBY6:exten-fibra-infinite-rem}If $(f_{\alpha,\beta},y)\in(\mathcal{O}(K_Y)\otimes\mathcal{I}(\varphi_Y))_y$ holds for any $y\in Y$, $\alpha\in \tilde E_{\beta}\backslash E_{\beta}$ and $\beta\in\tilde I_1$, the above result also holds when we replace  the ideal sheaf $\mathcal{I}\left(\max_{1\le j\le n_1}\left\{2\sum_{1\le k<\tilde m_j}p_{j,k}\pi_{1,j}^{*}(G_{\Omega_j}(\cdot,z_{j,k}))\right\}\right)$  by $\mathcal{I}(\varphi+\psi)$.
\end{Remark}
\subsection{Basic properties of the Green functions}

In this section, we recall some basic properties of the Green functions. Let $\Omega$ be an open Riemann surface, which admits a nontrivial Green function $G_{\Omega}$, and let $z_0\in\Omega$.

\begin{Lemma}[see \cite{S-O69}, see also \cite{Tsuji}] 	\label{l:green-sup}
		Let $w$ be a local coordinate on a neighborhood of $z_0$ satisfying $w(z_0)=0$.  $G_{\Omega}(z,z_0)=\sup_{v\in\Delta_{\Omega}^*(z_0)}v(z)$, where $\Delta_{\Omega}^*(z_0)$ is the set of negative subharmonic function on $\Omega$ such that $v-\log|w|$ has a locally finite upper bound near $z_0$. Moreover, $G_{\Omega}(\cdot,z_0)$ is harmonic on $\Omega\backslash\{z_0\}$ and $G_{\Omega}(\cdot,z_0)-\log|w|$ is harmonic near $z_0$.
	\end{Lemma}
	
	\begin{Lemma}[see \cite{GY-concavity3}]
	\label{l:green-sup2}Let $K=\{z_j:j\in\mathbb{Z}_{\ge1}\,\&\,j<\gamma \}$ be a discrete subset of $\Omega$, where $\gamma\in\mathbb{Z}_{>1}\cup\{+\infty\}$. Let $\psi$ be a negative subharmonic function on $\Omega$ such that $\frac{1}{2}v(dd^c\psi,z_j)\ge p_j$ for any $j$, where $p_j>0$ is a constant. Then $2\sum_{1\le j< \gamma}p_jG_{\Omega}(\cdot,z_j)$ is a subharmonic function on $\Omega$ satisfying that $2\sum_{1\le j<\gamma }p_jG_{\Omega}(\cdot,z_j)\ge\psi$ and $2\sum_{1\le j<\gamma }p_jG_{\Omega}(\cdot,z_j)$ is harmonic on $\Omega\backslash K$.
\end{Lemma}

\begin{Lemma}[see \cite{GY-concavity}]
\label{relative compactnees of Green function}
For any open neighborhood $U$ of $z_0$, there exists $t>0$ such that $\{G_{\Omega}(z,z_0)<-t\}$ is a relatively compact subset of $U$.
\end{Lemma}

\begin{Lemma}[see \cite{GY-concavity3}]
\label{approximate of Green function}
 There exists a sequence of open Riemann surfaces $\{\Omega_l\}_{l\in\mathbb{Z}^+}$ such that $z_0\in\Omega_l\Subset\Omega_{l+1}\Subset\Omega$, $\cup_{l\in\mathbb{Z}^+}\Omega_l=\Omega$, $\Omega_l$ has a smooth boundary $\partial\Omega_l$ in $\Omega$  and $e^{G_{\Omega_l}(\cdot,z_0)}$ can be smoothly extended to a neighborhood of $\overline{\Omega_l}$ for any $l\in\mathbb{Z}^+$, where $G_{\Omega_l}$ is the Green function of $\Omega_l$. Moreover, $\{{G_{\Omega_l}}(\cdot,z_0)-G_{\Omega}(\cdot,z_0)\}$ is decreasingly convergent to $0$ on $\Omega$ with respect to $l$.

\end{Lemma}

Let $M=(\prod_{1\le j\le n_1}\Omega_j)\times Y$, where $\Omega_j$ is an open Riemann surface and $Y$ is an $n_2$-dimensional complex manifold and $n_1+n_2=n$. Let $\pi_{1}$, $\pi_{1,j}$ and $\pi_2$ be the natural projections from $M$ to $\prod_{1\le j\le n_1}\Omega_j$, $\Omega_j$ and $Y$ respectively. Let $Z_j=\{z_{j,k}:1\le k<\tilde m_j\}$ be a discrete subset of $\Omega_j$ for any $1\le j\le n_1$, where $\tilde m_j\in\mathbb{Z}_{\ge2}\cup\{+\infty\}$. Denote that $Z_0:=\left(\prod_{1\le j\le n_1}Z_j\right)\times Y$.

Let $G=\max_{1\le j\le n_1}\left\{\tilde{\pi}_{j}^*\left(2\sum_{1\le k<\tilde m_j}p_{j,k}G_{\Omega_j}(\cdot,z_{j,k})\right)\right\}$ be a plurisubharmonic function on $\prod_{1\le j\le n_1}\Omega_j$, where $\sum_{1\le k<\tilde m_j}p_{j,k}G_{\Omega_j}(\cdot,z_{j,k})\not\equiv-\infty$ for any $j\in\{1,..,n\}$ and  $\tilde\pi_j$ is the natural projection from $\prod_{1\le j\le n_1}\Omega_j$ to $\Omega_j$.
 Let $\tilde{M}\subset M$ be an $n-$dimensional weakly pseudoconvex submanifold satisfying that $Z_0\subset\tilde{M}$.

Let $N\le 0$ be a plurisubharmonic function on $\tilde M$. Denote $\psi=\pi_1^{*}(G)+N$.
\begin{Lemma}
\label{decreasing property of l}
Let $l(t)$ be a positive Lebesgue measurable function on $(0,+\infty)$ satisfying that $l$ is decreasing on $(0,+\infty)$ and $\int_{0}^{+\infty}l(t)dt<+\infty$.

If $N\not\equiv0$, then there exists a Lebesgue measurable subset $\tilde{V}$ of $\tilde M$ such that $l(-\psi)<l\big(-\pi_1^{*}(G)\big)$ on $\tilde{V}$ and $\mu(\tilde{V})>0,$ where $\mu$ is the Lebesgue measure on $\tilde M$.
\end{Lemma}
\begin{proof}Let $V_1\Subset\prod_{1\le j\le n_1}\Omega_j\backslash \{z_{\beta}:\exists j\in\{1,2,\ldots,n_1\}\ s.t. 2\le\beta_j\le\tilde{m}_j\}$ be an open neighborhood of $z_{\beta^*}$, where $\beta^*=(1,1,\ldots,1)$. It follows from Lemma \ref{relative compactnees of Green function} that there exists $t_0>0$ such that $V_1\cap \{G<-t_0\}\Subset V_1$.

As $l(t)$ is decreasing and $\int_{0}^{+\infty}l(t)dt<+\infty$, there exists $t_1>t_0$ such that $l(t)<l(t_1)$ holds for any $t>t_1$.

For any $w\in Y$, let $U\Subset Y$ be an open neighborhood of $w$ in $Y$ such that $V_1\times U\subset \tilde M$.

As $\psi$ is upper semi-continuous function and $N\not\equiv0$ is a plurisubharmonic function on $\tilde M$, we have $$\sup\limits_{(z,w)\in(V_1\cap \{G\le-t_1\})\times U}\psi(z)<-t_1,$$
which implies that there exists $t_2\in (t_0,t_1)$ such that

$$\sup\limits_{(z,w)\in(V_1\cap \{G\le-t_2\})\times \tilde{U}}\psi(z)<-t_1,$$
where $\tilde{U}\Subset U$ is open.

Denote $t_3=-\sup_{(z,w)\in(V_1\cap \{G\le-t_2\})\times \tilde{U}}\psi(z)$. Then we know $-t_3<-t_1$.

Let $V=\{-t_1<G<-t_2\}\cap V_1$, then $\mu(V\times \tilde U)>0$.
As $l(t)$ is decreasing with respect to $t$, for any $(z,w)\in V\times \tilde U$, we have
$$l(-\psi)\le l(t_3)<l(t_1)\le l\big(\pi_1^*(-G)\big).$$
Lemma \ref{decreasing property of l} is proved.
\end{proof}

\subsection{Other lemmas}
We call a positive measurable function $c$ on $(S,+\infty)$ in class $\tilde{\mathcal{P}}_S$ if $\int_S^s c(l)e^{-l}dl<+\infty$ for some $s>S$ and $c(t)e^{-t}$ is decreasing with respect to $t$.
\begin{Lemma}[see \cite{GMY}]
\label{L2 method}
Let $B \in (0, +\infty)$ and $t_0 \ge S$ be arbitrarily given. Let $(M,\omega)$ be
an $n-$dimensional weakly pseudoconvex K\"ahler manifold. Let $\psi < -S$ be a
plurisubharmonic function on $M$. Let $\varphi$ be a plurisubharmonic function on $M$.
Let F be a holomorphic $(n,0)$ form on $\{\psi< -t_0\}$ such that
\begin{equation}\nonumber
\int_{K\cap \{\psi<-t_0\}} {|F|}^2<+\infty,
\end{equation}
for any compact subset $K$ of $M,$ and
\begin{equation}\nonumber
\int_M \frac{1}{B} \mathbb{I}_{\{-t_0-B< \psi < -t_0\}}  {|F|}^2
e^{{-}\varphi}\le C <+\infty.
\end{equation}
Then there exists a holomorphic $(n,0)$ form $\widetilde F$ on X, such that
\begin{equation}
\int_M {|\widetilde F-(1-b_{t_0,B}(\psi))F|}^2
e^{{-}\varphi+v_{t_0,B}(\psi)}c(-v_{t_0,B}(\psi))\le C\int^{t_0+B}_{S}c(t)e^{{-}t}dt.
\end{equation}
where $b_{t_0,B}(t)=\int^{t}_{-\infty}\frac{1}{B} \mathbb{I}_{\{-t_0-B< s < -t_0\}}ds$,
$v_{t_0,B}(t)=\int^{t}_{-t_0}b_{t_0,B}(s)ds-t_0$ and $c(t)\in \tilde{\mathcal{P}}_S$.
\end{Lemma}

\begin{Lemma}[see \cite{GMY}]
	\label{l:converge}
	Let $M$ be a complex manifold. Let $S$ be an analytic subset of $M$.  	
	Let $\{g_j\}_{j=1,2,...}$ be a sequence of nonnegative Lebesgue measurable functions on $M$, which satisfies that $g_j$ are almost everywhere convergent to $g$ on  $M$ when $j\rightarrow+\infty$,  where $g$ is a nonnegative Lebesgue measurable function on $M$. Assume that for any compact subset $K$ of $M\backslash S$, there exist $s_K\in(0,+\infty)$ and $C_K\in(0,+\infty)$ such that
	$$\int_{K}{g_j}^{-s_K}dV_M\leq C_K$$
	 for any $j$, where $dV_M$ is a continuous volume form on $M$.
	
 Let $\{F_j\}_{j=1,2,...}$ be a sequence of holomorphic $(n,0)$ form on $M$. Assume that $\liminf_{j\rightarrow+\infty}\int_{M}|F_j|^2g_j\leq C$, where $C$ is a positive constant. Then there exists a subsequence $\{F_{j_l}\}_{l=1,2,...}$, which satisfies that $\{F_{j_l}\}$ is uniformly convergent to a holomorphic $(n,0)$ form $F$ on $M$ on any compact subset of $M$ when $l\rightarrow+\infty$, such that
 $$\int_{M}|F|^2g\leq C.$$
\end{Lemma}

\begin{Lemma}[see \cite{G-R}]
\label{closedness}
Let $N$ be a submodule of $\mathcal O_{\mathbb C^n,o}^q$, $1\leq q<\infty$, let $f_j\in\mathcal O_{\mathbb C^n}(U)^q$ be a sequence of $q$-tuples holomorphic function in an open neighborhood $U$ of the origin $o$. Assume that the $f_j$ converges uniformly in $U$ towards  a $q-$tuples $f\in\mathcal O_{\mathbb C^n}(U)^q$, assume furthermore that all germs $(f_{j},o)$ belong to $N$. Then $(f,o)\in N$.	
\end{Lemma}

\begin{Lemma}[see \cite{GY-concavity4}]\label{l:m2}
Let $\psi=\max_{1\le j\le n}\{2p_j\log|w_j|\}$ be a plurisubharmonic function on $\mathbb{C}^n$, where $p_j>0$.
	Let $f=\sum_{\alpha\in\mathbb{Z}_{\ge0}^n}b_{\alpha}w^{\alpha}$ (Taylor expansion) be a holomorphic function on $\{\psi<-t_0\}$, where $t_0>0$. Denote that $q_{\alpha}:=\sum_{1\le j\le n}\frac{\alpha_j+1}{p_j}-1$ for any $\alpha\in\mathbb{Z}_{\ge0}^n$ and  $E_1:=\{\alpha\in\mathbb{Z}_{\ge0}^n:q_{\alpha}=0\}$. Let $k_1$ be  a constant satisfying $k_1+1>0$.Then
\begin{displaymath}\begin{split}
	\int_{\{-t-1-k_1<\psi<-t\}}|f|^2e^{-\psi}d\lambda_n=&\sum_{\alpha\in E_1}\frac{|b_{\alpha}|^2\pi^{n}(1+k_1)}{\prod_{1\le j\le n}(\alpha_j+1)}\\
	&+\sum_{\alpha\not\in E_1}\frac{|b_{\alpha}|^2\pi^{n}(q_{\alpha}+1)(e^{-q_{\alpha}t}-e^{-q_{\alpha}(t+1+k_1)})}{q_{\alpha}\prod_{1\le j\le n}(\alpha_j+1)}	
\end{split}\end{displaymath}
	 for any $t>t_0$.
\end{Lemma}
\begin{Remark}
Lemma \ref{l:m2}  is stated in \cite{GY-concavity4} in the case $k_1=0$, the same proof as in \cite{GY-concavity4} shows that Lemma \ref{l:m2} holds for $k_1$ which satisfying $k_1+1>0$.
\end{Remark}

\begin{Lemma}\label{l:m1}
Let $\psi=\max_{1\le j\le n}\{2p_j\log|w_j|\}$ be a plurisubharmonic function on $\mathbb{C}^n$, where $p_j>0$.
	Let $f=\sum_{\alpha\in \mathbb{Z}_{\ge0}^n}b_{\alpha}w^{\alpha}$ (Taylor expansion) be a holomorphic function on $\{\psi<-t_0\}$, where $t_0>0$. Let $c(t)$ be a nonnegative measurable function on $(t_0,+\infty)$. Denote that $q_{\alpha}:=\sum_{1\le j\le n}\frac{\alpha_j+1}{p_j}-1$ for any $\alpha\in\mathbb{Z}_{\ge0}^n$. Let $A$ be a real constant such that $t_0+A>0$. Then
	$$\int_{\{\psi+A<-t\}}|f|^2c(-\psi-A)d\lambda_n=\sum_{\alpha\in\mathbb{Z}_{\ge0}^n}\left(\int_t^{+\infty}c(s)e^{-(q_{\alpha}+1)s}ds\right)e^{-(q_{\alpha}+1)A}\frac{(q_{\alpha}+1)|b_{\alpha}|^2\pi^{n}}{\prod_{1\le j\le n}(\alpha_j+1)}$$
	holds for any $t\ge t_0$.
\end{Lemma}
	
\begin{proof}
	For any $t\ge t_0$, by direct calculations, we obtain that (note that $t+A>0$)
	\begin{equation}\label{eq:1127b}\begin{split}
		&\int_{\{\psi+A<-t\}}|w^{\alpha}|^2c(-\psi-A)d\lambda_n\\
		=&(2\pi)^n\int_{\left\{\max_{1\le j\le n}\left\{s_j^{p_j}\right\}<e^{-\frac{t+A}{2}}\,\&\,s_j>0\right\}}\prod_{1\le j\le n}s_j^{2\alpha_j+1}\cdot c\left(-\log\max_{1\le j\le n}\left\{s_j^{2p_j}\right\}-A\right)ds_1ds_2\ldots ds_n\\
		=&(2\pi)^n\frac{1}{\prod_{1\le j\le n}p_j}\\
		&\times\int_{\left\{\max_{1\le j\le n}\{r_j\}<e^{-\frac{t+A}{2}}\,\&\,r_j>0\right\}}\prod_{1\le j\le n}r_j^{\frac{2\alpha_j+2}{p_j}-1}\cdot c\left(-\log\max_{1\le j\le n}\left\{r_j^2\right\}-A\right)dr_1dr_2\ldots dr_n.
		\end{split}
	\end{equation}
 By the Fubini's theorem, we have
\begin{equation}
	\label{eq:211125f}\begin{split}
		&\int_{\left\{\max_{1\le j\le n}\{r_j\}<e^{-\frac{t+A}{2}}\,\&\,r_j>0\right\}}\prod_{1\le j\le n}r_j^{\frac{2\alpha_j+2}{p_j}-1}\cdot c\left(-\log\max_{1\le j\le n}\left\{r_j^2\right\}-A\right)dr_1dr_2\ldots dr_n\\
		=&\sum_{j'=1}^n\int_{0}^{e^{-\frac{t+A}{2}}}\left(\int_{\{0\le r_j<r_{j'},j\not=j'\}}\prod_{j\not=j'}r_j^{\frac{2\alpha_j+2}{p_j}-1}\cdot\wedge_{j\not=j'}dr_j\right)r_{j'}^{\frac{2\alpha_{j'}+2}{p_{j'}}-1}c(-2\log r_{j'}-A)dr_{j'}\\
		=&\sum_{j'=1}^n\left(\prod_{j\not=j'}\frac{p_j}{2\alpha_j+2}\right)\int_{0}^{e^{-\frac{t+A}{2}}}r_{j'}^{\sum_{1\le k\le n}\frac{2\alpha_k+2}{p_k}-1}c(-2\log r_{j'}-A)dr_{j'}\\
		=&(q_{\alpha}+1)e^{-(q_{\alpha}+1)A}\left(\int_t^{+\infty}c(s)e^{-(q_{\alpha}+1)s}ds\right)\prod_{1\le j\le n}\frac{p_j}{2\alpha_j+2}.
			\end{split}
\end{equation}		
Following from $\int_{\{\psi<-t\}}|f|^2c(-\psi)d\lambda_n=\sum_{\alpha\in\mathbb{Z}_{\ge0}^n}|b_{\alpha}|^2\int_{\{\psi<-t\}}|w^{\alpha}|^2c(-\psi)d\lambda_n$, equality \eqref{eq:1127b} and equality \eqref{eq:211125f}, we obtain that
\begin{equation*}
\int_{\{\psi<-t\}}|f|^2d\lambda_n=\sum_{\alpha\in\mathbb{Z}_{\ge0}^n}\left(\int_t^{+\infty}c(s)e^{-(q_{\alpha}+1)s}ds\right)e^{-(q_{\alpha}+1)A}\frac{(q_{\alpha}+1)|b_{\alpha}|^2\pi^n}{\prod_{1\le j\le n}(\alpha_j+1)}.
\end{equation*}
\end{proof}

\begin{Lemma}
\label{limit discussion}
Let $\Delta^n\subset \mathbb{C}^n$ be a polydisc. Let $\varphi$ be a bounded subharmonic function on $\Delta^n$. Assume that $v$ is a nonnegative continuous real function on $\Delta^n$.
Denote
$$I_t:=\int_{\{z\in\Delta^n:-t-1-k_2<\max\limits_{1\le j\le n}{\{2p_j\log|z_j|\}}<-t+k_1\}}v(z)\frac{\prod_{j=1}^{n}|z_j|^{2\alpha_j}}{\max\limits_{1\le j\le n}{\{|z_j|^{2p_j}\}}}
e^{-\varphi(z_1,\ldots,z_n)}dz\wedge d\bar{z},$$
$$J_t:=\int_{\{z\in\Delta^n:-t-1-k_2<\max\limits_{1\le j\le n}{\{2p_j\log|z_j|\}}<-t+k_1\}}v(z)\frac{\prod_{j=1}^{n}|z_j|^{2\beta_j}}{\max\limits_{1\le j\le n}{\{|z_j|^{2p_j}\}}}
e^{-\varphi(z_1,\ldots,z_n)}dz\wedge d\bar{z},$$
where $p_j,\alpha_j,\beta_j>0$ are constants for any $j$ satisfying $\sum_{1\le j\le n}\frac{\alpha_j+1}{p_j}=1$ and $\sum_{1\le j\le n}\frac{\beta_j+1}{p_j}>1$, $k_1,k_2$ are constants satisfying $k_1+k_2+1>0$.

Then
\begin{equation}
\label{limit for bounded psh}
\limsup_{t\to+\infty}I_t\le \frac{(2\pi)^{n}(1+k_1+k_2)}{\prod_{1\le j\le n}(\alpha_j+1)}v(0)e^{-\varphi(0)}.
\end{equation}
and
\begin{equation}\nonumber
\lim_{t\to+\infty}J_t=0.
\end{equation}
\end{Lemma}
\begin{proof}
The idea of the proof can be referred to Lemma 3.3 in \cite{ZZ2019}. For the convenience of the readers, we give a proof below.

Denote
$$I_{t+k_1}:=\int_{\{z\in\Delta^n:-t-1-k_1-k_2<\max\limits_{1\le j\le n}{\{2p_j\log|z_j|\}}<-t\}}v(z)\frac{\prod_{j=1}^{n}|z_j|^{2\alpha_j}}{\max\limits_{1\le j\le n}{\{|z_j|^{2p_j}\}}}
e^{-\varphi(z_1,\ldots,z_n)}dz\wedge d\bar{z},$$
Note that
$$\limsup_{t\to+\infty}I_t=\limsup_{t\to+\infty}I_{t+k_1},$$
then we only need to prove \eqref{limit for bounded psh} for $I_{t+k_1}$.

Denote
$$B_{\delta,t}:=\{z\in\Delta^n:
\varphi(e^{\frac{-t}{2p_1}}z_1,e^{\frac{-t}{2p_2}}z_2,\ldots,e^{\frac{-t}{2p_n}}z_n)
<(1+\delta)\varphi(0)\},$$
where $\delta\in(0,+\infty)$ and $t\in (0,+\infty)$.

Let $\lambda(B_{\delta,t})$ be the $2n-$dimensional Lebesgue measure of $B_{\delta,t}$.

Since the computation is local, we may assume that $\varphi$ is a negative upper semicontinuous function on $\Delta^n$. Note that $\varphi(0)>-\infty$. For any $\epsilon\in (0,1)$, there exists $t_{\epsilon}>0$ such that
$$\varphi(e^{\frac{-t}{2p_1}}z_1,e^{\frac{-t}{2p_2}}z_2,\ldots,e^{\frac{-t}{2p_n}}z_n)\le(1-\epsilon)\varphi(0)$$
for any $z\in \Delta^n$, when $t\ge t_\epsilon.$ We denote $\varphi(e^{\frac{-t}{2p_1}}z_1,e^{\frac{-t}{2p_2}}z_2,\ldots,e^{\frac{-t}{2p_n}}z_n)$ by $\varphi(e^{\frac{-t}{2p}}z)$ when there is no misunderstanding.

Note that for fixed $t$, $\varphi(e^{\frac{-t}{2p}}z)$ is subharmonic on $\Delta^n$ with respect to $z$. It follows from mean value inequality that, for all $t\ge t_{\epsilon}$, we have
\begin{equation*}
  \begin{split}
     \varphi(0) &
     \le\frac{1}{\pi}\int_{z_1\in\Delta} \varphi(e^{\frac{-t}{2p_1}}z_1,0\ldots,0)d\lambda_{z_1}\\
     &\le \frac{1}{\pi^n}\int_{z\in\Delta^n} \varphi(e^{\frac{-t}{2p_1}}z_1,e^{\frac{-t}{2p_2}}z_2,\ldots,e^{\frac{-t}{2p_n}}z_n)d\lambda_z\\
       & =\frac{1}{\pi^n}\int_{\Delta^n\backslash B_{\delta,t}} \varphi(e^{\frac{-t}{2p}}z)d\lambda_z+\frac{1}{\pi}\int_{B_{\delta,t}} \varphi(e^{\frac{-t}{2p}}z)d\lambda_z\\
       &\le\frac{(1-\epsilon)\varphi(0)}{\pi^n}(\pi^n-\lambda(B_{\delta,t}))+\frac{(1+\delta)\varphi(0)}{\pi^n}\lambda(B_{\delta,t})\\
       &=\varphi(0)(1-\epsilon+\frac{\delta+\epsilon}{\pi^n}\lambda(B_{\delta,t})).
  \end{split}
\end{equation*}
As $\varphi(0)<0$, we have
$$\lambda(B_{\delta,t})\le\frac{\pi^n\epsilon}{\delta+\epsilon}\le\frac{\pi^n\epsilon}{\delta}$$
for any $t\ge t_{\epsilon}$. Hence
$$\lim\limits_{t \to+\infty}\lambda(B_{\delta,t})=0.$$

Since $\varphi$ is bounded, we have $e^{-\varphi}\le C_1$ for some $C_1>0$.
As $v(t)$ is continuous, when $t$ is large enough, we have $$\sup\limits_{\{z\in\Delta^n:-t-1-k_1-k_2<\max\limits_{1\le j\le n}{\{2p_j\log|z_j|\}}<-t\}}v(z)\le C_2,$$
where $C_2>0$ is a constant independent of $t$. We denote $v(e^{\frac{-t}{2p_1}}z_1,e^{\frac{-t}{2p_2}}z_2,\ldots,e^{\frac{-t}{2p_n}}z_n)$ by $v(e^{\frac{-t}{2p}}z)$ when there is no misunderstandings.

 Let $z_i=e^{\frac{-t}{2p_i}}w_i$. We denote $v(e^{\frac{-t}{2p_1}}w_1,e^{\frac{-t}{2p_2}}w_2,\ldots,e^{\frac{-t}{2p_n}}w_n)$ by $v(e^{\frac{-t}{2p}}w)$ and denote $\varphi(e^{\frac{-t}{2p_1}}w_1,e^{\frac{-t}{2p_2}}w_2,\ldots,e^{\frac{-t}{2p_n}}w_n)$ by $\varphi(e^{\frac{-t}{2p}}w)$ for simplicity. By Lemma \ref{l:m2}, we have
\begin{equation}
\begin{split}
    I_{t+k_1}&=\int_{\{z\in\Delta:-t-1-k_1-k_2<\max\limits_{1\le j\le n}{\{2p_j\log|z_j|\}}<-t\}}v(z)\frac{\prod_{j=1}^{n}|z_j|^{2\alpha_j}}{\max\limits_{1\le j\le n}{\{|z_j|^{2p_j}\}}}e^{-\varphi(z)}dz\wedge d\bar{z} \\
     & =\int_{\{w\in\Delta^n:e^{-1-k_1-k_2}<\max\limits_{1\le j\le n}{\{|w_j|^{2p_j}\}}<1\}}
     v(e^{\frac{-t}{2p}}w)\frac{\prod_{j=1}^{n}|w_j|^{2\alpha_j}}{\max\limits_{1\le j\le n}{\{|w_j|^{2p_j}\}}}e^{-\varphi(e^{\frac{-t}{2p}}w)}dw\wedge d\bar{w}\\
    &= \int_{\{w\in\Delta^n:e^{-1-k_1-k_2}<\max\limits_{1\le j\le n}{\{|w_j|^{2p_j}\}}<1\}\cap B_{\delta,t}}
     v(e^{\frac{-t}{2p}}w)\frac{\prod_{j=1}^{n}|w_j|^{2\alpha_j}}{\max\limits_{1\le j\le n}{\{|w_j|^{2p_j}\}}}e^{-\varphi(e^{\frac{-t}{2p}}w)}dw\wedge d\bar{w}+\\
    &\int_{\{w\in\Delta^n:e^{-1-k_1-k_2}<\max\limits_{1\le j\le n}{\{|w_j|^{2p_j}\}}<1\}\backslash B_{\delta,t}}
     v(e^{\frac{-t}{2p}}w)\frac{\prod_{j=1}^{n}|w_j|^{2\alpha_j}}{\max\limits_{1\le j\le n}{\{|w_j|^{2p_j}\}}}e^{-\varphi(e^{\frac{-t}{2p}}w)}dw\wedge d\bar{w}\\
    &\le CC_1C_2\lambda(B_{\delta,t})+\left(\sup\limits_{\{e^{-1-k_1-k_2}<\max\limits_{1\le j\le n}{\{|w_j|^{2p_j}\}}<1\}}v(e^{\frac{-t}{2p}}w)\right)
    e^{-(1+\delta)\varphi(0)} \frac{(2\pi)^{n}(1+k_1+k_2)}{\prod_{1\le j\le n}(\alpha_j+1)}.
\end{split}
\end{equation}
Hence we know
$$\limsup_{t\to+\infty}I_{t+k_1}= \frac{(2\pi)^{n}(1+k_1+k_2)}{\prod_{1\le j\le n}(\alpha_j+1)}e^{-(1+\delta)\varphi(0)}v(0).$$
By the arbitrariness of $\delta$, we know \eqref{limit for bounded psh} holds for $I_{t+k_1}$.

For $J_t$, we know when $t$ is large enough, the function $v(z)$ and $e^{-\varphi(z)}$ are uniformly bounded by some constant $M>0$ with respect to $t$.
Then it follows from Lemma \ref{l:m2} that
\begin{equation}\nonumber
\begin{split}
\limsup\limits_{t\to+\infty}J_{t+k_1}
\le & M\limsup\limits_{t\to+\infty}\int_{\{z\in\Delta^n:-t-1-k_1-k_2<\max\limits_{1\le j\le n}{\{2p_j\log|z_j|\}}<-t\}}\frac{\prod_{j=1}^{n}|z_j|^{2\beta_j}}{\max\limits_{1\le j\le n}{\{|z_j|^{2p_j}\}}}
dz\wedge d\bar{z}\\
=& M\lim\limits_{t\to+\infty}\frac{\pi^{n}(q_{\alpha}+1)(e^{-q_{\alpha}t}-e^{-q_{\alpha}(t+1+k_1+k_2)})}{q_{\alpha}\prod_{1\le j\le n}(\alpha_j+1)},
\end{split}
\end{equation}
where $q_{\alpha}=\frac{\beta_j+1}{p_j}-1$. As $k_1+k_2+1>0$, we have $\lim\limits_{t\to+\infty}J_t=0.$

Lemma \ref{limit discussion} is proved.
\end{proof}

\begin{Lemma}
\label{limit discussion 2}
Let $\Delta^n\subset \mathbb{C}^n$ be a polydisc. Let $f=\sum_{\alpha\in E}b_{\alpha}w^{\alpha}$ (Taylor expansion) be a holomorphic function on $\Delta^n$, where  $E:=\left\{\alpha=(\alpha_1,\ldots,\alpha_{n}):\sum_{1\le j\le n_1}\frac{\alpha_j+1}{p_{j}}=1\right\}$. Assume that $v$ is a nonnegative continuous real function on $\Delta^n$.
Denote
$$S_t:=\int_{\{z\in\Delta^n:-t-1-k_2<\max\limits_{1\le j\le n}{\{2p_j\log|z_j|\}}<-t+k_1\}}v(z)\frac{|f|^2}{\max\limits_{1\le j\le n}{\{|z_j|^{2p_j}\}}}
e^{-\varphi(z_1,\ldots,z_n)}dz\wedge d\bar{z}.$$
Then we have
\begin{equation}\nonumber
\limsup_{t\to+\infty}S_t\le \sum_{\alpha\in E}\frac{|b_{\alpha}|^2(2\pi)^{n}(1+k_1+k_2)}{\prod_{1\le j\le n}(\alpha_j+1)}v(0)e^{-\varphi(0)}.
\end{equation}
\end{Lemma}
\begin{proof}
In the proof of Lemma \ref{limit discussion 2}, we follow the notations we used in the proof of Lemma \ref{limit discussion}. Let $z_i=e^{\frac{-t}{2p_i}}w_i$, by Lemma \ref{l:m2}, we have
\begin{equation}\nonumber
\begin{split}
S_t&=\int_{\{z\in\Delta^n:-t-1-k_2<\max\limits_{1\le j\le n}{\{2p_j\log|z_j|\}}<-t+k_1\}}v(z)\frac{|f|^2}{\max\limits_{1\le j\le n}{\{|z_j|^{2p_j}\}}}
e^{-\varphi(z_1,\ldots,z_n)}dz\wedge d\bar{z}\\
&= \int_{\{w\in\Delta^n:e^{-1-k_1-k_2}<\max\limits_{1\le j\le n}{\{|w_j|^{2p_j}\}}<1\}\cap B_{\delta,t}}
     v(e^{\frac{-t}{2p}}w)\frac{|f(w)|^2}{\max\limits_{1\le j\le n}{\{|w_j|^{2p_j}\}}}e^{-\varphi(e^{\frac{-t}{2p}}w)}dw\wedge d\bar{w}+\\
    &\int_{\{w\in\Delta^n:e^{-1-k_1-k_2}<\max\limits_{1\le j\le n}{\{|w_j|^{2p_j}\}}<1\}\backslash B_{\delta,t}}
     v(e^{\frac{-t}{2p}}w)\frac{|f(w)|^2}{\max\limits_{1\le j\le n}{\{|w_j|^{2p_j}\}}}e^{-\varphi(e^{\frac{-t}{2p}}w)}dw\wedge d\bar{w}\\
     &\le CC_1C_2\lambda(B_{\delta,t})+\left(\sup\limits_{\{e^{-1-k_1-k_2}<\max\limits_{1\le j\le n}{\{|w_j|^{2p_j}\}}<1\}}v(e^{\frac{-t}{2p}}w)\right)e^{-(1+\delta)\varphi(0)}\times\\
     &
    \int_{\{e^{-1-k_1-k_2}<\max\limits_{1\le j\le n}{\{|w_j|^{2p_j}\}}<1\}}
    \frac{|f(w)|^2}{\max\limits_{1\le j\le n}{\{|w_j|^{2p_j}\}}}dw\wedge d\bar{w}\\
    &= CC_1C_2\lambda(B_{\delta,t})+\left(\sup\limits_{\{e^{-1-k_1-k_2}<\max\limits_{1\le j\le n}{\{|w_j|^{2p_j}\}}<1\}}v(e^{\frac{-t}{2p}}w)\right)e^{-(1+\delta)\varphi(0)}\times\\
     &
    \sum_{\alpha\in E}\int_{\{e^{-1-k_1-k_2}<\max\limits_{1\le j\le n}{\{|w_j|^{2p_j}\}}<1\}}
    \frac{|b_{\alpha}|^2\prod_{j=1}^{n_1}|w_j|^{2\alpha_j}}{\max\limits_{1\le j\le n}{\{|w_j|^{2p_j}\}}}dw\wedge d\bar{w}.\\
\end{split}
\end{equation}
Hence  we know
\begin{equation}\nonumber
\limsup_{t\to+\infty}S_t\le \sum_{\alpha\in E}\frac{|b_{\alpha}|^2(2\pi)^{n}(1+k_1+k_2)}{\prod_{1\le j\le n}(\alpha_j+1)}v(0)e^{-\varphi(0)}.
\end{equation}

\end{proof}

\begin{Lemma}[see \cite{GY-concavity4}]
\label{growth rate of c(t)e(-t)}
Let $c(t)$ be a positive measurable function on $(0,+\infty)$, and let $a\in\mathbb{R}$. Assume that $\int_{t}^{+\infty}c(s)e^{-s}ds\in(0,+\infty)$ when $t$ near $+\infty$. Then we have

(1) $\lim\limits_{t\to+\infty}\frac{\int_{t}^{+\infty}c(s)e^{-as}ds}{\int_{t}^{+\infty}c(s)e^{-s}ds}=1$
if and only if $a=1$,

(2) $\lim\limits_{t\to+\infty}\frac{\int_{t}^{+\infty}c(s)e^{-as}ds}{\int_{t}^{+\infty}c(s)e^{-s}ds}=0$
if and only if $a>1$,

(3) $\lim\limits_{t\to+\infty}\frac{\int_{t}^{+\infty}c(s)e^{-as}ds}{\int_{t}^{+\infty}c(s)e^{-s}ds}=+\infty$
if and only if $a<1$.
\end{Lemma}

\begin{Lemma}[see \cite{GY-concavity3}]\label{tildecincreasing}
		If $c(t)$ is a positive measurable function on $(T,+\infty)$ such
		that $c(t)e^{-t}$ is decreasing on $(T,+\infty)$ and $\int_{T_1}^{+\infty}c(s)e^{-s}ds<+\infty$ for some $T_1>T$, then there exists a positive measurable function $\tilde{c}$ on $(T,+\infty)$ satisfying the
		following statements:
		
		(1). $\tilde{c}\geq c$ on $(T,+\infty)$;
		
		(2). $\tilde{c}(t)e^{-t}$ is strictly decreasing on $(T,+\infty)$ and $\tilde{c}$ is increasing on $(a,+\infty)$,	where $a>T$ is a real number;
		
		(3). $\int_{T_1}^{+\infty}\tilde{c}(s)e^{-s}ds<+\infty$.
		
		Moreover, if $\int_T^{+\infty}c(s)e^{-s}ds<+\infty$ and $c\in\mathcal{P}_T$, we can choose $\tilde{c}$ satisfying the above conditions, $\int_T^{+\infty}\tilde{c}(s)e^{-s}ds<+\infty$ and $\tilde{c}\in\mathcal{P}_T$.
	\end{Lemma}

\section{Preparations II: multiplier ideal sheaves and optimal $L^{2}$ extensions}

In this section, we recall and present some lemmas related to multiplier ideal sheaves and optimal $L^{2}$ extensions.

\subsection{Multiplier ideal sheaves}

	We need the following lemmas in the local cases.

Let $\Delta \subset \mathbb{C}$ be the unit disc.
Let $X=\Delta^{m}$ and let $Y=\Delta^{n-m}$.
Denote $M=X \times Y$. Let $\pi_1$ and $\pi_2$ be the natural projections from $M$ to $X$ and $Y$ respectively.

Let $\psi_1$ be a plurisubharmonic function on $X$. Let $\psi=\pi_1^*(\psi_1)$. Let $\varphi_1$ be a Lebesgue measurable function on $X$ and $\varphi_2\in Psh(Y)$. Denote $\varphi=\pi_1^*(\varphi_1)+\pi_2^*(\varphi_2)$.
\begin{Lemma}\label{e-varphic-psi}
		Assume that
		\begin{equation*}
			\int_M|f|^2e^{-\varphi}c(-\psi)<+\infty.
		\end{equation*}
		Then for any $w\in \Delta^{n-m}$,
		\begin{equation*}
			\int_{z\in\Delta^m}|f(z,w)|^2e^{-\varphi_1}c(-\psi_1)<+\infty.
		\end{equation*}
		
	\end{Lemma}
	
	\begin{proof}
		According to the Fubini's Theorem, we have
		\begin{flalign*}
			\begin{split}
				&\int_{z\in\Delta^m}|f(z,w)|^2e^{-\varphi_1}c(-\psi_1)\\
				\leq&\frac{1}{(\pi r^2) ^{n-m}}\int_{w'\in\Delta^{n-m}(w,r)}\left(\int_{z\in\Delta^m}|f(z,w')|^2e^{-\varphi_1}c(-\psi_1)\right)\\
				\leq&\frac{e^T}{(\pi r^2) ^{n-m}}\int_{w'\in\Delta^{n-m}(w,r)}\left(\int_{z\in\Delta^m}|f(z,w')|^2e^{-\varphi_1}c(-\psi_1)\right)e^{-\varphi_2}\\
				\leq&\frac{e^T}{(\pi r^2)^{n-m}}\int_{\Delta^n}|f|^2e^{-\varphi}c(-\psi)<+\infty,
			\end{split}
		\end{flalign*}
		where $r>0$ such that $\Delta^{n-m}(w,r)\Subset\Delta^{n-m}$, and $T:=-\sup_{w'\in\Delta^{n-m}(w,r)}\varphi_2(w')$.
	
	\end{proof}

\begin{Lemma}\label{f1zf2w}
		Let $f_1(z)$ be a holomorphic function on  $X$ such that $(f_1,o)\in \mathcal{I}(\varphi_1+\psi_1)_o$, and $f_2(w)$ be a holomorphic function on $Y$ such that $(f_2,w)\in\mathcal{I}(\varphi_2)_w$ for any $w\in Y$. Let $\tilde{f}(z,w)=f_1(z)f_2(w)$ on $M$, then $\big(\tilde{f},(o,w)\big)\in\mathcal{I}(\varphi+\psi)_{(o,w)}$ for any $(o,w)\in Y$.
	\end{Lemma}
	
	\begin{proof}
		According to $(f_1,o)\in \mathcal{I}(\varphi_1+\psi_1)_o$ and $(f_2,w)\in\mathcal{I}(\varphi_2)_w$, we can find some $r>0$ such that
		\begin{equation*}
			\int_{\Delta^m(o,r)}|f_1|^2e^{-\varphi_1-\psi_1}<+\infty,
		\end{equation*}
		and
		\begin{equation*}
			\int_{\Delta^{n-m}(w,r)}|f_2|^2e^{-\varphi_2}<+\infty.
		\end{equation*}
		Then using Fubini's Theorem, we get
		\begin{equation*}
			\int_{\Delta^n((o,w),r)}|\tilde{f}|^2e^{-\varphi-\psi}=\int_{\Delta^m(o,r)}|f_1|^2e^{-\varphi_1-\psi_1}\int_{\Delta^{n-m}(w,r)}|f_2|^2e^{-\varphi_2}<+\infty,
		\end{equation*}
		which means that $\big(\tilde{f},(o,w)\big)\in\mathcal{I}(\varphi+\psi)_{(o,w)}$.
	\end{proof}

Let $\Delta^{n_1}=\big\{w\in\mathbb{C}^{n_1}:|w_j|<1$ for any $j\in\{1,\ldots,n_1\}\big\}$ be product of the unit disks. Let $Y$ be an $n_2-$dimensional complex manifold, and let $M=\Delta^{n_1} \times Y$. Denote $n=n_1+n_2$. Let $\pi_1$ and $\pi_2$ be the natural projections from $M$ to $\Delta^{n_1}$ and $Y$ respectively.  Let $\rho_1$ be a nonnegative Lebesgue measurable function on $\Delta^{n_1}$ satisfying that $\rho_1(w)=\rho_1(|w_1|,\ldots,|w_{n_1}|)$ for any $w\in\Delta^{n_1}$ and the Lebesgue measure of $\{w\in\Delta^{n_1}: \rho_1(w)>0\}$ is  positive. Let $\rho_2$ be a nonnegative Lebesgue measurable function on $Y$, and denote that $\rho=\pi_1^*(\rho_1)\times\pi_2^*(\rho_2)$ on $M$.
\begin{Lemma}[see \cite{BGY-concavity6}]
	\label{decomp}	For any holomorphic $(n,0)$  form $F$ on $M$, there exists a unique sequence of  holomorphic  $(n_2,0)$   forms $\{F_{\alpha}\}_{\alpha\in\mathbb{Z}_{\ge0}^{n_1}}$ on $Y$ such that
	\begin{equation}\label{eq:1216c}
		F=\sum_{\alpha\in\mathbb{Z}_{\ge0}^{n_1}}\pi_1^*(w^{\alpha}dw_1\wedge\ldots\wedge dw_{n_1}) \wedge \pi_2^*(F_\alpha),
	\end{equation}
where the right term of the above equality is uniformly convergent on any compact subset of $M$. Moreover, if 	$\int_{M}|F|^2\rho<+\infty,$ we have
	\begin{equation}
	\label{eq:1216d}\int_{Y}|F_{\alpha}|^2\rho_2<+\infty
\end{equation}
for any $\alpha\in\mathbb{Z}_{\ge0}^{n_1}$.
\end{Lemma}

Let $\tilde{M}\subset M$ be an $n-$dimensional complex manifold satisfying that $\{o\}\times Y\subset \tilde{M}$, where $o$ is the origin in $\Delta^{n_1}$.

\begin{Lemma}[see \cite{BGY-concavity6}]
	\label{decomp-tildeM}For any holomorphic $(n,0)$  form $F$ on $\tilde{M}$, there exist a unique sequence of  holomorphic  $(n_2,0)$   forms $\{F_{\alpha}\}_{\alpha\in\mathbb{Z}_{\ge0}^{n_1}}$ on $Y$ and a neighborhood $M_2\subset \tilde{M}$ of $\{o\}\times Y$, such that
	\begin{equation*}
		F=\sum_{\alpha\in\mathbb{Z}_{\ge0}^{n_1}}\pi_1^*(w^{\alpha}dw_1\wedge\ldots\wedge dw_{n_1}) \wedge \pi_2^*(F_\alpha)
	\end{equation*}
on $M_2$, where the right term of the above equality is uniformly convergent on any compact subset of $M_2$. Moreover, if
	$\int_{\tilde{M}}|F|^2\rho<+\infty,$ we have
	\begin{equation*}
\int_{K}|F_{\alpha}|^2\rho_2<+\infty
\end{equation*}
for any compact subset $K$ of $Y$ and $\alpha\in\mathbb{Z}_{\ge0}^{n_1}$.
\end{Lemma}

Let $f=\sum_{\alpha\in\mathbb{Z}_{\ge0}^{n}}b_{\alpha}w^{\alpha}$ (Taylor expansion) be a holomorphic function  on $D=\{w\in\mathbb{C}^n:|w_j|<r_0$ for any $j\in\{1,\ldots,n\}\}$, where $r_0>0$. 	Let
$$\psi=\max_{1\le j\le n_1}\left\{2p_j\log|w_j|\right\}$$ be a plurisubharmonic function on $\mathbb{C}^n$, where $n_1\le n$ and $p_j>0$ is a constant for any $j\in\{1,\ldots,n_1\}$. We recall a characterization of  $\mathcal{I}(\psi)_o$, where $o$ is the origin in $\mathbb{C}^n$.
\begin{Lemma}[see \cite{BGY-concavity6}]\label{l:0}
$(f,o)\in\mathcal{I}(\psi)_{o}$ if and only if $\sum_{1\le j\le n_1}\frac{\alpha_j+1}{p_j}>1$ for any $\alpha\in\mathbb{Z}_{\ge0}^n$ satisfying $b_{\alpha}\not=0$.
\end{Lemma}

Let $\Omega=\Delta\subset \mathbb{C}$ be an unit disk. Let $Y=\Delta^n$ . Denote $M=\Omega \times Y$. Let $\pi_1$ and $\pi_2$ be the natural projections from $M$ to $\Omega$ and $Y$ respectively.

Let $\psi=\pi^*_1(2p\log|z|)+N$ be a plurisubharmonic function on $M$, where $N\le 0$ is a plurisubharmonic function on $M$ and $N|_{\{0\}\times \Delta^n}\not\equiv -\infty$.
Assume that there exist a holomorphic function $g$ on $\Delta$ and a function $\tilde{\psi}_2\in Psh(M)$ such that
$$\psi+\pi^{*}_1(\varphi_1)=\pi^{*}_1(2\log|g|)+\tilde{\psi}_2,$$
 where $ord_0(g)=q$. We assume that $g=dz^q h(z)$ on $\Delta$, where $d$ is a constant, $h(z)$ is a holomorphic function on $\Delta$ and $h(0)=1$.

Let $\varphi_2\in Psh(Y)$. Denote $\varphi:=\pi^{*}_1(\varphi_1)+\pi^{*}_2(\varphi_2)$.

	Let $F$ be a holomorphic $(n,0)$ form on $M$, where
	\begin{equation*}
		F=\sum_{j=k}^{\infty}\pi_1^*(z^jdz)\wedge \pi_2^*(F_j)
	\end{equation*}
	according to Lemma \ref{decomp}. Here $k\in\mathbb{N}$ and $F_j$ is a holomorphic $(n,0)$ form on $Y$ for any $j\geq k$.

Assume that $\int_M|F|^2e^{-\varphi}c(-\psi)<+\infty$ and $c(t)$ is increasing near $+\infty$.
As $\psi=\pi^*_1(2p\log|z|)+N\le \pi^*_1(2p\log|z|)$, when $t$ is large enough, we have
$$S:=\int_{\{\pi_1^*(2p\log|z|)<-t\}}|F|^2e^{-\varphi}c(-\pi_1^*(2p\log|z|))
\le\int_{\{\psi<-t\}}|F|^2e^{-\varphi}c(-\psi)<+\infty.$$

	\begin{Lemma}\label{k>k0}
		Let $c$ be a positive measurable function on $(0,+\infty)$ such that $c(t)e^{-t}$ is decreasing on $(0,+\infty)$, $c$ is increasing near $+\infty$, and $\int_0^{+\infty}c(s)e^{-s}ds<+\infty$. Assume that $k\ge q$, and
	$$S=\int_{\{\pi_1^*(2p\log|z|)<-t\}}|F|^2e^{-\varphi}c\big(-\pi_1^*(2p\log|z|)\big)
<+\infty.$$
		Then
		\[\big(F,(0,y)\big)\in\big(\mathcal{O}(K_M))\otimes\mathcal{I}(\pi_1^*(2\log|g|)+\pi_2^*(\varphi_2)\big)_{(0,y)}\]
		for any $y\in Y$.
	\end{Lemma}
	
	\begin{proof}

It follows from $M$ is a Stein manifold that there exist smooth plurisubharmonic functions $\tilde{\psi}_{2,l}$ on $M$ such that $\tilde{\psi}_{2,l}$ are decreasingly convergent to $\tilde{\psi}_{2}$. Since the computation is local, we assume that $|\tilde{\psi}_{2,l}(z,w)-\tilde{\psi}_{2,l}(0,w)|\le \epsilon$ for any $(z,w)\in M=\Delta\times \Delta^n$.

We also assume that $F=z^k\tilde{h}(z,w)dz\wedge dw$ on $M$ and $|h(z)-h(0)|<\epsilon$ for any $z\in \Delta$.
\begin{equation}\label{k>q formula1}
  \begin{split}
     S & =\int_{\{\pi_1^*(2p\log|z|)<-t\}}|F|^2e^{-\pi_1^*(2\log|g|)-\tilde{\psi}_2+\pi_1^*(2p\log|z|)+N-\pi_2^*(\varphi_2)}
     c(-\pi_1^*(2p\log|z|)) \\
       & \ge \int_{\{\pi_1^*(2p\log|z|)<-t\}}\frac{|z|^{2k+2p}|\tilde{h}(z,w)|^2}{d|z|^{2q}|h(z)|^2}
       e^{-\tilde{\psi}_{2,l}+N-\pi_2^*(\varphi_2)}
     c(-\pi_1^*(2p\log|z|))\\
     &\ge \int_{w\in\Delta^n}\int_{\{2p\log|z|<-t\}}|z|^{2k+2p-2q}e^N\frac{|\tilde{h}(z,w)|^2}{d|1+\epsilon|^2}
       e^{-\tilde{\psi}_{2,l}(0,w)-\epsilon-\pi_2^*(\varphi_2)}
     c(-\pi_1^*(2p\log|z|))\\
     &=\int_{w\in\Delta^n}\left(2\int_{\{2p\log|r|<-t\}}\int_{0}^{2\pi}|r|^{2k+2p-2q+1}e^{N(re^{i\theta},w)}
     |\tilde{h}(re^{i\theta},w)|^2c(-2p\log r)d\theta dr\right)\times\\
     &\frac{e^{-\tilde{\psi}_{2,l}(0,w)-\epsilon-\varphi_2(w)}}{d|1+\epsilon|^2}\\
     &\ge \frac{4\pi}{d|1+\epsilon|^2}e^{-\epsilon}
     \int_{\{2p\log|r|<-t\}}r^{2k+2p-2q+1}c(-2p\log r)dr
     \int_{w\in\Delta^n}|\tilde{h}(0,w)|^2 e^{N(0,w)-\tilde{\psi}_{2,l}(0,w)}e^{-\varphi_2(w)}\\
     &=\frac{2\pi}{pd|1+\epsilon|^2}e^{-\epsilon}
     \int_{t}^{+\infty}c(s)e^{-(\frac{k-q+1}{p}+1)s}ds
     \int_{w\in\Delta^n}|\tilde{h}(0,w)|^2 e^{N(0,w)-\tilde{\psi}_{2,l}(0,w)}e^{-\varphi_2(w)}.
  \end{split}
\end{equation}

Since $k\ge q$ and $\int_0^{+\infty}c(s)e^{-s}ds<+\infty$, we have
		\begin{equation*}
			\int_t^{+\infty}c(s)e^{-\left(\frac{k-q+1}{p}+1\right)s}ds<+\infty.
		\end{equation*}
		Letting $l\to+\infty$ in (\ref{k>q formula1}), it follows from $S<+\infty$ and Levi's Theorem that we have
		\begin{equation}\label{k>q formula2}
		\frac{2\pi}{pd|1+\epsilon|^2}e^{-\epsilon}
     \int_{t}^{+\infty}c(s)e^{-(\frac{k-q+1}{p}+1)s}ds
     \int_{w\in\Delta^n}|\tilde{h}(0,w)|^2 e^{N(0,w)-\tilde{\psi}_{2}(0,w)}e^{-\varphi_2(w)}<+\infty.
		\end{equation}

It follows from inequality \eqref{k>q formula2} that  we have
\begin{equation}
\label{contradiction}
\int_{w\in\Delta^n}|\tilde{h}(0,w)|^2 e^{N(0,w)-\tilde{\psi}_{2}(0,w)}e^{-\varphi_2(w)}<+\infty.
\end{equation}
 Note that $N|_{\{0\}\times \Delta^n}\not\equiv -\infty$. It follows from \eqref{contradiction} that there must exist $w_1\in Y$ such that
$\tilde{\psi}_2(0,w_1)>-\infty$.

Since $k\ge q$, we know $z^k\in\mathcal{I}(2\log |g|)_0$. It follows from Lemma \ref{f1zf2w} that
\begin{equation}
			\big(\pi_1^*(z^kdz)\wedge\pi_2^*(F_k),(0,y)\big)\in
\big(\mathcal{O}(K_M)\otimes\mathcal{I}(\pi_1^*(2\log|g|)+\pi_2^*(\varphi_2)\big)_{(0,y)},
		\end{equation}
for any $y\in Y$.

Let $0<\delta<1$ be a constant such that $w_1\in \Delta_{\delta}^n$. It follows from Fubini's Theorem and $c(t)e^{-t}$ is decreasing with respect to $t$ that
		\begin{flalign}
			\begin{split}
				&\int_{\Delta_{\delta}\times \Delta_{\delta}^n}|\pi_1^*(z^kdz)\wedge \pi_2^*(F_k)|^2e^{-\varphi}c(-\pi_1^*(2p\log|z|))\\
				=&\int_{\Delta_{\delta}}|z^k|^2e^{-\varphi_1}c(-2p\log|z|)|dz|^2\cdot\int_{\Delta_{\delta}^n}|F_k|^2e^{-\varphi_2}\\
				\leq&C\int_{\Delta_{\delta}}|z^k|^2e^{-\varphi_1-2p\log|z|}|dz|^2\cdot\int_{\Delta_{\delta}^n}|F_k|^2e^{-\varphi_2},\\
			\end{split}
		\end{flalign}
		where $C$ is a positive constant independent of $F$.

Consider
\begin{equation}
  \begin{split}
     I: & = \int_{\Delta_{\delta}}|z|^{2k}e^{-\varphi_1-2p\log|z|}|dz|^2\\
       &= \int_{\Delta_{\delta}}|z|^{2k}e^{-2\log |g|-u(z)}|dz|^2\\
         &= \int_{\Delta_{\delta}}|z|^{2k}e^{-2\log |g|+N(z,w)-\tilde{\psi}_2(z,w)}|dz|^2.
  \end{split}
\end{equation}
As $N$ is a plurisubharmonic function on $M$, $e^N$ has a upper bound $C_1$ on $\Delta_{\delta}\times \Delta_{\delta}^n$ (especially, $C_1$ is independent of $w$). Hence
$$I\le C_1 \int_{\Delta_{\delta}}|z|^{2k}e^{-2\log |g|-\tilde{\psi}_2(z,w)}|dz|^2=C_1
\int_{\Delta_{\delta}}|z|^{2k-2q}|h(z)|^2e^{-\tilde{\psi}_2(z,w)}|dz|^2.$$
Denote $M(w)=\int_{\Delta_{\delta}}|z|^{2k-2q}|h(z)|^2e^{-\tilde{\psi}_2(z,w)}|dz|^2$. We have $I\le M(w)$ for any $w\in \Delta_{\delta}^n$, especially $I\le M(w_1)$.

Next we prove $M(w_1)<+\infty.$
Note that $M(w_1)=\int_{\Delta_{\delta}}|z|^{2k-2q}|h(z)|^2e^{-\tilde{\psi}_2(z,w_1)}|dz|^2$. As $e^{-\tilde{\psi}_2(0,w_1)}>-\infty$, $k\ge q$ and $h(z)$ is a holomorphic function on $\Delta_{\delta}$, by H{\" o}lder inequality, we have
\begin{equation}\nonumber
\begin{split}
   M(w_1) & =\int_{\Delta_{\delta}}|z|^{2k-2q}|h(z)|^2e^{-\tilde{\psi}_2(z,w_1)}|dz|^2 \\
     & \le ( \int_{\Delta_{\delta}}|z|^{s(2k-2q)}|h(z)|^{2s}|dz|^2)^{\frac{1}{s}}
     (\int_{\Delta_{\delta}}e^{-\frac{s}{s-1}\tilde{\psi}_2(z,w_1)}|dz|^2)^{\frac{s-1}{s}}\\
     &< +\infty,
\end{split}
\end{equation}
where $s>1$ is a real number.
Hence we know $I< +\infty$. Then
$$\int_{\Delta_{\delta}\times \Delta_{\delta}^n}|\pi_1^*(z^kdz)\wedge \pi_2^*(F_k)|^2e^{-\varphi}c(-\pi_1^*(2p\log|z|))<+\infty.$$

As $S =\int_{\Delta_{\delta}\times \Delta_{\delta}^n}|F|^2e^{-\varphi}
     c(-\pi_1^*(2p\log|z|))<+\infty$, we have
		\begin{equation*}
			\int_{\Delta_{\delta}\times \Delta_{\delta}^n}|F-\pi_1^*(z^kdz)\wedge \pi_2^*(F_k)|^2e^{-\varphi}
     c(-\pi_1^*(2p\log|z|))<+\infty.
		\end{equation*}
		 Note that
		\begin{equation*}
			F-\pi_1^*(z^kdz)\wedge\pi_2^*(F_k)=\sum_{j=k+1}^{\infty}\pi_1^*(z^jdz)\wedge\pi_2^*(F_j)
		\end{equation*}
		and $k+1>q$. Using the same method as above, we can get that
		\begin{equation*}
			\big(\pi_1^*(z^{k+1}dz)\wedge\pi_2^*(F_{k+1}),(0,y)\big)\in
\big(\mathcal{O}(K_M)\otimes\mathcal{I}(\pi_1^*(2\log|g|)+\pi_2^*(\varphi_2)\big)_{(0,y)}
		\end{equation*}
		for any $y\in Y$ and
		\begin{equation*}
			\int_{\Delta_{\delta}\times \Delta_{\delta}^n}\left|F-\pi_1^*(z^kdz)\wedge\pi_2^*(F_k)-\pi_1^*(z^{k+1}dz)\wedge\pi_2^*(F_{k+1})\right|^2e^{-\varphi}c(-\pi_1^*(2p\log|z|))<+\infty.
		\end{equation*}
		By induction, we know that
		\begin{equation*}
			\big(\pi_1^*(z^jdz)\wedge\pi_2^*(F_j),(0,y)\big)\in
\big(\mathcal{O}(K_M))\otimes\mathcal{I}(\pi_1^*(2\log|g|)+\pi_2^*(\varphi_2)\big)_{(0,y)}
		\end{equation*}
		for any $j\geq k$, $y\in Y$. Then it follows from Lemma \ref{closedness} that
		\begin{equation*}
			\big(F,(0,y)\big)=\big(\sum_{j=k}^{\infty}\pi_1^*(z^jdz)\wedge\pi_2^*(F_j),(0,y)\big)
\in\big(\mathcal{O}(K_M)\otimes\mathcal{I}(\pi_1^*(2\log|g|)+\pi_2^*(\varphi_2)\big)_{(0,y)}
		\end{equation*}
		for any $y\in Y$.
	\end{proof}

Let $\Omega=\Delta$ be the unit disk in $\mathbb{C}$, where the coordinate is $z$. Let $Y=\Delta^{n}$ be the unit polydisc in $\mathbb{C}^n$, where the coordinate is $w=(w_1,\ldots,w_{n})$. Let $M=\Omega\times Y$. Let $\pi_1$, $\pi_2$ be the natural projections from $M$ to $\Omega$ and $Y$.
	
	Let $\psi_1=2p\log|z|+\psi_0$ on $\Omega$, where $p>0$ and $\psi_0$ is a negative subharmonic function on $\Omega$ with $\psi_0(0)>-\infty$. Let $\varphi_1$ be a Lebesgue measurable function on $\Omega$ such that $\varphi_1+\psi$ is a subharmonic function on $\Omega$. It follows from the Weierstrass Theorem on open Riemann surfaces (see \cite{OF81}) and
	the Siu's Decomposition Theorem, that  $\varphi_1+\psi_1=2\log|g|+2u$, where $g$ is a holomorphic function on $\Omega$ with $ord(g)_0=q\in\mathbb{N}$, $u$ is a subharmonic function on $\Omega$ such that $v(dd^cu,z)\in[0,1)$ for any $z\in \Omega$. Let $\varphi_2$ be a plurisubharmonic function on $Y$. Denote $\varphi:=\pi_1^*(\varphi_1)+\pi_2^*(\varphi_2)$ on $M$.

\begin{Lemma}
\label{equiv of multiplier ideal sheaf}
$\mathcal{I}\big(\pi^{*}_1(2\log|g|)+\pi^{*}_2(\varphi_2)\big)_{(0,y)}
=\mathcal{I}\big(\pi^{*}_1(\psi_1)+\pi^{*}_1(\varphi_1)+\pi^{*}_2(\varphi_2)\big)_{(0,y)}$ for any $y \in Y$.
\end{Lemma}
\begin{proof}
  It is easy to see that $\mathcal{I}\big(\pi^{*}_1(2\log|g|)+\pi^{*}_2(\varphi_2)\big)_{(0,y)}
  \supset\mathcal{I}\big(\pi^{*}_1(\psi_1+\varphi_1)+\pi^{*}_2(\varphi_2)\big)_{(0,y)}$ for any $y \in Y$.

  Now we prove $\mathcal{I}\big(\pi^{*}_1(2\log|g|)+\pi^{*}_2(\varphi_2)\big)_{(0,y)}
  \subset\mathcal{I}\big(\pi^{*}_1(\psi_1+\varphi_1)+\pi^{*}_2(\varphi_2)\big)_{(0,y)}$ for any $y \in Y$.

  Let $F\in \mathcal{I}\big(\pi^{*}_1(2\log|g|)+\pi^{*}_2(\varphi_2)\big)_{(0,y)}.$ Then by Lemma \ref{decomp} (although Lemma \ref{decomp} is stated for the holomorphic $(n,0)$ forms, since our case is local, the decomposition still holds for holomorphic functions), we know $F=\sum_{j=k}^{+\infty}z^j F_j$ on $M$, where the right hand side is uniformly convergent to $F$ on $M$,  $F_j(w)$ is a holomorphic function on $Y$ for any $j\ge k$ and $F_k\not\equiv 0$.

  Since the case is local, we also assume that $F=z^{k}h_1(z,w)$ on $M$, where $h_1(z,w)$ is a holomorphic function on $M$ satisfying $h_1(0,w)\neq 0$.
  $F\in \mathcal{I}\big(\pi^{*}_1(2\log|g|)+\pi^{*}_2(\varphi_2)\big)_{(0,y)}$ implies that $k\ge q$ and $$\int_{\Delta\times \Delta^n}|z|^{2k-2q}|h_1(z,w)|e^{-\pi_2^{*}(\varphi(w))}<+\infty.$$

  By Fubini's theorem and sub-mean value inequality of subharmonic functions, we have
  \begin{equation*}
    \begin{split}
       \int_{\Delta\times \Delta^n}|z|^{2k-2q}|h_1(z,w)|e^{-\pi_2^{*}(\varphi_2(w))} &
       =\int_{w\in \Delta^n} \int_{z\in\Delta}|z|^{2k-2q}|h_1(z,w)|e^{-\pi_2^{*}(\varphi_2(w))}\\
         & \ge C \int_{w\in \Delta^n} |h_1(0,w)|e^{-\varphi_2(w)}\\
         & =C \int_{w\in \Delta^n} |F_k|^2e^{-\varphi_2(w)},
    \end{split}
  \end{equation*}
  where $C>0$ is a constant.
Hence we have $(F_k,y)\in\mathcal{I}(\varphi_2)_y$ for any $y\in Y$.

As $v(dd^cu,0)\in[0,1)$, we have $(z^k,0) \in \mathcal{I}(2\log|g|)_0=\mathcal{I}(2\log|g|+2u)_0=\mathcal{I}(\psi_1+\varphi_1)_0$. It follows from Lemma \ref{f1zf2w} that we know $\big(z^k F_k,(0,y)\big)\in\mathcal{I}\big(\pi^{*}_1(\psi_1+\varphi_1)+\pi^{*}_2(\varphi_2)\big)_{(0,y)}$. Then we know
$$\int_{\Delta\times \Delta^n} |z^k F_k|^2e^{\pi_1^*(-2\log|g|+2u)-\pi_2^*(\varphi_2)}<+\infty,$$
which implies
 $$\int_{\Delta\times \Delta^n} |z^k F_k|^2e^{\pi_1^*(-2\log|g|)-\pi_2^*(\varphi_2)}<+\infty.$$
  Hence we have $\big(F-z^kF_k,(0,y)\big)\in \mathcal{I}\big(\pi^{*}_1(2\log|g|)+\pi^{*}_2(\varphi_2)\big)_{(0,y)}$.

  Denote $\tilde{F}_{k+1}=F-z^k F_k$ on $\Delta\times \Delta^n$. Note that $\tilde{F}_{k+1}=\sum_{j=k+1}^{+\infty}z^jF_j$ on $M$ and $\big(\tilde{F}_{k+1},(0,y)\big)\in \mathcal{I}\big(\pi^{*}_1(2\log|g|)+\pi^{*}_2(\varphi_2)\big)_{(0,y)}$.

  By using similar discussion as above, we know $$\big(z^{k+1} F_{k+1},(0,y)\big)\in\mathcal{I}\big(\pi^{*}_1(\psi_1+\varphi_1)+\pi^{*}_2(\varphi_2)\big)_{(0,y)},$$
 and
  $$\int_{\Delta\times \Delta^n} |z^{k+1} F_{k+1}|^2e^{\pi_1^*(-2\log|g|)-\pi_2^*(\varphi_2)}<+\infty.$$
   Hence $\big(\tilde{F}_{k+1}-z^{k+1}F_{k+1},(0,y)\big)\in \mathcal{I}\big(\pi^{*}_1(2\log|g|)+\pi^{*}_2(\varphi_2)\big)_{(0,y)}$. Denote $\tilde{F}_{k+2}=\tilde{F}_{k+1}-z^{k+1}F_{k+1}$ on $\Delta\times \Delta^n$. Note that $\tilde{F}_{k+2}=\sum_{j=k+2}^{+\infty}z^j F_j$ on $M$ and $\big(\tilde{F}_{k+2},(0,y)\big)\in \mathcal{I}\big(\pi^{*}_1(2\log|g|)+\pi^{*}_2(\varphi_2)\big)_{(0,y)}$.

By induction, we know that for any $j\ge k$,
   $$\big(z^{j} F_{j},(0,y)\big)\in\mathcal{I}\big(\pi^{*}_1(\psi_1+\varphi_1)+\pi^{*}_2(\varphi_2)\big)_{(0,y)}$$
 holds. Then it follows from Lemma \ref{closedness} that we know
    $$\big(F,(0,y)\big)\in\mathcal{I}\big(\pi^{*}_1(\psi_1+\varphi_1)+\pi^{*}_2(\varphi_2)\big)_{(0,y)}$$
  for any $y\in Y=\Delta^n$.

\end{proof}
 We recall a well known result about multiplier ideal sheaves.
 \begin{Lemma}[see \cite{BGY-concavity6}]
 	\label{l:phi1+phi2}Let $\Phi_1$ and $\Phi_2$ be plurisubharmonic functions on $\Delta^n$ satisfying $\Phi_2(o)>-\infty$, where $n\in \mathbb{Z}_{>0}$ and $o$ is the origin in $\Delta^n$. Then $\mathcal{I}(\Phi_1)_o=\mathcal{I}(\Phi_1+\Phi_2)_o$.
 \end{Lemma}

\subsection{Optimal jet $L^2$ extensions}
\label{some construction before extension}
Let $\Omega$ be an open Riemann surface with nontrivial Green functions. Let $Z_{\Omega}=\{z_j:j \in \mathbb{N}_+ \& j<\gamma\}$ be a subset of $\Omega$ of discrete points, where $\gamma\in \mathbb{Z}^+_{\ge 2}$ or $\gamma=+\infty$. Let $Y$ be an $n-$dimensional weakly pseudoconvex K\"ahler manifold. Denote $M=\Omega \times Y$. Let $\pi_1$ and $\pi_2$ be the natural projections from $M$ to $\Omega$ and $Y$ respectively. Denote $Z_0:=Z_{\Omega}\times Y$. Denote $Z_j:=\{z_j\}\times Y$.

Let $\tilde{M}\subset M$ be an $n-$dimensional weakly pseudoconvex K\"ahler manifold satisfying that $Z_0\subset \tilde{M}$.  Let $F$ be a holomorphic $(n,0)$ form on a neighborhood $U_0\subset \tilde{M}$ of $Z_0$.

Let $\psi$ be a plurisubharmonic function on $\tilde{M}$.
It follows from Siu's decomposition theorem that $$dd^{c}\psi=\sum\limits_{j\ge1}2p_j[Z_j]+\sum\limits_{i\ge 1}\lambda_i[A_i]+R,$$
where $[Z_j]$ and $[A_i]$ are the currents of integration over an irreducible $(n-1)-$dimensional analytic set, and where $R$ is a closed positive current with the property that $dimE_c(R)<n-1$ for every $c>0$, where $E_c(R)=\{x\in \tilde M: v(R,x)\ge c\}$ is the upperlevel sets of Lelong number. We assume that $p_j>0$ for any $1\le j< \gamma$.

Then $N:=\psi-\pi^{*}_1\big(\sum\limits_{j\ge1}2p_jG_{\Omega}(z,z_j)\big)$ is a plurisubharmonic function on $\tilde M$. We assume that $N\le0$.

Let $\varphi_1$ be a Lebesgue measurable function on $\Omega$ such that $\psi+\pi^{*}_1(\varphi)$ is a plurisubharmonic function on $\tilde M$. With similar discussion as above, by Siu's decomposition theorem, we have

$$dd^{c}(\psi+\pi^{*}_1(\varphi))=\sum\limits_{j\ge1}2\tilde{q}_j[Z_j]+\sum\limits_{i\ge 1}\tilde{\lambda}_i[\tilde{A}_i]+\tilde{R},$$
where $\tilde{q}_j\ge 0$ for any $1\le j< \gamma$.

By Weierstrass theorem on open Riemann surface, there exists a holomorphic function $g$ on $\Omega$ such that $ord_{z_j}(g)=q_j:=[\tilde{q}_j]$ for any $z_j\in Z_{\Omega}$ and $g(z)\neq 0$ for any $z\notin Z_{\Omega}$, where $[q]$ equals to the integer part of the nonnegative real number $q$.
Then we know that there exists a plurisubharmonic function $\tilde{\psi}_2\in Psh(\tilde M)$ such that $$\psi+\pi^{*}_1(\varphi_1)=\pi^{*}_1(2\log|g|)+\tilde{\psi}_2.$$

Let $\varphi_2\in Psh(Y)$. Denote $\varphi:=\pi^{*}_1(\varphi_1)+\pi^{*}_2(\varphi_2)$.

For $1\le j<\gamma$, let $(V_{j},\tilde{z}_j)$ be a local coordinated open neighborhood of $z_j$ in $\Omega$ satisfying $V_{j}\Subset \Omega$, $\tilde{z}_j(z_j)=0$ under the local coordinate and $V_{j}\cap V_{k}=\emptyset$ for any $j\neq k$. Denote $V_0:=\cup_{1\le j<\gamma} V_j$. We assume that $g=d_j\tilde{z}_j^{q_j}h_j(z)$ on $V_j$, where $d_j$ is a constant, $h_j(z)$ is a holomorphic function on $V_j$ and $h(z_j)=1$.

Let $c(t)$ be a positive measurable function on $(0,+\infty)$ satisfying that $c(t)e^{-t}$ is decreasing and $\int_{0}^{+\infty}c(s)e^{-s}ds<+\infty$.

 We have the following lemma.

 \begin{Lemma}
 \label{optimal extension}
 Let $F$ be a holomorphic $(n,0)$ form on $U_0$ such that for $1\le j <\gamma$, $F=\pi^{*}_1(\tilde{z}_j^{k_j}f_jdz_j)\wedge\pi^{*}_2(F_j)$ on $U_j\Subset U_0\cap (V_j\times Y)$, where $U_j$ is an open neighborhood of $Z_j$ in $\tilde{M}$, $k_j$ is a nonnegative integer, $f_j$ is a holomorphic function on $V_j$ satisfying $f_j(z_j)=a_j\neq 0$ and $F_j$ is a holomorphic $(n-1,0)$ form on $Y$.

 Denote $I_F:= \{j:1\le j < \gamma \& k_j+1-q_j\le 0\}$. Assume that $k_j+1=q_j$ for any $j\in I_F$ and $\tilde{\psi}_2|_{Z_j}$ is not identically $-\infty$ on $Z_j$.
 If
 \begin{equation}
 \label{measure finite}
 \sum_{j\in I_F}\frac{2\pi|a_j|^2}{p_j|d_j|^2}\int_Y|F_j|^2e^{-\varphi_2-\tilde{\psi}_2(z_j,w)}<+\infty
 \end{equation}
   and
 \begin{equation}
 \label{measure finite}
 \int_{Y} |F_j|^2e^{-\varphi_2-\tilde{\psi_2}(z_j,w)}<+\infty,
 \end{equation}
 for any $j\notin I_F$.
 Then there exists a holomorphic $(n,0)$ form $\tilde{F}$ on $\tilde M$ such that
 $\big(\tilde{F}-F,(z_j,y)\big)\in
 \big(\mathcal{O}(K_{\tilde M})\otimes\mathcal{I}(\pi^{*}_1(2\log |g|)+\pi^{*}_2(\varphi_2))\big)_{(z_j,y)}$ for any $1\le j <\gamma$ and $y \in Y$ and
  \begin{equation}
 \label{estimate in L2}
 \begin{split}
    \int_{\tilde M}|\tilde{F}|^2c(-\psi)e^{-\varphi}
 \le(\int_{0}^{+\infty}c(s)e^{-s}ds)
 \sum_{j\in I_F}\frac{2\pi|a_j|^2}{p_j|d_j|^2}\int_Y|F_j|^2e^{-\varphi_2-\tilde{\psi}_2(z_j,w)}<+\infty.
 \end{split}
 \end{equation}
 \end{Lemma}

 \begin{proof}
  Denote $G=\sum\limits_{j\ge1}2p_jG_{\Omega}(z,z_j)$. Note that $\psi\le \pi^{*}_1(G)$ and $c(t)e^{-t}$ is decreasing. We have
 $$c(-\psi)e^{-\varphi}\le c(-\pi^{*}_1(G))e^{-N-\pi^{*}_1(\varphi_1)-\pi^{*}_2(\varphi_2)}.$$
To prove Lemma \ref{optimal extension}, it suffice to prove
   \begin{equation}
 \label{goal in L2}
 \begin{split}
    &\int_{\tilde M}|\tilde{F}|^2c\big(-\pi^{*}_1(G)\big)e^{-N-\pi^{*}_1(\varphi_1)-\pi^{*}_2(\varphi_2)}\\
 \le&(\int_{0}^{+\infty}c(s)e^{-s}ds)
 \sum_{j\in I_F}\frac{2\pi|a_j|^2}{p_j|d_j|^2}\int_Y|F_j|^2e^{-\varphi_2-\tilde{\psi}_2(z_j,w)}<+\infty.
 \end{split}
 \end{equation}

The following remark shows that we only need to prove formula \eqref{goal in L2} when $Z_{\Omega}$ is a finite set.
 \begin{Remark}

 It follows from Lemma \ref{approximate of Green function} that there exists a sequence of open Riemann surfaces $\{\Omega_l\}_{l\in\mathbb{Z}^+}$ such that $\Omega_l\Subset \Omega_{l+1}\Subset \Omega$, $\cup_{l\in \mathbb{Z}^+}\Omega_l=\Omega$ and $\{G_{\Omega_l}(\cdot,z_0)-G_{\Omega}(\cdot,z_0)\}$ is decreasingly convergent to $0$ on $\Omega$ with respect to $l$ for any $z_0\in \Omega$.

 Denote $Z_l:=\Omega_l\cap Z_0$. As $Z_{\Omega}$ is a subset of $\Omega$ of discrete points, $Z_l$ is a set of finite points.

 Denote $$G_l:=\sum\limits_{z_j\in Z_l}2p_jG_{\Omega}(z,z_j)$$ and
 $$\varphi_{1,l}:=\varphi_1+G-G_l.$$
 Then we have $N+\pi_1^*(G_l)+\pi_1^*(\varphi_{1,l})=N+\pi_1^*(G)+\pi_1^*(\varphi_1)=\psi+\pi_1^*(\varphi_1).$

 Let $I_l:=I_F\cap\{j: z_j\in Z_l\}$. Denote $\tilde{M}_l:=(\Omega_l\times Y)\cap \tilde{M}$. We note that $\tilde{M}_l$ is weakly pseudoconvex K\"ahler. Now we assume that the formula \eqref{goal in L2} holds on $\tilde{M}_l$, i.e. we have
   \begin{equation}\nonumber
 \begin{split}
    &\int_{\tilde{M}_l}|\tilde{F}_l|^2c\big(-\pi^{*}_1(G_l)\big)e^{-N-\pi^{*}_1(\varphi_{1,l})-\pi^{*}_2(\varphi_2)}\\
 \le&(\int_{0}^{+\infty}c(s)e^{-s}ds)
 \sum_{j\in I_l}\frac{2\pi|a_j|^2}{p_j|d_j|^2}\int_Y|F_j|^2e^{-\varphi_2-\tilde{\psi}_2(z_j,w)}<+\infty,
 \end{split}
 \end{equation}
 where  $F_l$ is a holomorphic $(n,0)$ form on $M_l$ satisfying $(F_l-F,(z_j,y))\in
 (\mathcal{O}(K_M)\otimes\mathcal{I}(\pi^{*}_1(2\log|g_0|)+\pi^{*}_2(\varphi_2)))_{(z_j,y)}$ for any $l$, $z_j\in Z_l$, and $y\in Y$.

 As $G\le G_l$ and $c(t)e^{-t}$ is decreasing on $(0,+\infty)$, we have
    \begin{equation}
 \label{goal in L2 for finite points}
 \begin{split}
 &\int_{\tilde{M}_l}|\tilde{F}_l|^2c\big(-\pi^{*}_1(G)\big)e^{-N-\pi^{*}_1(\varphi_{1})-\pi^{*}_2(\varphi_2)}\\
   \le &\int_{\tilde{M}_l}|\tilde{F}_l|^2c\big(-\pi^{*}_1(G_l)\big)e^{-N-\pi^{*}_1(\varphi_{1,l})-\pi^{*}_2(\varphi_2)}\\
 \le&(\int_{0}^{+\infty}c(s)e^{-s}ds)
 \sum_{j\in I_l}\frac{2\pi|a_j|^2}{p_j|d_j|^2}\int_Y|F_j|^2e^{-\varphi_2-\tilde{\psi}_2(z_j,w)}<+\infty.
 \end{split}
 \end{equation}

Note that $\pi^{*}_1(G)$ is continuous on $\tilde{M}\backslash Z_0$, where $Z_0$ is a closed analytic subset of $\tilde{M}$ and $N+\pi_1^*(G)+\pi_1^*(\varphi_1)=\psi+\pi_1^*(\varphi_1)$ is a plurisubharmonic function on $M$. For any compact subset $K$ of $\tilde{M}\backslash Z_0$, there exist $\hat{l}$ (depending on the choice of $K$) such that $K \subset \subset \tilde M_{\hat{l}}$ and $C_K>0$ such that $\frac{e^{N+\pi^{*}_1(\varphi_{1})+\pi^{*}_2(\varphi_2)}}{c(-\pi^{*}_1(G))}=
\frac{e^{N+\pi^{*}_1(\varphi_{1})+\pi^{*}_2(\varphi_2)+\pi^{*}_1(G)}}{c(-\pi^{*}_1(G))e^{\pi^{*}_1(G)}}
\le C_K$ on $K$. It follows from Lemma \ref{l:converge} and diagonal method that there exists a subsequence of $\{F_l\}$ (also denoted by $\{F_l\}$), which is compactly convergent to a holomorphic $(n,0)$ form $\tilde{F}$ on $\tilde{M}$. Combining formula \eqref{goal in L2 for finite points} and Fatou's lemma, we have

    \begin{equation}\nonumber
 \begin{split}
 &\int_{\tilde{M}_l}|\tilde{F}|^2c(-\pi^{*}_1(G))e^{-N-\pi^{*}_1(\varphi_{1})-\pi^{*}_2(\varphi_2)}\\
 \le&\liminf_{l\to+\infty}\int_{\tilde{M} _l}|\tilde{F}_l|^2c(-\pi^{*}_1(G))e^{-N-\pi^{*}_1(\varphi_{1})-\pi^{*}_2(\varphi_2)}\\
   \le &\liminf_{l\to+\infty}\int_{\tilde{M}_l}|\tilde{F}_l|^2c(-\pi^{*}_1(G_l))e^{-N-\pi^{*}_1(\varphi_{1,l})-\pi^{*}_2(\varphi_2)}\\
 \le&\liminf_{l\to+\infty}(\int_{0}^{+\infty}c(s)e^{-s}ds)
 \sum_{j\in I_l}\frac{2\pi|a_j|^2}{p_j|d_j|^2}\int_Y|F_j|^2e^{-\varphi_2-\tilde{\psi}_2(z_j,w)}\\
 \le &(\int_{0}^{+\infty}c(s)e^{-s}ds)
 \sum_{j\in I_F}\frac{2\pi|a_j|^2}{p_j|d_j|^2}\int_Y|F_j|^2e^{-\varphi_2-\tilde{\psi}_2(z_j,w)}<+\infty.
 \end{split}
 \end{equation}

As $\{F_l\}$ is compactly convergent to $\tilde{F}$ on $\tilde{M}$ and $(F_l-F,(z_j,y))\in
 (\mathcal{O}(K_{\tilde{M}})\otimes\mathcal{I}(\pi^{*}_1(2\log|g_0|)+\pi^{*}_2(\varphi_2)))_{(z_j,y)}$ for any $l$, $z_j\in Z_l$, and $y\in Y$. It follows from Lemma \ref{closedness} that
 $(\tilde{F}-F,(z_j,y))\in
 (\mathcal{O}(K_{\tilde{M}})\otimes\mathcal{I}(\pi^{*}_1(2\log|g_0|)+\pi^{*}_2(\varphi_2)))_{(z_j,y)}$ for any $z_j\in Z_{\Omega}$, and $y\in Y$.

 \end{Remark}

We continue to prove Lemma \ref{optimal extension}. Now we assume that $\gamma=m+1$ i.e. $Z_{\Omega}=\{z_j:1\le j \le m\}$ and $I_F=\{1,2,\ldots,m_1\}$, where $m_1\le m$.

   Denote $\tilde{\psi}_{2,l}=\max\{\tilde{\psi}_2,-l\}$, where $l$ is a positive integer. As $\tilde{\psi}_2$ is plurisubharmonic, we know $\{\tilde{\psi}_{2,l}\}_{l=1}^{+\infty}$ is a sequence of plurisubharmonic functions on $\tilde M$ decreasingly convergent to $\tilde{\psi}_2$. We also note that every $\tilde{\psi}_{2,l}$ is lower bounded.

   When $t_0$ is large enough, we know that $\{G<-t\}\Subset V_0$, for any $t>t_0$.
    As $\tilde{M}$ is a weakly pseudoconvex K\"ahler manifold, there exists a sequence of weakly pseudoconvex K\"ahler manifolds $\tilde {M}_s$ satisfying $\tilde {M}_1\Subset \tilde{M}_2 \cdots \Subset \tilde{M}_s\Subset\cdots \tilde{M}$ and $\cup_{s\in \mathbb{N}^+} \tilde {M}_s=\tilde{M}$.

    It is easy to verify that $\int_{\{\pi^{*}_1(G)<-t\}\cap \tilde{M}_s}|F|^2<+\infty,$ and it follows from formula \eqref{measure finite} and Fubini's theorem that
    $$\int_{\tilde{M}_s}\mathbb{I}_{\{-t-1<\pi^{*}_1(G)<-t\}}|F|^2
    e^{\pi^{*}_1(-2\log|g|)-\tilde{\psi}_{2,l}-\pi^{*}_2(\varphi_2)}<+\infty.$$

It follows from Lemma \ref{L2 method} that there exists a holomorphic $(n,0)$ form $F_{t,l,s}$ on $M$ such that
    \begin{equation}\label{original estimate}
    \begin{split}
      &\int_{\tilde{M}_s}|F_{t,l,s}-(1-b_{t,1}(\pi^{*}_1(G)))F|^2
      e^{\pi^{*}_1(-2\log|g|)-\tilde{\psi}_{2,l}-\pi^{*}_2(\varphi_2)+v_{t,1}(\pi^{*}_1(G))}c(-v_{t,1}(\pi^{*}_1(G)))\\
    \le&(\int_{0}^{t+1}c(s)e^{-s}ds)\int_{\tilde{M}_s}\mathbb{I}_{\{-t-1<\pi^{*}_1(G)<-t\}}|F|^2
    e^{\pi^{*}_1(-2\log|g|)-\tilde{\psi}_{2,l}-\pi^{*}_2(\varphi_2)}<+\infty,
    \end{split}
    \end{equation}
    for any $t\ge t_0$.

As $b_{t,1}(\hat{t})=0$ for $\hat{t}$ largely enough, then we know  $(F_{t,l,s}-F,(z_0,y))\in
 (\mathcal{O}(K_{\tilde{M}})\otimes\mathcal{I}(\pi^{*}_1(2\log|g_0|)+\pi^{*}_2(\varphi_2)))_{(z_j,y)}$ for any $(z_j,y) \in Z_0\cap \tilde{M}_s$.

For any $\epsilon>0$, there exists $t_1>t_0$ such that
\\
(1) for any $j\in\{1,2,\ldots,m\}$\\
\centerline{$\sup\limits_{\tilde z_j\in\{G<-t_1\}\cap V_j}|g_1(\tilde z)-g_1(z_j)|<\epsilon,$}

where $g_1(z)$ is a smooth function on $V_0$ satisfying $g_1(\tilde z)|_{V_j}:=G-2p_j\log|\tilde{z}_j|$.
\\
(2) for any $j\in\{1,2,\ldots,m\}$\\
\centerline{$\sup\limits_{\tilde z_j\in\{G<-t_1\}\cap V_j}|f_j(\tilde{z}_j)-a_j|<\epsilon.$}
\\
(3) for any $j\in\{1,2,\ldots,m\}$\\
\centerline{$\sup\limits_{\tilde z_j\in\{G<-t_1\}\cap V_j}|h_j(\tilde{z}_j)-1|<\epsilon.$}

 For any $(z_j,y) \in Z_0\cap \tilde{M}_s$, letting $(U_y,w)$ be a small local coordinated open neighborhood of $y$ and shrinking $V_j$ if necessary, we have $V_j\times U_y\Subset U_j$ for any $j$. Recall that $V_0:=\cup_{j\ge 1}V_j$. Assume that $F_j=\tilde{h}_j(w)dw$ on $U_y$, where $\tilde{h}_j$ is a holomorphic function on $U_y$ and $dw=dw_1\wedge dw_2\wedge\ldots\wedge dw_n$. There exists $t_2>t_1$ such that when $t>t_2$, $(\{G<-t\}\times U_y)\cap (V_j\times U_y) \Subset V_j\times U_y$, for any $j$.

 When $t>t_2$, direct calculation shows
 \begin{equation}\nonumber
 \begin{split}
     & \int_{V_0\times U_y} \mathbb{I}_{\{-t-1<\pi^{*}_1(G)<-t\}}|F|^2
    e^{\pi^{*}_1(-2\log|g|)-\tilde{\psi}_{2,l}-\pi^{*}_2(\varphi_2)}\\
      =& \sum_{j=1}^{m}\int_{\{-t-1<G<-t\}\times U_y}\frac{|\tilde{z}_j|^{2k_j}|f_j|^2|F_j|^2}{|d_j|^{2}|\tilde{z}|^{2q_j}|h_j|^{2}}
      e^{-\tilde{\psi}_{2,l}-\pi^{*}_2(\varphi_2)}|d\tilde{z}_j|^2|dw|^2\\
      \le & \sum_{j=1}^{m}\int_{y\in U_y}\left(\int_{\{-t-1-\epsilon-g_1(z_0)<2p_j\log|\tilde{z}|<-t+\epsilon-g_1(z_0)\}}
     |\tilde{z}_j|^{2k_j-2q_j} \frac{|f_j|^2}{|d_j|^{2}|h_j|^{2}} e^{-\tilde{\psi}_{2,l}(z,w)}|d\tilde{z}_j|^2\right)\times \\
     &
     |\tilde{h}_j|^2
      e^{-\varphi_2}|dw|^2.
 \end{split}
 \end{equation}

Note that $2k_j-2q_j=-2$ for any $1\le j\le m_1$ and $2k_j-2q_j\ge 0$ for $m_1< j\le m$. When $t$ is laege enough, for any $j$, the integral $$\int_{\{-t-1-\epsilon-g_1(z_0)<2p_j\log|\tilde{z}|<-t+\epsilon-g_1(z_0)\}}
     |\tilde{z}_j|^{2k_j-2q_j}
     \frac{|f_j|^2}{|d_j|^{2}|h_j|^{2}} e^{-\tilde{\psi}_{2,l}(z,w)}|d\tilde{z}_j|^2$$ is uniformly bounded with respect to $t$. It follows from \eqref{measure finite} that $\int_{U_y} |\tilde{h}_j|^2
      e^{-\varphi_2}|dw|^2<+\infty$ for any $j$.
Then, by Fatou's lemma and Lemma \ref{limit discussion} (we use Lemma \ref{limit discussion} for the case $n=1$), we have

\begin{equation}\nonumber
 \begin{split}
     & \limsup_{t\to+\infty}\int_{V_0\times U_y} \mathbb{I}_{\{-t-1<\pi^{*}_1(G)<-t\}}|F|^2
    e^{\pi^{*}_1(-2\log|g|)-\tilde{\psi}_{2,l}-\pi^{*}_2(\varphi_2)}\\
      \le &\limsup_{t\to+\infty}  \sum_{j=1}^{m}\int_{y\in U_y}(\int_{\{-t-1-\epsilon-g_1(z_0)<2p_j\log|\tilde{z}|<-t+\epsilon-g_1(z_0)\}}
     |\tilde{z}_j|^{2k_j-2q_j}\times \\
     &\frac{|f_j|^2}{|d_j|^{2}|h_j|^{2}} e^{-\tilde{\psi}_{2,l}(z,w)}|d\tilde{z}_j|^2)
     |\tilde{h}_j|^2
      e^{-\varphi_2}|dw|^2\\
      \le& \sum_{j=1}^{m}\int_{y\in U_y}\limsup_{t\to+\infty} (\int_{\{-t-1-\epsilon-g_1(z_0)<2p_j\log|\tilde{z}|<-t+\epsilon-g_1(z_0)\}}
     |\tilde{z}_j|^{2k_j-2q_j}\times \\
     &\frac{|f_j|^2}{|d_j|^{2}|h_j|^{2}} e^{-\tilde{\psi}_{2,l}(z,w)}|d\tilde{z}_j|^2)
     |\tilde{h}_j|^2
      e^{-\varphi_2}|dw|^2\\
      \le&\sum_{j=1}^{m_1}\int_{y\in U_y}\frac{2\pi(1+2\epsilon)|a_j|^2}{p_j|d_j|^{2}}e^{-\tilde{\psi}_{2,l}(z_0,w)}|\tilde{h}_j|^2
      e^{-\varphi_2}|dw|^2.
 \end{split}
 \end{equation}
Let $\epsilon\to +\infty$, we have
\begin{equation}\nonumber
 \begin{split}
     & \limsup_{t\to+\infty}\int_{V_0\times U_y} \mathbb{I}_{\{-t-1<\pi^{*}_1(G)<-t\}}|F|^2
    e^{\pi^{*}_1(-2\log|g|)-\tilde{\psi}_{2,l}-\pi^{*}_2(\varphi_2)}\\
      \le&\sum_{j=1}^{m_1}\int_{y\in U_y}\frac{2\pi|a_j|^2}{p_j|d_j|^{2}}e^{-\tilde{\psi}_{2,l}(z_0,w)}|\tilde{h}_j|^2
      e^{-\varphi_2}|dw|^2.
 \end{split}
 \end{equation}

 As $y$ and $U_y$ are arbitrarily chosen, we have
 \begin{equation}\label{estimate for measure}
 \begin{split}
     \limsup_{t\to+\infty}& \int_{\tilde{M}_s} \mathbb{I}_{\{-t-1<\pi^{*}_1(G)<-t\}}|F|^2
    e^{\pi^{*}_1(-2\log|g|)-\tilde{\psi}_{2,l}-\pi^{*}_2(\varphi_2)}\\
 \le&\sum_{j=1}^{m_1}\frac{2\pi|a_j|^2}{p_j|d_j|^{2}}\int_{(\{z_j\}\times Y)\cap \tilde{M}_s}|F_j|^2e^{-\tilde{\psi}_{2,l}(z_0,w)-\varphi_2}.
 \end{split}
 \end{equation}
 Since $v_{t,1}(\pi^{*}_1(G))\ge \pi^{*}_1(G)$ and $c(t)e^{-t}$ is decreasing with respect to $t$, it follows from \eqref{original estimate} and \eqref{estimate for measure} that
 \begin{equation}
 \label{estimate for Ftls}
 \begin{split}
 &\limsup_{t\to +\infty}  \int_{\tilde{M}_s}|F_{t,l,s}-(1-b_{t,1}(\pi^{*}_1(G)))F|^2
      e^{\pi^{*}_1(-2\log|g|)-\tilde{\psi}_{2,l}-\pi^{*}_2(\varphi_2)+\pi^{*}_1(G)}c(-\pi^{*}_1(G))\\
 \le&\limsup_{t\to +\infty}  \int_{\tilde{M}_s}|F_{t,l,s}-(1-b_{t,1}(\pi^{*}_1(G)))F|^2
      e^{\pi^{*}_1(-2\log|g|)-\tilde{\psi}_{2,l}-\pi^{*}_2(\varphi_2)+v_{t,1}(\pi^{*}_1(G))}c(-v_{t,1}(\pi^{*}_1(G)))\\
 \le& \limsup_{t\to+\infty}(\int_{0}^{t+1}c(t_1)e^{-t_1}dt_1) \int_{\tilde{M}_s} \mathbb{I}_{\{-t-1<\pi^{*}_1(G)<-t\}}|F|^2
    e^{\pi^{*}_1(-2\log|g|)-\tilde{\psi}_{2,l}-\pi^{*}_2(\varphi_2)}\\
    \le&(\int_{0}^{+\infty}c(t_1)e^{-t_1}dt_1)
    \sum_{j=1}^{m_1}\frac{2\pi|a_j|^2}{p_j|d_j|^{2}}\int_{(\{z_j\}\times Y)\cap \tilde{M}_s}|F_j|^2e^{-\tilde{\psi}_{2,l}(z_0,w)-\varphi_2}.
 \end{split}
 \end{equation}
Note that $k_j-q_j=-1$ for $1\le j\le m_1$, $k_j-q_j\ge 0$ for $m_1< j\le m$, when $t$ is large enough, we have
\begin{equation}\nonumber
\begin{split}
    & \int_{(\{\pi^{*}_1(G)<-t\}\times Y)\cap \tilde{M}_s}|F|^2 e^{\pi^{*}_1(-2\log|g|+G)-\tilde{\psi}_{2,l}-\pi^{*}_2(\varphi_2)}c(-\pi^{*}_1(G))\\
     \le& \sum_{j=1}^{m} C_1\int_{\{G<-t\}}|z_j|^{2k_j}|f_j|^{2}e^{-2\log|g_0|}e^{-G}c(-G)\int_{\{z_j\}\times Y}|F_j|^2e^{-\varphi_2}\\
     \le&C_2 \sum_{j=1}^{m}\int_{t}^{+\infty}c(t_1)e^{-t_1}dt_1\int_{(\{z_j\}\times Y)\cap \tilde{M}_s}|F_j|^2e^{-\varphi_2}<+\infty,
\end{split}
\end{equation}
where $C_1$ and $C_2$ are constants.
Hence we have
$$\limsup_{t\to +\infty}  \int_{\tilde{M}_s}|(1-b_{t,1}(\pi^{*}_1(G)))F|^2
      e^{\pi^{*}_1(-2\log|g|)-\tilde{\psi}_{2,l}-\pi^{*}_2(\varphi_2)+\pi^{*}_1(G)}c(-\pi^{*}_1(G))<+\infty.$$
Combining with \eqref{estimate for Ftls}, we know
$$\limsup_{t\to +\infty}  \int_{\tilde{M}_s}|F_{t,l,s}|^2
      e^{\pi^{*}_1(-2\log|g|)-\tilde{\psi}_{2,l}-\pi^{*}_2(\varphi_2)+\pi^{*}_1(G)}c(-\pi^{*}_1(G))<+\infty.$$
By Lemma \ref{l:converge}, we know there exists a subsequence of $\{F_{t,l,s}\}_{t\to+\infty}$ (still denoted by $\{F_{t,l,s}\}_{t\to+\infty}$) compactly convergent to a holomorphic $(n,0)$ form $F_{l,s}$ on $M_s$. It follows from \eqref{estimate for Ftls} and Fatou's lemma that
 \begin{equation}
 \begin{split}
 &\int_{\tilde{M}_s}|F_{l,s}|^2e^{\pi^{*}_1(-2\log|g|)-\tilde{\psi}_{2,l}-\pi^{*}_2(\varphi_2)+\pi^{*}_1(G)}c(-\pi^{*}_1(G))\\
 \le&
 \liminf_{t\to +\infty}  \int_{\tilde{M}_s}|F_{t,l,s}-(1-b_{t,1}(\pi^{*}_1(G)))F|^2
      e^{\pi^{*}_1(-2\log|g|)-\tilde{\psi}_{2,l}-\pi^{*}_2(\varphi_2)+\pi^{*}_1(G)}c(-\pi^{*}_1(G))\\
    \le&(\int_{0}^{+\infty}c(t_1)e^{-t_1}dt_1)
    \sum_{j=1}^{m_1}\frac{2\pi|a_j|^2}{p_j|d_j|^{2}}\int_{(\{z_j\}\times Y)\cap \tilde{M}_s}|F_j|^2e^{-\tilde{\psi}_{2,l}(z_0,w)-\varphi_2}<+\infty.
 \end{split}
 \end{equation}
As $\tilde{\psi}_{2,l}$ is decreasingly convergent to $\tilde{\psi}_2$, when $l\to+\infty$, and $\int_{Y}|F_j|^2e^{-\tilde{\psi}_{2}(z_j,w)-\pi^{*}_2(\varphi_2)}<+\infty$ for any $j$,
 \begin{equation}
 \label{estimate for Fls}
 \begin{split}
 &\int_{\tilde{M}_s}|F_{l,s}|^2e^{\pi^{*}_1(-2\log|g|)-\tilde{\psi}_{2,l}-\pi^{*}_2(\varphi_2)+\pi^{*}_1(G)}c(-\pi^{*}_1(G))\\
    \le&(\int_{0}^{+\infty}c(t_1)e^{-t_1}dt_1)
     \sum_{j=1}^{m_1}\frac{2\pi|a_j|^2}{p_j|d_j|^{2}}\int_{(\{z_j\}\times Y)\cap \tilde{M}_s}|F_j|^2e^{-\tilde{\psi}_{2}(z_j,w)-\varphi_2}<+\infty.
 \end{split}
 \end{equation}
As $\tilde{\psi}_{2,l}$ is decreasingly convergent to $\tilde{\psi}_2$, for any compact subset $K\subset \tilde{M}_s$, we have $$\inf_{l}\inf_{K}e^{-\tilde{\psi}_{2,l}}\ge\inf_{K}e^{-\tilde{\psi}_{2,1}}>0,$$
then it follows from \eqref{estimate for Fls} that
$$\sup_l\int_K|F_{l,s}|^2e^{\pi^{*}_1(-2\log|g|)-\pi^{*}_2(\varphi_2)+\pi^{*}_1(G)}c(-\pi^{*}_1(G))<+\infty.$$
Hence it follows from Lemma \ref{l:converge} and diagonal method that there exists a subsequence of $\{F_{l,s}\}_{l\to+\infty}$ (still denoted by $\{F_{l,s}\}_{l\to+\infty}$) compactly convergent to a holomorphic $(n,0)$ form $F_{s}$ on $\tilde{M}_s$. It follows from \eqref{measure finite}, \eqref{estimate for Fls} and Fatou's lemma that
 \begin{equation}
 \label{estimate for Fs}
 \begin{split}
 &\int_{\tilde{M}_s}|F_{s}|^2e^{\pi^{*}_1(-2\log|g|)-\tilde{\psi}_{2}-\pi^{*}_2(\varphi_2)+\pi^{*}_1(G)}c(-\pi^{*}_1(G))\\
 \le&\liminf_{l\to+\infty}\int_{\tilde{M}_s}|F_{l,s}|^2e^{\pi^{*}_1(-2\log|g|)-\tilde{\psi}_{2,l}-\pi^{*}_2(\varphi_2)+\pi^{*}_1(G)}c(-\pi^{*}_1(G))\\
    \le&(\int_{0}^{+\infty}c(t_1)e^{-t_1}dt_1)
    \sum_{j=1}^{m_1}\frac{2\pi|a_j|^2}{p_j|d_j|^{2}}\int_{(\{z_j\}\times Y)\cap \tilde{M}_s}|F_j|^2e^{-\tilde{\psi}_{2}(z_j,w)-\varphi_2}\\
    \le&(\int_{0}^{+\infty}c(t_1)e^{-t_1}dt_1)
    \sum_{j=1}^{m_1}\frac{2\pi|a_j|^2}{p_j|d_j|^{2}}\int_{Y}|F_j|^2e^{-\tilde{\psi}_{2}(z_j,w)-\varphi_2}<+\infty.\\
 \end{split}
 \end{equation}

Again using Lemma \ref{l:converge} and diagonal method, we know that there exists a subsequence of $\{F_{s}\}_{s\to+\infty}$ (still denoted by $\{F_{s}\}_{s\to+\infty}$) compactly convergent to a holomorphic $(n,0)$ form $\tilde{F}$ on $\tilde{M}$. It follows from \eqref{estimate for Fs} and Fatou's lemma that
 \begin{equation}
 \begin{split}
 &\int_{\tilde{M}}|\tilde{F}|^2e^{\pi^{*}_1(-2\log|g|)-\tilde{\psi}_{2}-\pi^{*}_2(\varphi_2)+\pi^{*}_1(G)}c(-\pi^{*}_1(G))\\
 \le&\liminf_{s\to+\infty}\int_{M_s}|F_{l,s}|^2e^{\pi^{*}_1(-2\log|g|)-\tilde{\psi}_{2,l}-\pi^{*}_2(\varphi_2)+\pi^{*}_1(G)}c(-\pi^{*}_1(G))\\
    \le&(\int_{0}^{+\infty}c(t_1)e^{-t_1}dt_1)
    \sum_{j=1}^{m_1}\frac{2\pi|a_j|^2}{p_j|d_j|^{2}}\int_{Y}|F_j|^2e^{-\tilde{\psi}_{2}(z_j,w)-\varphi_2}<+\infty.
 \end{split}
 \end{equation}
  \label{estimate for tildeF}
Note that $N+\pi^{*}_1(G)+\pi^{*}_1(\varphi_1)=2\log|g|+\tilde{\psi}_2$. We have
\begin{equation}
 \begin{split}
 &\int_{\tilde{M}}|\tilde{F}|^2e^{-\pi^{*}_1(\varphi_1)-N-\pi^{*}_2(\varphi_2)}c(-\pi^{*}_1(G))\\
    \le&(\int_{0}^{+\infty}c(t_1)e^{-t_1}dt_1)
    \sum_{j=1}^{m_1}\frac{2\pi|a_j|^2}{p_j|d_j|^{2}}\int_{Y}|F_j|^2e^{-\tilde{\psi}_{2}(z_j,w)-\varphi_2}<+\infty.
 \end{split}
 \end{equation}
 It follows from $(F_{t,l,s}-F,(z_j,y))\in
 (\mathcal{O}(K_{\tilde M})\otimes\mathcal{I}(\pi^{*}_1(2\log|g|)+\pi^{*}_2(\varphi_2)))_{(z_j,y)}$ for any $(z_j,y) \in Z_0\cap \tilde{M}_s$, Lemma \ref{closedness} and the compactly convergence of all the sequences, we know that
 $$(\tilde{F}-F,(z_j,y))\in
 (\mathcal{O}(K_M)\otimes\mathcal{I}(\pi^{*}_1(2\log|g|)+\pi^{*}_2(\varphi_2)))_{(z_j,y)}$$ for any $y \in Y$.

Lemma \ref{optimal extension} is proved.
\end{proof}

Let $\Omega_j$  be an open Riemann surface, which admits a nontrivial Green function $G_{\Omega_j}$ for any  $1\le j\le n_1$. Let $Y$ be an $n_2-$dimensional weakly pseudoconvex K\"ahler manifold, and let $K_Y$ be the canonical (holomorphic) line bundle on $Y$. Let
$M=\left(\prod_{1\le j\le n_1}\Omega_j\right)\times Y$
 be an $n-$dimensional complex manifold, where $n=n_1+n_2$, and let $K_M$ be the canonical (holomorphic) line bundle on $M$. Let $\pi_{1}$, $\pi_{1,j}$ and $\pi_2$ be the natural projections from $M$ to $\prod_{1\le j\le n_1}\Omega_j$, $\Omega_j$ and $Y$ respectively.

 Let $\tilde{M}\subset M$ be an $n-$dimensional weakly pseudoconvex K\"ahler manifold satisfying that $Z_0\subset \tilde{M}$.  Let $F$ be a holomorphic $(n,0)$ form on a neighborhood $U_0\subset \tilde{M}$ of $Z_0$.

 Let $Z_j=\{z_{j,k}:1\le k<\tilde m_j\}$ be a discrete subset of $\Omega_j$ for any  $j\in\{1,\ldots,n_1\}$, where $\tilde m_j\in\mathbb{Z}_{\ge2}\cup\{+\infty\}$. Denote that $Z_0:=\left(\prod_{1\le j\le n_1}Z_j\right)\times Y$.

Let $p_{j,k}$ be a positive number for any $1\le j\le n_1$ and $1\le k<\tilde m_j$, which satisfies that $\sum_{1\le k<\tilde m_j}p_{j,k}G_{\Omega_j}(\cdot,z_{j,k})\not\equiv-\infty$ for any $1\le j\le n_1$.
Denote that
$$G:=\max_{1\le j\le n_1}\left\{2\sum_{1\le k<\tilde m_j}p_{j,k}\pi_{1,j}^{*}(G_{\Omega_j}(\cdot,z_{j,k}))\right\}.$$
Let $N\le0$ be a plurisubharmonic function on $\tilde{M}$. Denote $\psi:=G+N$.

Let $\varphi_X$ be a Lebesgue measurable function on $\prod_{1\le j\le n_1}\Omega_j$. Assume that $\pi_1^*(\varphi_X)+N$ is a plurisubharmonic function on $\tilde{M}$ and  $(\pi_1^*(\varphi_X)+N)|_{Z_0}\not\equiv -\infty$. Denote $\Phi=\pi_1^*(\varphi_X)+N$.
 Let $\varphi_Y$ be a plurisubharmonic function on $Y$,
and
$$\varphi:=\pi_1^*(\varphi_X)+\pi_2^*(\varphi_Y)$$
 on $M$.

 Let $w_{j,k}$ be a local coordinate on a neighborhood $V_{z_{j,k}}\Subset\Omega_{j}$ of $z_{j,k}\in\Omega_j$ satisfying $w_{j,k}(z_{j,k})=0$ for any $1\le j\le n_1$ and $1\le k<\tilde m_j$, where $V_{z_{j,k}}\cap V_{z_{j,k'}}=\emptyset$ for any $j$ and $k\not=k'$.

  Denote that $\tilde I_1:=\{\beta=(\beta_1,\ldots,\beta_{n_1}):1\le \beta_j<\tilde m_j$ for any $j\in\{1,\ldots,n_1\}\}$, $V_{\beta}:=\prod_{1\le j\le n_1}V_{z_{j,\beta_j}}$  and $w_{\beta}:=(w_{1,\beta_1},\ldots,w_{n_1,\beta_{n_1}})$ is a local coordinate on $V_{\beta}$ of $z_{\beta}:=(z_{1,\beta_1},\ldots,z_{n_1,\beta_{n_1}})\in\prod_{1\le j\le n_1}\Omega_j$ for any $\beta=(\beta_1,\ldots,\beta_{n_1})\in\tilde I_1$.
 Denote that $E_{\beta}:=\left\{\alpha=(\alpha_1,\ldots,\alpha_{n_1}):\sum_{1\le j\le n_1}\frac{\alpha_j+1}{p_{j,\beta_j}}=1\,\&\,\alpha_j\in\mathbb{Z}_{\ge0}\right\}$ and $\tilde E_{\beta}:=\left\{\alpha=(\alpha_1,\ldots,\alpha_{n_1}):\sum_{1\le j\le n_1}\frac{\alpha_j+1}{p_{j,\beta_j}}\ge1\,\&\,\alpha_j\in\mathbb{Z}_{\ge0}\right\}$ for any $\beta\in\tilde I_1$.
Let $U_0$ be an open neighborhood of $Z_0$ in $\tilde{M}$. Let $f$ be a holomorphic $(n,0)$ form on $U_0$  such that
$$f=\sum_{\alpha\in\tilde E_{\beta}}\pi_1^*(w_{\beta}^{\alpha}dw_{1,\beta_1}\wedge\ldots\wedge dw_{n_1,\beta_{n_1}})\wedge\pi_2^*(f_{\alpha,\beta})$$ on $U_0\cap(V_{\beta}\times Y)$, where $f_{\alpha,\beta}$ is a holomorphic $(n_2,0)$ form on $Y$ for any $\alpha\in\tilde E_{\beta}$ and $\beta\in\tilde I_1$. Denote that
\begin{equation*}
c_{j,k}:=\exp\lim_{z\rightarrow z_{j,k}}\left(\frac{\sum_{1\le k_1<\tilde m_j}p_{j,k_1}G_{\Omega_j}(z,z_{j,k_1})}{p_{j,k}}-\log|w_{j,k}(z)|\right)
\end{equation*}
 for any $j\in\{1,\ldots,n\}$ and $1\le k<\tilde m_j$ (following from Lemma \ref{l:green-sup} and Lemma \ref{l:green-sup2}, we get that the above limit exists).

 \begin{Lemma}
 	\label{p:exten-fibra}
 Let $c$ be a positive function on $(0,+\infty)$ such that $\int_{0}^{+\infty}c(t)e^{-t}dt<+\infty$ and $c(t)e^{-t}$ is decreasing on $(0,+\infty)$. Assume that $f_{\alpha,\beta}\in \mathcal{I}(\varphi_Y)_y$ for any $y\in Y$, where $\alpha\in \tilde{E}_{\beta}\backslash E_{\beta}$ and $\beta\in \tilde{I}_1$, and
\begin{equation}
\label{measure-finite-2}
\sum_{\beta\in\tilde I_1}\sum_{\alpha\in E_{\beta}}\frac{(2\pi)^{n_1}\int_Y|f_{\alpha,\beta}|^2
e^{-\varphi_{Y}-\left(N+\pi_1^*(\varphi_X)\right)(z_{\beta},w')}}{\prod_{1\le j\le n_1}(\alpha_j+1)c_{j,\beta_j}^{2\alpha_{j}+2}}<+\infty.
\end{equation}
Then there exists a holomorphic $(n,0)$ form $F$ on $\tilde{M}$ satisfying that $\big(F-f,(z,y)\big)\in\big(\mathcal{O}(K_{\tilde{M}})\otimes\mathcal{I}(G+\pi_2^*(\varphi_Y))\big)_{(z,y)}$ for any $(z,y)\in Z_0$ and
\begin{displaymath}
	\begin{split}
	\int_{\tilde{M}}|F|^2e^{-\varphi}c(-\psi)
\le\left(\int_0^{+\infty}c(s)e^{-s}ds\right)\sum_{\beta\in\tilde I_1}\sum_{\alpha\in E_{\beta}}\frac{(2\pi)^{n_1}
\int_Y|f_{\alpha,\beta}|^2e^{-\varphi_{Y}-\left(N+\pi_1^*(\varphi_X)\right)(z_{\beta},w')}}{\prod_{1\le j\le n_1}(\alpha_j+1)c_{j,\beta_j}^{2\alpha_{j}+2}}.	
	\end{split}
\end{displaymath}
\end{Lemma}

\begin{Remark}\label{remark after extension}
If we don't assume that $f_{\alpha,\beta}\in \mathcal{I}(\varphi_Y)_y$ for any $y\in Y$, where $\alpha\in \tilde{E}_{\beta}\backslash E_{\beta}$ and $\beta\in \tilde{I}_1$, we can still find a holomorphic $(n,0)$ form $F$ on $\tilde{M}$ satisfying that $\big(F-f,(z,y)\big)\in\big(\mathcal{O}(K_{\tilde{M}})\otimes\mathcal{I}(G)\big)_{(z,y)}$ for any $(z,y)\in Z_0$ and
\begin{displaymath}
	\begin{split}
	\int_{\tilde{M}}|F|^2e^{-\varphi}c(-\psi)
\le\left(\int_0^{+\infty}c(s)e^{-s}ds\right)\sum_{\beta\in\tilde I_1}\sum_{\alpha\in E_{\beta}}\frac{(2\pi)^{n_1}
\int_Y|f_{\alpha,\beta}|^2e^{-\varphi_{Y}-\left(N+\pi_1^*(\varphi_X)\right)(z_{\beta},w')}}{\prod_{1\le j\le n_1}(\alpha_j+1)c_{j,\beta_j}^{2\alpha_{j}+2}}
	\end{split}
\end{displaymath}
as long as \eqref{measure-finite-2} holds.
\end{Remark}

\begin{proof}[Proof of Lemma \ref{p:exten-fibra}]

As $\psi=G+N\le G$, we know that $c(-\psi)e^{-\varphi}\le c(-G)e^{-\varphi-N}$. To prove Lemma \ref{p:exten-fibra}, it suffice to prove that there exists a holomorphic $(n,0)$ form $F$ on $\tilde{M}$ satisfying that $(F-f,z)\in(\mathcal{O}(K_{\tilde{M}})\otimes\mathcal{I}(G+\pi_2^*(\varphi_Y)))_{z}$ for any $z\in Z_0$ and
\begin{displaymath}
	\begin{split}
	\int_{\tilde{M}}|F|^2e^{-\varphi-N}c(-G)
\le\left(\int_0^{+\infty}c(s)e^{-s}ds\right)\sum_{\beta\in\tilde I_1}\sum_{\alpha\in E_{\beta}}\frac{(2\pi)^{n_1}
\int_Y|f_{\alpha,\beta}|^2e^{-\varphi_{Y}-\left(N+\pi_1^*(\varphi_X)\right)(z_{\beta},w')}}{\prod_{1\le j\le n_1}(\alpha_j+1)c_{j,\beta_j}^{2\alpha_{j}+2}}.	
	\end{split}
\end{displaymath}

 The following Remark shows that it suffices to prove Proposition \ref{p:exten-fibra} for the case $\tilde m_j<+\infty$ for any $j\in\{1,\ldots,n_1\}$.
\begin{Remark} Assume that Proposition \ref{p:exten-fibra} holds for the case $\tilde m_j<+\infty$ for any $j\in\{1,\ldots,n_1\}$.
For any $j\in\{1,\ldots,n_1\}$, it follows from Lemma \ref{approximate of Green function} that there exists a sequence of Riemann surfaces $\{\Omega_{j,l}\}_{l\in\mathbb{Z}_{\ge1}}$, which satisfies that $\Omega_{j,l}\Subset\Omega_{j,l+1}\Subset\Omega_{j}$ for any $l$, $\cup_{l\in\mathbb{Z}_{\ge1}}\Omega_{j,l}=\Omega_j$ and $\{G_{\Omega_{j,l}}(\cdot,z)-G_{\Omega_j}(\cdot,z)\}_{l\in\mathbb{Z}_{\ge1}}$ is decreasingly convergent to $0$ with respect to $l$ for any $z\in\Omega_j$. As ${Z}_j$ is a discrete subset of $\Omega_j$, $ Z_{j,l}:=\Omega_{j,l}\cap{Z}_{j}$ is a set of finite points. Denote that $\tilde{M}_l:=\left((\prod_{1\le j\le n_1}\Omega_{j,l})\times Y\right)\cap \tilde{M}$ and $G_l:=\max_{1\le j\le n_1}\left\{\pi_{1,j}^*\left(\sum_{z_{j,k}\in {Z}_{j,l}}2p_{j,k}G_{\Omega_{j,l}}(\cdot,z_{j,k})\right)\right\}$  on $\tilde{M}_l$. Note that $\tilde{M}_l$ is weakly pseudoconvex K\"ahler manifold.
Denote that $$c_{j,k,l}=\exp\lim_{z\rightarrow z_{j,k}}\left(\frac{\sum_{z_{j,k_1}\in Z_{j,l}}p_{j,k_1}G_{\Omega_{j,l}}(z,z_{j,k_1})}{p_{j,k}}-\log|w_{j,k}(z)|\right)$$
for any $1\le j\le n_1$, $l\in\mathbb{Z}_{\ge1}$ and $1\le k<\tilde m_j$ satisfying $z_{j,k}\in Z_{j,l}$.
Hence $c_{j,k,l}$ is decreasingly convergent to $c_{j,k}$ with respect to $l$, $G_l$ is decreasingly convergent to $G$ with respect to $l$ and $\cup_{l\in\mathbb{Z}_{\ge1}}\tilde{M}_l=\tilde{M}$.

Then there exists a holomorphic $(n,0)$ form $F_l$ on $\tilde{M}_l$ such that $(F_l-f,(z_{\beta},y))\in(\mathcal{O}(K_{\tilde{M}_l})\otimes(\mathcal{I}(G_l)+\pi_2^*(\varphi_Y))_{(z_{\beta},y)}=(\mathcal{O}(K_{\tilde{M}})\otimes\mathcal{I}(G+\pi_2^*(\varphi_Y)))_{(z_{\beta},y)}$ for any $\beta\in\{\tilde\beta\in\tilde{I}_1:z_{\tilde\beta}\in \prod_{1\le j\le n_1}\Omega_{j,l}\}$ and $y\in Y$, and $F_l$ satisfies
\begin{displaymath}
	\begin{split}
		&\int_{\tilde{M}_l}|F_l|^2e^{-\varphi-N}c(-G_l)\\
		\leq&\left(\int_0^{+\infty}c(s)e^{-s}ds\right)\sum_{\beta\in\{\tilde\beta\in\tilde{I}_1:z_{\tilde\beta}\in \prod_{1\le j\le n_1}\Omega_{j,l}\}}\sum_{\alpha\in E_{\beta}}\frac{(2\pi)^{n_1}\int_Y|f_{\alpha,\beta}|^2e^{-\varphi_{Y}-(N+\pi_1^*(\varphi_X))(z_{\beta},w')}}{\prod_{1\le j\le n_1}(\alpha_j+1)c_{j,\beta_j,l}^{2\alpha_{j}+2}}\\
		\le& \left(\int_{0}^{+\infty}c(s)e^{-s}ds\right)\sum_{\beta\in\tilde I_1}\sum_{\alpha\in E_{\beta}}\frac{(2\pi)^{n_1}\int_Y|f_{\alpha,\beta}|^2e^{-\varphi_{Y}-\left(N+\pi_1^*(\varphi_X)\right)(z_{\beta},w')}}{\prod_{1\le j\le n_1}(\alpha_j+1)c_{j,\beta_j}^{2\alpha_{j}+2}}\\
		<&+\infty.
	\end{split}
\end{displaymath}
Since $\psi\le\psi_l$ and $c(t)e^{-t}$ is decreasing on $(0,+\infty)$, we have
\begin{equation}
	\label{eq:211223c}\begin{split}
		&\int_{\tilde{M}_l}|F_l|^2e^{-\varphi-N-G_l+G}c(-G)\\
		\le&\int_{\tilde{M}_l}|F_l|^2e^{-\varphi-N}c(-G_l)\\
		\leq&\left(\int_{0}^{+\infty}c(s)e^{-s}ds\right)\sum_{\beta\in\tilde I_1}\sum_{\alpha\in E_{\beta}}\frac{(2\pi)^{n_1}\int_Y|f_{\alpha,\beta}|^2e^{-\varphi_{Y}-\left(N+\pi_1^*(\varphi_X)\right)(z_{\beta},w')}}{\prod_{1\le j\le n_1}(\alpha_j+1)c_{j,\beta_j}^{2\alpha_{j}+2}}.
	\end{split}
\end{equation}
Note that $\psi$ is continuous on $M\backslash Z_0$, $\psi_l$ is continuous on $M_l\backslash Z_0$ and $Z_0$ is a closed complex submanifold of $M$.
 For any compact subset $K$ of $M\backslash Z_0$, there exist $l_K>0$ such that $K\Subset M_{l_K}$ and $C_K>0$  such that $\frac{e^{\varphi+N+G_l-G}}{c(-G)}\le C_K$ for any $l\ge l_K$.
 It follows from Lemma \ref{l:converge} and the diagonal method that there exists a subsequence of $\{F_l\}$, denoted still by $\{F_l\}$, which is uniformly convergent to a holomorphic $(n,0)$ form $F$ on any compact subset of $M$. It follows from Fatou's Lemma and inequality \eqref{eq:211223c} that
\begin{displaymath}
	\begin{split}
		\int_{\tilde{M}}|F|^2e^{-\varphi-N}c(-G)&=\int_{\tilde{M}}\lim_{l\rightarrow+\infty}|F_l|^2e^{-\varphi-N-G_l+G}c(-G)\\
		&\leq\liminf_{l\rightarrow+\infty}\int_{\tilde{M}_l}|F_l|^2e^{-\varphi-N-G_l+G}c(-G)\\
		&\le\left(\int_{0}^{+\infty}c(s)e^{-s}ds\right)\sum_{\beta\in\tilde I_1}\sum_{\alpha\in E_{\beta}}\frac{(2\pi)^{n_1}\int_Y|f_{\alpha,\beta}|^2e^{-\varphi_{Y}-\left(N+\pi_1^*(\varphi_X)\right)(z_{\beta},w')}}{\prod_{1\le j\le n_1}(\alpha_j+1)c_{j,\beta_j}^{2\alpha_{j}+2}}.	
	\end{split}
\end{displaymath}
 Since $\{F_l\}$ is uniformly convergent to   $F$ on any compact subset of $\tilde{M}$ and $(F_l-f,(z_{\beta},y))\in(\mathcal{O}(K_{\tilde{M}})\otimes\mathcal{I}(G+\pi_2^*(\varphi_Y)))_{(z_{\beta},y)}$ for any $\beta\in\left\{\tilde\beta\in\tilde{I}_1:z_{\tilde\beta}\in \prod_{1\le j\le n_1}\Omega_{j,l}\right\}$ and $y\in Y$, it follows from Lemma \ref{closedness} that  $(F-f,(z_{\beta},y))\in(\mathcal{O}(K_{\tilde{M}})\otimes\mathcal{I}(G+\pi_2^*(\varphi_Y)))_{(z_{\beta},y)}$ for any $\beta\in\tilde I_1$ and $y\in Y$.
\end{Remark}

In the following, we assume that $\tilde m_j<+\infty$ for any $1\le j\le n_1$.
Denote that $m_j=\tilde{m}_j-1$.

As $\tilde{M}$ is a weakly pseudoconvex K\"ahler manifold, there exists a sequence of weakly pseudoconvex K\"ahler manifolds $\tilde {M}_s$ satisfying $\tilde {M}_1\Subset \tilde{M}_2 \cdots \Subset \tilde{M}_s\Subset\cdots \tilde{M}$ and $\cup_{s\in \mathbb{N}^+} \tilde {M}_s=\tilde{M}$.

Recall that $\Phi=\pi_1^*(\varphi_X)+N\in Psh(M)$. Denote $\Phi_l=\max\{\Phi,-l\}$, where $l$ is a positive integer. We note that $\Phi_l$ is a bounded plurisubharmonic function on $\tilde{M}$.

It follows from Lemma \ref{l:green-sup} and Lemma \ref{l:green-sup2} that there exists a local coordinate $\tilde{w}_{j,k}$ on a neighborhood $\tilde{V}_{z_{j,k}}\Subset V_{z_{j,k}}$ of $z_{j,k}$ satisfying $\tilde{w}_{j,k}(z_{j,k})=0$ and
 $$|\tilde{w}_{j,k}|=\exp\left({\frac{\sum_{1\le k_1\le m_j}p_{j,k_1}G_{\Omega_j}(\cdot,z_{j,k_1})}{p_{j,k}}}\right)$$
  on $\tilde{V}_{z_{j,k}}.$

Denote that $\tilde{V}_{\beta}:=\prod_{1\le j\le n_1}\tilde V_{j,\beta_j}$ for any $\beta\in\tilde I_1$. Let $\tilde{f}$ be a holomorphic $(n,0)$ form on $(\cup_{\beta\in\tilde I_1}\tilde{V}_{\beta}\times Y)\cap U_0$ satisfying
$$\tilde{f}=\sum_{\alpha\in E_{\beta}}c_{\alpha,\beta}\pi_1^*(\tilde{w}_{\beta}^{\alpha}d\tilde{w}_{1,\beta_1}\wedge d\tilde{w}_{2,\beta_2}\wedge\ldots\wedge d\tilde{w}_{n_1,\beta_{n_1}})\wedge\pi_2^*(f_{\alpha,\beta})$$
on $(\tilde V_{\beta}\times Y)\cap U_0$,  where $c_{\alpha,\beta}=\prod_{1\le j\le n_1}\left(\lim_{z\rightarrow z_{j,\beta_j}}\frac{w_{j,\beta_j}(z)}{\tilde{w}_{j,\beta_j}(z)}\right)^{\alpha_j+1}$.
It follows from $f_{\alpha,\beta}\in \mathcal{I}(\varphi_Y)_y$ for any $y\in Y$, where $\alpha\in \tilde{E}_{\beta}$ and $\beta\in \tilde{I}_1$, and Lemma \ref{l:0} that
\begin{equation}\label{for remark after extension}
(f-\tilde{f},(z,y))\in(\mathcal{O}(K_{M})\otimes\mathcal{I}(G+\pi_2^*(\varphi_Y)))_{(z,y)}
\end{equation}
 for any $(z,y)\in Z_0$.

Denote that
 $G_1:=\max_{1\le j\le n_1}\left\{\tilde\pi_{j}^*\left(2\sum_{1\le k<\tilde{m}_j}p_{j,k}G_{\Omega_j}(\cdot,z_{j,k})\right)\right\}$
  on $\prod_{1\le j\le n_1}\Omega_j$, where $\tilde\pi_j$ is the natural projection from $\prod_{1\le j\le n_1}\Omega_j$ to $\Omega_j$ and  $G=\pi_1^*(G_1)$.
It follows from Lemma \ref{l:green-sup2} and Lemma \ref{relative compactnees of Green function} that there exists $t_0>0$ such that $\{G_1<-t_0\}\Subset \cup_{\beta\in\tilde I_1}\tilde V_{\beta}$, which implies that $\int_{\{G<-t\}\cap \tilde{M}_{s}}|\tilde f|^2<+\infty$. It follows from \eqref{measure-finite-2}, $\Phi_l$ is bounded and Fubini's theorem that we know
$$\int_{\tilde{M}_{s}}\mathbb{I}_{\{-t-1<G<-t\}}|\tilde f|^2e^{-G-\pi_2^*(\varphi_Y)-\Phi_l}<+\infty.$$

Using Lemma \ref{L2 method}, there exists a holomorphic $(n,0)$ form $F_{l,s,t}$ on $\tilde{M}_{s}$ such that
\begin{equation}
	\label{eq:1125m}
	\begin{split}
&\int_{\tilde{M}_{s}}|F_{l,s,t}-(1-b_{t,1}(\psi))\tilde{f}|^{2}e^{-G-\pi_2^*(\varphi_Y)-\Phi_l+v_{t,1}(G)}c(-v_{t,1}(G))\\
\leq& \left(\int_{0}^{t+1}c(t)e^{-t}dt\right) \int_{\tilde{M}_{s}}\mathbb{I}_{\{-t-1<G<-t\}}|\tilde f|^2e^{-G-\pi_2^*(\varphi_Y)-\Phi_l},
\end{split}
\end{equation}
where $t\ge t_0$. Note that $b_{t,1}(t_1)=0$ for large enough $t_1$, then $(F_{l,s,t}-\tilde f,(z,y))\in(\mathcal{O}(K_{M})\otimes\mathcal{I}(G+\pi_2^*(\varphi_Y)))_{(z,y)}$ for any $(z,y)\in Z_0\cap \tilde{M}_{s}$.

For any $(z_{\beta},y) \in Z_0\cap \tilde{M}_s$, letting $(U_y,w)$ be a small local coordinated open neighborhood of $y$ and shrinking $\tilde{V}_{\beta}$ if necessary, we have $\tilde{V}_{\beta}\times U_y\Subset U_0\cap(V_{\beta}\times Y)$ for any $\beta\in\tilde{I}_1$. Assume that $F_j=\tilde{h}_j(w)dw$ on $U_y$, where $\tilde{h}_j$ is a holomorphic function on $U_y$ and $dw=dw_1\wedge dw_2\wedge\ldots\wedge dw_n$. There exists $t_2>t_1$ such that when $t>t_2$, $(\{G_1<-t\}\times U_y)\cap (\tilde V_{\beta}\times U_y) \Subset (V_{\beta}\times Y)\cap \tilde{M}$, for any $\beta\in\tilde{I}_1$.

 Now we consider
 \begin{equation}
 \label{multi L2 formula 1}
 \begin{split}
  & \int_{(\cup_{\beta\in \tilde{I}_1}\tilde{V}_{\beta})\times U_y}\mathbb{I}_{\{-t-1<G<-t\}}|\tilde f|^2e^{-G-\pi_2^*(\varphi_Y)-\Phi_l}\\
      =&\sum_{\beta\in\tilde{I}_1}\int_{w'\in U_y}(\int_{\{-t-1<G_1<-t\}}|f(\tilde{w},w')|^2\frac{e^{-\Phi_l(\tilde{w},w')}}{\max_{1\le j \le n}\{|\tilde{w}_{j,\beta_j}|^{2p_{j,\beta_j}}\}})e^{-\varphi_Y}.
 \end{split}
 \end{equation}

Note that $|c_{\alpha,\beta}|=\frac{1}{\prod_{1\le j\le n_1}c_{j,\beta_j}^{\alpha_j+1}}$ and $\int_{w'\in Y_s}|f_{\alpha,\beta}|^2e^{-\varphi_Y}<+\infty$. It follows from \eqref{multi L2 formula 1}, Fatou's Lemma and Lemma \ref{limit discussion 2} that we have

 \begin{equation}
 \label{multi L2 formula 3}
 \begin{split}
  & \limsup_{t\to +\infty}\int_{\cup_{\beta\in \tilde{I}_1}\tilde{V}_{\beta}\times U_y}\mathbb{I}_{\{-t-1<G<-t\}}|\tilde f|^2e^{-G-\pi_2^*(\varphi_Y)-\Phi_l}\\
      =&\limsup_{t\to +\infty} \sum_{\beta\in\tilde{I}_1}\int_{w'\in U_y}(\int_{\{-t-1<G_1<-t\}}|f(\tilde{w},w')|^2\frac{e^{-\Phi_l(\tilde{w},w')}}{\max_{1\le j \le n}\{|\tilde{w}_{j,\beta_j}|^{2p_{j,\beta_j}}\}})e^{-\varphi_Y}\\
      \le  & \sum_{\beta\in\tilde{I}_1}\int_{w'\in U_y}\limsup_{t\to +\infty}(\int_{\{-t-1<G_1<-t\}}|f(\tilde{w},w')|^2\frac{e^{-\Phi_l(\tilde{w},w')}}{\max_{1\le j \le n}\{|\tilde{w}_{j,\beta_j}|^{2p_{j,\beta_j}}\}})e^{-\varphi_Y}\\
      \le &\sum_{\beta\in\tilde{I}_1}\sum_{\alpha\in E_{\beta}}
       \frac{(2\pi)^{n_1}\int_{U_y}|f_{\alpha,\beta}|^2e^{-\varphi_{Y}-\Phi_l(z_{\beta},w')}}{\prod_{1\le j\le n_1}(\alpha_j+1)c_{j,\beta_j}^{2\alpha_{j}+2}}.
 \end{split}
 \end{equation}

Note that $y$ and $U_y$ are arbitrarily chosen, $v_{t,1}(\psi)\ge\psi$ and $c(t)e^{-t}$ is decreasing. Combining inequalities \eqref{eq:1125m} and \eqref{multi L2 formula 3}, then we have

\begin{equation}
	\label{multi L2 formula 4}
	\begin{split}
&\int_{\tilde{M}_{s}}|F_{l,s,t}-(1-b_{t,1}(\psi))\tilde{f}|^{2}e^{-\pi_2^*(\varphi_Y)-\Phi_l}c(-G)\\
\le & \int_{\tilde{M}_{s}}|F_{l,s,t}-(1-b_{t,1}(\psi))\tilde{f}|^{2}e^{-G-\pi_2^*(\varphi_Y)-\Phi_l+v_{t,1}(G)}c(-v_{t,1}(G))\\
\leq& \left(\int_{0}^{t+1}c(t_1)e^{-t_1}dt_1\right) \int_{\tilde{M}_{s}}\mathbb{I}_{\{-t-1<G<-t\}}|\tilde f|^2e^{-G-\pi_2^*(\varphi_Y)-\Phi_l}\\
\le &\left(\int_{0}^{t+1}c(t_1)e^{-t_1}dt_1\right)\sum_{\beta\in\tilde{I}_1}\sum_{\alpha\in E_{\beta}}
       \frac{(2\pi)^{n_1}\int_{(\{z_{\beta}\}\times Y)\cap \tilde{M}_s}|f_{\alpha,\beta}|^2e^{-\varphi_{Y}-\Phi_l(z_{\beta},w')}}{\prod_{1\le j\le n_1}(\alpha_j+1)c_{j,\beta_j}^{2\alpha_{j}+2}}.
\end{split}
\end{equation}

Note that $G$ is continuous on $\tilde{M}\backslash Z_0$. For any open set $K\Subset \tilde{M}_{s}\backslash Z_0$, as $b_{t,1}(t_1)=1$ for any $t_1$ large enough and $c(t_2)e^{-t_2}$ is decreasing with respect to $t_2$, we get that there exists a constant $C_K>0$ such that
$$\int_{K}|(1-b_{t,1}(\psi))\tilde{f}|^2e^{-\pi_2^*(\varphi_Y)-\Phi_l}c(-G)\le C_K\int_{\{G<-t_1\}\cap K}|\tilde{f}|^2<+\infty$$ for any $t>t_1$,
which implies that
$$\limsup_{t\rightarrow+\infty}\int_{K}|F_{l,s,t}|^2e^{-\pi_2^*(\varphi_Y)-\Phi_l}c(-G)<+\infty.$$
Using Lemma \ref{l:converge} and the diagonal method,
we obtain that
there exists a subsequence of $\{F_{l,s,t}\}_{t\rightarrow+\infty}$ denoted by $\{F_{l,s,t_m}\}_{m\rightarrow+\infty}$
uniformly convergent on any compact subset of $\tilde{M}_{s}\backslash Z_0$. As $Z_0$ is a closed complex submanifold of $\tilde{M}$, we obtain that $\{F_{l,s,t_m}\}_{m\rightarrow+\infty}$ is uniformly convergent to a holomorphic $(n,0)$ form $F_{l,s}$ on $\tilde{M}_{s}$ on any compact subset of $\tilde{M}_{s}$. Then it follows from inequality \eqref{multi L2 formula 4} and Fatou's Lemma that
\begin{displaymath}
\begin{split}
&\int_{\tilde{M}_{s}}|F_{l,s}|^{2}e^{-\pi_2^*(\varphi_Y)-\Phi_l}c(-G)
\\=&\int_{\tilde{M}_{s}}\liminf_{m\rightarrow+\infty}|F_{l,s,t_m}-(1-b_{t_m,1}(\psi))\tilde{f}|^{2}e^{-\pi_2^*(\varphi_Y)-\Phi_l}c(-G)\\
\leq&\liminf_{m\rightarrow+\infty}\int_{\tilde{M}_{s}}|F_{l,s,t_m}-(1-b_{t_m,1}(\psi))\tilde{f}|^{2}e^{-\pi_2^*(\varphi_Y)-\Phi_l}c(-G)\\
\leq&\left(\int_0^{+\infty}c(t_1)e^{-t_1}dt_1\right)\sum_{\beta\in\tilde{I}_1}\sum_{\alpha\in E_{\beta}} \frac{(2\pi)^{n_1}\int_{(\{z_{\beta}\}\times Y)\cap \tilde{M}_s}|f_{\alpha,\beta}|^2e^{-\varphi_{Y}-\Phi_l(z_{\beta},w')}}{\prod_{1\le j\le n_1}(\alpha_j+1)c_{j,\beta_j}^{2\alpha_{j}+2}}\\
	<&+\infty.
\end{split}	
\end{displaymath}

As $\Phi_l(z_{\beta},w')$ is decreasingly convergent to $\Phi(z_{\beta},w')$ for any $\beta\in \tilde{I}_1$, then we have
\begin{equation}
	\label{eq:1126e}\begin{split}
		&\limsup_{l\rightarrow+\infty}\int_{\tilde{M}_{s}}|F_{l,s}|^{2}e^{-\pi_2^*(\varphi_Y)-\Phi_l}c(-G)\\
		\leq&\left(\int_0^{+\infty}c(t_1)e^{-t_1}dt_1\right)\sum_{\beta\in\tilde{I}_1}\sum_{\alpha\in E_{\beta}} \frac{(2\pi)^{n_1}\int_{(\{z_{\beta}\}\times Y)\cap \tilde{M}_s}|f_{\alpha,\beta}|^2e^{-\varphi_{Y}-\Phi(z_{\beta},w')}}{\prod_{1\le j\le n_1}(\alpha_j+1)c_{j,\beta_j}^{2\alpha_{j}+2}}\\
=&\left(\int_0^{+\infty}c(t_1)e^{-t_1}dt_1\right)\sum_{\beta\in\tilde{I}_1}\sum_{\alpha\in E_{\beta}} \frac{(2\pi)^{n_1}\int_{(\{z_{\beta}\}\times Y)\cap \tilde{M}_s}|f_{\alpha,\beta}|^2e^{-\varphi_{Y}-\left(N+\pi_1^*(\varphi_X)\right)(z_{\beta},w')}}{\prod_{1\le j\le n_1}(\alpha_j+1)c_{j,\beta_j}^{2\alpha_{j}+2}}\\
	<&+\infty.
	\end{split}
\end{equation}
Note that $G$ is continuous on $\tilde{M}\backslash Z_0$ and $Z_0$ is a closed complex submanifold of $\tilde{M}$.
Using Lemma \ref{l:converge},
we obtain that
there exists a subsequence of $\{F_{l,s}\}_{l\rightarrow+\infty}$ (also denoted by $\{F_{l,s}\}_{l\rightarrow+\infty}$)
uniformly convergent to a holomorphic $(n,0)$ form $F_{s}$ on $\tilde{M}_{s}$ on any compact subset of $\tilde{M}_{s}$, which satisfies that
\begin{equation*}
\begin{split}
&\int_{\tilde{M}_{s}}|F_{s}|^2e^{-\pi_2^*(\varphi_Y)-N-\varphi_X}c(-G)\\
\le &\left(\int_0^{+\infty}c(t_1)e^{-t_1}dt_1\right)\sum_{\beta\in\tilde{I}_1}\sum_{\alpha\in E_{\beta}} \frac{(2\pi)^{n_1}\int_{(\{z_{\beta}\}\times Y)\cap \tilde{M}_s}|f_{\alpha,\beta}|^2e^{-\varphi_{Y}-\left(N+\pi_1^*(\varphi_X)\right)(z_{\beta},w')}}{\prod_{1\le j\le n_1}(\alpha_j+1)c_{j,\beta_j}^{2\alpha_{j}+2}}.
\end{split}
\end{equation*}

As $\cup_{s\in\mathbb{Z}_{\ge1}}\tilde{M}_{s}=\tilde{M}$, we have
\begin{equation}
	\label{eq:211223d}\begin{split}
	&\limsup_{s\rightarrow+\infty}\int_{\tilde{M}_{s}}|F_{s}|^2e^{-\varphi-N}c(-\psi)\\
	\leq
		&\lim_{s\rightarrow+\infty}\sum_{\beta\in\tilde I_1}\sum_{\alpha\in E_{\beta}}\frac{(2\pi)^{n_1}\int_{(\{z_{\beta}\}\times Y)\cap \tilde{M}_s}|f_{\alpha,\beta}|^2e^{-\varphi_Y-(N+\varphi_X)(z_{\beta},w')}}{\prod_{1\le j\le n_1}(\alpha_j+1)c_{j,\beta_j}^{2\alpha_{j}+2}}\\		
	=&\sum_{\beta\in\tilde I_1}\sum_{\alpha\in E_{\beta}}\frac{(2\pi)^{n_1}\int_{Y}|f_{\alpha,\beta}|^2e^{-\varphi_Y-\left(N+\pi_1^*(\varphi_X)\right)(z_{\beta},w')}}{\prod_{1\le j\le n_1}(\alpha_j+1)c_{j,\beta_j}^{2\alpha_{j}+2}}\\
	<&+\infty.
	\end{split}
\end{equation}
Note that $\psi$ is continuous on $\tilde{M}\backslash Z_0$, $Z_0$ is a closed complex submanifold of $\tilde{M}$ and $\cup_{s\in\mathbb{Z}_{\ge1}}\tilde{M}_{s}=\tilde{M}$.
Using Lemma \ref{l:converge} and  the diagonal method,
we get that
there exists a subsequence of $\{F_{s}\}$ (also denoted by $\{F_{s}\}$)
uniformly convergent to a holomorphic $(n,0)$ form $F$ on $\tilde M$ on any compact subset of $\tilde M$.  Then it follows from inequality \eqref{eq:211223d} and  Fatou's Lemma that
\begin{equation*}
\begin{split}
	\int_{\tilde M}|F|^2e^{-\varphi}c(-\psi)&=\int_{\tilde M}\liminf_{s\rightarrow+\infty}\mathbb{I}_{\tilde M_{s}}|F_{l'}|^2e^{-\varphi}c(-\psi)\\
	&\le\liminf_{s\rightarrow+\infty}\int_{\tilde M_{s}}|F_{l'}|^2e^{-\varphi}c(-\psi)\\
	&\leq\sum_{\beta\in\tilde I_1}\sum_{\alpha\in E_{\beta}}\frac{(2\pi)^{n_1}\int_{Y}|f_{\alpha,\beta}|^2e^{-\varphi_Y-\left(N+\pi_1^*(\varphi_X)\right)(z_{\beta},w')}}{\prod_{1\le j\le n_1}(\alpha_j+1)c_{j,\beta_j}^{2\alpha_{j}+2}}.\end{split}
\end{equation*}
Following from Lemma \ref{closedness}, we have $(F-f,(z,y))\in(\mathcal{O}(K_{\tilde M})\otimes\mathcal{I}(G+\pi_2^*(\varphi_Y)))_{(z,y)}$ for any $(z,y)\in Z_0$.

Thus, Proposition \ref{p:exten-fibra} holds.
\end{proof}

\begin{proof}[Proof of Remark \ref{remark after extension}]
If we don't assume that $f_{\alpha,\beta}\in \mathcal{I}(\varphi_Y)_y$ for any $y\in Y$, where $\alpha\in \tilde{E}_{\beta}\backslash E_{\beta}$ and $\beta\in \tilde{I}_1$, it follows from Lemma \ref{l:0} that for any $(z,y)\in Z_0$, we will have
\begin{equation}
(f-\tilde{f},(z,y))\in(\mathcal{O}(K_{M})\otimes\mathcal{I}(G))_{(z,y)}.
\end{equation}

Replace the formula \eqref{for remark after extension} by $(f-\tilde{f},(z,y))\in(\mathcal{O}(K_{M})\otimes\mathcal{I}(G))_{(z,y)}$.
Then if \eqref{measure-finite-2} holds, by the same proof as Proposition \ref{p:exten-fibra}, we know Remark \ref{remark after extension} holds.
\end{proof}
\subsection{Other Calculations}
Let $\Omega$ be an open Riemann surface with nontrivial Green functions. Let $Z_{\Omega}=\{z_j:j \in \mathbb{N}_+ \& j<\gamma\}$ be a subset of $\Omega$ of discrete points. Let $Y$ be an $(n-1)-$dimensional weakly pseudoconvex K\"ahler manifold. Denote $M=\Omega \times Y$. Let $\pi_1$ and $\pi_2$ be the natural projections from $M$ to $\Omega$ and $Y$ respectively. Denote $Z_0:=Z_{\Omega}\times Y$. Denote $Z_j:=\{z_j\}\times Y$.

Let $\psi$ be a plurisubharmonic function on $M$.
It follows from Siu's decomposition theorem that $$dd^{c}\psi=\sum\limits_{j\ge1}2p_j[Z_j]+\sum\limits_{i\ge 1}\lambda_i[A_i]+R,$$
where $[Z_j]$ and $[A_i]$ are the currents of integration over an irreducible $(n-1)-$dimensional analytic set, and where $R$ is a closed positive current with the property that $dimE_c(R)<n-1$ for every $c>0$. We assume that $p_j>0$ for any $1\le j< \gamma$.

Then $N:=\psi-\pi^{*}_1\big(\sum\limits_{j\ge1}2p_jG_{\Omega}(z,z_j)\big)$ is a plurisubharmonic function on $M$. We assume that $N\le0$ and $N|_{Z_j}$ is not identically $-\infty$ for any $j$.

Let $\varphi_1$ be a Lebesgue measurable function on $\Omega$ such that $\psi+\pi^{*}_1(\varphi)$ is a plurisubharmonic function on $M$. With Similar discussion as above, by Siu's decomposition theorem, we have

$$dd^{c}(\psi+\pi^{*}_1(\varphi))=\sum\limits_{j\ge1}2\tilde{q}_j[Z_j]+\sum\limits_{i\ge 1}\tilde{\lambda}_i[\tilde{A}_i]+\tilde{R},$$
where $\tilde{q}_j\ge 0$ for any $1\le j< \gamma$.

By Weierstrass theorem on open Riemann surfaces, there exists a holomorphic function $g$ on $\Omega$ such that $ord_{z_j}(g)=q_j:=[\tilde{q}_j]$ for any $z_j\in Z_{\Omega}$ and $g(z)\neq 0$ for any $z\notin Z_{\Omega}$, where $[q]$ equals to the integer part of the nonnegative real number $q$.
Then we know that there exists a plurisubharmonic function $\tilde{\psi}_2\in Psh(M)$ such that $$\psi+\pi^{*}_1(\varphi_1)=\pi^{*}_1(2\log|g|)+\tilde{\psi}_2.$$

Let $\varphi_2\in Psh(Y)$. Denote $\varphi:=\pi^{*}_1(\varphi_1)+\pi^{*}_2(\varphi_2)$.

For $1\le j<\gamma$, let $(V_{j},\tilde{z}_j)$ be a local coordinated open neighborhood of $z_j$ in $\Omega$ satisfying $V_{j}\Subset \Omega$, $\tilde{z}_j(z_j)=0$ under the local coordinate and $V_{j}\cap V_{k}=\emptyset$ for any $j\neq k$. Denote $V_0:=\cup_{1\le j<\gamma} V_j$. We assume that $g=d_j\tilde{z}_j^{q_j}h_j(z)$ on $V_j$, where $d_j$ is a constant, $h_j(z)$ is a holomorphic function on $V_j$ and $h_j(z_j)=1$.

Let $c(t)$ be a positive measurable function on $(0,+\infty)$ satisfying that $c(t)e^{-t}$ is decreasing and $\int_{0}^{+\infty}c(s)e^{-s}ds<+\infty$.

\begin{Lemma}
\label{slope bigger than measure}
Assume that $c(t)$ is increasing near $+\infty$. Let $F$ be a holomorphic $(n,0)$ form on $M$ such that $F=\sum_{l=k_j}^{+\infty}\pi^{*}_1(\tilde{z}_j^ld\tilde{z}_j)\wedge\pi^{*}_2(F_{j,l})$ on $V_j\times Y$ $(1\le j<\gamma)$, where $k_j$ is a nonnegative integer, $F_{j,l}$ is a holomorphic $(n-1,0)$ form on $Y$, for any $l,j$, and $F_j:=F_{j,k_j}\not\equiv 0$ on $Y$. Denote that
$$I_F:=\{j:k_j+1-q_j\le 0 \& 1\le j<\gamma\}.$$
Assume that
\begin{equation}\nonumber
  \liminf_{t\to+\infty}\frac{\int_{\{\psi<-t\}}|F|^2e^{-\varphi}c(-\psi)}{\int_{t}^{+\infty}c(s)e^{-s}ds}<+\infty.
\end{equation}
Then $k_j+1-q_j=0$ for any $j\in I_F$, and
\begin{equation}
  \liminf_{t\to+\infty}\frac{\int_{\{\psi<-t\}}|F|^2e^{-\varphi}c(-\psi)}{\int_{t}^{+\infty}c(s)e^{-s}ds}
  \ge
 \sum_{j\in I_F}\frac{2\pi}{p_j|d_j|^2}\int_Y|F_j|^2e^{-\varphi_2-\tilde{\psi}_2(z_j,w)}.
\end{equation}
\end{Lemma}
\begin{proof}

As $\tilde{\psi}_2$ and $N$ are upper semi-continuous functions on $M$, there exists continuous functions $\tilde{\psi}_{2,\alpha}$ and $N_{\beta}$ on $M$ decreasingly convergent to $\tilde{\psi}_2$ and $N$  respectively.

   There exists some $t_1\ge0$ such that $c(t)$ is increasing on $[t_1,+\infty)$. Recall that $\psi=\pi^{*}_1(\sum\limits_{j\ge1}2p_j\pi^{*}_1(G_{\Omega}(z,z_j)))+N$ and denote $G=\sum\limits_{j\ge1}2p_j\pi^{*}_1(G_{\Omega}(z,z_j))$. For any $t>t_1$,
   \begin{equation}\label{bigger 1}
   \begin{split}
      &\int_{\{\psi<-t\}}|F|^2e^{-\varphi}c(-\psi)  \\
        =& \int_{\{G+N<-t\}}|F|^2e^{-2\pi^{*}_1(\log|g|)-\tilde{\psi}_2+\pi^{*}_1(G)+N-\pi^{*}_2(\varphi_2)}c(-(G+N))\\
    \ge&\int_{\{G+N_{\beta}<-t\}}|F|^2e^{-2\pi^{*}_1(\log|g|)-\tilde{\psi}_{2,\alpha}+\pi^{*}_1(G)+N-\pi^{*}_2(\varphi_2)}c(-(G+N_{\beta})).
   \end{split}
   \end{equation}

For any $y\in Y$, let $(U_y,w)$ be a local coordinated open neighborhood of $y$ in $Y$ satisfying $U_y\Subset Y$. We assume that $F=\tilde{z_j}^{k_j}\tilde{h}_j(\tilde{z_j},w)d\tilde{z}_j\wedge dw$ on $V_j\times U_y$.

Let $m\in\mathbb{N}$. For any $0<\epsilon<\frac{1}{2}$, let $s_0$ be large enough and shrink $U_y$ if necessary such that,
\\
(1) for any $j\in\{1,2,\ldots,m\},$
\\
 \centerline{$\tilde{V}_j:=\{|\tilde{z}_j|<s_0,z\in V_j\}\Subset V_j$,}
\\
(2) for any $j\in\{1,2,\ldots,m\}$ and any $w\in U_y$, we have
\\
\centerline{$\sup\limits_{\tilde{z}\in \tilde{V}_j }|\tilde{\psi}_{2,\alpha}(\tilde{z},w)-\tilde{\psi}_{2,\alpha}(z_0,w)|<\epsilon.$}
\\
(3) Denote $H_j(\tilde{z}_j,w):=G-2p_j\log|\tilde{z}_j|+N_{\beta}(\tilde{z}_j,w)+\epsilon$ on $\tilde{V}_j\times U_y$. Then for any $w\in U_y$, we have
\\
 \centerline{$\sup\limits_{\tilde{z}\in \tilde{V}_j}|H(\tilde{z}_j,w)-H(z_j,w)|<\epsilon.$}
\\
(4)
Recall that $g=d_j\tilde{z}^{q_j}h_j(\tilde{z}_j)$ on $V_j$, where $d_j$ is a constant, $h_j(\tilde{z})$ is a holomorphic function on $V_j$ and $h(z_j)=1$. We assume that
\\
\centerline{$\sup\limits_{\tilde{z}\in \tilde{V}_j }|h_j(\tilde{z}_j)-1|<\epsilon.$}
\\
(5)
Denote that $2g_j(\tilde{z}_j)=G-2p_j\log|\tilde{z}_j|$ on $\tilde{V}_j$. Note that $g_j$ is a harmonic function on $\tilde{V}_j$.

It follows from Lemma \ref{relative compactnees of Green function} that there exists $t'>t_1$ such that $$\left(\{G+N_{\beta}<-t\}\cap (\tilde{V}_j\times U_y)\right)\Subset \tilde{V}_j\times U_y,$$ for any $t>t'$.
We also have $G+N_{\beta}\le 2p_j\log|\tilde{z}_j|+H_j(z_j,w)$ on $V_j\times U_y$.

Following from \eqref{bigger 1}, for any $t>t'$, direct calculation shows that

  \begin{equation}
   \begin{split}
      &\int_{\{\psi<-t\}\cap (V_0\times U_y)}|F|^2e^{-\varphi}c(-\psi)  \\
\ge&\int_{\{G+N_{\beta}<-t\}\cap (V_0\times U_y)}|F|^2e^{-2\pi^{*}_1(\log|g|)-\tilde{\psi}_{2,\alpha}+\pi^{*}_1(G)+N-\pi^{*}_2(\varphi_2)}c(-(G+N_{\beta}))\\
\ge &  \sum_{j=1}^{m} \int_{\{2p_j\log|\tilde{z}_j|+H_j(z_j,w)<-t\}\cap (\tilde{V}_j\times U_y)}
|\tilde{z}_j|^{2k_j+2p_j-2q_j}\frac{|\tilde{h}_j(\tilde{z}_j,w)|^2}{|d_j|^{2}|h_j(\tilde{z}_j)|^{2}}\times\\
&e^{-\tilde{\psi}_{2,l,\alpha}(z_j,w)-\epsilon+2g_j(\tilde{z}_j)+N-\pi^{*}_2(\varphi_2)}c(-2p_j\log|\tilde{z}_j|-H_j(z_j,w))\\
\ge&\sum_{j=1}^{m}\int_{w\in U_y}(\int_{\{2p_j\log|\tilde{z}_j |+H_j(z_j,w)<-t\}}|\tilde{h}_j(\tilde{z}_j ,w)|^2|\tilde{z}_j|^{2k_j+2p_j-2q_j}e^{2g_j(\tilde{z}_j )+N}\times\\
&c(-2p_j\log|\tilde{z}_j |-H_j(z_j ,w))|d\tilde{z}_j |^2)
\frac{e^{-\tilde{\psi}_{2,\alpha}(z_j,w)-\epsilon-\varphi_2(w)}}{|d|^{2}|1+\epsilon|^{2}}|dw|^2\\
=&\sum_{j=1}^{m}\int_{w\in U_y}(2\int_{0}^{e^{\frac{-t-H_j(z_j ,w)}{2p_j}}}\int_{0}^{2\pi}|\tilde{h}_j(re^{-\theta},w)|^2r^{2k_j+2p_j-2q_j+1}
e^{2g_j(re^{i\theta})+N(re^{i\theta},w)}\times\\
&c(-2p_j\log r-H_j(z_j,w))d\theta dr)\frac{e^{-\tilde{\psi}_{2,\alpha}(z_j,w)-\epsilon-\varphi_2(w)}}{|d|^{2}|1+\epsilon|^{2}}|dw|^2\\
\ge&\sum_{j=1}^{m}
\int_{w\in U_y}4\pi(\int_{0}^{e^{\frac{-t-H_j(z_j,w)}{2p_j}}}r^{2k_j+2p_j-2q_j+1}c(-2p_j\log r-H_j(z_j,w)))\times\\
&|\tilde{h}_j(z_j,w)|^2e^{2g_j(z_j)+N(z_j,w)}\frac{e^{-\tilde{\psi}_{2,\alpha}(z_j,w)-\epsilon-\varphi_2(w)}}{|d|^{2}|1+\epsilon|^{2}}|dw|^2\\
=&\sum_{j=1}^{m}\frac{2\pi}{p_j|d_j|^{2}|1+\epsilon|^{2}}\left(\int_{t}^{+\infty}c(s)e^{-(\frac{k_j+1-q_j}{p_j}+1)s}ds\right)\times\\
&\int_{w\in U_y}|\tilde{h}_j(z_j,w)|^2e^{-(\frac{k_j+1-q_j}{p_j}+1)H_j(z_j,w)}e^{2g_j(z_j)+N(z_j,w)-\tilde{\psi}_{2,\alpha}(z_j,w)-\epsilon-\varphi_2(w)}|dw|^2.
   \end{split}
   \end{equation}
As $\int_{0}^{+\infty}c(s)e^{-s}ds<+\infty$, hence we have
  \begin{equation}\label{bigger 2}
   \begin{split}
      &\frac{\int_{\{\psi<-t\}\cap (V_0\times U_y)}|F|^2e^{-\varphi}c(-\psi)}{\int_{t}^{+\infty}c(s)e^{-s}ds}  \\
\ge&
   \sum_{j=1}^{m}\frac{2\pi}{p_j|d_j|^{2}|1+\epsilon|^{2}}\left(\frac{\int_{t}^{+\infty}c(s)e^{-(\frac{k+1-q_2}{p}+1)s}ds}{\int_{t}^{+\infty}c(s)e^{-s}ds}\right)\times\\
&\int_{w\in U_y}|\tilde{h}_j(z_j,w)|^2e^{-(\frac{k_j+1-q_j}{p_j}+1)H_j(z_j,w)}e^{2g_j(z_j)+N(z_j,w)-\tilde{\psi}_{2,\alpha}(z_j,w)-\epsilon-\varphi_2(w)}|dw|^2.
   \end{split}
   \end{equation}

   Denote
   $$I_m:=\{1\le j\le m: k_j+1-q_j\le 0\}.$$
Note that
\begin{equation}\nonumber
  \liminf_{t\to+\infty}\frac{\int_{\{\psi<-t\}}|F|^2e^{-\varphi}c(-\psi)}{\int_{t}^{+\infty}c(s)e^{-s}ds}<+\infty,
\end{equation}
 $N(z_j,w)\not\equiv -\infty$ and $|\tilde{h}(z_j,w)|^2\not\equiv 0$. It follows from Lemma \ref{growth rate of c(t)e(-t)} and \eqref{bigger 2} that we know $k_j+1-q_j=0$ for any $j\in I_m$ and
  \begin{equation}\nonumber
   \begin{split}
      &\liminf_{t\to+\infty}\frac{\int_{\{\psi<-t\}\cap (V_0\times U_y)}|F|^2e^{-\varphi}c(-\psi)}{\int_{t}^{+\infty}c(s)e^{-s}ds}  \\
\ge&\sum_{j\in I_m}\frac{2\pi}{p_j|d_j|^{2}}
\int_{w\in U_y}|\tilde{h}_j(z_j,w)|^2e^{-H(z_j,w)}e^{2g_j(z_j)+N(z_j,w)-\tilde{\psi}_{2,\alpha}(z_j,w)-\varphi_2(w)}|dw|^2\\
 =&  \sum_{j\in I_m}\frac{2\pi}{p_j|d_j|^{2}}
\int_{w\in U_y}|\tilde{h}_j(z_j,w)|^2e^{-2g_j(z_j,0)+N_{\beta}(z_j,w)}e^{2g_j(z_j)+N(z_j,w)-\tilde{\psi}_{2,\alpha}(z_j,w)-\varphi_2(w)}|dw|^2\\
=&   \sum_{j\in I_m}\frac{2\pi}{p_j|d_j|^{2}}
\int_{w\in U_y}|\tilde{h}_j(z_j,w)|^2e^{N(z_j,w)-N_{\beta}(z_j,w)-\tilde{\psi}_{2,\alpha}(z_j,w)-\varphi_2(w)}|dw|^2.
   \end{split}
   \end{equation}
By Monotone convergence theorem, letting $\alpha \to +\infty$ and $\beta \to +\infty$, we have
  \begin{equation}\label{bigger 3}
   \begin{split}
      &\liminf_{t\to+\infty}\frac{\int_{\{\psi<-t\}\cap (V_0\times U_y)}|F|^2e^{-\varphi}c(-\psi)}{\int_{t}^{+\infty}c(s)e^{-s}ds}  \\
\ge&   \sum_{j\in I_m}\frac{2\pi}{p_j|d_j|^{2}}
\int_{w\in U_y}|\tilde{h}_j(z_j,w)|^2e^{-\tilde{\psi}_{2}(z_j,w)-\varphi_2(w)}|dw|^2.
   \end{split}
   \end{equation}
As $y$ and $U_y$ are arbitrarily chosen, it follows from $Y$ is a weakly pseudoconvex K\"ahler manifold that we have
  \begin{equation}\label{bigger 3}
   \begin{split}
      &\liminf_{t\to+\infty}\frac{\int_{\{\psi<-t\}}|F|^2e^{-\varphi}c(-\psi)}{\int_{t}^{+\infty}c(s)e^{-s}ds}  \\
\ge&   \sum_{j\in I_m}\frac{2\pi}{p_j|d_j|^{2}}
\int_{Y}|F_j|^2e^{-\tilde{\psi}_{2}(z_j,w)-\varphi_2(w)}.
   \end{split}
   \end{equation}

   Let $m\to +\infty$ and we have

   \begin{equation}\label{bigger 4}
   \begin{split}
      &\liminf_{t\to+\infty}\frac{\int_{\{\psi<-t\}}|F|^2e^{-\varphi}c(-\psi)}{\int_{t}^{+\infty}c(s)e^{-s}ds}
\ge  \sum_{j\in I_F}\frac{2\pi}{p_j|d_j|^{2}}
\int_{Y}|F_j|^2e^{-\tilde{\psi}_{2}(z_j,w)-\varphi_2(w)}.
   \end{split}
   \end{equation}

   Especially, for any $j\in I_F$, we have
$$\int_{Y}|F_j|^2e^{-\tilde{\psi}_{2}(z_j,w)-\varphi_2(w)}<+\infty.$$
\end{proof}

Denote $M=\prod_{1\le j\le n_1}\Omega_j\times Y$.
Let $p_{j,k}$ be a positive number for any $1\le j\le n_1$ and $1\le k<\tilde m_j$, which satisfies that $\sum_{1\le k<\tilde m_j}p_{j,k}G_{\Omega_j}(\cdot,z_{j,k})\not\equiv-\infty$ for any $1\le j\le n_1$.
Recall that
$$G:=\max_{1\le j\le n_1}\left\{2\sum_{1\le k<\tilde m_j}p_{j,k}\pi_{1,j}^{*}(G_{\Omega_j}(\cdot,z_{j,k}))\right\}.$$
Let $N\le0$ be a plurisubharmonic function on $M$ satisfying that $N|_{Z_0}\not\equiv -\infty$. Denote $\psi:=G+N$.

Let $\varphi_X:=\sum_{1\le j\le n_1}\varphi_j(z)$, where each $\varphi_j$ is an upper semi-continuous function on $\Omega_j$ satisfying $\varphi_j(z_j)\neq -\infty$ for any $z_j\in \Omega_j$. We assume that $\pi_1^*(\varphi_X)+N$ is a plurisubharmonic function on $M$. Let $\varphi_Y$ be a plurisubharmonic function on $Y$. Denote $\varphi:=\pi_1^*(\varphi_X)+\pi_2^*(\varphi_Y)$.

It follows from Lemma \ref{l:green-sup} and Lemma \ref{l:green-sup2} that there exists a local coordinate $w_{j,k}$  on a neighborhood $V_{z_{j,k}}\Subset\Omega_{j}$ of $z_{j,k}\in\Omega_j$ satisfying $w_{j,k}(z_{j,k})=0$ and
$$\log|w_{j,k}|=\frac{1}{p_{j,k}}\sum_{1\le k<\tilde m_j}p_{j,k}G_{\Omega_j}(\cdot,z_{j,k})$$
 for any $j\in\{1,\ldots,n_1\}$ and $1\le k<\tilde{m}_j$, where $V_{z_{j,k}}\cap V_{z_{j,k'}}=\emptyset$ for any $j$ and $k\not=k'$.

 Denote that $\tilde I_1:=\{(\beta_1,\ldots,\beta_{n_1}):1\le \beta_j< \tilde m_j$ for any $j\in\{1,\ldots,n_1\}\}$, $V_{\beta}:=\prod_{1\le j\le n_1}V_{z_{j,\beta_j}}$ for any $\beta=(\beta_1,\ldots,\beta_n)\in\tilde I_1$ and $w_{\beta}:=(w_{1,\beta_1},\ldots,w_{n,\beta_n})$ is a local coordinate on $V_{\beta}$ of $z_{\beta}:=(z_{1,\beta_1},\ldots,z_{n,\beta_n})\in M$.

 Let
 $$G_1=\max_{1\le j\le n_1}\left\{\tilde\pi_{j}^*\left(2\sum_{1\le k<\tilde{m}_j}p_{j,k}G_{\Omega_j}(\cdot,z_{j,k})\right)\right\}$$
  on $\prod_{1\le j\le n_1}\Omega_j$, where $\tilde\pi_j$ is the natural projection from $\prod_{1\le j\le n_1}\Omega_j$ to $\Omega_j$. Note that $G=\pi_1^*(G_1)$.

Let $F$ be a holomorphic $(n,0)$ form on $\{\psi<-t_0\}\subset M$ for some $t_0>0$ satisfying $\int_{\{\psi<-t_0\}}|F|^2e^{-\varphi}c(-\psi)<+\infty$. For any $\beta\in \tilde{I}_1$, it follows from Lemma \ref{decomp} that there exists a sequence of  holomorphic  $(n_2,0)$  forms $\{F_{\alpha,\beta}\}_{\alpha\in\mathbb{Z}_{\ge0}^{n_1}}$ on $Y$ such that
$$F=\sum_{\alpha\in\mathbb{Z}_{\ge0}^{n_1}}\pi_1^*(w_{\beta}^{\alpha}dw_{1,\beta_1}\wedge\ldots\wedge dw_{n_1,\beta_{n_1}})\wedge\pi_2^*(F_{\alpha,\beta})$$
on $V_{\beta}\times Y$.

Denote that $E_{\beta}:=\left\{\alpha\in\mathbb{Z}_{\ge0}^{n_1}:\sum_{1\le j\le n_1}\frac{\alpha_j+1}{p_{j,\beta_j}}=1\right\}$, $E_{1,\beta}:=\left\{\alpha\in\mathbb{Z}_{\ge0}^{n_1}:\sum_{1\le j\le n_1}\frac{\alpha_j+1}{p_{j,\beta_j}}<1\right\}$ and $E_{2,\beta}:=\left\{\alpha\in\mathbb{Z}_{\ge0}^{n_1}:\sum_{1\le j\le n_1}\frac{\alpha_j+1}{p_{j,\beta_j}}>1\right\}$.

Assume that $c(t)$ is increasing near $+\infty$.
\begin{Lemma}
	\label{l:limit}If $\liminf_{t\rightarrow+\infty}\frac{\int_{\{\psi<-t\}}|F|^2e^{-\varphi}c(-\psi)}{\int_t^{+\infty}c(s)e^{-s}ds}<+\infty$, we have $F_{\alpha,\beta}\equiv0$ for any $\alpha\in E_{1,\beta}$ and $\beta\in\tilde I_1$, and
	$$\liminf_{t\rightarrow+\infty}\frac{\int_{\{\psi<-t\}}|F|^2e^{-\varphi}c(-\psi)}{\int_t^{+\infty}c(s)e^{-s}ds}\ge
\sum_{\beta\in\tilde I_1}\sum_{\alpha\in E_{\beta}}\frac{(2\pi)^{n_1}e^{-\sum_{1\le j\le n_1}\varphi_j(z_{j,\beta_j})}}{\prod_{1\le j\le n_1}(\alpha_j+1)}\int_{Y}|F_{\alpha,\beta}|^2e^{-\varphi_Y-N(z_{\beta},w')}.$$
\end{Lemma}
\begin{proof}
	 As $\varphi_X$ is an upper semi-continuous functions on $X:=\prod_{1\le j\le n_1}\Omega_j$ and $N$ is an upper semi-continuous function on $M$, there exist continuous functions $\varphi_{X,l}$ and $N_{\gamma}$ on $X$ and $M$ decreasingly convergent to $\varphi_X$ and $N$ respectively.

When $t$ is large enough, $c(t)$ is increasing, then we have
\begin{equation}\nonumber
\int_{\{\psi<-t\}}|F|^2e^{-\varphi}c(-\psi)\ge
\int_{\{G+N_{\gamma}<-t\}}|F|^2e^{-\pi_1^*(\varphi_{X,l})-\pi_2^*(\varphi_{Y})}c(-G-N_{\gamma}).
\end{equation}

As $Y$ is a weakly pseudoconvex K\"ahler manifold, there exist open weakly pseudoconvex K\"ahler manifolds $Y_1\Subset\ldots\Subset Y_{s}\Subset Y_{s+1}\Subset\ldots$ such that $\cup_{s\in\mathbb{Z}_{\ge1}}Y_{s}=Y$.
	
Fix $Y_s$. For any $\beta\in \tilde{I}_1$, there exists $t_{\beta}>t_{0}$ such that $\{G+N_{\gamma}<-t_{\beta}\}\cap (V_{\beta}\times Y_{s})\Subset (V_{\beta}\times Y_{s+1})$.

On $V_{\beta}\times Y_s$, for any $\epsilon>0$, there exists $t_{\epsilon}$ large enough such that when $t>t_{\epsilon}$,
\\
(1) for any $(w,w')\in \{G+N_{\gamma}<-t\}$,
\\
 \centerline{$|\varphi_{X,l}(w)-\varphi_{X,l}(z_\beta)|<\epsilon$.}
\\
(2) Denote $H(w,w'):=N_{{\gamma}}(w,w')+\epsilon$.
For any $(w,w')\in \{G+N_{\gamma}<-t\}$,
\\
 \centerline{$|H(w,w')-H(z_\beta,w')|<\epsilon$.}
Then we have $G+H(z_{\beta},w')\ge G+N_s(w,w')$, for any $(w,w')\in\{G+N_{\gamma}<-t_{\beta}\}\cap (V_{\beta}\times Y_s)$.

Let
 $$G_1=\max_{1\le j\le n_1}\left\{\tilde\pi_{j}^*\left(2\sum_{1\le k<\tilde{m}_j}p_{j,k}G_{\Omega_j}(\cdot,z_{j,k})\right)\right\}$$
  on $\prod_{1\le j\le n_1}\Omega_j$, where $\tilde\pi_j$ is the natural projection from $\prod_{1\le j\le n_1}\Omega_j$ to $\Omega_j$. Note that $G=\pi_1^*(G_1)$.

  For any $t\ge \max{\{t_{\beta},t_\epsilon\}}$, note that $\{G_1<-t\}=\prod_{1\le j\le n_1}\left\{|w_{j,\beta_j}|<e^{-\frac{t}{2p_{j,\beta_j}}}\right\}$ and $F=\sum_{\alpha\in\mathbb{Z}_{\ge0}^{n_1}}\pi_1^*(w_{\beta}^{\alpha}dw_{1,\beta_1}\wedge\ldots\wedge dw_{n_1,\beta_{n_1}})\wedge\pi_2^*(F_{\alpha,\beta})$ on $\{G_1<-t\}\times Y$, then we have
\begin{equation}\begin{split}
	\label{eq:1218a}&\int_{\{\psi<-t\}\cap(V_{\beta}\times Y_s)}|F|^2e^{-\varphi}c(-\psi)\\
	\ge&\int_{\{G+N_{\gamma}<-t\}\cap(V_{\beta}\times Y_s)}|F|^2e^{-\pi_1^*(\varphi_{X,l})-\pi_2^*(\varphi_{Y})}c(-G-N_{\gamma})	\\
\ge&e^{-\epsilon}\int_{\{G+H(z_{\beta},w')<-t\}\cap(V_{\beta}\times Y_s)}|F|^2e^{-\pi_1^*(\varphi_{X,l}(z_{\beta},w'))-\pi_2^*(\varphi_{Y})}c(-G-H(z_{\beta},w'))\\
=&e^{-\epsilon}\sum_{\alpha\in\mathbb{Z}_{\ge0}^{n_1}}\int_{w'\in Y_s}\left(\int_{\{G_1+H(z_{\beta},w')<-t\}}|w_{\beta}^{\alpha}dw_{1,\beta_1}\wedge\ldots\wedge dw_{n_1,\beta_{n_1}}|^2c(-G-H(z_{\beta},w'))\right)\times\\
&e^{-\varphi_{X,l}(z_{\beta})}|F_{\alpha,\beta}|^2e^{-\varphi_Y}.
\end{split}\end{equation}
Denote that $q_{\alpha}:=\sum_{1\le j\le n_1}\frac{\alpha_j+1}{p_{j,\beta_j}}-1$. It follows from Lemma \ref{l:m1} and inequality \eqref{eq:1218a} that
\begin{equation*}
	\begin{split}
		&\int_{\{\psi<-t\}\cap(V_{\beta}\times Y_s)}|F|^2e^{-\varphi}c(-\psi)\\
		\ge& e^{-\varphi_{X,l}(z_{\beta})-\epsilon}\sum_{\alpha\in\mathbb{Z}_{\ge0}^{n_1}}\left(\int_t^{+\infty}c(s)e^{-(q_{\alpha}+1)s}ds\right)\frac{(q_{\alpha}+1)(2\pi)^{n_1}}{\prod_{1\le j\le n_1}(\alpha_j+1)}\int_{Y_s}|F_{\alpha,\beta}|^2e^{-\varphi_Y-(q_{\alpha}+1)H(z_{\beta},w')}.
	\end{split}
\end{equation*}
It follows from $\liminf_{t\rightarrow+\infty}\frac{\int_{\{\psi<-t\}}|F|^2e^{-\varphi}c(-\psi)}{\int_t^{+\infty}c(s)e^{-s}ds}<+\infty$ and Lemma \ref{growth rate of c(t)e(-t)} that
$$F_{\alpha,\beta}\equiv0$$
 for any $\alpha$ satisfying $q_{\alpha}<0$ and
\begin{equation*}
\liminf_{t\rightarrow+\infty}\frac{\int_{\{\psi<-t\}\cap(V_{\beta}\times Y_s)}|F|^2e^{-\varphi}c(-\psi)}{\int_t^{+\infty}c(s)e^{-s}ds}\ge e^{-\varphi_{X,l}(z_{\beta})-\epsilon}\sum_{\alpha\in E_{\beta}}\frac{(2\pi)^{n_1}\int_{Y_s}|F_{\alpha,\beta}|^2e^{-\varphi_Y-H(z_\beta,w')}}{\prod_{1\le j\le n_1}(\alpha_j+1)}.
\end{equation*}
Note that $H(w,w'):=N_{\gamma}(w,w')+\epsilon$. Letting $\epsilon\rightarrow 0$, $\gamma\rightarrow+\infty$ and $s\rightarrow+\infty$, we have
\begin{equation}
	\label{eq:1218b}
	\begin{split}
		&\liminf_{t\rightarrow+\infty}\frac{\int_{\{\psi<-t\}\cap(V_{\beta}\times Y)}|F|^2e^{-\varphi}c(-\psi)}{\int_t^{+\infty}c(s)e^{-s}ds}\\
		\ge&\sum_{\alpha\in E_{\beta}}\frac{(2\pi)^{n_1}e^{-\sum_{1\le j\le n_1}\varphi_j(z_{j,\beta_j})}}{\prod_{1\le j\le n_1}(\alpha_j+1)}\int_Y|F_{\alpha,\beta}|^2e^{-\varphi_Y-N(z_{\beta},w')}.
			\end{split}
\end{equation}

Note that $V_{\beta}\cap V_{\tilde\beta}=\emptyset$ for any $\beta\not=\tilde\beta$ and $\{\psi_1<-t_{\beta}\}\cap V_{\beta} \Subset V_{\beta}$ for any $\beta\in\tilde I_1$. It follows from inequality \eqref{eq:1218b} that
$$\liminf_{t\rightarrow+\infty}\frac{\int_{\{\psi<-t\}}|F|^2e^{-\varphi}c(-\psi)}{\int_t^{+\infty}c(s)e^{-s}ds}\ge\sum_{\beta\in\tilde I_1}\sum_{\alpha\in E_{\beta}}\frac{(2\pi)^{n_1}e^{-\sum_{1\le j\le n_1}\varphi_j(z_{j,\beta_j})}}{\prod_{1\le j\le n_1}(\alpha_j+1)}\int_{Y}|F_{\alpha,\beta}|^2e^{-\varphi_Y-N(z_{\beta},w')}.$$
Thus, Lemma \ref{l:limit} holds.
\end{proof}

Let $\tilde{M}$ be an open complex submanifold of $M$ satisfying that $Z_0=\{z_{\beta}:\beta\in\tilde I_1\}\times Y\subset \tilde{M}$, and let $K_{\tilde{M}}$  be the  canonical (holomorphic) line bundle on  $M_1$. Let $F_1$ be a holomorphic $(n,0)$ form on $\{\psi<-t_0\}\cap \tilde{M}$ for $t_0>0$ satisfying that $\int_{\{\psi<-t_0\}\cap \tilde{M}}|F_1|^2e^{-\varphi}c(-\psi)<+\infty$. For any $\beta\in\tilde I_1$, it follows from Lemma \ref{decomp-tildeM} that there exist a sequence of holomorphic $(n_2,0)$ forms $\{F_{\alpha,\beta}\}_{\alpha\in\mathbb{Z}_{\ge0}^{n_1}}$ on $Y$ and an open subset $U_{\beta}$ of $\{\psi<-t_0\}\cap \tilde{M}\cap (V_{\beta}\times Y)$ such that
$$F_1=\sum_{\alpha\in\mathbb{Z}_{\ge0}^{n_1}}\pi_1^*(w_{\beta}^{\alpha}dw_{1,\beta_1}\wedge\ldots\wedge dw_{n_1,\beta_{n_1}})\wedge\pi_2^*(F_{\alpha,\beta})$$
on $U_{\beta}$ and
$$\int_{K}|F_{\alpha,\beta}|^2e^{-\varphi_Y}<+\infty$$
for any $\alpha\in\mathbb{Z}_{\ge0}^{n_1}$ and compact subset $K$ of $Y$. Using the similar method in Lemma \ref{l:limit}, we have the following Remark.
\begin{Remark}
	\label{l:limit2}If $\liminf_{t\rightarrow+\infty}\frac{\int_{\{\psi<-t\}\cap \tilde{M}}|F|^2e^{-\varphi}c(-\psi)}{\int_t^{+\infty}c(s)e^{-s}ds}<+\infty$, we have $F_{\alpha,\beta}\equiv0$ for any $\alpha\in E_{1,\beta}$ and $\beta\in\tilde I_1$, and

$$\liminf_{t\rightarrow+\infty}\frac{\int_{\{\psi<-t\}\cap \tilde{M}}|F|^2e^{-\varphi}c(-\psi)}{\int_t^{+\infty}c(s)e^{-s}ds}\ge
\sum_{\beta\in\tilde I_1}\sum_{\alpha\in E_{\beta}}\frac{(2\pi)^{n_1}e^{-\sum_{1\le j\le n_1}\varphi_j(z_{j,\beta_j})}}{\prod_{1\le j\le n_1}(\alpha_j+1)}\int_{Y}|F_{\alpha,\beta}|^2e^{-\varphi_Y-N}.$$

	\end{Remark}
\begin{proof}
 As $\varphi_X$ is an upper semi-continuous functions on $X:=\prod_{1\le j\le n_1}\Omega_j$ and $N$ is an upper semi-continuous function on $M$, there exist continuous functions $\varphi_{X,l}$ and $N_{\gamma}$ on $X$ and $M$ decreasingly convergent to $\varphi_X$ and $N$ respectively.

When $t$ is large enough, $c(t)$ is increasing, then we have
\begin{equation}\nonumber
\int_{\{\psi<-t\}\cap \tilde{M}}|F|^2e^{-\varphi}c(-\psi)\ge
\int_{\{G+N_{\gamma}<-t\}\cap \tilde{M}}|F|^2e^{-\pi_1^*(\varphi_{X,l})-\pi_2^*(\varphi_{Y})}c(-G-N_{\gamma}).
\end{equation}

As $Y$ is a weakly pseudoconvex K\"ahler manifold, there exist open weakly pseudoconvex K\"ahler manifolds $Y_1\Subset\ldots\Subset Y_{s}\Subset Y_{s+1}\Subset\ldots$ such that $\cup_{s\in\mathbb{Z}_{\ge1}}Y_{s}=Y$.
	
For any $\beta\in\tilde I_1$ and any $Y_s$ , there exists open subset $\hat{V}_{\beta}\Subset V_{\beta}$ and $t_{\beta,s}>t_{0}$ such that $\{G+N_{\gamma}<-t_{\beta,s}\}\cap (\hat{V}_{\beta}\times Y_s)\Subset (\hat{V}_{\beta}\times Y_{s+1})\Subset U_{\beta}$.

On $\hat{V}_{\beta}\times Y_s$, for any $\epsilon>0$, there exists $t_{\epsilon}$ large enough such that when $t>t_{\epsilon}$,
\\
(1) for any $(w,w')\in \{G+N_{\gamma}<-t\}$,
\\
 \centerline{$|\varphi_{X,l}(w)-\varphi_{X,l}(z_\beta)|<\epsilon$.}
\\
(2) Denote $H(w,w'):=N_{{\gamma}}(w,w')+\epsilon$.
For any $(w,w')\in \{G+N_{\gamma}<-t\}$,
\\
 \centerline{$|H(w,w')-H(z_\beta,w')|<\epsilon$.}
Then we have $G+H(z_{\beta},w')\ge G+N_s(w,w')$, for any $(w,w')\in\{G+N_{\gamma}<-t_{\beta}\}\cap (\hat{V}_{\beta}\times Y_s)$.

Recall that
 $G_1=\max_{1\le j\le n_1}\left\{\tilde\pi_{j}^*\left(2\sum_{1\le k<\tilde{m}_j}p_{j,k}G_{\Omega_j}(\cdot,z_{j,k})\right)\right\}$
  on $\prod_{1\le j\le n_1}\Omega_j$, where $\tilde\pi_j$ is the natural projection from $\prod_{1\le j\le n_1}\Omega_j$ to $\Omega_j$. Note that $G=\pi_1^*(G_1)$.
For any $t\ge \max{\{t_{\beta,s},t_\epsilon\}}$, note that $\{G_1<-t\}=\prod_{1\le j\le n_1}\left\{|w_{j,\beta_j}|<e^{-\frac{t}{2p_{j,\beta_j}}}\right\}$ and $F=\sum_{\alpha\in\mathbb{Z}_{\ge0}^{n_1}}\pi_1^*(w_{\beta}^{\alpha}dw_{1,\beta_1}\wedge\ldots\wedge dw_{n_1,\beta_{n_1}})\wedge\pi_2^*(F_{\alpha,\beta})$ on $U_{\beta}$, then we have
\begin{equation}\begin{split}
	\label{eq:1218a2}&\int_{\{\psi<-t\}\cap(\hat{V}_{\beta}\times Y_s)}|F|^2e^{-\varphi}c(-\psi)\\
	\ge&\int_{\{G+N_{\gamma}<-t\}\cap(\hat{V}_{\beta}\times Y_s)}|F|^2e^{-\pi_1^*(\varphi_{X,l})-\pi_2^*(\varphi_{Y})}c(-G-N_{\gamma})	\\
\ge&e^{-\epsilon}\int_{\{G+H(z_{\beta},w')<-t\}\cap(\hat{V}_{\beta}\times Y_s)}|F|^2e^{-\pi_1^*(\varphi_{X,l}(z_{\beta},w'))-\pi_2^*(\varphi_{Y})}c(-G-H(z_{\beta},w'))\\
=&e^{-\epsilon}\sum_{\alpha\in\mathbb{Z}_{\ge0}^{n_1}}\int_{w'\in Y_s}\left(\int_{\{G_1+H(z_{\beta},w')<-t\}}|w_{\beta}^{\alpha}dw_{1,\beta_1}\wedge\ldots\wedge dw_{n_1,\beta_{n_1}}|^2c(-G-H(z_{\beta},w'))\right)\times\\
&e^{-\varphi_{X,l}(z_{\beta})}|F_{\alpha,\beta}|^2e^{-\varphi_Y}.
\end{split}\end{equation}
Denote that $q_{\alpha}:=\sum_{1\le j\le n_1}\frac{\alpha_j+1}{p_{j,\beta_j}}-1$. It follows from Lemma \ref{l:m1} and inequality \eqref{eq:1218a2} that
\begin{equation*}
	\begin{split}
		&\int_{\{\psi<-t\}\cap(\hat{V}_{\beta}\times Y_s)}|F|^2e^{-\varphi}c(-\psi)\\
		\ge& e^{-\varphi_{X,l}(z_{\beta})-\epsilon}\sum_{\alpha\in\mathbb{Z}_{\ge0}^{n_1}}\left(\int_t^{+\infty}c(s)e^{-(q_{\alpha}+1)s}ds\right)\frac{(q_{\alpha}+1)(2\pi)^{n_1}}{\prod_{1\le j\le n_1}(\alpha_j+1)}\int_{Y_s}|F_{\alpha,\beta}|^2e^{-\varphi_Y-(q_{\alpha}+1)H(z_{\beta},w')}.
	\end{split}
\end{equation*}
It follows from $\liminf_{t\rightarrow+\infty}\frac{\int_{\{\psi<-t\}\cap \tilde{M}}|F|^2e^{-\varphi}c(-\psi)}{\int_t^{+\infty}c(s)e^{-s}ds}<+\infty$ and Lemma \ref{growth rate of c(t)e(-t)} that
$$F_{\alpha,\beta}\equiv0$$
 for any $\alpha$ satisfying $q_{\alpha}<0$ and
\begin{equation*}
\liminf_{t\rightarrow+\infty}\frac{\int_{\{\psi<-t\}\cap(\hat{V}_{\beta}\times Y_s)}|F|^2e^{-\varphi}c(-\psi)}{\int_t^{+\infty}c(s)e^{-s}ds}\ge e^{-\varphi_{X,l}(z_{\beta})-\epsilon}\sum_{\alpha\in E_{\beta}}\frac{(2\pi)^{n_1}\int_{Y_s}|F_{\alpha,\beta}|^2e^{-\varphi_Y-H(z_\beta,w')}}{\prod_{1\le j\le n_1}(\alpha_j+1)}.
\end{equation*}
Note that $H(w,w'):=N_{\gamma}(w,w')+\epsilon$. Letting $\epsilon\rightarrow 0$, $\gamma\rightarrow+\infty$ and $s\rightarrow+\infty$, we have
\begin{equation}
	\label{eq:1218b2}
	\begin{split}
		&\liminf_{t\rightarrow+\infty}\frac{\int_{\{\psi<-t\}\cap (\hat{V}_{\beta}\times Y_s)}|F|^2e^{-\varphi}c(-\psi)}{\int_t^{+\infty}c(s)e^{-s}ds}\\
		\ge&\sum_{\alpha\in E_{\beta}}\frac{(2\pi)^{n_1}e^{-\sum_{1\le j\le n_1}\varphi_j(z_{j,\beta_j})}}{\prod_{1\le j\le n_1}(\alpha_j+1)}\int_Y|F_{\alpha,\beta}|^2e^{-\varphi_Y-N(z_{\beta},w')}.
			\end{split}
\end{equation}

Note that $V_{\beta}\cap V_{\tilde\beta}=\emptyset$ implies that $\hat{V}_{\beta}\cap \hat{V}_{\tilde\beta}=\emptyset$ for any $\beta\not=\tilde\beta$. It follows from inequality \eqref{eq:1218b2} that
$$\liminf_{t\rightarrow+\infty}\frac{\int_{\{\psi<-t\}\cap\tilde{M}}|F|^2e^{-\varphi}c(-\psi)}{\int_t^{+\infty}c(s)e^{-s}ds}\ge\sum_{\beta\in\tilde I_1}\sum_{\alpha\in E_{\beta}}\frac{(2\pi)^{n_1}e^{-\sum_{1\le j\le n_1}\varphi_j(z_{j,\beta_j})}}{\prod_{1\le j\le n_1}(\alpha_j+1)}\int_{Y}|F_{\alpha,\beta}|^2e^{-\varphi_Y-N(z_{\beta},w')}.$$
Thus, Remark \ref{l:limit2} holds.

\end{proof}

\section{Proofs of Theorem \ref{finite points}, Proposition \ref{infinite points} and Proposition \ref{tildeM}}

\subsection{Proof of Theorem \ref{finite points}}

\begin{proof}
We firstly give the proof of the sufficiency in Theorem \ref{finite points}.

As $\psi+\pi_1^*(\varphi_1)=\pi_1^*(2\sum_{j=1}^mp_jG_{\Omega}(\cdot,z_j)+\varphi_1)$ is a plurisubharmonic function on $M$, by definition, we know that $2\sum_{j=1}^mp_jG_{\Omega}(\cdot,z_j)+\varphi_1$ is a subharmonic function on $\Omega$. By the construction of $g\in \mathcal{O}_{\Omega}$, we have $2u(z):=2\sum_{j=1}^mp_jG_{\Omega}(\cdot,z_j)+\varphi_1-2\log|g|$ is a subharmonic function on $\Omega$ which satisfies $v(dd^cu,z)\in[0,1)$ for any $z\in Z_{\Omega}$. Then it follows from Lemma \ref{equiv of multiplier ideal sheaf} that we know $\mathcal{I}(\varphi+\psi)_{(z_j,y)}=\mathcal{I}\big(\pi_1^{*}(2\log |g|)+\pi_2^{*}(\varphi_2)\big)_{(z_j,y)}$ for any $(z_j,y)\in Z_0$. It follows from Theorem \ref{BGY-RESULT-FINITE POINTS} that we know the sufficiency part of  Theorem \ref{finite points} holds.

Now we prove the necessity part of Theorem \ref{finite points}.

Assume that $G\big(h^{-1}(r)\big)$ is linear with respect to $r\in (0,\int_0^{+\infty}c(s)e^{-s}ds]$. Then according to Lemma \ref{linear}, there exists a unique holomorphic $(n,0)$ form $F$ on $M$ satisfying $(F-f)\in H^0(Z_0,\Big(\mathcal{O}(K_M)\otimes\mathcal{I}\big(\pi^{*}_1(2\log |g|)+\pi^{*}_2(\varphi_2)\big)\Big)|_{Z_0})$, and $G(t;c)=\int_{\{\psi<-t\}}|F|^2e^{-\varphi}c(-\psi)$ for any $t\geq 0$. Then according to Lemma \ref{linear}, Remark \ref{ctildec} and Lemma \ref{tildecincreasing}, we can assume that $c$ is increasing near $+\infty$.

Following from Lemma \ref{decomp}, for any $j\in\{1,2,\ldots,m\}$, we assume that
		\begin{equation*}
			F=\sum_{l=k_j}^{\infty}\pi_1^*(\tilde{z}_j^ld\tilde{z}_j)\wedge \pi_2^*(F_{j,l})
		\end{equation*}
		on $V_{j}\times Y$, where $k_j\in\mathbb{N}$, $F_{j,l}$ is a holomorphic $(n-1,0)$ form on $Y$ for any $l\geq k_j$, and $\tilde{F}_j:=F_{j,k_j}\not\equiv 0$.

We firstly recall the construction of $\psi$ and $\varphi$.

Recall that $\psi=\pi^{*}_1\big(\sum\limits_{j=1}^m 2p_jG_{\Omega}(z,z_j)\big)+N$ is a plurisubharmonic function on $M$. We assume that $N\le0$ is a plurisubharmonic function on $M$ and $N|_{Z_j}$ is not identically $-\infty$ for any $j$. $\varphi_1$ is a Lebesgue measurable function on $\Omega$ such that $\psi+\pi^{*}_1(\varphi)$ is a plurisubharmonic function on $M$. We also note that, by Siu's decomposition theorem and Weierstrass theorem on open Riemann surfaces, we have
$$\psi+\pi^{*}_1(\varphi_1)=\pi^{*}_1(2\log|g|)+\tilde{\psi}_2,$$
where $g$ is a holomorphic function on $\Omega$ and $\tilde{\psi}_2\in Psh(M)$. Denote $ord_{z_j}(g)=p_j$ and $d_j:=\lim_{z\rightarrow z_j}(g/\tilde{z}_j^{q_j})(z)$.

		Now we prove that $\psi=\pi_1^*\big(2\sum_{j=1}^mp_jG_{\Omega}(\cdot,z_j)\big)$. As $G(h^{-1}(r))$ is linear and $c$ is increasing near $+\infty$, using Lemma \ref{slope bigger than measure}, we can get that $k_j-q_j+1=0$ for any $j\in I_F$ and
		\begin{equation}\label{geqfinite finite points}
			\frac{\int_M|F|^2e^{-\varphi}c(-\psi)}{\int_0^{+\infty}c(s)e^{-s}ds}\geq \sum_{j\in I_F}\frac{2\pi}{p_j|d_j|^2}\int_Y|F_j|^2e^{-\varphi_2-\tilde{\psi}_2(z_j,w)},
		\end{equation}
		where $I_F:=\{j:k_j-q_j+1\leq 0\}$. Especially, $\sum_{j\in I_F}\frac{2\pi}{p_j|d_j|^2}\int_Y|F_j|^2e^{-\varphi_2-\tilde{\psi}_2(z_j,w)}<+\infty$ and $\tilde{\psi}_2$ is not identically $-\infty$ on $Z_j$ for any $j\in I_F$.

Let $V_j\times U_y$ be an open neighborhood of $(z_j,y)$. Assume that $F=\tilde{z}_j^{k_j}h(\tilde{z}_j,w)d\tilde{z}_j\wedge dw$ on $V_j\times U_y$, where $h(z_j,w)dw=\tilde{F}_j$.  Note that $h(z_j,w)$ is not identically zero on $\{z_j\}\times U_y$, there must exist $\hat{w}\in U_y$ such that $h(z_j,\hat{w})\neq 0$. Then we know $|h(z_j,w)|^2$ has a positive lower bound on $\tilde{V}_j\times U_{\hat{w}}$, where $\tilde{V}_j\times U_{\hat{w}}$ is a small open neighborhood of $(z_j,\hat{w})$.
 Then, according to Lemma \ref{e-varphic-psi}, $\psi\le\pi_1^*\big(2\sum_{j=1}^mp_jG_{\Omega}(\cdot,z_j)\big)$ and $c(t)$ is increasing near $+\infty$, we know that for any $w\in U_{\hat{w}}$, we have $\int_{\tilde{V}_j}|\tilde{z}_j^{k_j}d\tilde{z}_j|^2e^{-\varphi_1}c(-\sum_{j=1}^m2p_jG_{\Omega}(\cdot,z_j))<+\infty$.
It follows from Lemma \ref{f1zf2w} that we have $$\int_{\tilde{V}_j\times U_{\hat{w}}}|\pi_1^*(\tilde{z}_j^{k_j}d\tilde{z}_j)\wedge\pi_2^*(\tilde{F}_j)|^2
e^{-\varphi}c\big(-\pi_1^*(\sum_{j=1}^m2p_jG_{\Omega}(\cdot,z_j))\big)<+\infty.$$

Then by  Lemma \ref{k>k0}, we have
		\begin{equation*}
			\big(F-\pi_1^*(\tilde{z}_j^{k_j}d\tilde{z}_j)\wedge\pi_2^*(\tilde{F}_j),(z_j,y)\big)
\in\big(\mathcal{O}(K_M)\otimes\mathcal{I}(\pi^{*}_1(2\log |g|)+\pi^{*}_2(\varphi_2))\big)_{(z_j,y)}
		\end{equation*}
		for any $j\in I_F$, $y\in Y$. And according to Lemma \ref{k>k0} we also have
		\begin{equation*}
			\big(F,(z_j,y)\big)\in\big(\mathcal{O}(K_M)\otimes\mathcal{I}(\pi^{*}_1(2\log |g|)+\pi^{*}_2(\varphi_2))\big)_{(z_j,y)}
		\end{equation*}
		for any $j\in\{1,2,\ldots,m\}\setminus I_F$, $y\in Y$.

Denote that $\tilde{\psi}:=\pi_1^*\big(2\sum_{j=1}^mp_jG_{\Omega}(\cdot,z_j)\big)$, $\tilde{\varphi}_1:=\pi_1^*(\varphi_1)+\psi-\tilde{\psi}$, and $\tilde{\varphi}:=\tilde{\varphi}_1+\pi_2^*(\varphi_2)$. Then according to Lemma \ref{optimal extension} \big($F=\pi_1^*(\tilde{z}_j^{k_j}dw_j)\wedge\pi_2^*(\tilde{F}_j)$ on $V_j\times Y$ for $j\in I_F$ and $F\equiv0$ on $V_j\times Y$ for $j\notin I_F$ in Lemma \ref{optimal extension}\big), there exists a holomorphic $(n,0)$ form $\tilde{F}$ on $M$ such that $\big(\tilde{F}-\pi_1^*(\tilde{z}_j^{k_j}dw_j)\wedge\pi_2^*(\tilde{F}_j),(z_j,y)\big)
\in\big(\mathcal{O}(K_M)\otimes\mathcal{I}(\pi^{*}_1(2\log |g|)+\pi^{*}_2(\varphi_2))\big)_{(z_j,y)}$ for any $j\in I_F$, $y\in Y$, $\big(\tilde{F},(z_j,y)\big)\in\big(\mathcal{O}(K_M)\otimes\mathcal{I}(\pi^{*}_1(2\log |g|)+\pi^{*}_2(\varphi_2))\big)_{(z_j,y)}$ for any $j\in\{1,2,\ldots,m\}\setminus I_F$, $y\in Y$, and
		\begin{equation}\label{leqfinite finite points}
			\int_M|\tilde{F}|^2e^{-\tilde{\varphi}}c(-\tilde{\psi})\leq \left({\int_0^{+\infty}c(s)e^{-s}ds}\right)\sum_{j\in I_F}\frac{2\pi}{p_j|d_j|^2}\int_Y|F_j|^2e^{-\varphi_2-\tilde{\psi}_2(z_j,w)}.
		\end{equation}
		Then $\big(\tilde{F}-F,(z_j,y)\big)\in\big(\mathcal{O}(K_M)\otimes\mathcal{I}(\pi^{*}_1(2\log |g|)+\pi^{*}_2(\varphi_2))\big)_{(z_j,y)}$ for any $(z_j,y)\in Z_0$. Combining inequality (\ref{geqfinite finite points}) with inequality (\ref{leqfinite finite points}), we have that
		\begin{equation}\label{inequality in finite pts}
			\int_M|\tilde{F}|^2e^{-\varphi}c(-\psi)\geq \int_M|F|^2e^{-\varphi}c(-\psi)\geq \int_M|\tilde{F}|^2e^{-\tilde{\varphi}}c(-\tilde{\psi}).
		\end{equation}
		As $c(t)e^{-t}$ is decreasing with respect to $t$ and $\psi\le \tilde{\psi}$, we have $\int_M|\tilde{F}|^2e^{-\tilde{\varphi}}c(-\tilde{\psi})\ge \int_M|\tilde{F}|^2e^{-\varphi}c(-\psi)$. Then all ``$\ge$" in \eqref{inequality in finite pts} should be ``$=$". It follows from Lemma \ref{decreasing property of l} that $N=0$ and $\psi=\pi_1^*\big(2\sum_{j=1}^mp_jG_{\Omega}(\cdot,z_j)\big)$.

As $\psi+\pi_1^*(\varphi_1)=\pi_1^*(2\sum_{j=1}^mp_jG_{\Omega}(\cdot,z_j)+\varphi_1)$ is a plurisubharmonic function on $M$, by definition, we know that $2\sum_{j=1}^mp_jG_{\Omega}(\cdot,z_j)+\varphi_1$ is a subharmonic function on $\Omega$. By the construction of $g\in \mathcal{O}_{\Omega}$, we have $2u(z):=2\sum_{j=1}^mp_jG_{\Omega}(\cdot,z_j)+\varphi_1-2\log|g|$ is a subharmonic function on $\Omega$ which satisfies $v(dd^cu,z)\in[0,1)$ for any $z\in Z_0$. Then it follows from Lemma \ref{equiv of multiplier ideal sheaf} that we know $\mathcal{I}(\varphi+\psi)_{(z_j,y)}=\mathcal{I}\big(\pi_1^{*}(2\log |g|)+\pi_2^{*}(\varphi_2)\big)_{(z_j,y)}$ for any $(z_j,y)\in Z_0$.

Then it follows from Theorem \ref{BGY-RESULT-FINITE POINTS} that we know Theorem \ref{finite points} holds.
\end{proof}

\subsection{Proof of Proposition \ref{infinite points}}

\begin{proof}[Proof of Proposition \ref{infinite points}]
Assume that $G\big(h^{-1}(r)\big)$ is linear with respect to $r\in (0,\int_0^{+\infty}c(s)e^{-s}ds]$. Then according to Lemma \ref{linear}, there exists a unique holomorphic $(n,0)$ form $F$ on $M$ satisfying $(F-f)\in H^0(Z_0,\big(\mathcal{O}(K_M)\otimes\mathcal{I}(\pi^{*}_1(2\log |g|)+\pi^{*}_2(\varphi_2))\big)|_{Z_0})$, and $G(t;c)=\int_{\{\psi<-t\}}|F|^2e^{-\varphi}c(-\psi)$ for any $t\geq 0$. Then according to Lemma \ref{linear}, Remark \ref{ctildec} and Lemma \ref{tildecincreasing}, we can assume that $c$ is increasing near $+\infty$.

Following from Lemma \ref{decomp}, for any $j\ge 1$, we assume that
		\begin{equation*}
			F=\sum_{l=k_j}^{\infty}\pi_1^*(\tilde{z}_j^ld\tilde{z}_j)\wedge \pi_2^*(F_{j,l})
		\end{equation*}
		on $V_{j}\times Y$, where $k_j\in\mathbb{N}$, $F_{j,l}$ is a holomorphic $(n-1,0)$ form on $Y$ for any $l\geq k_j$, and $\tilde{F}_j:=F_{j,k_j}\not\equiv 0$.

We firstly recall the construction of $\psi$ and $\varphi$.

Recall that $\psi=\pi^{*}_1\big(\sum\limits_{j\ge1}2p_jG_{\Omega}(z,z_j)\big)+N$ is a plurisubharmonic function on $M$. We assume that $N\le 0$ is a plurisubharmonic function on $M$ and $N|_{Z_j}$ is not identically $-\infty$ for any $j$. $\varphi_1$ is a Lebesgue measurable function on $\Omega$ such that $\psi+\pi^{*}_1(\varphi)$ is a plurisubharmonic function on $M$. We also note that, by Siu's decomposition theorem and Weierstrass theorem on open Riemann surface, we have
$$\psi+\pi^{*}_1(\varphi_1)=\pi^{*}_1(2\log|g|)+\tilde{\psi}_2,$$
where $g$ is a holomorphic function on $\Omega$ and $\tilde{\psi}_2\in Psh(M)$. Denote $ord_{z_j}(g)=p_j$ and $d_j:=\lim_{z\rightarrow z_j}(g/\tilde{z}_j^{q_j})(z)$.

		Now we prove that $\psi_1=2\sum_{j=1}^{\gamma}p_jG_{\Omega}(\cdot,z_j)$. As $G(h^{-1}(r))$ is linear and $c$ is increasing near $+\infty$, using Lemma \ref{slope bigger than measure}, we can get that $k_j-q_j+1=0$ for any $j\in I_F$ and
		\begin{equation}\label{geqfinite infinite points}
			\frac{\int_M|F|^2e^{-\varphi}c(-\psi)}{\int_0^{+\infty}c(s)e^{-s}ds}\geq \sum_{j\in I_F}\frac{2\pi}{p_j|d_j|^2}\int_Y|F_j|^2e^{-\varphi_2-\tilde{\psi}_2(z_j,w)},
		\end{equation}
		where $I_F:=\{j:k_j-q_j+1\leq 0\}$. Especially, $\sum_{j\in I_F}\frac{2\pi}{p_j|d_j|^2}\int_Y|F_j|^2e^{-\varphi_2-\tilde{\psi}_2(z_j,w)}<+\infty$ and $\tilde{\psi}_2$ is not identically $-\infty$ on $Z_j$ for any $j\in I_F$.

Let $V_j\times U_y$ be an open neighborhood of $(z_j,y)$. Assume that $F=\tilde{z}_j^{k_j}h(\tilde{z}_j,w)d\tilde{z}_j\wedge dw$ on $V_j\times U_y$, where $h(z_j,w)dw=\tilde{F}_j$.  Note that $h(z_j,w)$ is not identically zero on $\{z_j\}\times U_y$, there must exist $\hat{w}\in U_y$ such that $h(z_j,\hat{w})\neq 0$. Then we know $|h(z_j,w)|^2$ has a positive lower bound on $\tilde{V}_j\times U_{\hat{w}}$, where $\tilde{V}_j\times U_{\hat{w}}$ is a small open neighborhood of $(z_j,\hat{w})$.
 Then, according to Lemma \ref{e-varphic-psi}, $\psi\le\pi_1^*(\sum_{j=1}^{\gamma}2p_jG_{\Omega}(\cdot,z_j))$ and $c(t)$ is increasing near $+\infty$, we know that for any $w\in U_{\hat{w}}$, we have $\int_{\tilde{V}_j}|\tilde{z}_j^{k_j}d\tilde{z}_j|^2e^{-\varphi_1}c(-\sum_{j=1}^m2p_jG_{\Omega}(\cdot,z_j))<+\infty$.
It follows from Lemma \ref{f1zf2w} that we have $$\int_{\tilde{V}_j\times U_{\hat{w}}}|\pi_1^*(\tilde{z}_j^{k_j}d\tilde{z}_j)\wedge\pi_2^*(\tilde{F}_j)|^2
e^{-\varphi}c(-\pi_1^*(\sum_{j=1}^m2p_jG_{\Omega}(\cdot,z_j)))<+\infty.$$

Then by Lemma \ref{k>k0}, we have
		\begin{equation*}
			\big(F-\pi_1^*(\tilde{z}_j^{k_j}d\tilde{z}_j)\wedge\pi_2^*(\tilde{F}_j),(z_j,y)\big)
\in\big(\mathcal{O}(K_M)\otimes\mathcal{I}(\pi^{*}_1(2\log |g|)+\pi^{*}_2(\varphi_2))\big)_{(z_j,y)}
		\end{equation*}
		for any $j\in I_F$, $y\in Y$. And according to Lemma \ref{k>k0} we also have
		\begin{equation*}
			\big(F,(z_j,y)\big)\in\big(\mathcal{O}(K_M)\otimes\mathcal{I}(\pi^{*}_1(2\log |g|)+\pi^{*}_2(\varphi_2))\big)_{(z_j,y)}
		\end{equation*}
		for any $j\in\{1,2,\ldots,m\}\setminus I_F$, $y\in Y$.

Denote that $\tilde{\psi}:=\pi_1^*\big(2\sum_{j=1}^{\gamma}p_jG_{\Omega}(\cdot,z_j)\big)$, $\tilde{\varphi}_1:=\pi_1^*(\varphi_1)+\psi-\tilde{\psi}$, and $\tilde{\varphi}:=\tilde{\varphi}_1+\pi_2^*(\varphi_2)$. Then according to Lemma \ref{optimal extension} \big($F=\pi_1^*(\tilde{z}_j^{k_j}dw_j)\wedge\pi_2^*(\tilde{F}_j)$ on $V_j\times Y$ for $j\in I_F$ and $F\equiv0$ on $V_j\times Y$ for $j\notin I_F$ in Lemma \ref{optimal extension}\big), there exists a holomorphic $(n,0)$ form $\tilde{F}$ on $M$ such that $\big(\tilde{F}-\pi_1^*(\tilde{z}_j^{k_j}dw_j)\wedge\pi_2^*(\tilde{F}_j),(z_j,y)\big)
\in\big(\mathcal{O}(K_M)\otimes\mathcal{I}(\pi^{*}_1(2\log |g|)+\pi^{*}_2(\varphi_2))\big)_{(z_j,y)}$ for any $j\in I_F$, $y\in Y$, $\big(\tilde{F},(z_j,y)\big)\in\big(\mathcal{O}(K_M)\otimes\mathcal{I}(\pi^{*}_1(2\log |g|)+\pi^{*}_2(\varphi_2))\big)_{(z_j,y)}$ for any $j\in\{1,2,\ldots,m\}\setminus I_F$, $y\in Y$, and
		\begin{equation}\label{leqfinite infinite points}
			\int_M|\tilde{F}|^2e^{-\tilde{\varphi}}c(-\tilde{\psi})\leq \left({\int_0^{+\infty}c(s)e^{-s}ds}\right)\sum_{j\in I_F}\frac{2\pi}{p_j|d_j|^2}\int_Y|F_j|^2e^{-\varphi_2-\tilde{\psi}_2(z_j,w)}.
		\end{equation}
		 Then $\big(\tilde{F}-F,(z_j,y)\big)\in\big(\mathcal{O}(K_M)\otimes\mathcal{I}(\pi^{*}_1(2\log |g|)+\pi^{*}_2(\varphi_2))\big)_{(z_j,y)}$ for any $(z_j,y)\in Z_0$. Combining inequality (\ref{geqfinite infinite points}) with inequality (\ref{leqfinite infinite points}), we have that
		\begin{equation}\label{inequality in infinite pts}
			\int_M|\tilde{F}|^2e^{-\varphi}c(-\psi)\geq \int_M|F|^2e^{-\varphi}c(-\psi)\geq \int_M|\tilde{F}|^2e^{-\tilde{\varphi}}c(-\tilde{\psi}).
		\end{equation}
As $c(t)e^{-t}$ is decreasing with respect to $t$ and $\psi\le \tilde{\psi}$, we have $\int_M|\tilde{F}|^2e^{-\tilde{\varphi}}c(-\tilde{\psi})\ge \int_M|\tilde{F}|^2e^{-\varphi}c(-\psi)$. Then all ``$\ge$" in \eqref{inequality in infinite pts} should be ``$=$". It follows from Lemma \ref{decreasing property of l} that $N=0$ and $\psi=\pi_1^*\big(2\sum_{j=1}^{\gamma}p_jG_{\Omega}(\cdot,z_j)\big)$.		

As $\psi+\pi_1^*(\varphi_1)=\pi_1^*(2\sum_{j=1}^{\gamma}p_jG_{\Omega}(\cdot,z_j)+\varphi_1)$ is a plurisubharmonic function on $M$, by definition, we know that $2\sum_{j=1}^{\gamma}p_jG_{\Omega}(\cdot,z_j)+\varphi_1$ is a subharmonic function on $\Omega$. By the construction of $g\in \mathcal{O}_{\Omega}$, we have $2u(z):=2\sum_{j=1}^{\gamma}p_jG_{\Omega}(\cdot,z_j)+\varphi_1-2\log|g|$ is a subharmonic function on $\Omega$ which satisfies $v(dd^cu,z)\in[0,1)$ for any $z\in Z_0$. Then it follows from Lemma \ref{equiv of multiplier ideal sheaf} that we know $\mathcal{I}(\varphi+\psi)_{(z_j,y)}=\mathcal{I}\big(\pi_1^{*}(2\log |g|)+\pi_2^{*}(\varphi_2)\big)_{(z_j,y)}$ for any $(z_j,y)\in Z_0$.

Then it follows from Proposition \ref{BGY-RESULT-INFINITE POINTS} that we know Proposition \ref{infinite points} holds.
\end{proof}
\subsection{Proof of Proposition \ref{tildeM}}

\begin{proof}
Assume that $\tilde{G}\big(h^{-1}(r)\big)$ is linear with respect to $r\in (0,\int_0^{+\infty}c(s)e^{-s}ds]$. Then according to Lemma \ref{linear}, there exists a unique holomorphic $(n,0)$ form $F$ on $\tilde M$ satisfying $(F-f)\in H^0(Z_0,\big(\mathcal{O}(K_{\tilde{M}})\otimes\mathcal{I}(\pi^{*}_1(2\log |g|)+\pi^{*}_2(\varphi_2))\big)|_{Z_0})$, and $\tilde{G}(t;c)=\int_{\{\psi<-t\}}|F|^2e^{-\varphi}c(-\psi)$ for any $t\geq 0$. Then according to Lemma \ref{linear}, Remark \ref{ctildec} and Lemma \ref{tildecincreasing}, we can assume that $c$ is increasing near $+\infty$.

We firstly recall the construction of $\psi$ and $\varphi$.

Recall that $\psi=\pi^{*}_1\big(\sum\limits_{j\ge1}2p_jG_{\Omega}(z,z_j)\big)+N$ is a plurisubharmonic function on $\tilde M$. We assume that $N\le 0$ is a plurisubharmonic function on $\tilde{M}$ and $N|_{Z_j}$ is not identically $-\infty$ for any $j$. $\varphi_1$ is a Lebesgue measurable function on $\Omega$ such that $\psi+\pi^{*}_1(\varphi)$ is a plurisubharmonic function on $\tilde M$. We also note that, by Siu's decomposition theorem and Weierstrass theorem on open Riemann surface, we have
$$\psi+\pi^{*}_1(\varphi_1)=\pi^{*}_1(2\log|g|)+\tilde{\psi}_2,$$
where $g$ is a holomorphic function on $\Omega$ and $\tilde{\psi}_2\in Psh(\tilde M)$. Denote $ord_{z_j}(g)=p_j$ and $d_j:=\lim_{z\rightarrow z_j}(g/\tilde{z}_j^{q_j})(z)$.

Following from Lemma \ref{decomp-tildeM}, for any $j\in\{1,2,\ldots,m\}$, we assume that
		\begin{equation*}
			F=\sum_{l=k_j}^{\infty}\pi_1^*(\tilde{z}_j^ld\tilde{z}_j)\wedge \pi_2^*(F_{j,l})
		\end{equation*}
		on $U_j\Subset (V_j\times Y)\cap \tilde{M}$ is a neighborhood of $Z_j:=\{z_j\}\times Y$ in $\tilde{M}$, and $k_j\in\mathbb{N}$, $F_{j,l}$ is a holomorphic $(n-1,0)$ form on $Y$ for any $l\geq k_j$, and $\tilde{F}_j:=F_{j,k_j}\not\equiv 0$.

	Let $W$ be an open subset of $Y$ such that $W\Subset Y$. Then for any $j$, $1\leq j<\gamma$, there exists $r_{j,W}>0$ such that $V_{j,W}\times W\subset\tilde{M}$, where $V_{j,W}:=\{z\in\Omega : |\tilde{z}_j(z)|<r_{j,W}\}$.

As $\tilde{G}(h^{-1}(r))$ is linear and $c$ is increasing near $+\infty$, using Lemma \ref{slope bigger than measure}, we can get that $k_j-q_j+1=0$ for any $j\in I_F$ and
		\begin{equation}\nonumber
			\frac{\int_{\tilde{M}}|F|^2e^{-\varphi}c(-\psi)}{\int_0^{+\infty}c(s)e^{-s}ds}\geq \sum_{j\in I_F}\frac{2\pi}{p_j|d_j|^2}\int_W|F_j|^2e^{-\varphi_2-\tilde{\psi}_2(z_j,w)},
		\end{equation}
where $I_F:=\{j: 1\leq j<\gamma \ \&\ \tilde{k}_j+1-k_j\leq 0\}$.

By the arbitrariness of $W$, we have $\int_Y|\tilde{F}_j|^2e^{-\varphi_2}<+\infty$ for any $j\in I_F$, and
	\begin{equation}\label{geqfinite tildeM}
		\frac{\int_{\tilde{M}}|F|^2e^{-\varphi}c(-\psi)}{\int_0^{+\infty}c(s)e^{-s}ds}\geq \sum_{j\in I_F}\frac{2\pi}{p_j|d_j|^2}\int_Y|F_j|^2e^{-\varphi_2-\tilde{\psi}_2(z_j,w)}.
	\end{equation}
 Especially, we know $\sum_{j\in I_F}\frac{2\pi}{p_j|d_j|^2}\int_Y|F_j|^2e^{-\varphi_2-\tilde{\psi}_2(z_j,w)}<+\infty$ and $\tilde{\psi}_2$ is not identically $-\infty$ on $Z_j$ for any $j\in I_F$.

 Let $V_j\times U_y$ be an open neighborhood of $(z_j,y)$ in $\tilde{M}$. Assume that $F=\tilde{z}_j^{k_j}h(\tilde{z}_j,w)d\tilde{z}_j\wedge dw$ on $V_j\times U_y$, where $h(z_j,w)dw=\tilde{F}_j$.  Note that $h(z_j,w)$ is not identically zero on $\{z_j\}\times U_y$, there must exist $\hat{w}\in U_y$ such that $h(z_j,\hat{w})\neq 0$. Then we know $|h(z_j,w)|^2$ has a positive lower bound on $\tilde{V}_j\times U_{\hat{w}}$, where $\tilde{V}_j\times U_{\hat{w}}\Subset V_j\times U_y$ is a small open neighborhood of $(z_j,\hat{w})$.
 Then, according to Lemma \ref{e-varphic-psi}, $\psi\le\sum_{j=1}^{\gamma}2p_jG_{\Omega}(\cdot,z_j)$ and $c(t)$ is increasing near $+\infty$, we know that for any $w\in U_{\hat{w}}$, we have $\int_{\tilde{V}_j}|\tilde{z}_j^{k_j}d\tilde{z}_j|^2e^{-\varphi_1}c(-\sum_{j=1}^{\gamma}2p_jG_{\Omega}(\cdot,z_j))<+\infty$.
It follows from Lemma \ref{f1zf2w} that we have $$\int_{\tilde{V}_j\times U_{\hat{w}}}|\pi_1^*(\tilde{z}_j^{k_j}d\tilde{z}_j)\wedge\pi_2^*(\tilde{F}_j)|^2
e^{-\varphi}c\big(-\pi_1^*(\sum_{j=1}^{\gamma}2p_jG_{\Omega}(\cdot,z_j))\big)<+\infty.$$

Then by Lemma \ref{k>k0}, we have
		\begin{equation*}
			\big(F-\pi_1^*(\tilde{z}_j^{k_j}d\tilde{z}_j)\wedge\pi_2^*(\tilde{F}_j),(z_j,y)\big)
\in\big(\mathcal{O}(K_M)\otimes\mathcal{I}(\pi^{*}_1(2\log |g|)+\pi^{*}_2(\varphi_2))\big)_{(z_j,y)}
		\end{equation*}
		for any $j\in I_F$, $y\in Y$. And according to Lemma \ref{k>k0} we also have
		\begin{equation*}
			\big(F,(z_j,y)\big)\in\big(\mathcal{O}(K_M)\otimes\mathcal{I}(\pi^{*}_1(2\log |g|)+\pi^{*}_2(\varphi_2))\big)_{(z_j,y)}
		\end{equation*}
		for any $j\in\{1,2,\ldots,m\}\setminus I_F$, $y\in Y$.

Now we prove that $N\equiv0$. Denote that $\tilde{\psi}:=\pi_1^*(2\sum_{j=1}^{\gamma}p_jG_{\Omega}(\cdot,z_j))$, $\tilde{\varphi}_1:=\pi_1^*(\varphi_1)+\psi-\tilde{\psi}$, and $\tilde{\varphi}:=\tilde{\varphi}_1+\pi_2^*(\varphi_2)$.
Then according to Lemma \ref{optimal extension} \big($F=\pi_1^*(\tilde{z}_j^{k_j}dw_j)\wedge\pi_2^*(\tilde{F}_j)$ on $U_j$ for $j\in I_F$ and $F\equiv0$ on $U_j$ for $j\notin I_F$ in Lemma \ref{optimal extension}\big), there exists a holomorphic $(n,0)$ form $\tilde{F}$ on $\tilde M$ such that $\big(\tilde{F}-\pi_1^*(\tilde{z}_j^{k_j}dw_j)\wedge\pi_2^*(\tilde{F}_j),(z_j,y)\big)
\in\big(\mathcal{O}(K_{\tilde{M}})\otimes\mathcal{I}(\pi^{*}_1(2\log |g|)+\pi^{*}_2(\varphi_2))\big)_{(z_j,y)}$ for any $j\in I_F$, $y\in Y$, $\big(\tilde{F},(z_j,y)\big)\in\big(\mathcal{O}(K_{\tilde{M}})\otimes\mathcal{I}(\pi^{*}_1(2\log |g|)+\pi^{*}_2(\varphi_2))\big)_{(z_j,y)}$ for any $j\ge 1 \& j\notin I_F$ and any $y\in Y$, and
		\begin{equation}\label{leqfinite}
			\int_{\tilde{M}}|\tilde{F}|^2e^{-\tilde \varphi}c(-\tilde \psi)\leq \left({\int_0^{+\infty}c(s)e^{-s}ds}\right)\sum_{j\in I_F}\frac{2\pi}{p_j|d_j|^2}\int_Y|F_j|^2e^{-\varphi_2-\tilde{\psi}_2(z_j,w)}.
		\end{equation}
		Note that we have $\big(\tilde{F}-F,(z_j,y)\big)\in\big(\mathcal{O}(K_{\tilde{M}})\otimes\mathcal{I}(\pi^{*}_1(2\log |g|)+\pi^{*}_2(\varphi_2))\big)_{(z_j,y)}$ for any $(z_j,y)\in Z_0$.

It follows from \eqref{geqfinite tildeM} and \eqref{leqfinite} that, we have
\begin{equation}\label{PROP 1.6-3}
\int_{\tilde{M}}|\tilde{F}|^2e^{-\varphi}c(-\psi)\ge\int_{\tilde{M}}|F|^2e^{-\varphi}c(-\psi)\ge \int_{\tilde{M}}|\tilde{F}|^2e^{-\tilde \varphi}c(-\tilde \psi).
\end{equation}
 As $c(t)e^{-t}$ is decreasing with respect to $t$ and $\psi\le \tilde{\psi}$, we have $\int_{\tilde M}|\tilde{F}|^2e^{-\tilde{\varphi}}c(-\tilde{\psi})\ge \int_{\tilde M}|\tilde{F}|^2e^{-\varphi}c(-\psi)$. Then all ``$\ge$" in \eqref{PROP 1.6-3} should be ``$=$". It follows from Lemma \ref{decreasing property of l} that $N=0$ and $\psi=\pi_1^*\big(2\sum_{j=1}^{\gamma}p_jG_{\Omega}(\cdot,z_j)\big)$.

  Using Lemma \ref{optimal extension} \big($\tilde{M}\sim M$ and $F=\pi_1^*(\tilde{z}_j^{k_j}dw_j)\wedge\pi_2^*(\tilde{F}_j)$ for $j\in I_F$ and $F\equiv0$ for $j\notin I_F$ in Lemma \ref{optimal extension}\big), there exists a holomorphic $(n,0)$ form $\hat{F}$ on $M$ such that $\big(\hat{F}-\pi_1^*(\tilde{z}_j^{k_j}dw_j)\wedge\pi_2^*(\tilde{F}_j),(z_j,y)\big)
  \in\big(\mathcal{O}(K_M)\otimes\mathcal{I}(\pi^{*}_1(2\log |g|)+\pi^{*}_2(\varphi_2))\big)_{(z_j,y)}$ for any $j\in I_F$, $y\in Y$, $\big(\hat{F},(z_j,y)\big)\in\big(\mathcal{O}(K_M)\otimes\mathcal{I}(\pi^{*}_1(2\log |g|)+\pi^{*}_2(\varphi_2))\big)_{(z_j,y)}$ for any $j\ge 1 \& j\notin I_F$ and any $y\in Y$, and
		\begin{equation}\label{leqfinite}
			\int_M|\hat{F}|^2e^{-\varphi}c(-\psi)\leq \left({\int_0^{+\infty}c(s)e^{-s}ds}\right)\sum_{j\in I_F}\frac{2\pi}{p_j|d_j|^2}\int_Y|F_j|^2e^{-\varphi_2-\tilde{\psi}_2(z_j,w)}.
		\end{equation}
		Then $\big(\hat{F}-F,(z_j,y)\big)\in\big(\mathcal{O}(K_M)\otimes\mathcal{I}(\pi^{*}_1(2\log |g|)+\pi^{*}_2(\varphi_2))\big)_{(z_j,y)}$ for any $(z_j,y)\in Z_0$.

Now According to the choice of $F$, we have that
	\begin{flalign}\label{tildeM-2}
		\begin{split}
			\tilde{G}(0)=&\int_{\tilde{M}}|F|^2e^{-\varphi}c(-\psi)\leq \int_{\tilde{M}}|\hat{F}|^2e^{-\varphi}c(-\psi)\\
			\leq&\int_M|\hat{F}|^2e^{-\varphi}c(-\psi)\\
			\leq&\left(\int_0^{+\infty}c(s)e^{-s}ds\right)\sum_{j\in I_F}\frac{2\pi}{p_j|d_j|^2}\int_Y|F_j|^2e^{-\varphi_2-\tilde{\psi}_2(z_j,w)}.
		\end{split}
	\end{flalign}	
	 Combining inequality (\ref{geqfinite tildeM}) with inequality (\ref{tildeM-2}), we get that
	\begin{equation*}
		\int_{\tilde{M}}|\hat{F}|^2e^{-\varphi}c(-\psi)=\int_M|\hat{F}|^2e^{-\varphi}c(-\psi).
	\end{equation*}
	As $\tilde{F}\not\equiv 0$, the above equality implies that $\tilde{M}=M$.

\end{proof}
\section{ Proofs of Theorem \ref{thm:linear-fibra-single}, Theorem \ref{thm:linear-fibra-finite}, Theorem \ref{thm:linear-fibra-infinite} and Proposition \ref{p:M=M_1}}
In this section, we prove Theorem \ref{thm:linear-fibra-single}, Theorem \ref{thm:linear-fibra-finite}, Theorem \ref{thm:linear-fibra-infinite} and Proposition \ref{p:M=M_1}.

\subsection{Proof of Theorem \ref{thm:linear-fibra-single}}
\begin{proof}
We firstly give the proof of the sufficiency of Theorem \ref{thm:linear-fibra-single}.
It follows from Theorem \ref{GBY6-single pt} that $G\big(h^{-1}(r);\mathcal{I}(\varphi+\psi)\big)$ is linear with respect to $r$.

 When $N\equiv 0$, then $\varphi_j$ is a subharmonic function on $\Omega_j$ satisfying $\varphi_j(z_j)>-\infty$ for any $1\le j\le n_1$,
it follows from Lemma \ref{l:phi1+phi2} that $\mathcal{I}(\varphi+\psi)_{(z_j,y)}=\mathcal{I}\big(\hat G+\pi_2^*(\varphi_Y)\big)_{(z_j,y)}$ for any $(z_j,y)\in Z_0$.

Hence $G\big(h^{-1}(r);\mathcal{I}(\hat{G}+\pi_2^*(\varphi_Y))\big)$ is linear with respect to $r$. The sufficiency of Theorem \ref{thm:linear-fibra-single} is proved.

We prove the necessity part of Theorem \ref{thm:linear-fibra-single}.

Assume that $G\big(h^{-1}(r)\big)$ is linear with respect to $r\in (0,\int_0^{+\infty}c(s)e^{-s}ds]$. Then according to Lemma \ref{linear}, there exists a unique holomorphic $(n,0)$ form $F$ on $M$ satisfying $(F-f)\in H^0(Z_0,\big(\mathcal{O}(K_M)\otimes\mathcal{I}(\hat G+\pi_2^*(\varphi_Y)\big)|_{Z_0})$, and $G(t;c)=\int_{\{\psi<-t\}}|F|^2e^{-\varphi}c(-\psi)$ for any $t\geq 0$. Then according to Lemma \ref{linear}, Remark \ref{ctildec} and Lemma \ref{tildecincreasing}, we can assume that $c$ is increasing near $+\infty$.

It follows from Lemma \ref{decomp} that there exists a sequence of  holomorphic  $(n_2,0)$  forms $\{F_{\alpha}\}_{\alpha\in\mathbb{Z}_{\ge0}^{n_1}}$ on $Y$ such that
$$F=\sum_{\alpha\in\mathbb{Z}_{\ge0}^{n_1}}\pi_1^*(w^{\alpha}dw_{1}\wedge\ldots\wedge dw_{n_1})\wedge\pi_2^*(F_{\alpha})$$
on $V_{0}\times Y$.

Denote that $E_0:=\left\{\alpha\in\mathbb{Z}_{\ge0}^{n_1}:\sum_{1\le j\le n_1}\frac{\alpha_j+1}{p_{j}}=1\right\}$, $E_{1}:=\left\{\alpha\in\mathbb{Z}_{\ge0}^{n_1}:\sum_{1\le j\le n_1}\frac{\alpha_j+1}{p_{j}}<1\right\}$ and $E_{2}:=\left\{\alpha\in\mathbb{Z}_{\ge0}^{n_1}:\sum_{1\le j\le n_1}\frac{\alpha_j+1}{p_{j}}>1\right\}$.

Now we prove $N\equiv 0$. It follows from Lemma \ref{l:limit} that
we have $F_{\alpha}\equiv0$ for any $\alpha\in E_{1}$ , and
	\begin{equation}
\label{multi single 1}
\liminf_{t\rightarrow+\infty}\frac{\int_{\{\psi<-t\}}|F|^2e^{-\varphi}c(-\psi)}
{\int_t^{+\infty}c(s)e^{-s}ds}\ge\sum_{\alpha\in E_{0}}\frac{(2\pi)^{n_1}e^{-\sum_{1\le j\le n_1}\varphi_j(z_{j})}}{\prod_{1\le j\le n_1}(\alpha_j+1)c_{j}^{2\alpha_{j}+2}(z_j)}\int_{Y}|F_{\alpha,\beta}|^2e^{-\varphi_Y-N(z_0,w')}.
\end{equation}

Note that $\psi=\hat G+N\le \hat G$, $c(t)$ is increasing near $+\infty$, $\varphi_X$ is upper semi-continuous on $\prod_{j=1}^{n_1}\Omega_j$. When $t$ is large enough, we have
	\begin{equation}\nonumber
\begin{split}
&\int_{\{\psi<-t\}}|F|^2e^{-\varphi}c(-\psi)\ge\int_{\{\hat G<-t\}}|F|^2e^{-\pi_1^*\big(-\sum_{1\le j\le n_1}\varphi_j(z_{j})\big)-
\pi_2^*(\varphi_Y)}c(-\hat G)\\
\ge &
C_0\int_{\{\hat G<-t\}}|F|^2e^{-
\pi_2^*(\varphi_Y)}c(-\hat G),
\end{split}
\end{equation}
where $C_0>0$ is a constant, then it follows from $\int_{\{\psi<-t\}}|F|^2e^{-\varphi}c(-\psi)<+\infty$ and Lemma \ref{decomp} that for any $\alpha\in E_{2}:=\left\{\alpha\in\mathbb{Z}_{\ge0}^{n_1}:\sum_{1\le j\le n_1}\frac{\alpha_j+1}{p_{j}}>1\right\}$, we know $F_{\alpha}\in\mathcal{I}(\varphi_Y)$ for any $y\in Y$.

Recall that
$$\hat G=\max_{1\le j\le n_1}\left\{2p_j\pi_{1,j}^{*}\big(G_{\Omega_j}(\cdot,z_j)\big)\right\}.$$
Denote $\tilde{\varphi}:=\varphi+\psi-\hat G$. It follows from Lemma \ref{p:exten-fibra} that there exists a holomorphic $(n,0)$ form $\hat{F}$ on $M$ satisfying that $(\hat{F}-F,z)\in\big(\mathcal{O}(K_M)\otimes\mathcal{I}(\hat G+\pi_2^*(\varphi_Y))\big)_{z}$ for any $z\in Z_0$ and
\begin{equation}
\label{multi single 2}
	\begin{split}
	&\int_{M}|\hat{F}|^2e^{-\tilde{\varphi}}c(-\hat G)\\
\le&\left(\int_0^{+\infty}c(s)e^{-s}ds\right)\sum_{\alpha\in E_{0}}\frac{(2\pi)^{n_1}e^{-\sum_{1\le j\le n_1}\varphi_j(z_{j})}
\int_Y|f_{\alpha}|^2e^{-\varphi_{Y}-N(z_0,w')}}{\prod_{1\le j\le n_1}(\alpha_j+1)c_{j}^{2\alpha_{j}+2}(z_j)}.	
	\end{split}
\end{equation}
It follows from \eqref{multi single 1} and \eqref{multi single 2} that we have
\begin{equation}\label{multi single 3}
	\begin{split}
	\int_{M}|\hat{F}|^2e^{-\tilde{\varphi}}c(-\hat G)\le\int_M |F|^2e^{-\varphi}c(-\psi)
\le\int_M |\hat{F}|^2e^{-\varphi}c(-\psi)
	\end{split}
\end{equation}
As $c(t)e^{-t}$ is decreasing with respect to $t$ and $\psi\le \hat G$, we have $\int_{M}|\hat{F}|^2e^{-\tilde{\varphi}}c(-\hat G)\ge\int_M |\hat{F}|^2e^{-\varphi}c(-\psi)$. Then all ``$\ge$" in \eqref{multi single 3} should be ``$=$". It follows from Lemma \ref{decreasing property of l} that we know $N\equiv 0$ and $\psi=\max_{1\le j\le n_1}\{2p_j\pi_{1,j}^{*}\big(G_{\Omega_j}(\cdot,z_j)\big)\}$.

As $N\equiv 0$, we know $\varphi_j$ is a subharmonic function on $\Omega_j$ satisfying $\varphi_j(z_j)>-\infty$ for any $1\le j\le n_1$ and
it follows from Lemma \ref{l:phi1+phi2} that $\mathcal{I}(\varphi+\psi)_{(z_j,y)}=\mathcal{I}\big(\hat G+\pi_2^*(\varphi_Y)\big)_{(z_j,y)}$ for any $(z_j,y)\in Z_0$. Note that $\psi=\max_{1\le j\le n_1}\{2p_j\pi_{1,j}^{*}\big(G_{\Omega_j}(\cdot,z_j)\big)\}$.
It follows from  Theorem \ref{GBY6-single pt} that the necessity part of Theorem \ref{thm:linear-fibra-single} holds.
\end{proof}

\subsection{Proof of Theorem \ref{thm:linear-fibra-finite}}
\begin{proof}
We firstly give the proof of the sufficiency of Theorem \ref{thm:linear-fibra-finite}.
It follows from Theorem \ref{GBY6-finitie pts} that $G\big(h^{-1}(r);\mathcal{I}(\varphi+\psi)\big)$ is linear with respect to $r$.

 When $N\equiv 0$, then $\varphi_j$ is a subharmonic function on $\Omega_j$ satisfying $\varphi_j(z_j)>-\infty$ for any $1\le j\le n_1$ and
it follows from Lemma \ref{l:phi1+phi2} that $\mathcal{I}(\varphi+\psi)_{(z_j,y)}=\mathcal{I}\big(\hat G+\pi_2^*(\varphi_Y)\big)_{(z_j,y)}$ for any $(z_j,y)\in Z_0$.

Hence $G\big(h^{-1}(r);\mathcal{I}(\hat G+\pi_2^*(\varphi_Y))\big)$ is linear with respect to $r$.  The sufficiency of Theorem \ref{thm:linear-fibra-finite} is proved.

We prove the necessity part of Theorem \ref{thm:linear-fibra-finite}.

Assume that $G\big(h^{-1}(r)\big)$ is linear with respect to $r\in (0,\int_0^{+\infty}c(s)e^{-s}ds]$. Then according to Lemma \ref{linear}, there exists a unique holomorphic $(n,0)$ form $F$ on $M$ satisfying $(F-f)\in H^0(Z_0,\big(\mathcal{O}(K_M)\otimes\mathcal{I}(\hat G+\pi_2^*(\varphi_Y)\big)|_{Z_0})$, and $G(t;c)=\int_{\{\psi<-t\}}|F|^2e^{-\varphi}c(-\psi)$ for any $t\geq 0$. Then according to Lemma \ref{linear}, Remark \ref{ctildec} and Lemma \ref{tildecincreasing}, we can assume that $c$ is increasing near $+\infty$.

It follows from Lemma \ref{decomp} that there exists a sequence of  holomorphic  $(n_2,0)$  forms $\{F_{\alpha,\beta}\}_{\alpha\in\mathbb{Z}_{\ge0}^{n_1}}$ on $Y$ such that
$$F=\sum_{\alpha\in\mathbb{Z}_{\ge0}^{n_1}}\pi_1^*(w^{\alpha}_{\beta}dw_{1,\beta_1}\wedge\ldots\wedge dw_{n_1,\beta_{n_1}})\wedge\pi_2^*(F_{\alpha,\beta})$$
on $V_{\beta}\times Y$, for any $\beta\in\tilde{I}_1$. Denote that $E_{\beta}:=\left\{\alpha\in\mathbb{Z}_{\ge0}^{n_1}:\sum_{1\le j\le n_1}\frac{\alpha_j+1}{p_{j,\beta_j}}=1\right\}$, $E_{\beta,1}:=\left\{\alpha\in\mathbb{Z}_{\ge0}^{n_1}:\sum_{1\le j\le n_1}\frac{\alpha_j+1}{p_{j,\beta_j}}<1\right\}$ and $E_{\beta,2}:=\left\{\alpha\in\mathbb{Z}_{\ge0}^{n_1}:\sum_{1\le j\le n_1}\frac{\alpha_j+1}{p_{j,\beta_j}}>1\right\}$.

Now we prove $N\equiv 0$. It follows from Lemma \ref{l:limit} that we have
we have $F_{\alpha}\equiv0$ for any $\alpha\in E_{\beta,1}$, and
	\begin{equation}
\label{multi finite 1}
\liminf_{t\rightarrow+\infty}\frac{\int_{\{\psi<-t\}}|F|^2e^{-\varphi}c(-\psi)}
{\int_t^{+\infty}c(s)e^{-s}ds}\ge\sum_{\beta\in\tilde I_1}\sum_{\alpha\in E_{\beta}}\frac{(2\pi)^{n_1}e^{-\sum_{1\le j\le n_1}\varphi_j(z_{j,\beta_j})}}{\prod_{1\le j\le n_1}(\alpha_j+1)c_{j,\beta_j}^{2\alpha_{j}+2}}\int_{Y}|F_{\alpha,\beta}|^2e^{-\varphi_Y-N(z_{\beta},w')}.
\end{equation}

Note that $\psi=\hat G+N\le \hat G$, $c(t)$ is increasing near $+\infty$, $\varphi_X$ is upper semi-continuous on $\prod_{j=1}^{n_1}\Omega_j$. When $t$ is large enough, we have
	\begin{equation}\nonumber
\begin{split}
&\int_{\{\psi<-t\}}|F|^2e^{-\varphi}c(-\psi)\ge\int_{\{\hat G<-t\}}|F|^2e^{-\pi_1^*\big(\sum_{1\le j\le n_1}\pi_{1,j}^*(\varphi_j)\big)-
\pi_2^*(\varphi_Y)}c(-\hat G)\\
\ge &
C_0\int_{\{\hat G<-t\}}|F|^2e^{-
\pi_2^*(\varphi_Y)}c(-\hat G),
\end{split}
\end{equation}
then it follows from $\int_{\{\psi<-t\}}|F|^2e^{-\varphi}c(-\psi)<+\infty$ and Lemma \ref{decomp} that for any $\alpha\in E_{\beta,2}:=\left\{\alpha\in\mathbb{Z}_{\ge0}^{n_1}:\sum_{1\le j\le n_1}\frac{\alpha_j+1}{p_{j,\beta_j}}>1\right\}$ and any $\beta\in \mathcal{I}_1$, we know $F_{\alpha,\beta}\in\mathcal{I}(\varphi_Y)$ for any $y\in Y$.

Recall that
$$\hat{G}:=\max_{1\le j\le n_1}\left\{\pi_{1,j}^*\left(2\sum_{1\le k\le m_j}p_{j,k}G_{\Omega_j}(\cdot,z_{j,k})\right)\right\},$$
Denote $\tilde{\varphi}:=\varphi+\psi-\hat G$. It follows from Lemma \ref{p:exten-fibra} that there exists a holomorphic $(n,0)$ form $\hat{F}$ on $M$ satisfying that $(\hat{F}-F,z)\in\big(\mathcal{O}(K_M)\otimes\mathcal{I}(\hat G+\pi_2^*(\varphi_Y))\big)_{z}$ for any $z\in Z_0$ and
\begin{equation}
\label{multi finite 2}
	\begin{split}
	&\int_{M}|\hat{F}|^2e^{-\tilde{\varphi}}c(-\hat G)\\
\le&\left(\int_0^{+\infty}c(s)e^{-s}ds\right)\sum_{\beta\in\tilde I_1}\sum_{\alpha\in E_{\beta}}\frac{(2\pi)^{n_1}e^{-\sum_{1\le j\le n_1}\varphi_j(z_{j,\beta_j})}
\int_Y|f_{\alpha,\beta}|^2e^{-\varphi_{Y}-N(z_{\beta},w')}}{\prod_{1\le j\le n_1}(\alpha_j+1)c_{j,\beta_j}^{2\alpha_{j}+2}}.	
	\end{split}
\end{equation}
It follows from \eqref{multi finite 1} and \eqref{multi finite 2} that we have
\begin{equation}\label{multi finite 3}
	\begin{split}
	\int_{M}|\hat{F}|^2e^{-\tilde{\varphi}}c(-\hat G)\le\int_M |F|^2e^{-\varphi}c(-\psi)
\le\int_M |\hat{F}|^2e^{-\varphi}c(-\psi).
	\end{split}
\end{equation}
As $c(t)e^{-t}$ is decreasing with respect to $t$ and $\psi\le \hat G$, we have $\int_{M}|\hat{F}|^2e^{-\tilde{\varphi}}c(-\hat G)\ge\int_M |\hat{F}|^2e^{-\varphi}c(-\psi)$. Then all ``$\ge$" in \eqref{multi finite 3} should be ``$=$". It follows from Lemma \ref{decreasing property of l} that we know $N\equiv 0$ and $\psi=\max_{1\le j\le n_1}\left\{\pi_{1,j}^*\left(2\sum_{1\le k\le m_j}p_{j,k}G_{\Omega_j}(\cdot,z_{j,k})\right)\right\}$.

As $N\equiv 0$, we know $\varphi_j$ is a subharmonic function on $\Omega_j$ satisfying $\varphi_j(z_j)>-\infty$ for any $1\le j\le n_1$ and
it follows from Lemma \ref{l:phi1+phi2} that $\mathcal{I}(\varphi+\psi)_{(z_j,y)}=\mathcal{I}\big(\hat G+\pi_2^*(\varphi_Y)\big)_{(z_j,y)}$ for any $(z_j,y)\in Z_0$. Note that $\psi=\max_{1\le j\le n_1}\left\{\pi_{1,j}^*\left(2\sum_{1\le k\le m_j}p_{j,k}G_{\Omega_j}(\cdot,z_{j,k})\right)\right\}$. It follows from Theorem \ref{GBY6-finitie pts} that the necessity part of Theorem \ref{thm:linear-fibra-finite} holds.
\end{proof}

\subsection{Proof of Theorem \ref{thm:linear-fibra-infinite}}
\begin{proof}

We prove Theorem \ref{thm:linear-fibra-infinite} by contradiction.

Assume that $G\big(h^{-1}(r)\big)$ is linear with respect to $r\in (0,\int_0^{+\infty}c(s)e^{-s}ds]$. Then according to Lemma \ref{linear}, there exists a unique holomorphic $(n,0)$ form $F$ on $M$ satisfying $(F-f)\in H^0(Z_0,\big(\mathcal{O}(K_M)\otimes\mathcal{I}(\hat G+\pi_2^*(\varphi_Y)\big)|_{Z_0})$, and $G(t;c)=\int_{\{\psi<-t\}}|F|^2e^{-\varphi}c(-\psi)$ for any $t\geq 0$. Then according to Lemma \ref{linear}, Remark \ref{ctildec} and Lemma \ref{tildecincreasing}, we can assume that $c$ is increasing near $+\infty$.

It follows from Lemma \ref{decomp} that there exists a sequence of  holomorphic  $(n_2,0)$  forms $\{F_{\alpha,\beta}\}_{\alpha\in\mathbb{Z}_{\ge0}^{n_1}}$ on $Y$ such that
$$F=\sum_{\alpha\in\mathbb{Z}_{\ge0}^{n_1}}\pi_1^*(w^{\alpha}_{\beta}dw_{1,\beta_1}\wedge\ldots\wedge dw_{n_1,\beta_{n_1}})\wedge\pi_2^*(F_{\alpha,\beta})$$
on $V_{\beta}\times Y$ for any $\beta\in\tilde{I}_1$. Denote that $E_{\beta}:=\left\{\alpha\in\mathbb{Z}_{\ge0}^{n_1}:\sum_{1\le j\le n_1}\frac{\alpha_j+1}{p_{j,\beta_j}}=1\right\}$, $E_{\beta,1}:=\left\{\alpha\in\mathbb{Z}_{\ge0}^{n_1}:\sum_{1\le j\le n_1}\frac{\alpha_j+1}{p_{j,\beta_j}}<1\right\}$ and $E_{\beta,2}:=\left\{\alpha\in\mathbb{Z}_{\ge0}^{n_1}:\sum_{1\le j\le n_1}\frac{\alpha_j+1}{p_{j,\beta_j}}>1\right\}$.

Now we prove $N\equiv 0$. It follows from Lemma \ref{l:limit} that we have
we have $F_{\alpha}\equiv0$ for any $\alpha\in E_{\beta,1}$, and
	\begin{equation}
\label{multi infinite 1}
\liminf_{t\rightarrow+\infty}\frac{\int_{\{\psi<-t\}}|F|^2e^{-\varphi}c(-\psi)}
{\int_t^{+\infty}c(s)e^{-s}ds}\ge\sum_{\beta\in\tilde I_1}\sum_{\alpha\in E_{\beta}}\frac{(2\pi)^{n_1}e^{-\sum_{1\le j\le n_1}\varphi_j(z_{j,\beta_j})}}{\prod_{1\le j\le n_1}(\alpha_j+1)c_{j,\beta_j}^{2\alpha_{j}+2}}\int_{Y}|F_{\alpha,\beta}|^2e^{-\varphi_Y-N(z_{\beta},w')}.
\end{equation}

Note that $\psi=\hat G+N\le\hat  G$, $c(t)$ is increasing near $+\infty$, $\sum_{1\le j\le n_1}\pi_{1,j}^*(\varphi_j)$  is upper semi-continuous on $\prod_{j=1}^{n_1}\Omega_j$. When $t$ is large enough, we have
	\begin{equation}\nonumber
\begin{split}
&\int_{\{\psi<-t\}}|F|^2e^{-\varphi}c(-\psi)\ge\int_{\{G<-t\}}|F|^2e^{-\sum_{1\le j\le n_1}\pi_{1,j}^*(\varphi_j) -
\pi_2^*(\varphi_Y)}c(-\hat G)\\
\ge &
C_0\int_{\{\hat G<-t\}}|F|^2e^{-
\pi_2^*(\varphi_Y)}c(-\hat G),
\end{split}
\end{equation}
where $C_0$ is a constant.
Then it follows from $\int_{\{\psi<-t\}}|F|^2e^{-\varphi}c(-\psi)<+\infty$ and Lemma \ref{decomp} that for any $\alpha\in E_{\beta,2}:=\left\{\alpha\in\mathbb{Z}_{\ge0}^{n_1}:\sum_{1\le j\le n_1}\frac{\alpha_j+1}{p_{j,\beta_j}}>1\right\}$ and any $\beta\in \mathcal{I}_1$, we know $F_{\alpha,\beta}\in\mathcal{I}(\varphi_Y)$ for any $y\in Y$.

Recall that
$$\hat{G}:=\max_{1\le j\le n_1}\left\{\pi_{1,j}^*\left(2\sum_{1\le k\le m_j}p_{j,k}G_{\Omega_j}(\cdot,z_{j,k})\right)\right\}.$$
Denote $\tilde{\varphi}:=\varphi+\psi-\hat G$. It follows from Lemma \ref{p:exten-fibra} that there exists a holomorphic $(n,0)$ form $\hat{F}$ on $M$ satisfying that $(\hat{F}-F,z)\in\big(\mathcal{O}(K_M)\otimes\mathcal{I}(\hat G+\pi_2^*(\varphi_Y))\big)_{z}$ for any $z\in Z_0$ and
\begin{equation}
\label{multi infinite 2}
	\begin{split}
	&\int_{M}|\hat{F}|^2e^{-\tilde{\varphi}}c(-\hat G)\\
\le&\left(\int_0^{+\infty}c(s)e^{-s}ds\right)\sum_{\beta\in\tilde I_1}\sum_{\alpha\in E_{\beta}}\frac{(2\pi)^{n_1}e^{-\sum_{1\le j\le n_1}\varphi_j(z_{j,\beta_j})}
\int_Y|F_{\alpha,\beta}|^2e^{-\varphi_{Y}-N(z_{\beta},w')}}{\prod_{1\le j\le n_1}(\alpha_j+1)c_{j,\beta_j}^{2\alpha_{j}+2}}.	
	\end{split}
\end{equation}
It follows from \eqref{multi infinite 1} and \eqref{multi infinite 2} that we have
\begin{equation}\label{multi infinite 3}
	\begin{split}
	\int_{M}|\hat{F}|^2e^{-\tilde{\varphi}}c(-\hat G)\le\int_M |F|^2e^{-\varphi}c(-\psi)
\le\int_M |\hat{F}|^2e^{-\varphi}c(-\psi).
	\end{split}
\end{equation}
As $c(t)e^{-t}$ is decreasing with respect to $t$ and $\psi\le \hat G$, we have $\int_{M}|\hat{F}|^2e^{-\tilde{\varphi}}c(-\hat G)\ge\int_M |\hat{F}|^2e^{-\varphi}c(-\psi)$. Then all ``$\ge$" in \eqref{multi infinite 3} should be ``$=$". It follows from Lemma \ref{decreasing property of l} that we know $N\equiv 0$ and $\psi=\max_{1\le j\le n_1}\left\{\pi_{1,j}^*\left(2\sum_{1\le k< \tilde{m}_j}p_{j,k}G_{\Omega_j}(\cdot,z_{j,k})\right)\right\}$.

As $N\equiv 0$, we know $\varphi_j$ is a subharmonic function on $\Omega_j$ satisfying $\varphi_j(z_j)>-\infty$ for any $1\le j\le n_1$ and
it follows from Lemma \ref{l:phi1+phi2} that $\mathcal{I}(\varphi+\psi)_{(z_j,y)}=\mathcal{I}\big(\hat G+\pi_2^*(\varphi_Y)\big)_{(z_j,y)}$ for any $(z_j,y)\in Z_0$. Note that $\psi=\max_{1\le j\le n_1}\left\{\pi_{1,j}^*\left(2\sum_{1\le k< \tilde{m}_j}p_{j,k}G_{\Omega_j}(\cdot,z_{j,k})\right)\right\}$. It follows from Theorem \ref{GBY6-infinite points} that we get a contradiction.
\end{proof}

\subsection{Proof of proposition \ref{p:M=M_1}}

\begin{proof}
As $\tilde{G}(h^{-1}(r))$ is linear with respect to $r\in (0,\int_0^{+\infty}c(s)e^{-s}ds]$, then according to Lemma \ref{linear}, there exists a unique holomorphic $(n,0)$ form $F$ on $\tilde{M}$ satisfying $(F-f)\in H^0(Z_0,\big(\mathcal{O}(K_{\tilde{M}})\otimes\mathcal{I}(\hat G+\pi_2^*(\varphi_Y)\big)|_{Z_0})$, and $G(t)=\int_{\{\psi<-t\}\cap \tilde{M}}|F|^2e^{-\varphi}c(-\psi)$ for any $t\geq 0$. Then according to Lemma \ref{linear}, Remark \ref{ctildec} and Lemma \ref{tildecincreasing}, we can assume that $c$ is increasing near $+\infty$.

It follows from Lemma \ref{decomp-tildeM} that there exists a sequence of  holomorphic  $(n_2,0)$  forms $\{F_{\alpha,\beta}\}_{\alpha\in\mathbb{Z}_{\ge0}^{n_1}}$ on $Y$ such that
$$F=\sum_{\alpha\in\mathbb{Z}_{\ge0}^{n_1}}\pi_1^*(w^{\alpha}_{\beta}dw_{1,\beta_1}\wedge\ldots\wedge dw_{n_1,\beta_{n_1}})\wedge\pi_2^*(F_{\alpha,\beta})$$
on a open neighborhood $U_{\beta}\subset (V_{\beta}\times Y)\cap \tilde{M}$ of $\{z_{\beta}\}\times Y$ for any $\beta\in\tilde{I}_1$. Denote that $E_{\beta}:=\left\{\alpha\in\mathbb{Z}_{\ge0}^{n_1}:\sum_{1\le j\le n_1}\frac{\alpha_j+1}{p_{j,\beta_j}}=1\right\}$, $E_{\beta,1}:=\left\{\alpha\in\mathbb{Z}_{\ge0}^{n_1}:\sum_{1\le j\le n_1}\frac{\alpha_j+1}{p_{j,\beta_j}}<1\right\}$ and $E_{\beta,2}:=\left\{\alpha\in\mathbb{Z}_{\ge0}^{n_1}:\sum_{1\le j\le n_1}\frac{\alpha_j+1}{p_{j,\beta_j}}>1\right\}$.

 It follows from Remark \ref{l:limit2} that we have
we have $F_{\alpha,\beta}\equiv0$ for any $\alpha\in E_{\beta,1}$, and
	\begin{equation}
\label{multi m1=m 1}
\frac{G(0)}
{\int_t^{+\infty}c(s)e^{-s}ds}\ge\sum_{\beta\in\tilde I_1}\sum_{\alpha\in E_{\beta}}\frac{(2\pi)^{n_1}e^{-\sum_{1\le j\le n_1}\varphi_j(z_{j,\beta_j})}}{\prod_{1\le j\le n_1}(\alpha_j+1)c_{j,\beta_j}^{2\alpha_{j}+2}}\int_{Y}|F_{\alpha,\beta}|^2e^{-\varphi_Y-N(z_{\beta},w')}.
\end{equation}

Note that $\psi=\hat G+N\le \hat G$, $c(t)$ is increasing near $+\infty$, $\sum_{1\le j\le n_1}\pi_{1,j}^*(\varphi_j)$ is upper semi-continuous on $\prod_{j=1}^{n_1}\Omega_j$. When $t$ is large enough, we have
	\begin{equation}\nonumber
\begin{split}
&\int_{\{\psi<-t\}\cap\tilde{M}}|F|^2e^{-\varphi}c(-\psi)\ge\int_{\{\hat G<-t\}\cap\tilde{M}}|F|^2e^{-\sum_{1\le j\le n_1}\pi_{1,j}^*(\varphi_j)-
\pi_2^*(\varphi_Y)}c(-\hat G)\\
\ge &
C_0\int_{\{\hat G<-t\}\cap\tilde{M}}|F|^2e^{-
\pi_2^*(\varphi_Y)}c(-\hat G),
\end{split}
\end{equation}
then it follows from $\int_{\{\psi<-t\}}|F|^2e^{-\varphi}c(-\psi)<+\infty$ and Lemma \ref{decomp-tildeM} that for any $\alpha\in E_{\beta,2}:=\left\{\alpha\in\mathbb{Z}_{\ge0}^{n_1}:\sum_{1\le j\le n_1}\frac{\alpha_j+1}{p_{j,\beta_j}}>1\right\}$ and any $\beta\in \mathcal{I}_1$, we know $F_{\alpha,\beta}\in\mathcal{I}(\varphi_Y)$ for any $y\in Y$.

Recall that $\hat G:=\max_{1\le j\le n_1}\left\{2\sum_{1\le k<\tilde m_j}p_{j,k}\pi_{1,j}^{*}(G_{\Omega_j}(\cdot,z_{j,k}))\right\}$ and $\tilde{\varphi}=\varphi+N$. It follows from Lemma \ref{p:exten-fibra} that there exists a holomorphic $(n,0)$ form $\tilde{F}$ on $\tilde{M}$ satisfying that $(\tilde{F}-F,z)\in(\mathcal{O}(K_{\tilde{M}})\otimes\mathcal{I}(\hat{G}+\pi_2^*(\varphi_Y)))_{z}$ for any $z\in Z_0$ and
\begin{equation}
\label{multi m1=m 2}
	\begin{split}
	&\int_{\tilde M}|\tilde{F}|^2e^{-\tilde\varphi}c(-\hat G)\\
\le&\left(\int_0^{+\infty}c(s)e^{-s}ds\right)\sum_{\beta\in\tilde I_1}\sum_{\alpha\in E_{\beta}}\frac{(2\pi)^{n_1}e^{-\sum_{1\le j\le n_1}\varphi_j(z_{j,\beta_j})}
\int_Y|f_{\alpha,\beta}|^2e^{-\varphi_{Y}-N(z_{\beta},w')}}{\prod_{1\le j\le n_1}(\alpha_j+1)c_{j,\beta_j}^{2\alpha_{j}+2}}.	
	\end{split}
\end{equation}
Combining with \eqref{multi m1=m 1} and \eqref{multi m1=m 2}, we have
\begin{equation}
\label{multi m1=m 3}\int_{\tilde M}|\tilde F|^2e^{-\varphi}c(-\psi)
\ge\int_{\tilde M}|F|^2e^{-\varphi}c(-\psi)\ge\int_{\tilde M}|\tilde{F}|^2e^{-\tilde\varphi}c(-\hat G).
\end{equation}
As $c(t)e^{-t}$ is decreasing with respect to $t$ and $\psi\le \hat G$, we have $\int_{\tilde M}|\tilde{F}|^2e^{-\tilde{\varphi}}c(-\hat G)\ge \int_{\tilde M}|\tilde{F}|^2e^{-\varphi}c(-\psi)$. Then all ``$\ge$" in \eqref{multi m1=m 3} should be ``$=$". It follows from Lemma \ref{decreasing property of l} that $N=0$ and $\psi=\max_{1\le j\le n_1}\left\{\pi_{1,j}^*\left(2\sum_{1\le k\le m_j}p_{j,k}G_{\Omega_j}(\cdot,z_{j,k})\right)\right\}$.

It follows from Lemma \ref{p:exten-fibra} ($\tilde{M}\sim M$) that there exists a holomorphic $(n,0)$ form $\hat{F}$ on $M$ satisfying that $(\hat{F}-F,z)\in\big(\mathcal{O}(K_M)\otimes\mathcal{I}(\hat G+\pi_2^*(\varphi_Y))\big)_{z}$ for any $z\in Z_0$ and
\begin{equation}
	\begin{split}
\label{multi m1=m 4}
	&\int_{M}|\hat{F}|^2e^{-\varphi}c(-\psi)\\
\le&\left(\int_0^{+\infty}c(s)e^{-s}ds\right)\sum_{\beta\in\tilde I_1}\sum_{\alpha\in E_{\beta}}\frac{(2\pi)^{n_1}e^{-\sum_{1\le j\le n_1}\varphi_j(z_{j,\beta_j})}
\int_Y|f_{\alpha,\beta}|^2e^{-\varphi_{Y}-N(z_{\beta},w')}}{\prod_{1\le j\le n_1}(\alpha_j+1)c_{j,\beta_j}^{2\alpha_{j}+2}}.	
	\end{split}
\end{equation}
 We also have
	\begin{equation}
\label{multi m1=m 5}
\frac{\tilde G(0)}
{\int_t^{+\infty}c(s)e^{-s}ds}=\frac{\int_{ \tilde{M}}|F|^2e^{-\varphi}c(-\psi)}
{\int_t^{+\infty}c(s)e^{-s}ds}
\le\frac{\int_{\tilde{M}}|\hat{F}|^2e^{-\varphi}c(-\psi)}
{\int_t^{+\infty}c(s)e^{-s}ds}.
\end{equation}
Combining \eqref{multi m1=m 1}, \eqref{multi m1=m 4} and \eqref{multi m1=m 5}, we have
$$\int_{ \tilde{M}}|\hat{F}|^2e^{-\varphi}c(-\psi)=\int_{ M}|\hat{F}|^2e^{-\varphi}c(-\psi).$$
As $\hat{F}\neq 0$, the above equality shows that $M=\tilde{M}$.

\end{proof}

\section{ Proofs of Theorem \ref{app-L2 equality} and Theorem \ref{app-L2 inequality}}
In this section, we prove Theorem \ref{app-L2 equality} and Theorem \ref{app-L2 inequality}.

\subsection{Proof of Theorem \ref{app-L2 equality}}
\begin{proof}

It follows from Lemma \ref{optimal extension} that we know there exists a holomorphic $(n,0)$ form $F$ on $\tilde M$ such that $(F-f,(z_j,y))\in\big(\mathcal{O}(K_{\tilde{M}})\otimes\mathcal{I}(\pi_1^*(2\log |g|)+\pi_2^*(\varphi_2))\big)_{(z_j,y)}$ for any $(z_j,y)\in Z_0$ and

	\begin{equation}
			\int_{\tilde{M}}|F|^2e^{-\varphi}c(-\psi)\leq
\left(\int_0^{+\infty}c(s)e^{-s}ds\right)
 \sum_{j=1}^{m}\frac{2\pi|a_j|^2}{p_j|d_j|^2}\int_Y|F_j|^2e^{-\varphi_2-\tilde{\psi}_2(z_j,w)}.
		\end{equation}

In the following, we prove the characterization of the holding of the equality $\left(\int_0^{+\infty}c(s)e^{-s}ds\right)
 \sum_{j=1}^{m}\frac{2\pi|a_j|^2}{p_j|d_j|^2}\int_Y|F_j|^2e^{-\varphi_2-\tilde{\psi}_2(z_j,w)}=\inf\big\{$ $ \int_{\tilde{M}}|\tilde{F}|^2e^{-\varphi}c(-\psi):\tilde{F}$ is a holomorphic $(n,0)$ form on $\tilde{M}$ such that $\big(\tilde{F}-f,(z_j,y)\big)\in\big(\mathcal{O}(K_{\tilde{M}})\otimes\mathcal{I}(\pi_1^*(2\log |g|)+\pi_2^*(\varphi_2))\big)_{(z_j,y)}$ for any $(z_j,y)\in Z_0\big\}$.

We firstly give the proof of the sufficiency in Theorem \ref{app-L2 equality}.

As $\psi+\pi_1^*(\varphi_1)=\pi_1^*(2\sum_{j=1}^mp_jG_{\Omega}(\cdot,z_j)+\varphi_1)$ is a plurisubharmonic function on $M$, by definition, we know that $2\sum_{j=1}^mp_jG_{\Omega}(\cdot,z_j)+\varphi_1$ is a subharmonic function on $\Omega$. By the construction of $g\in \mathcal{O}_{\Omega}$ , we have $2u(z):=2\sum_{j=1}^mp_jG_{\Omega}(\cdot,z_j)+\varphi_1-2\log|g|$ is a subharmonic function on $\Omega$ which satisfies $v(dd^cu,z)\in[0,1)$ for any $z\in Z_{\Omega}$. Then it follows from Lemma \ref{equiv of multiplier ideal sheaf} that we know $\mathcal{I}(\varphi+\psi)_{(z_j,y)}=\mathcal{I}\big(\pi_1^{*}(2\log |g|)+\pi_2^{*}(\varphi_2)\big)_{(z_j,y)}$ for any $(z_j,y)\in Z_0$. It follows from Theorem \ref{GBY-APP1} that we know the sufficiency part of  Theorem \ref{app-L2 equality} holds.

Next we prove the necessity part of Theorem \ref{app-L2 equality}.

Denote that $\tilde{\psi}:=\pi_1^*\big(2\sum_{j=1}^mp_jG_{\Omega}(\cdot,z_j)\big)$, $\tilde{\varphi}_1:=\pi_1^*(\varphi_1)+\psi-\tilde{\psi}$, and $\tilde{\varphi}:=\tilde{\varphi}_1+\pi_2^*(\varphi_2)$.
It follows from Lemma \ref{optimal extension} that there exists a holomorphic $(n,0)$ form $\tilde{F}_1\neq 0$ on $\tilde {M}$ such that $\big(\tilde{F}_1-f,(z_j,y)\big)\in\big(\mathcal{O}(K_{\tilde M})\otimes\mathcal{I}(\pi_1^*(2\log |g|)+\pi_2^*(\varphi_2))\big)_{(z_j,y)}$ for any $(z_j,y)\in Z_0$ and
		\begin{flalign}
			\begin{split}
				&\int_{\tilde M}|\tilde{F}_1|^2e^{-\tilde{\varphi}}c\left(-\tilde{\psi}\right)\\
				\leq&\left(\int_0^{+\infty}c(s)e^{-s}ds\right)
\sum_{j=1}^{m}\frac{2\pi|a_j|^2}{p_j|d_j|^2}\int_Y|F_j|^2e^{-\varphi_2-\tilde{\psi}_2(z_j,w)}.
			\end{split}
		\end{flalign}	

As $\left(\int_0^{+\infty}c(s)e^{-s}ds\right)
 \sum_{j=1}^{m}\frac{2\pi|a_j|^2}{p_j|d_j|^2}\int_Y|F_j|^2e^{-\varphi_2-\tilde{\psi}_2(z_j,w)}=\inf\big\{$ $ \int_{\tilde M}|\tilde{F}|^2e^{-\varphi}c(-\psi):\tilde{F}$ is a holomorphic $(n,0)$ form on $\tilde{M}$ such that $\big(\tilde{F}-f,(z_j,y)\big)\in\big(\mathcal{O}(K_{\tilde{M}})\otimes\mathcal{I}(\pi_1^*(2\log |g|)+\pi_2^*(\varphi_2))\big)_{(z_j,y)}$ for any $(z_j,y)\in Z_0\big\}$ holds, we know
 $$\int_{\tilde{M}}|\tilde{F}_1|^2e^{-\varphi}c\left(-\psi\right)
				\geq\left(\int_0^{+\infty}c(s)e^{-s}ds\right)
\sum_{j=1}^{m}\frac{2\pi|a_j|^2}{p_j|d_j|^2}\int_Y|F_j|^2e^{-\varphi_2-\tilde{\psi}_2(z_j,w)}.$$
Note that $c(t)e^{-t}$ is decreasing on $(0,+\infty)$, $\psi\leq \tilde{\psi}$. Hence we must have
$$\int_{\tilde{M}}|\tilde{F}_1|^2e^{-\varphi}c\left(-\psi\right)
=\int_{\tilde M}|\tilde{F}_1|^2e^{-\tilde{\varphi}}c\left(-\tilde{\psi}\right).$$
It follows from Lemma \ref{decreasing property of l} that we have $N\equiv 0$ and $\psi=\pi_1^*\big(2\sum_{j=1}^mp_jG_{\Omega}(\cdot,z_j)\big)$.

It follows from Lemma \ref{optimal extension} ($\tilde{M}\sim M$) that there exists a holomorphic $(n,0)$ form $\hat{F}\neq 0$ on $M$ such that $\big(\hat{F}-f,(z_j,y)\big)\in\big(\mathcal{O}(K_{ M})\otimes\mathcal{I}(\pi_1^*(2\log |g|)+\pi_2^*(\varphi_2))\big)_{(z_j,y)}$ for any $(z_j,y)\in Z_0$ and
		\begin{flalign}
			\begin{split}
				&\int_{ M}|\hat{F}|^2e^{-\varphi}c\left(-\psi\right)\\
				\leq&\left(\int_0^{+\infty}c(s)e^{-s}ds\right)
\sum_{j=1}^{m}\frac{2\pi|a_j|^2}{p_j|d_j|^2}\int_Y|F_j|^2e^{-\varphi_2-\tilde{\psi}_2(z_j,w)}.
			\end{split}
		\end{flalign}	
As $\left(\int_0^{+\infty}c(s)e^{-s}ds\right)
 \sum_{j=1}^{m}\frac{2\pi|a_j|^2}{p_j|d_j|^2}\int_Y|F_j|^2e^{-\varphi_2-\tilde{\psi}_2(z_j,w)}=\inf\big\{$ $ \int_{\tilde M}|\tilde{F}|^2e^{-\varphi}c(-\psi):\tilde{F}$ is a holomorphic $(n,0)$ form on $\tilde{M}$ such that $\big(\tilde{F}-f,(z_j,y)\big)\in\big(\mathcal{O}(K_{\tilde{M}})\otimes\mathcal{I}(\pi_1^*(2\log |g|)+\pi_2^*(\varphi_2))\big)_{(z_j,y)}$ for any $(z_j,y)\in Z_0\big\}$ holds, we know
	\begin{flalign*}
			\begin{split}
&\left(\int_0^{+\infty}c(s)e^{-s}ds\right)
\sum_{j=1}^{m}\frac{2\pi|a_j|^2}{p_j|d_j|^2}\int_Y|F_j|^2e^{-\varphi_2-\tilde{\psi}_2(z_j,w)}\\
\leq&\int_{\tilde M}|\hat{F}|^2e^{-\varphi}c\left(-\psi\right)\\
				\leq&\int_{ M}|\hat{F}|^2e^{-\varphi}c\left(-\psi\right)\\
				\leq&\left(\int_0^{+\infty}c(s)e^{-s}ds\right)
\sum_{j=1}^{m}\frac{2\pi|a_j|^2}{p_j|d_j|^2}\int_Y|F_j|^2e^{-\varphi_2-\tilde{\psi}_2(z_j,w)}.
			\end{split}
		\end{flalign*}	
As $\hat{F}\neq 0$, we have $\tilde M=M$.

Now $\psi+\pi_1^*(\varphi_1)=\pi_1^*(2\sum_{j=1}^mp_jG_{\Omega}(\cdot,z_j)+\varphi_1)$ is a plurisubharmonic function on $M$. By definition, we know that $2\sum_{j=1}^mp_jG_{\Omega}(\cdot,z_j)+\varphi_1$ is a subharmonic function on $\Omega$. By the construction of $g\in \mathcal{O}_{\Omega}$, we have $2u(z):=2\sum_{j=1}^mp_jG_{\Omega}(\cdot,z_j)+\varphi_1-2\log|g|$ is a subharmonic function on $\Omega$  and satisfies $v(dd^cu,z)\in[0,1)$ for any $z\in Z_{\Omega}$. Then it follows from Lemma \ref{equiv of multiplier ideal sheaf} that we know $\mathcal{I}(\varphi+\psi)_{(z_j,y)}=\mathcal{I}\big(\pi_1^{*}(2\log |g|)+\pi_2^{*}(\varphi_2)\big)_{(z_j,y)}$ for any $(z_j,y)\in Z_0$.

Then it follows from the necessity part of Theorem \ref{GBY-APP1} that we know Theorem \ref{app-L2 equality} holds.

\end{proof}

\subsection{Proof of Theorem \ref{app-L2 inequality}}

\begin{proof}
It follows from Lemma \ref{optimal extension} that we know there exists a holomorphic $(n,0)$ form $F$ on $M$ such that $\big(F-f,(z_j,y)\big)\in(\mathcal{O}(K_{\tilde M})\otimes\mathcal{I}\big(\pi_1^*(2\log |g|)+\pi_2^*(\varphi_2)\big)_{(z_j,y)}$ for any $(z_j,y)\in Z_0$ and
		\begin{equation}\nonumber
			\int_{\tilde M}|F|^2e^{-\varphi}c(-\psi)\leq
\left(\int_0^{+\infty}c(s)e^{-s}ds\right)
 \sum_{j=1}^{+\infty}\frac{2\pi|a_j|^2}{(k_j+1)|d_j|^2}\int_Y|F_j|^2e^{-\varphi_2-\tilde{\psi}_2(z_j,w)}.
		\end{equation}

Now it suffices to show that the equality $\left(\int_0^{+\infty}c(s)e^{-s}ds\right)
 \sum_{j=1}^{+\infty}\frac{2\pi|a_j|^2}{(k_j+1)|d_j|^2}\int_Y|F_j|^2e^{-\varphi_2-\tilde{\psi}_2(z_j,w)}
 =\inf\big\{ \int_{\tilde M}|\tilde{F}|^2e^{-\varphi}c(-\psi):\tilde{F}$ is a holomorphic $(n,0)$ form on $\tilde M$ such that $\big(\tilde{F}-f,(z_j,y)\big)\in\big(\mathcal{O}(K_{\tilde M})\otimes\mathcal{I}(\pi_1^*(2\log |g|)+\pi_2^*(\varphi_2))\big)_{(z_j,y)}$ for any $(z_j,y)\in Z_0\big\}$ can not hold. We assume that the equality holds to get a contradiction.

Denote that $\tilde{\psi}:=\pi_1^*\big(2\sum_{j=1}^m(k_j+1)G_{\Omega}(\cdot,z_j)\big)$, $\tilde{\varphi}_1:=\pi_1^*(\varphi_1)+\psi-\tilde{\psi}$, and $\tilde{\varphi}:=\tilde{\varphi}_1+\pi_2^*(\varphi_2)$.
It follows from Lemma \ref{optimal extension} that there exists a holomorphic $(n,0)$ form $\tilde{F}_1\neq 0$ on $\tilde M$ such that $\big(\tilde{F}_1-f,(z_j,y)\big)\in\big(\mathcal{O}(K_{\tilde M})\otimes\mathcal{I}(\pi_1^*(2\log |g|)+\pi_2^*(\varphi_2))\big)_{(z_j,y)}$ for any $(z_j,y)\in Z_0$ and
		\begin{flalign}
			\begin{split}
				&\int_{\tilde M}|\tilde{F}_1|^2e^{-\tilde{\varphi}}c\left(-\tilde{\psi}\right)\\
				\leq&\left(\int_0^{+\infty}c(s)e^{-s}ds\right)
\sum_{j=1}^{+\infty}\frac{2\pi|a_j|^2}{(k_j+1)|d_j|^2}\int_Y|F_j|^2e^{-\varphi_2-\tilde{\psi}_2(z_j,w)}.
			\end{split}
		\end{flalign}

As $\left(\int_0^{+\infty}c(s)e^{-s}ds\right)
 \sum_{j=1}^{+\infty}\frac{2\pi|a_j|^2}{(k_j+1)|d_j|^2}\int_Y|F_j|^2e^{-\varphi_2-\tilde{\psi}_2(z_j,w)}=\inf\big\{$ $ \int_{\tilde M}|\tilde{F}|^2e^{-\varphi}c(-\psi):\tilde{F}$ is a holomorphic $(n,0)$ form on $\tilde M$ such that $\big(\tilde{F}-f,(z_j,y)\big)\in\big(\mathcal{O}(K_{\tilde M})\otimes\mathcal{I}(\pi_1^*(2\log |g|)+\pi_2^*(\varphi_2))\big)_{(z_j,y)}$ for any $(z_j,y)\in Z_0\big\}$, we know
 $$\int_{\tilde M}|\tilde{F}_1|^2e^{-\varphi}c\left(-\psi\right)
				\geq\left(\int_0^{+\infty}c(s)e^{-s}ds\right)
\sum_{j=1}^{+\infty}\frac{2\pi|a_j|^2}{(k_j+1)|d_j|^2}\int_Y|F_j|^2e^{-\varphi_2-\tilde{\psi}_2(z_j,w)}.$$
Note that $c(t)e^{-t}$ is decreasing on $(0,+\infty)$ and $\psi\leq \tilde{\psi}$ . Hence we must have
$$\int_{\tilde M}|\tilde{F}_1|^2e^{-\varphi}c\left(-\psi\right)
=\int_{\tilde M}|\tilde{F}_1|^2e^{-\tilde{\varphi}}c\left(-\tilde{\psi}\right).$$
Hence, by Lemma \ref{decreasing property of l}, we have $N\equiv 0$ and $\psi=\pi_1^*\big(2\sum_{j=1}^{+\infty}(k_j+1)G_{\Omega}(\cdot,z_j)\big)$.

It follows from Lemma \ref{optimal extension} ($\tilde{M}\sim M$) that there exists a holomorphic $(n,0)$ form $\hat{F}\neq 0$ on $M$ such that $\big(\hat{F}-f,(z_j,y)\big)\in\big(\mathcal{O}(K_{ M})\otimes\mathcal{I}(\pi_1^*(2\log |g|)+\pi_2^*(\varphi_2))\big)_{(z_j,y)}$ for any $(z_j,y)\in Z_0$ and
		\begin{flalign}
			\begin{split}
				&\int_{ M}|\hat{F}|^2e^{-\varphi}c\left(-\psi\right)\\
				\leq&\left(\int_0^{+\infty}c(s)e^{-s}ds\right)
\sum_{j=1}^{m}\frac{2\pi|a_j|^2}{(k_j+1)|d_j|^2}\int_Y|F_j|^2e^{-\varphi_2-\tilde{\psi}_2(z_j,w)}.
			\end{split}
		\end{flalign}	
As $\left(\int_0^{+\infty}c(s)e^{-s}ds\right)
 \sum_{j=1}^{m}\frac{2\pi|a_j|^2}{p_j|d_j|^2}\int_Y|F_j|^2e^{-\varphi_2-\tilde{\psi}_2(z_j,w)}=\inf\big\{$ $ \int_{\tilde M}|\tilde{F}|^2e^{-\varphi}c(-\psi):\tilde{F}$ is a holomorphic $(n,0)$ form on $\tilde{M}$ such that $\big(\tilde{F}-f,(z_j,y)\big)\in\big(\mathcal{O}(K_{\tilde{M}})\otimes\mathcal{I}(\pi_1^*(2\log |g|)+\pi_2^*(\varphi_2))\big)_{(z_j,y)}$ for any $(z_j,y)\in Z_0\big\}$ holds, we know
	\begin{flalign*}
			\begin{split}
&\left(\int_0^{+\infty}c(s)e^{-s}ds\right)
\sum_{j=1}^{m}\frac{2\pi|a_j|^2}{(k_j+1)|d_j|^2}\int_Y|F_j|^2e^{-\varphi_2-\tilde{\psi}_2(z_j,w)}\\
\leq&\int_{\tilde M}|\hat{F}|^2e^{-\varphi}c\left(-\psi\right)\\
				\leq&\int_{ M}|\hat{F}|^2e^{-\varphi}c\left(-\psi\right)\\
				\leq&\left(\int_0^{+\infty}c(s)e^{-s}ds\right)
\sum_{j=1}^{m}\frac{2\pi|a_j|^2}{(k_j+1)|d_j|^2}\int_Y|F_j|^2e^{-\varphi_2-\tilde{\psi}_2(z_j,w)}.
			\end{split}
		\end{flalign*}	
As $\hat{F}\neq 0$, we have $\tilde M=M$.

Now $\psi+\pi_1^*(\varphi_1)=\pi_1^*(2\sum_{j=1}^{+\infty}(k_j+1)G_{\Omega}(\cdot,z_j)+\varphi_1)$ is a plurisubharmonic function on $M$, by definition, we know that $2\sum_{j=1}^m(k_j+1)G_{\Omega}(\cdot,z_j)+\varphi_1$ is a subharmonic function on $\Omega$. By the construction of $g\in \mathcal{O}_{\Omega}$, we have $2u(z):=2\sum_{j=1}^{+\infty}(k_j+1)G_{\Omega}(\cdot,z_j)+\varphi_1-2\log|g|$ is a subharmonic function on $\Omega$ and satisfies $v(dd^cu,z)\in[0,1)$ for any $z\in Z_{\Omega}$. Then it follows from Lemma \ref{equiv of multiplier ideal sheaf} that we know $\mathcal{I}(\varphi+\psi)_{(z_j,y)}=\mathcal{I}\big(\pi_1^{*}(2\log |g|)+\pi_2^{*}(\varphi_2)\big)_{(z_j,y)}$ for any $(z_j,y)\in Z_0$.

Then it follows from  Theorem \ref{GBY-APP2} that we know Theorem \ref{app-L2 inequality} holds.

\end{proof}

\section{ Proofs of Theorem \ref{thm:exten-fibra-single} and Theorem \ref{thm:exten-fibra-finite} and Theorem \ref{thm:exten-fibra-infinite}}
\label{section of multi extension}

In this section, we prove Theorem \ref{thm:exten-fibra-single} and Theorem \ref{thm:exten-fibra-finite} and Theorem \ref{thm:exten-fibra-infinite}.

\subsection{Proofs of Theorem \ref{thm:exten-fibra-single} and Remark \ref{rem:exten-fibra-single}}

\begin{proof}[Proof of Theorem \ref{thm:exten-fibra-single}]
It follows from Remark \ref{remark after extension} that there exists a holomorphic $(n,0)$ form $F$ on $\tilde{M}$ satisfying that $(F-f,z)\in\left(\mathcal{O}(K_{\tilde{M}})\otimes\mathcal{I}\left(\max_{1\le j\le n_1}\left\{2p_j\pi_{1,j}^{*}(G_{\Omega_j}(\cdot,z_j))\right\}\right)\right)_{z}$ for any $z\in Z_0$ and
\begin{displaymath}
	\begin{split}
	&\int_{\tilde{M}}|F|^2e^{-\varphi}c(-\psi)\\
	\le&\left(\int_0^{+\infty}c(s)e^{-s}ds\right)\sum_{\alpha\in E}\frac{(2\pi)^{n_1}\int_Y|f_{\alpha}|^2e^{-\varphi_Y-\left(N+\pi_{1,j}^*(\sum_{1\le j\le n_1}\varphi_j)\right)(z_0,w)}}{\prod_{1\le j\le n_1}(\alpha_j+1)c_{j}(z_j)^{2\alpha_{j}+2}}.	
	\end{split}
\end{displaymath}

In the following, we prove the characterization of the holding of the equality. It follows from Theorem \ref{GBY6-exten-fibra-single} that we only need to show the necessity.

Denote $\tilde{\varphi}:=\varphi+\psi-\hat G$. It follows from Remark \ref{remark after extension} that there exists a holomorphic $(n,0)$ form $\hat{F}$ on $\tilde M$ satisfying that $(\hat{F}-F,z)\in\big(\mathcal{O}(K_{\tilde M})\otimes\mathcal{I}(\hat G)\big)_{z}$ for any $z\in Z_0$ and
\begin{equation}
	\begin{split}
	&\int_{\tilde M}|\hat{F}|^2e^{-\tilde{\varphi}}c(-\hat G)\\
\le&\left(\int_0^{+\infty}c(s)e^{-s}ds\right)\sum_{\alpha\in E}\frac{(2\pi)^{n_1}
\int_Y|f_{\alpha}|^2e^{-\varphi_{Y}-\left(N+\pi_{1,j}^*(\sum_{1\le j\le n_1}\varphi_j)\right)(z_0,w)}}{\prod_{1\le j\le n_1}(\alpha_j+1)c_{j}^{2\alpha_{j}+2}(z_j)}.	
	\end{split}
\end{equation}

When the equality $\inf\big\{\int_{\tilde{M}}|\tilde{F}|^2e^{-\varphi}c(-\psi):\tilde{F}\in H^0(\tilde{M},\mathcal{O}(K_{\tilde{M}}))\,\&\, (\tilde{F}-f,z)\in(\mathcal{O}\left(K_{\tilde{M}})\otimes\mathcal{I}\left(\max_{1\le j\le n_1}\left\{2p_j\pi_{1,j}^{*}(G_{\Omega_j}(\cdot,z_j))\right\}\right)\right)_{z}$ for any $z\in Z_0\big\}=\left(\int_0^{+\infty}c(s)e^{-s}ds\right)\times\sum_{\alpha\in E}\frac{(2\pi)^{n_1}\int_Y|f_{\alpha}|^2e^{-\varphi_Y-\left(N+\pi_{1,j}^*(\sum_{1\le j\le n_1}\varphi_j)\right)(z_0,w)}}{\prod_{1\le j\le n_1}(\alpha_j+1)c_{j}(z_j)^{2\alpha_{j}+2}}$ holds, we have
$$\int_{\tilde{M}}|\hat{F}|^2e^{-\varphi}c(-\psi)\ge \left(\int_0^{+\infty}c(s)e^{-s}ds\right)\sum_{\alpha\in E}\frac{(2\pi)^{n_1}
\int_Y|f_{\alpha}|^2e^{-\varphi_{Y}-\left(N+\pi_{1,j}^*(\sum_{1\le j\le n_1}\varphi_j)\right)(z_0,w)}}{\prod_{1\le j\le n_1}(\alpha_j+1)c_{j}^{2\alpha_{j}+2}(z_j)}.$$

Note that $c(t)e^{-t}$ is decreasing on $(0,+\infty)$, $\psi\leq \hat G$. Hence we must have
$$\int_{\tilde{M}}|\hat{F}|^2e^{-\varphi}c\left(-\psi\right)
=\int_{\tilde M}|\hat{F}|^2e^{-\tilde{\varphi}}c\left(-\hat G\right).$$
It follows from Lemma \ref{decreasing property of l} that we have $N\equiv 0$ and $\psi=\max_{1\le j\le n_1}\left\{2p_j\pi_{1,j}^{*}(G_{\Omega_j}(\cdot,z_j))\right\}$.

As $N\equiv 0$, we know $\varphi_j$ is a subharmonic function on $\Omega_j$ satisfying $\varphi_j(z_j)>-\infty$ for any $1\le j\le n_1$ and
it follows from Lemma \ref{l:phi1+phi2} that $\mathcal{I}(\varphi+\psi)_{(z_j,y)}=\mathcal{I}\big(\hat G+\pi_2^*(\varphi_2)\big)_{(z_j,y)}$ for any $(z_j,y)\in Z_0$. Note that $\psi=\max_{1\le j\le n_1}\big\{2p_j\pi_{1,j}^{*}( G_{\Omega_j}(\cdot,z_j))\big\}$.
Then It follows from the necessity part of Theorem \ref{GBY6-exten-fibra-single} that Theorem \ref{thm:exten-fibra-single} holds.
\end{proof}

\begin{proof}[Proof of Remark \ref{rem:exten-fibra-single}]
As $(f_{\alpha},y)\in\big(\mathcal{O}(K_Y)\otimes\mathcal{I}(\varphi_Y)\big)_y$ for any $y\in Y$ and $\alpha\in \tilde E\backslash E$, it follows from Lemma \ref{p:exten-fibra} that there exists a holomorphic $(n,0)$ form $F$ on $\tilde M$ satisfying that $(F-f,(z,y))\in\left(\mathcal{O}(K_{\tilde M})\otimes\mathcal{I}\left(\hat G+\pi_2^*(\varphi_Y)\right)\right)_{(z,y)}$ for any $(z,y)\in Z_0$ and
\begin{equation}\nonumber
	\begin{split}
	&\int_{\tilde{M}}|F|^2e^{-\varphi}c(-\psi)\\
	\le&\left(\int_0^{+\infty}c(s)e^{-s}ds\right)\sum_{\alpha\in E}\frac{(2\pi)^{n_1}\int_Y|f_{\alpha}|^2e^{-\varphi_Y-\left(N+\pi_{1,j}^*(\sum_{1\le j\le n_1}\varphi_j)\right)(z_0,w)}}{\prod_{1\le j\le n_1}(\alpha_j+1)c_{j}^{2\alpha_{j}+2}(z_j)}.	
	\end{split}
\end{equation}

In the following, we prove the characterization of the holding of the equality.

When $N\equiv 0$, we know $\varphi_j$ is a subharmonic function on $\Omega_j$ satisfying $\varphi_j(z_j)>-\infty$ for any $1\le j\le n_1$ and
it follows from Lemma \ref{l:phi1+phi2} that $\mathcal{I}(\varphi+\psi)_{(z_j,y)}=\mathcal{I}\big(\hat G+\pi_2^*(\varphi_2)\big)_{(z_j,y)}$ for any $(z_j,y)\in Z_0$. Note that $\psi=\max_{1\le j\le n_1}\big\{2p_j\pi_{1,j}^{*}(G_{\Omega_j}(\cdot,z_j))\big\}$.
It follows from Remark \ref{GBY6-exten-fibra-single-rem} that we have the sufficiency of Remark \ref{rem:exten-fibra-single}.

Next we show the necessity of Remark \ref{rem:exten-fibra-single}.

Denote $\tilde{\varphi}:=\varphi+\psi-\hat G$. It follows from Remark \ref{remark after extension} that there exists a holomorphic $(n,0)$ form $\hat{F}$ on $\tilde M$ satisfying that $(\hat{F}-F,z)\in\big(\mathcal{O}(K_{\tilde M})\otimes\mathcal{I}(\hat G+\pi_2^*(\varphi_Y))\big)_{z}$ for any $z\in Z_0$ and
\begin{equation}
	\begin{split}\label{rem-extension-single 1}
	&\int_{\tilde M}|\hat{F}|^2e^{-\tilde{\varphi}}c(-\hat G)\\
\le&\left(\int_0^{+\infty}c(s)e^{-s}ds\right)\sum_{\alpha\in E}\frac{(2\pi)^{n_1}
\int_Y|f_{\alpha}|^2e^{-\varphi_{Y}-\left(N+\pi_{1,j}^*(\sum_{1\le j\le n_1}\varphi_j)\right)(z_0,w)}}{\prod_{1\le j\le n_1}(\alpha_j+1)c_{j}^{2\alpha_{j}+2}(z_j)}.	
	\end{split}
\end{equation}

 Now we know the equality $\inf\big\{\int_{\tilde{M}}|\tilde{F}|^2e^{-\varphi}c(-\psi):\tilde{F}\in H^0(\tilde{M},\mathcal{O}(K_{\tilde{M}}))\,\&\, \big(\tilde{F}-f,(z,y)\big)\in\big(\mathcal{O}(K_{\tilde{M}})\otimes\mathcal{I}(\hat G+\pi_2^*(\varphi_Y))\big)_{(z,y)}$ for any $(z,y)\in Z_0\big\}=\left(\int_0^{+\infty}c(s)e^{-s}ds\right)\times\sum_{\alpha\in E}\frac{(2\pi)^{n_1}\int_Y|f_{\alpha}|^2e^{-\varphi_Y-\left(N+\pi_{1,j}^*(\sum_{1\le j\le n_1}\varphi_j)\right)(z_0,w)}}{\prod_{1\le j\le n_1}(\alpha_j+1)c_{j}(z_j)^{2\alpha_{j}+2}}$ holds, then we have

\begin{equation}
 \label{rem-extension-single 2}
\begin{split}
&\int_{\tilde{M}}|\hat F|^2e^{-\varphi}c(-\psi)\\
\ge& \left(\int_0^{+\infty}c(s)e^{-s}ds\right)\sum_{\alpha\in E}\frac{(2\pi)^{n_1}
\int_Y|f_{\alpha}|^2e^{-\varphi_{Y}-\left(N+\pi_{1,j}^*(\sum_{1\le j\le n_1}\varphi_j)\right)(z_0,w)}}{\prod_{1\le j\le n_1}(\alpha_j+1)c_{j}^{2\alpha_{j}+2}(z_j)}.
\end{split}
\end{equation}

It follows from \eqref{rem-extension-single 1} and \eqref{rem-extension-single 2} that we  have
$$\int_{\tilde{M}}|\hat F|^2e^{-\varphi}c\left(-\psi\right)
=\int_{\tilde M}|\hat F|^2e^{-\tilde{\varphi}}c\left(-\hat G\right).$$
Note that $c(t)e^{-t}$ is decreasing on $(0,+\infty)$, $\psi\leq \hat G$.
It follows from Lemma \ref{decreasing property of l} that we have $N\equiv 0$ and $\psi=\max_{1\le j\le n_1}\left\{2p_j\pi_{1,j}^{*}(G_{\Omega_j}(\cdot,z_j))\right\}$.

As $N\equiv 0$, we know $\varphi_j$ is a subharmonic function on $\Omega_j$ satisfying $\varphi_j(z_j)>-\infty$ for any $1\le j\le n_1$ and
it follows from Lemma \ref{l:phi1+phi2} that $\mathcal{I}(\varphi+\psi)_{(z_j,y)}=\mathcal{I}\big(\hat G+\pi_2^*(\varphi_Y)\big)_{(z_j,y)}$ for any $(z_j,y)\in Z_0$. Note that $\psi=\max_{1\le j\le n_1}\big\{2p_j\pi_{1,j}^{*}(G_{\Omega_j}(\cdot,z_j))\big\}$.
Then It follows from Remark \ref{GBY6-exten-fibra-single-rem} that Remark \ref{rem:exten-fibra-single} holds.
\end{proof}

\subsection{Proofs of Theorem \ref{thm:exten-fibra-finite} and Remark \ref{rem:exten-fibra-finite}}

\begin{proof}[Proof of Theorem \ref{thm:exten-fibra-finite}]
It follows from Remark \ref{remark after extension} that there exists a holomorphic $(n,0)$ form $F$ on $\tilde{M}$ satisfying that $(F-f,z)\in\left(\mathcal{O}(K_{\tilde{M}})\otimes\mathcal{I}\left(\max_{1\le j\le n_1}\left\{2\sum_{1\le k\le m_j}p_{j,k}\pi_{1,j}^{*}(G_{\Omega_j}(\cdot,z_{j,k}))\right\}\right)\right)_{z}$ for any $z\in Z_0$ and
\begin{displaymath}
	\begin{split}
	&\int_{\tilde{M}}|F|^2e^{-\varphi}c(-\psi)\\
	\le&\left(\int_0^{+\infty}c(s)e^{-s}ds\right)\sum_{\beta\in \tilde{I}_1}\sum_{\alpha\in E_{\beta}}\frac{(2\pi)^{n_1}\int_Y|f_{\alpha,\beta}|^2e^{-\varphi_Y-\left(N+\pi_{1,j}^*(\sum_{1\le j\le n_1}\varphi_j)\right)(z_{\beta},w)}}{\prod_{1\le j\le n_1}(\alpha_j+1)c_{j,\beta_j}^{2\alpha_{j}+2}}.	
	\end{split}
\end{displaymath}
	
In the following, we prove the characterization of the holding of the equality. It follows from Theorem \ref{GBY6:exten-fibra-finite} that we only need to show the necessity.

Denote $\tilde{\varphi}:=\varphi+\psi-\hat G$. It follows from Remark \ref{remark after extension} that there exists a holomorphic $(n,0)$ form $\hat{F}$ on $\tilde M$ satisfying that $(\hat{F}-F,z)\in\big(\mathcal{O}(K_{\tilde M})\otimes\mathcal{I}(\hat G)\big)_{z}$ for any $z\in Z_0$ and
\begin{equation}
	\begin{split}
	&\int_{\tilde M}|\hat{F}|^2e^{-\tilde{\varphi}}c(-\hat G)\\
\le&\left(\int_0^{+\infty}c(s)e^{-s}ds\right)\sum_{\beta\in \tilde{I}_1}\sum_{\alpha\in E_{\beta}}\frac{(2\pi)^{n_1}\int_Y|f_{\alpha,\beta}|^2e^{-\varphi_Y-\left(N+\pi_{1,j}^*(\sum_{1\le j\le n_1}\varphi_j)\right)(z_{\beta},w)}}{\prod_{1\le j\le n_1}(\alpha_j+1)c_{j,\beta_j}^{2\alpha_{j}+2}}.
	\end{split}
\end{equation}

When the equality $\inf\bigg\{\int_{\tilde{M}}|\tilde{F}|^2e^{-\varphi}c(-\psi):\tilde{F}\in H^0(\tilde{M},\mathcal{O}(K_{\tilde{M}}))\,\&\, (\tilde{F}-f,z)\in\left(\mathcal{O}(K_{\tilde{M}})\otimes\mathcal{I}\left(\max_{1\le j\le n_1}\left\{2\sum_{1\le k\le m_j}p_{j,k}\pi_{1,j}^{*}(G_{\Omega_j}(\cdot,z_{j,k}))\right\}\right)\right)_{z}$ for any $z\in Z_0\bigg\}=\left(\int_0^{+\infty}c(s)e^{-s}ds\right)\sum_{\beta\in \tilde{I}_1}\sum_{\alpha\in E_{\beta}}\frac{(2\pi)^{n_1}\int_Y|f_{\alpha,\beta}|^2e^{-\varphi_Y-\left(N+\pi_{1,j}^*(\sum_{1\le j\le n_1}\varphi_j)\right)(z_{\beta},w)}}{\prod_{1\le j\le n_1}(\alpha_j+1)c_{j,\beta_j}^{2\alpha_{j}+2}}$ holds,
we have
$$\int_{\tilde{M}}|\hat{F}|^2e^{-\varphi}c(-\psi)\ge \left(\int_0^{+\infty}c(s)e^{-s}ds\right)\sum_{\beta\in \tilde{I}_1}\sum_{\alpha\in E_{\beta}}\frac{(2\pi)^{n_1}\int_Y|f_{\alpha,\beta}|^2e^{-\varphi_Y-\left(N+\pi_{1,j}^*(\sum_{1\le j\le n_1}\varphi_j)\right)(z_{\beta},w)}}{\prod_{1\le j\le n_1}(\alpha_j+1)c_{j,\beta_j}^{2\alpha_{j}+2}}.$$

Note that $c(t)e^{-t}$ is decreasing on $(0,+\infty)$, $\psi\leq \hat G$. Hence we must have
$$\int_{\tilde{M}}|\hat{F}|^2e^{-\varphi}c\left(-\psi\right)
=\int_{\tilde M}|\hat{F}|^2e^{-\tilde{\varphi}}c\left(-\hat G\right).$$
It follows from Lemma \ref{decreasing property of l} that we have $N\equiv 0$ and $\psi=\max_{1\le j\le n_1}\left\{2\sum_{1\le k\le m_j}p_{j,k}\pi_{1,j}^{*}(G_{\Omega_j}(\cdot,z_{j,k}))\right\}$.

As $N\equiv 0$, we know $\varphi_j$ is a subharmonic function on $\Omega_j$ satisfying $\varphi_j(z_j)>-\infty$ for any $1\le j\le n_1$ and
it follows from Lemma \ref{l:phi1+phi2} that $\mathcal{I}(\varphi+\psi)_{(z_j,y)}=\mathcal{I}\big(\hat G+\pi_2^*(\varphi_Y)\big)_{(z_j,y)}$ for any $(z_j,y)\in Z_0$. Note that $\psi=\max_{1\le j\le n_1}\left\{2\sum_{1\le k\le m_j}p_{j,k}\pi_{1,j}^{*}(G_{\Omega_j}(\cdot,z_{j,k}))\right\}$. Then It follows from the necessity part of Theorem \ref{GBY6:exten-fibra-finite} that Theorem \ref{thm:exten-fibra-finite} holds.
\end{proof}

\begin{proof}[Proof of Remark \ref{rem:exten-fibra-finite}]
As $(f_{\alpha,\beta},y)\in\big(\mathcal{O}(K_Y)\otimes\mathcal{I}(\varphi_Y)\big)_y$ holds for any $y\in Y$,  $\alpha\in \tilde E_{\beta}\backslash E_{\beta}$ and $\beta\in \tilde{I}_1$, it follows from Lemma \ref{p:exten-fibra} that there exists a holomorphic $(n,0)$ form $F$ on $M$ satisfying that $\big(F-f,(z,y)\big)\in\left(\mathcal{O}(K_{\tilde M})\otimes\mathcal{I}\left(\hat G+\pi_2^*(\varphi_Y)\right)\right)_{(z,y)}$ for any $(z,y)\in Z_0$ and
\begin{equation}\nonumber
	\begin{split}
	&\int_{\tilde{M}}|F|^2e^{-\varphi}c(-\psi)\\
	\le&\left(\int_0^{+\infty}c(s)e^{-s}ds\right)\sum_{\beta\in \tilde{I}_1}\sum_{\alpha\in E_{\beta}}\frac{(2\pi)^{n_1}\int_Y|f_{\alpha,\beta}|^2e^{-\varphi_Y-\left(N+\pi_{1,j}^*(\sum_{1\le j\le n_1}\varphi_j)\right)(z_{\beta},w)}}{\prod_{1\le j\le n_1}(\alpha_j+1)c_{j,\beta_j}^{2\alpha_{j}+2}}.		
	\end{split}
\end{equation}

In the following, we prove the characterization of the holding of the equality.

When $N\equiv 0$, we know $\varphi_j$ is a subharmonic function on $\Omega_j$ satisfying $\varphi_j(z_j)>-\infty$ for any $1\le j\le n_1$ and
it follows from Lemma \ref{l:phi1+phi2} that $\mathcal{I}(\varphi+\psi)_{(z_j,y)}=\mathcal{I}\big(\hat G+\pi_2^*(\varphi_Y)\big)_{(z_j,y)}$ for any $(z_j,y)\in Z_0$. Note that $\psi=\max_{1\le j\le n_1}\{2p_j\pi_{1,j}^{*}(G_{\Omega_j}(\cdot,z_j))\}$.
It follows from Remark \ref{GBY6:exten-fibra-finite-rem} that we have the sufficiency of Remark \ref{rem:exten-fibra-finite}.

Next we show the necessity of Remark \ref{rem:exten-fibra-finite}.

Denote $\tilde{\varphi}:=\varphi+\psi-\hat G$. It follows from Remark \ref{remark after extension} that there exists a holomorphic $(n,0)$ form $\hat{F}$ on $\tilde M$ satisfying that $(\hat{F}-F,z)\in\big(\mathcal{O}(K_{\tilde M})\otimes\mathcal{I}(\hat G+\pi_2^*(\varphi_Y))\big)_{(z,y)}$ for any $(z,y)\in Z_0$ and
\begin{equation}\label{rem-extension-finite 1}
	\begin{split}
	&\int_{\tilde M}|\hat{F}|^2e^{-\tilde{\varphi}}c(-\hat G)\\
\le&\left(\int_0^{+\infty}c(s)e^{-s}ds\right)\sum_{\beta\in \tilde{I}_1}\sum_{\alpha\in E_{\beta}}\frac{(2\pi)^{n_1}\int_Y|f_{\alpha,\beta}|^2e^{-\varphi_Y-\left(N+\pi_{1,j}^*(\sum_{1\le j\le n_1}\varphi_j)\right)(z_{\beta},w)}}{\prod_{1\le j\le n_1}(\alpha_j+1)c_{j,\beta_j}^{2\alpha_{j}+2}}.
	\end{split}
\end{equation}

 Now we know the equality $\inf\big\{\int_{\tilde{M}}|\tilde{F}|^2e^{-\varphi}c(-\psi):\tilde{F}\in H^0(\tilde{M},\mathcal{O}(K_{\tilde{M}}))\,\&\, (\tilde{F}-f,(z,y)\in(\mathcal{O}\left(K_{\tilde{M}})\otimes\mathcal{I}\left(\hat G+\pi_2^*(\varphi_Y)\right)\right)_{(z,y)}$ for any $(z,y)\in Z_0\big\}=\left(\int_0^{+\infty}c(s)e^{-s}ds\right)$\\$\sum_{\beta\in \tilde{I}_1}\sum_{\alpha\in E_{\beta}}\frac{(2\pi)^{n_1}\int_Y|f_{\alpha,\beta}|^2e^{-\varphi_Y-\left(N+\pi_{1,j}^*(\sum_{1\le j\le n_1}\varphi_j)\right)(z_{\beta},w)}}{\prod_{1\le j\le n_1}(\alpha_j+1)c_{j,\beta_j}^{2\alpha_{j}+2}}$ holds, then we have
\begin{equation}
\label{rem-extension-finite 2}
\begin{split}
&\int_{\tilde{M}}|\hat F|^2e^{-\varphi}c(-\psi)\\
\ge& \left(\int_0^{+\infty}c(s)e^{-s}ds\right)\sum_{\beta\in \tilde{I}_1}\sum_{\alpha\in E_{\beta}}\frac{(2\pi)^{n_1}\int_Y|f_{\alpha,\beta}|^2e^{-\varphi_Y-\left(N+\pi_{1,j}^*(\sum_{1\le j\le n_1}\varphi_j)\right)(z_{\beta},w)}}{\prod_{1\le j\le n_1}(\alpha_j+1)c_{j,\beta_j}^{2\alpha_{j}+2}}.
\end{split}
\end{equation}

Note that $c(t)e^{-t}$ is decreasing on $(0,+\infty)$, $\psi\leq \hat G$.
It follows from \eqref{rem-extension-finite 1} and \eqref{rem-extension-finite 2} that we  have
$$\int_{\tilde{M}}|\hat F|^2e^{-\varphi}c\left(-\psi\right)
=\int_{\tilde M}|\hat F|^2e^{-\varphi-N}c\left(-\hat G\right).$$
It follows from Lemma \ref{decreasing property of l} that we have $N\equiv 0$ and $\psi=\max_{1\le j\le n_1}\left\{2\sum_{1\le k\le m_j}p_{j,k}\pi_{1,j}^{*}(G_{\Omega_j}(\cdot,z_{j,k}))\right\}$.

As $N\equiv 0$, we know $\varphi_j$ is a subharmonic function on $\Omega_j$ satisfying $\varphi_j(z_j)>-\infty$ for any $1\le j\le n_1$,
it follows from Lemma \ref{l:phi1+phi2} that $\mathcal{I}(\varphi+\psi)_{(z_j,y)}=\mathcal{I}\big(\hat G+\pi_2^*(\varphi_Y)\big)_{(z_j,y)}$ for any $(z_j,y)\in Z_0$. Note that $\psi=\max_{1\le j\le n_1}\left\{2\sum_{1\le k\le m_j}p_{j,k}\pi_{1,j}^{*}(G_{\Omega_j}(\cdot,z_{j,k}))\right\}$.
Then It follows from Remark \ref{GBY6:exten-fibra-finite-rem} that Remark \ref{rem:exten-fibra-finite} holds.
\end{proof}

\subsection{Proofs of Theorem \ref{thm:exten-fibra-infinite} and Remark \ref{rem:exten-fibra-infinite}}

\begin{proof}[Proof of Theorem \ref{thm:exten-fibra-infinite}]
It follows from Remark \ref{remark after extension} that there exists a holomorphic $(n,0)$ form $F$ on $\tilde{M}$ satisfying that $(F-f,z)\in\left(\mathcal{O}(K_{\tilde{M}})\otimes\mathcal{I}\left(\max_{1\le j\le n_1}\left\{2\sum_{1\le k< \tilde{m}_j}p_{j,k}\pi_{1,j}^{*}(G_{\Omega_j}(\cdot,z_{j,k}))\right\}\right)\right)_{z}$ for any $z\in Z_0$ and
\begin{displaymath}
	\begin{split}
	&\int_{\tilde{M}}|F|^2e^{-\varphi}c(-\psi)\\
	\le&\left(\int_0^{+\infty}c(s)e^{-s}ds\right)\sum_{\beta\in \tilde{I}_1}\sum_{\alpha\in E_{\beta}}\frac{(2\pi)^{n_1}\int_Y|f_{\alpha,\beta}|^2e^{-\varphi_Y-\left(N+\pi_{1,j}^*(\sum_{1\le j\le n_1}\varphi_j)\right)(z_{\beta},w)}}{\prod_{1\le j\le n_1}(\alpha_j+1)c_{j,\beta_j}^{2\alpha_{j}+2}}.	
	\end{split}
\end{displaymath}

In the following, we prove the characterization of the holding of the equality. It follows from Theorem \ref{GBY6:exten-fibra-finite} that we only need to show the necessity.

We show the equality $\inf\bigg\{\int_{\tilde{M}}|\tilde{F}|^2e^{-\varphi}c(-\psi):\tilde{F}\in H^0(\tilde{M},\mathcal{O}(K_{\tilde{M}}))\,\&\, (\tilde{F}-f,z)\in\left(\mathcal{O}(K_{\tilde{M}})\otimes\mathcal{I}\left(\max_{1\le j\le n_1}\left\{2\sum_{1\le k\le m_j}p_{j,k}\pi_{1,j}^{*}(G_{\Omega_j}(\cdot,z_{j,k}))\right\}\right)\right)_{z}$ for any $z\in Z_0\bigg\}=\left(\int_0^{+\infty}c(s)e^{-s}ds\right)\sum_{\beta\in \tilde{I}_1}\sum_{\alpha\in E_{\beta}}\frac{(2\pi)^{n_1}\int_Y|f_{\alpha,\beta}|^2e^{-\varphi_Y-\left(N+\pi_{1,j}^*(\sum_{1\le j\le n_1}\varphi_j)\right)(z_{\beta},w)}}{\prod_{1\le j\le n_1}(\alpha_j+1)c_{j,\beta_j}^{2\alpha_{j}+2}}$ can not hold. We assume that the equality holds to get a contradiction.

Denote $\tilde{\varphi}:=\varphi+\psi-\hat G$. It follows from Remark \ref{remark after extension} that there exists a holomorphic $(n,0)$ form $\hat{F}$ on $\tilde M$ satisfying that $(\hat{F}-F,z)\in\big(\mathcal{O}(K_{\tilde M})\otimes\mathcal{I}(\hat G)\big)_{z}$ for any $z\in Z_0$ and
\begin{equation}
	\begin{split}
	&\int_{\tilde M}|\hat{F}|^2e^{-\tilde{\varphi}}c(-\hat G)\\
\le&\left(\int_0^{+\infty}c(s)e^{-s}ds\right)\sum_{\beta\in \tilde{I}_1}\sum_{\alpha\in E_{\beta}}\frac{(2\pi)^{n_1}\int_Y|f_{\alpha,\beta}|^2e^{-\varphi_Y-\left(N+\pi_{1,j}^*(\sum_{1\le j\le n_1}\varphi_j)\right)(z_{\beta},w)}}{\prod_{1\le j\le n_1}(\alpha_j+1)c_{j,\beta_j}^{2\alpha_{j}+2}}.
	\end{split}
\end{equation}

As the equality holds,
we have
$$\int_{\tilde{M}}|\hat{F}|^2e^{-\varphi}c(-\psi)\ge \left(\int_0^{+\infty}c(s)e^{-s}ds\right)\sum_{\beta\in \tilde{I}_1}\sum_{\alpha\in E_{\beta}}\frac{(2\pi)^{n_1}\int_Y|f_{\alpha,\beta}|^2e^{-\varphi_Y-\left(N+\pi_{1,j}^*(\sum_{1\le j\le n_1}\varphi_j)\right)(z_{\beta},w)}}{\prod_{1\le j\le n_1}(\alpha_j+1)c_{j,\beta_j}^{2\alpha_{j}+2}}.$$

Note that $c(t)e^{-t}$ is decreasing on $(0,+\infty)$, $\psi\leq \hat G$. Hence we must have
$$\int_{\tilde{M}}|\hat{F}|^2e^{-\varphi}c\left(-\psi\right)
=\int_{\tilde M}|\hat{F}|^2e^{-\tilde{\varphi}}c\left(-G\right).$$
It follows from Lemma \ref{decreasing property of l} that we have $\psi=\max_{1\le j\le n_1}\left\{2\sum_{1\le k< \tilde{m}_j}p_{j,k}\pi_{1,j}^{*}(G_{\Omega_j}(\cdot,z_{j,k}))\right\}$.

As $N\equiv 0$, we know $\varphi_j$ is a subharmonic function on $\Omega_j$ satisfying $\varphi_j(z_j)>-\infty$ for any $1\le j\le n_1$ and
it follows from Lemma \ref{l:phi1+phi2} that $\mathcal{I}(\varphi+\psi)_{(z_j,y)}=\mathcal{I}\big(G+\pi_2^*(\varphi_Y)\big)_{(z_j,y)}$ for any $(z_j,y)\in Z_0$. Note that $\psi=\max_{1\le j\le n_1}\left\{2\sum_{1\le k< \tilde{m}_j}p_{j,k}\pi_{1,j}^{*}(G_{\Omega_j}(\cdot,z_{j,k}))\right\}$.
Then It follows from Theorem \ref{GBY6:exten-fibra-infinite} that we get the contradiction and Theorem \ref{thm:exten-fibra-infinite} is proved.
\end{proof}

\begin{proof}[Proof of Remark \ref{rem:exten-fibra-infinite}]
As $(f_{\alpha,\beta},y)\in\big(\mathcal{O}(K_Y)\otimes\mathcal{I}(\varphi_Y)\big)_y$ holds for any $y\in Y$,  $\alpha\in \tilde E_{\beta}\backslash E_{\beta}$ and $\beta\in \tilde{I}_1$, it follows Lemma \ref{p:exten-fibra} that there exists a holomorphic $(n,0)$ form $F$ on $\tilde M$ satisfying that $\big(F-f,(z,y)\big)\in\left(\mathcal{O}(K_{\tilde M})\otimes\mathcal{I}\left(\hat G+\pi_2^*(\varphi_Y)\right)\right)_{(z,y)}$ for any $(z,y)\in Z_0$ and
\begin{equation}\nonumber
	\begin{split}
	&\int_{\tilde{M}}|F|^2e^{-\varphi}c(-\psi)\\
	\le&\left(\int_0^{+\infty}c(s)e^{-s}ds\right)\sum_{\beta\in \tilde{I}_1}\sum_{\alpha\in E_{\beta}}\frac{(2\pi)^{n_1}\int_Y|f_{\alpha,\beta}|^2e^{-\varphi_Y-\left(N+\pi_{1,j}^*(\sum_{1\le j\le n_1}\varphi_j)\right)(z_{\beta},w)}}{\prod_{1\le j\le n_1}(\alpha_j+1)c_{j,\beta_j}^{2\alpha_{j}+2}}.		
	\end{split}
\end{equation}

In the following, we assume that
 the equality $\inf\big\{\int_{\tilde{M}}|\tilde{F}|^2e^{-\varphi}c(-\psi):\tilde{F}\in H^0\big(\tilde{M},\mathcal{O}(K_{\tilde{M}})\big)\,\&\, \big(\tilde{F}-f,(z,y)\big)\in
 \big(\mathcal{O}(K_{\tilde{M}})\otimes\mathcal{I}(G+\pi_2^*(\varphi_Y))\big)_{(z,y)}$ for any $(z,y)\in Z_0\big\}=\left(\int_0^{+\infty}c(s)e^{-s}ds\right)\sum_{\beta\in \tilde{I}_1}\sum_{\alpha\in E_{\beta}}\frac{(2\pi)^{n_1}\int_Y|f_{\alpha,\beta}|^2e^{-\varphi_Y-\left(N+\pi_{1,j}^*(\sum_{1\le j\le n_1}\varphi_j)\right)(z_{\beta},w)}}{\prod_{1\le j\le n_1}(\alpha_j+1)c_{j,\beta_j}^{2\alpha_{j}+2}}$ holds to get a contradiction.

Denote $\tilde{\varphi}:=\varphi+\psi-\hat G$. It follows from Remark \ref{remark after extension} that there exists a holomorphic $(n,0)$ form $\hat{F}$ on $\tilde M$ satisfying that $\big(\hat{F}-F,z\big)\in\big(\mathcal{O}(K_{\tilde M})\otimes\mathcal{I}(\hat G+\pi_2^*(\varphi_Y))\big)_{(z,y)}$ for any $(z,y)\in Z_0$ and
\begin{equation}\label{rem-extension-infinite 1}
	\begin{split}
	&\int_{\tilde M}|\hat{F}|^2e^{-\tilde{\varphi}}c(-\hat G)\\
\le&\left(\int_0^{+\infty}c(s)e^{-s}ds\right)\sum_{\beta\in \tilde{I}_1}\sum_{\alpha\in E_{\beta}}\frac{(2\pi)^{n_1}\int_Y|f_{\alpha,\beta}|^2e^{-\varphi_Y-\left(N+\pi_{1,j}^*(\sum_{1\le j\le n_1}\varphi_j)\right)(z_{\beta},w)}}{\prod_{1\le j\le n_1}(\alpha_j+1)c_{j,\beta_j}^{2\alpha_{j}+2}}.
	\end{split}
\end{equation}

Since the equality holds,  we have
\begin{equation}
\label{rem-extension-infinite 2}
\begin{split}
&\int_{\tilde{M}}|\hat F|^2e^{-\varphi}c(-\psi)\\
\ge& \left(\int_0^{+\infty}c(s)e^{-s}ds\right)\sum_{\beta\in \tilde{I}_1}\sum_{\alpha\in E_{\beta}}\frac{(2\pi)^{n_1}\int_Y|f_{\alpha,\beta}|^2e^{-\varphi_Y-\left(N+\pi_{1,j}^*(\sum_{1\le j\le n_1}\varphi_j)\right)(z_{\beta},w)}}{\prod_{1\le j\le n_1}(\alpha_j+1)c_{j,\beta_j}^{2\alpha_{j}+2}}.
\end{split}
\end{equation}

Note that $c(t)e^{-t}$ is decreasing on $(0,+\infty)$, $\psi\leq \hat G$. It follows from \eqref{rem-extension-infinite 1} and \eqref{rem-extension-infinite 2} that we  have
$$\int_{\tilde{M}}|F|^2e^{-\varphi}c\left(-\psi\right)
=\int_{\tilde M}|F|^2e^{-\varphi-N}c\left(-\hat G\right).$$

It follows from Lemma \ref{decreasing property of l} that we have $N\equiv 0$ and $\psi=\max_{1\le j\le n_1}\left\{2\sum_{1\le k< \tilde{m}_j}p_{j,k}\pi_{1,j}^{*}(G_{\Omega_j}(\cdot,z_{j,k}))\right\}$.

As $N\equiv 0$, we know $\varphi_j$ is a subharmonic function on $\Omega_j$ satisfying $\varphi_j(z_j)>-\infty$ for any $1\le j\le n_1$ and
it follows from Lemma \ref{l:phi1+phi2} that $\mathcal{I}(\varphi+\psi)_{(z_j,y)}=\mathcal{I}\big(\hat G+\pi_2^*(\varphi_2)\big)_{(z_j,y)}$ for any $(z_j,y)\in Z_0$. Note that $\psi=\max_{1\le j\le n_1}\left\{2\sum_{1\le k< \tilde{m}_j}p_{j,k}\pi_{1,j}^{*}(G_{\Omega_j}(\cdot,z_{j,k}))\right\}$.
Then It follows from Remark \ref{GBY6:exten-fibra-infinite-rem} that Remark \ref{rem:exten-fibra-infinite} holds.
\end{proof}


\vspace{.1in} {\em Acknowledgements}.\
The second author and the third author were supported by National Key R\&D Program of China 2021YFA1003100.
The second author was supported by NSFC-11825101, NSFC-11522101 and NSFC-11431013. The third author was supported by China Postdoctoral Science Foundation 2022T150687.

\bibliographystyle{references}
\bibliography{xbib}

\end{document}